\documentclass{memo-l}

\usepackage{amsmath, amssymb}
\usepackage{tikz}
\usepackage{enumitem}
\makeatletter
\def\l@section{\@tocline{1}{0pt}{1pc}{}{}}
\def\l@subsection{\@tocline{2}{0pt}{1pc}{4.6em}{}}
\def\l@subsubsection{\@tocline{3}{0pt}{1pc}{7.6em}{}}
\renewcommand{\tocsection}[3]{%
  \indentlabel{\@ifnotempty{#2}{\makebox[2.3em][l]{%
    \ignorespaces#1 #2\hfill}}}#3}
\renewcommand{\tocsubsection}[3]{%
  \indentlabel{\@ifnotempty{#2}{\hspace*{2.3em}\makebox[2.3em][l]{%
    \ignorespaces#1 #2\hfill}}}#3}
\renewcommand{\tocsubsubsection}[3]{%
  \indentlabel{\@ifnotempty{#2}{\hspace*{4.6em}\makebox[3em][l]{%
    \ignorespaces#1 #2\hfill}}}#3}
\makeatother

\newcounter{my_enumerate_counter}
\newcommand{\pushcounter}{\setcounter{my_enumerate_counter}{\value{enumi}}}
\newcommand{\popcounter}{\setcounter{enumi}{\value{my_enumerate_counter}}}

\DeclareMathOperator{\SF}{\mathcal{SF}}
\DeclareMathOperator{\SFn}{\mathcal{SF}_n}
\DeclareMathOperator{\SFns}{\mathcal{SF}_{ns}}

\DeclareMathOperator{\eq}{eq}
\DeclareMathOperator{\sr}{sr}
\DeclareMathOperator{\tsr}{tsr}
\DeclareMathOperator{\rr}{rr}
\newcommand{\bfD}{\mathbf D} 
\newcommand{\bfE}{\mathbf E} 
\newcommand{\bfT}{T}

\newcommand{\bbPT}{\mathbb P_{\bfT}}
\newcommand{\bbPTS}{\mathbb P_{\bfT,\Sigma}}

\newcommand{\fST}{{\mathfrak S}_{\bfT}}
\newcommand{\fSST}{{\mathfrak S}_{\Sigma,\bfT}}
\newcommand{\fTT}{{\mathfrak T}_{\Sigma,\bfT}} 
\newcommand{\fT}{{\mathfrak T}_{\bfT}}
\newcommand{\fG}{\mathfrak G}

\renewcommand{\fTT}{{\mathfrak T}_{\bfT}} 

\newcommand{\fSTf}{{\mathfrak S}_{0,\bfT}}
\newcommand{\bbPTf}{\mathbb P_{0,\bfT}}
\newcommand{\bfF}{\mathbf F} 

\DeclareMathOperator{\Ad}{Ad}
\DeclareMathOperator{\RC}{rc}
\DeclareMathOperator{\UNP}{unp}
\DeclareMathOperator{\AUNP}{aunp}
\DeclareMathOperator{\drank}{dr}

\DeclareMathOperator{\nuc}{nuc}
\DeclareMathOperator{\nucdim}{nucdim}

\DeclareMathOperator{\dKK}{d_{KK}}

\DeclareMathOperator{\id}{id}
\newcommand{\cU}{\mathcal U}
\newcommand{\cO}{\mathcal O}

\newcommand{\bbD}{{\mathbb D}}

\newcommand{\bbZ}{{\mathbb Z}}
\newcommand{\bbT}{\mathbb T}

\newcommand{\bbN}{{\mathbb N}}
\newcommand{\bbC}{\mathbb C}
\newcommand{\bbQ}{\mathbb Q}
\newcommand{\bbR}{\mathbb R}

\newcommand{\cR}{{\mathcal R}}

\newcommand{\cW}{{\mathcal W}}
\newcommand{\calL}{\mathcal L}
\newcommand{\calLp}{\mathcal L^+}
\newcommand{\cX}{{\mathcal X}}

\newcommand{\cZ}{{\mathcal Z}}

\newcommand{\cE}{{\mathcal E}}
\newcommand{\fM}{\mathfrak M}

\newcommand{\fC}{\mathfrak C}
\newcommand{\rs}{\restriction}

\newcommand{\cC}{\mathcal C}
\newcommand{\cF}{\mathcal F}

\newcommand{\cB}{\mathcal B}
\newcommand{\cK}{\mathcal K}
\newcommand{\cQ}{\mathcal Q}

\newcommand{\calD}{\mathcal D}
\newcommand{\e}{\varepsilon}
\newtheorem{thm}{Theorem}[section]
\newtheorem{theorem}[thm]{Theorem}
\newtheorem{coro}[thm]{Corollary}
\newtheorem{corollary}[thm]{Corollary}

\newtheorem{conjecture}[thm]{Conjecture}

\newtheorem{question}[thm]{Question}
\newtheorem*{question*}{Question}
\newtheorem{claim}[thm]{Claim}

\newtheorem{quest}{Question}
\newtheorem{lemma}[thm]{Lemma}

\newtheorem{prop}[thm]{Proposition}
\newtheorem{proposition}[thm]{Proposition}

\theoremstyle{definition}
\newtheorem{remark}[thm]{Remark}

\newtheorem{definition}[thm]{Definition}
\newtheorem{problem}[thm]{Problem}
\newtheorem{example}[thm]{Example}

\DeclareMathOperator{\dist}{dist}

\DeclareMathOperator{\rc}{rc}
\DeclareMathOperator{\Cu}{Cu}
\DeclareMathOperator{\Diag}{Diag}
\DeclareMathOperator{\Elem}{Elem}

\newcommand{\cP}{\mathcal P} 
 
\DeclareMathOperator{\Th}{Th}
\DeclareMathOperator{\Thminus}{Th^{--}}
\DeclareMathOperator{\ThA}{Th_\forall}
\DeclareMathOperator{\Mod}{Mod}
\DeclareMathOperator{\diag}{diag}
\newcommand{\ASA}{\prod_n M_n(\bbC)/\bigoplus_n M_n(\bbC)}

\newcommand{\aea}{$\forall\exists$-axio\-ma\-ti\-za\-ble}

\newcommand{\udt}{definable by uniform families of formulas} 
\newcommand{\udut}{definable by uniform families of existential formulas} 
\newcommand{\snus}{separable, nuclear, unital and simple}
\DeclareMathOperator{\Sym}{Sym}
\DeclareMathOperator{\sgn}{sgn}
\DeclareMathOperator{\Sp}{sp}
\newcommand{\cstar}{$\mathrm{C}^*$}
\newcommand{\cst}{\mathrm{C}^*}

\newcommand{\cS}{\mathcal S}
\DeclareMathOperator{\dom}{dom}
\newcommand{\rng}{\mbox{rng}}
\newcommand{\ccS}{\mathfrak S}
\newcommand{\ccF}{\mathfrak F} 
\newcommand{\ccW}{\mathfrak W}
\newcommand{\ccR}{\mathfrak R}
\newcommand{\Met}{\mbox{Met}}

\DeclareMathOperator{\vphiALn}{\varphi_{AL,n}}

\newcommand{\dotminus}{ 
\buildrel\textstyle\ .\over{\hbox{ 
\vrule height3pt depth0pt width0pt}{\smash-} 
}\ }

 \numberwithin{equation}{section}
 \DeclareMathOperator{\piMVN}{\pi_{MvN}}
 
 \DeclareMathOperator{\Ell}{Ell}

\renewcommand{\phi}{\varphi}

\numberwithin{section}{chapter}
\numberwithin{equation}{chapter}

\makeindex

\begin{document}

\frontmatter

\title{Model theory of \cstar-algebras}

\author[I.\ Farah et al.]{I.\ Farah, B.\ Hart, M.\ Lupini, L.\ Robert, A.\ Tikuisis, \\A.\ Vignati and W.\ Winter}

%
%
%
%
%
%

\subjclass{46L05, 46L35, 03C20, 03C98, 03E15, 03C25}
\keywords{Model theory, continuous logic, nuclear \cstar-algebras, model theoretic forcing} 
\date{July 1, 2016}

\begin{abstract} 
A number of significant properties of  \cstar-algebras 
can be expressed in continuous logic, or at least in terms of definable (in 
a model-theoretic sense)  sets. 
Certain sets, such as  the set of projections or the unitary group,  are uniformly definable across all  
\cstar-algebras. On the other hand, the 
definability of some other sets, such as the connected component of the identity in the unitary group of a unital \cstar-algebra, or the set of elements that are Cuntz--Pedersen equivalent to $0$, depends on structural properties of the \cstar-algebra in question. 
Regularity properties required in the  Elliott programme for classification of nuclear \cstar-algebras 
imply  the definability of some of these sets. In fact any known pair of \snus{} \cstar-algebras
with the same Elliott invariant can be distinguished by their first-order theory. Although parts of the Elliott invariant of a classifiable (in the technical \cstar-algebraic sense) \cstar-algebra
can be reconstructed from its model-theoretic imaginaries,  
 the information provided by the theory is largely complementary  to the information provided by the Elliott invariant. 
We prove that all standard invariants
employed to verify non-isomorphism of pairs of \cstar-algebras indistinguishable by their 
K-theoretic invariants (the divisibility properties of the Cuntz semigroup,  
the radius of comparison, and the existence of finite or infinite projections) 
are invariants
of the theory of a \cstar-algebra. 

Many of our results are stated and proved for arbitrary metric structures. 
We present a self-contained treatment of the imaginaries (most importantly, definable sets and quotients) and  
 a  self-contained description of the Henkin construction of generic \cstar-algebras and other 
 metric structures.

Our results readily provide model-theoretic reformulations 
of a  number of outstanding questions  at the heart of the structure and classification theory of nuclear \cstar-algebras. 
The existence of counterexamples to the Toms--Winter conjecture and some 
other outstanding conjectures 
can be reformulated in terms of the existence of a model of a certain theory that 
omits a certain sequence of types.  
Thus the existence of a counterexample is equivalent to the assertion that   
the set of counterexamples is generic. 
  Finding interesting examples of \cstar-algebras is in some cases reduced to finding interesting examples 
 of theories of \cstar-algebras. This follows from one of our main technical results, a proof that 
 some non-elementary (in model-theoretic sense) classes of \cstar-algebras are `\udt.'
 This was known for UHF and AF algebras, and  we extend the result to 
  \cstar-algebras that are nuclear, have nuclear dimension $\leq n$, 
 decomposition rank $\leq n$, are simple, are quasidiagonal, Popa algebras, and are tracially AF. 
   We show that some theories of \cstar-algebras do not have nuclear models. 
As an application of the model-theoretic vantage point, we give a proof that
various properties of \cstar-algebras are preserved by small perturbations in the Kadison--Kastler distance. 
\end{abstract} 

\maketitle

\tableofcontents


\mainmatter

\chapter{Introduction}
Around 1990 George Elliott proposed a bold conjecture that a certain subcategory~$\cC$ of separable nuclear \cstar-algebras could be classified by $K$-theoretic invariants $\Ell$ which themselves form a category. 
Moreover there would be an isomorphism of categories between $\cC$ represented on the top and the category of invariants.

\begin{tikzpicture}
 \matrix[row sep=1cm,column sep=2cm] {
&\node  (A1) {$A$};
& \node (A2) {$B$};
\\
&\node  (B1) {$\Ell(A)$};
& \node (B2) {$\Ell(B)$};
\\
}; 
\draw (A1) edge[->] node [above] {$\Phi$} (A2) ;
\draw (A1) edge[->]   (B1) ;
\draw (A2) edge[->]    (B2) ;
\draw (B1) edge[->] node [above] {$\Ell(\Phi)$} (B2) ;
\end{tikzpicture}

In this memoir we take a look at  the Elliott programme and nuclear  \cstar-algebras in general  through a model-theoretic lens.
Being axiomatizable, 
the category of \cstar-algebras is equivalent to a category of models of a metric theory  (\cite{FaHaSh:Model2}).  
Neither this theory nor the language in which it is expressed are uniquely determined, 
as one can freely add predicates for arbitrary definable\footnote{In this introduction 
and elsewhere the term `definable' is used in its technical model-theoretic sense.}
  sets (such as the set of all projections or the norm on $M_n(A)$) 
to the structure without affecting an equivalence of categories.
The clean way to package all definable data is as follows.
Among all categories of metric models equivalent to the category of \cstar-algebras there is a universal 
category. The metric structure in this category associated to a given \cstar-algebra $A$ is denoted $A^{\eq}$, 
and it contains information about a variety of  objects associated with $A$, 
such as the cone of completely positive maps between $A$ and finite-dimensional algebras.

In its initial form,  first stated  in \cite{elliott1993classification}, 
the Elliott conjecture  (restated in model-theoretic language)  asserted  that the map sending a
\cstar-algebra $A$ into a discrete object in $A^{\eq}$ is an equivalence  
of categories. 
This was confirmed (fifteen years before the conjecture was posed) in    \cite{Ell:On} 
in the case when $\cC$ is the class of inductive limits of finite-dimensional \cstar-algebras (known as 
\emph{AF algebras})\index{C@\cstar-algebra!approximately finite (AF)}  
and the invariant is $D_0(A)$, the abelianization of the set of projections in $A$ with a partially defined addition. 
This local semigroup was later replaced by the 
ordered group $K_0(A)$. 
Although  $K_0(A)$ is not necessarily  in $A^{\eq}$,  
its fragments, consisting of each quotient of the set of projections in $M_n(A)$, for $n\in \bbN$, are. 
Every AF algebra $A$ has two important regularity properties. 
It has \emph{real rank zero},\index{real rank zero} 
meaning that the invertible self-adjoint elements are dense  in $A_{\mathrm{sa}}$, the space of self-adjoint elements of~$A$.\index{A@$A_{\mathrm{sa}}$}
It also has \emph{stable rank one},\index{stable rank one} 
meaning that the invertible  elements are dense  in $A$. 
Both of these properties are axiomatizable. 

Inductive limits of tensor products of $\mathrm{C}(\bbT)$, the algebra of continuous functions on the circle, and 
finite-dimensional algebras are called \emph{AT algebras}. They always have stable rank one but $\mathrm{C}(\bbT)$ itself does not have
real rank zero. 
In~\cite{elliott1993classification} Elliott classified separable AT algebras of real rank zero using the functors $K_0$ and $K_1$. 
Like $K_0(A)$, $K_1(A)$ is an abelian group which  typically does not belong  to~$A^{\eq}$.  
However,  in the case of AT algebras of real rank zero  it is 
a  quotient of a quantifier-free definable set (the unitary group) by a 
definable relation (see Proposition~\ref{P.K1.sr}).

The same pair of invariants, $K_0(A)$ and $K_1(A)$, were used in the classification of
simple, separable, nuclear and purely infinite \cstar-algebras (nowadays known as Kirchberg algebras)\index{C@\cstar-algebra!Kirchberg} 
by Kirchberg and Phillips (see \cite[Proposition 2.5.1]{Ror:Classification}, \cite{Phi:Classification}).\footnote{This result has an additional assumption, more about that later.} In this case both $K$-groups belong to~$A^{\eq}$.  

The construction of isomorphisms in classification results typically uses
the \emph{Elliott intertwining argument}
to lift  an isomorphism between Elliott invariants $\Ell(A)$ and $\Ell(B)$ 
to an  isomorphism between 
\cstar-algebras $A$ and $B$. In the first step, $A$ and $B$ are presented as limits of a sequence of building blocks:
$A=\lim_n A_n$, $B=\lim_n B_n$.
(If $A$ and $B$ don't come with such sequences, one might use $A_n=A$ and $B_n=B$ for all $n$.) 
In the second step 
one finds partial  lifts of $\Phi$ and $\Phi^{-1}$ between the building blocks such that the triangles in the
following diagram 
approximately commute. 

\begin{tikzpicture}
 \matrix[row sep=1cm,column sep=1cm] {
\node  (A1) {$A_1$};
& \node (A2) {$A_2$};
& \node (A3) {$A_3$};
& \node (A4) {$A_4$};
& \node (Adots) {$\dots$};
& \node (A) {$A=\lim_n A_n$};
\\
\node  (B1) {$B_1$};
& \node (B2) {$B_2$};
& \node (B3) {$B_3$};
& \node (B4) {$B_4$};
& \node (Bdots) {$\dots$};
& \node (B) {$B=\lim_n B_n$};
\\
};
\draw (A1) edge[->] node [left] {$\Phi_1$} (B1) ;
\draw (A2) edge[->] node  [left] {$\Phi_2$} (B2) ;
\draw (A3) edge[->] node [left] {$\Phi_3$} (B3);
\draw (A4) edge[->] node [left] {$\Phi_4$} (B4);
\draw (B1) edge[->] node [left] {$\Psi_1$} (A2);
\draw (B2) edge[->] node [left] {$\Psi_2$} (A3);
\draw (B3) edge[->] node [left] {$\Psi_3$}  (A4);
\draw (B4) edge[->] (Adots);
\draw (A1) edge[->] (A2);
\draw (A2) edge[->] (A3);
\draw (A3) edge[->] (A4);
\draw (A4) edge[->] (Adots);
\draw (B1) edge[->] (B2);
\draw (B2) edge[->] (B3);
\draw (B3) edge[->] (B4);
\draw (B4) edge[->] (Bdots);

\end{tikzpicture}

These maps converge to an isomorphism from $A$ onto $B$ and its inverse. 
A full classification result would also assert that every morphism between the invariants 
lifts to a morphism between $A$ and $B$.

Stably finite algebras comprise 
a vast class of \cstar-algebras not covered by the Kirchberg--Philips result. 
By a deep result of Haagerup all such exact algebras have a tracial state.
The space $T(A)$\index{T@$T(A)$} of all tracial states of $A$, up to affine homeomorphism, 
is  included in the Elliott invariant. 
The space $T(A)$ is  a weak$^*$-compact, convex set (and even a Choquet simplex)  
and it does not belong to $A^{\eq}$ unless it is finite-dimensional (and even its 
finite-dimensionality does not guarantee membership in $A^{\eq}$).\footnote{The spaces 
$T(A)$ up to affine homeomorphism are not even classifiable by countable structures, 
but that is a different story.}
By a classical result of Cuntz and Pedersen (\cite{CuPe:Equivalence})
$T(A)$ is isomorphic to 
the state space  of the  (real, ordered) Banach space $A_{\mathrm{sa}}/A_0$. 
Here $A_0$ is the closed linear span of self-adjoint commutators $[a,a^*]=aa^*-a^*a$.  
While $A_{\mathrm{sa}}$ is always definable,  $A_0$
 is definable in well-behaved (e.g., exact and $\cZ$-stable)  \cstar-algebras 
 but there are examples of \cstar-algebras (and even monotracial \cstar-algebras) in 
 which~$A_0$ is not definable. 
Other well-studied problems in classification and structure of \cstar-algebras are closely related 
to the study of definability of distinguished subsets in \cstar-algebras, such as the 
connected component of the identity in the unitary group.

The Elliott conjecture was refuted in 2004. In  \cite{Ror:Simple} 
R\o rdam constructed  a simple, nuclear \cstar-algebra that 
was neither tracial nor a Kirchberg algebra, since it contains both a finite and an infinite projection. 
An even  more devastating counterexample was constructed by Toms (\cite{To:On}) who produced
two stably finite, nonisomorphic algebras that cannot be distinguished by any functor continuous with respect to inductive limits and  
homotopy-invariance.\footnote{Looking ahead, the theory of a \cstar-algebra is neither a functor nor continuous with respect to inductive limits.}
Prior to these developments, Jiang and Su (\cite{JiangSu}) constructed a  stably finite,  infinite-dimensional \cstar-algebra $\cZ$ with the same invariant as $\bbC$, and it has turned out that $\cZ$\index{C@\cstar-algebra!Jiang--Su algebra, $\cZ$}\index{Z@$\cZ$}
 plays a positive role in the Elliott programme. 
In its current form, the Elliott programme aims to classify \snus{}
algebras that absorb $\cZ$ tensorially (such algebras are also said to be 
\emph{$\cZ$-stable}).\index{Z@$\cZ$-stable} 
$\cZ$-stable algebras  are axiomatizable among separable \cstar-algebras. 
In addition to the regularity properties that make $\cZ$-stable algebras amenable to $K$-theoretic classification, 
these algebras have an abundance of definable sets.  
Notably,  Cuntz--Pedersen equivalence is  definable by a 
result of Ozawa  \cite[Theorem 6]{Oza:Dixmier} (see also \cite[Corollary~4.3]{NgRobert}).

In the following pages we develop the model theory of nuclear \cstar-algebras and their invariants. 
Nuclear \cstar-algebras are a generalization of AF algebras. 
This class consists of \cstar-algebras satisfying 
the \emph{completely positive approximation property}\index{CPAP (completely positive approximation property)}
  (CPAP),\footnote{This is a theorem, not a definition.} which means that the identity map approximately factors through finite-dimensional \cstar-algebras via completely positive contractive (c.p.c.) maps:
 
 \begin{center}
   \begin{tikzpicture}
 \matrix[row sep=.6cm,column sep=.6cm] {
\node  (A1) {$A$};
&& \node (A2) {$A$};&;
\\
&\node  (B1) {$F$};
\\
};
\draw (A1) edge[->] node [above] {$\id$}  (A2);
\draw (A1) edge[->] node [left] {$\phi$} (B1) ;
\draw (B1) edge[->] node [right] {$\psi$} (A2);
\end{tikzpicture}
 \end{center}

At least in the separable case they can also be written as inductive limits of finite-dimensional \cstar-algebras

   \begin{tikzpicture}
 \matrix[row sep=.6cm,column sep=.6cm] {
\node  (A1) {$A$};
&& \node (A2) {$A$};
&& \node (A3) {$A$};
&& \node (Adots) {$\dots$};
& \node (A) {$A$};
\\
&\node  (B1) {$F_1$};
&& \node (B2) {$F_2$};
&& \node (B3) {$F_3$};
\\
};
\draw (A1) edge[->] node [left] {$\phi_1$} (B1) ;
\draw (A2) edge[->] node  [left] {$\phi_2$} (B2) ;
\draw (A3) edge[->] node [left] {$\phi_3$} (B3);
\draw (B1) edge[->] node [right] {$\psi_1$} (A2);
\draw (B2) edge[->] node [right] {$\psi_2$} (A3);
\draw (B3) edge[->] node [right] {$\psi_3$}  (Adots);
\draw (B3) edge[->] (Adots);
\draw (A1) edge[->] node [above] {$\id$}  (A2);
\draw (A2) edge[->]  node [above] {$\id$}(A3);
\draw (A3) edge[->] node [above] {$\id$} (Adots);

\end{tikzpicture}

\noindent 
and approximately commuting triangles which also keep track of the
multiplicative structure in a suitable way. Using generalized limits as defined 
in \cite{BlaKir} in this way one obtains the class of quasidiagonal nuclear algebras. 
If the limit is taken in the category of operator systems with order units, and 
endowed with a (unique) multiplicative structure only a posteriori, then 
the limits are exactly the nuclear \cstar-algebras. 

It is worth noting that  for every finite-dimensional \cstar-algebra $F$ 
the set of all c.p.c.\  maps from $A$ into $F$, as well as the set of all c.p.c.\  maps from 
$F$ to $A$, belongs to $A^{\eq}$. 
Variations of CPAP in 
which stronger conditions than complete positivity are imposed on the maps $\psi_n$  
are used  
to define important subclasses of nuclear \cstar-algebras, such as algebras with finite nuclear dimension or finite 
decomposition rank. None of these classes is elementary as the above diagram does not survive taking ultrapowers. 
In spite of this, we shall  give a characterization  of 
these classes of algebras in model-theoretic terms. Our results extend those of  
 \cite{Mitacs2012} where it was demonstrated that AF algebras and UHF algebras are exactly the algebras that 
omit a certain sequence of existential types. 

Our motivation for this  work is manifold; see also \cite{FarahICM} for more information on the interactions of logic and \cstar-algebras.
\smallskip

\paragraph{\bf Constructing novel examples of nuclear \cstar-algebras.}
This is our most ambitious and long-term goal. Our results readily provide model-theoretic reformulations 
of a  number of outstanding questions  at the heart of the structure and classification theory of nuclear \cstar-algebras. 
The existence of counterexamples to the Toms--Winter conjecture and some 
other outstanding conjectures about the structure of nuclear \cstar-algebras 
can be reformulated in terms of the existence of a model of a certain theory that 
omits a certain sequence of types.  
Thus the existence of a counterexample is equivalent to the assertion that   
the set of counterexamples is generic.

While at present we do not have examples of the desired kind, we hope our methods will contribute to a systematic study of how they could possibly arise.

\smallskip
\paragraph{\bf Model theory of \cstar-algebras.} The theory of every infinite-dim\-ensional \cstar-algebra is unstable, in a model theoretical sense, (\cite{FaHaSh:Model2}) 
and the only theory of an infinite-dimensional \cstar-algebra with elimination of quantifiers is the theory of complex-valued continuous functions on Cantor space (\cite{eagle2015quantifier}). 
In spite (or perhaps because of) this unruliness, \cstar-algebras have an abundance of interesting definable sets. 
Certain sets, such as  the set of projections or the unitary group,  are uniformly definable across all  
\cstar-algebras. On the other hand, 
definability of some other sets, such as the connected component of the identity in the unitary group, or the set of elements that are Cuntz--Pedersen equivalent to $0$, depends on structural properties of the \cstar-algebra in question. 
Not surprisingly, regularity properties required by the Elliott conjecture (see \cite{EllTo:Regularity}) 
imply  the definability of these sets. The most prominent regularity property---tensorial absorption of the 
Jiang--Su algebra $\cZ$---is axiomatizable.

\smallskip
\paragraph{\bf Developing the toolbox.} We present a self-contained presentation of imaginaries\index{imaginaries} and $A^{\eq}$
in the general setting of the logic of metric structures (not necessarily \cstar-algebras). We prove the appropriate version 
of the Beth definability theorem
and translate some of the 
standard preservation and axiomatizability results from discrete first-order model theory into  the metric context. 

\smallskip
\paragraph{\bf The theory as an invariant.} 
The Elliott classification programme (when successful) provides an equivalence between the category of \cstar-algebras and the category of Elliott invariants. 
Parts of the Elliott invariant $\Ell(A)$ of a classifiable\footnote{The usage of term  `classifiable' should not be confused 
with that from model theory, as in \cite{Sh:Classification}.  
See  \cite[\S 2.5]{Ror:Classification} for a technical definition of `classifiable \cstar-algebra.'}  
\cstar-algebra $A$ 
can be reconstructed from $A^{\eq}$, while $\Ell(A)$ provides a template for reconstructing $A$.\footnote{A very  general classification 
theorem of this sort was proved in \cite{Ell:Towards}.}

The theory of a \cstar-algebra is neither a functor nor continuous with respect to inductive limits. 
On the other hand, it is a Borel map from the Borel space of separable \cstar-algebras into the dual space of 
the normed vector space of sentences. It provides a smooth invariant which cannot even distinguish all separable 
AF algebras up to isomorphism. 
Although there is some overlap, the information provided by the theory is largely complementary  to the information provided by the Elliott invariant.  To the best of our knowledge, 
all known  \snus{} algebras are classified by their theory together with the Elliott invariant.  We now give two examples.  

The peculiarity of R\o rdam's example of a simple \cstar-algebra with an infinite and a finite projection (\cite{Ror:Simple})
is concealed from its Elliott invariant, but it is evident to its theory (Proposition~\ref{P.definable.projection}).  
Also,  numerical values of the radius of comparison (the invariant used by Toms to distinguish between infinitely many  
\snus{} algebras with the same Elliott invariant in \cite{To:Infinite}) can be read off from the theory of  \cstar-algebras (Theorem~\ref{P.rc}). 
A very optimistic interpretation of these observations is the following:

\begin{quest} \label{Q1} 
Assume $A$ and $B$ are  two elementarily equivalent, simple, separable, nuclear, unital \cstar-algebras
with the same Elliott invariant. Are $A$ and $B$ 
necessarily isomorphic? 
\end{quest}

The answer to this question is likely negative (see Example~\ref{Ex.Counterexample} for a non-nuclear example) 
but it seems that proving this would require a novel construction of nuclear \cstar-algebras.

\smallskip
\paragraph{\bf Byproducts.} 
Our results provide  a uniform  proof that 
various classes of \cstar-subalgebras of $B(H)$  are stable under small perturbations  in the Kadison--Kastler metric (see \cite{christensen2012perturbations},   
Corollary \ref{C.dkk.1} and  Corollary \ref{C.dkk.2}). 
We also show that these classes of algebras are Borel in the standard Borel space 
of \cstar-algebras (see \cite{Kec:C*}), 
 answering a question of  Kechris (see Corollary~\ref{C.Borel}).

This work also precipitated  the study of omitting types in the logic of metric structures 
 (\cite{FaMa:Omitting}). Although in general the set of (partial) types omissible in a model 
 of some theory is highly complex (more precisely,~$\mathbf\Sigma^1_2$), types relevant to \cstar-algebras
 belong to a class 
 for which there is a natural omitting types theorem.

\smallskip
\paragraph{\bf Acknowledgments.} We would like to thank I.\ Ben Yaacov, 
B.\ Blackadar, G.\ Cherlin, C.\ Eckhardt, 
G.A.\ Elliott, D.\ Enders, I.\ Goldbring, I.\ Hirshberg, 
M.\ Magidor, R.\ Nest, N.\ C.\ Phillips, D.\ Sherman, T.\ Sinclair, G.\ Sza\-bo, T.\ Tsankov, and S.\ White for valuable discussions. 
A part of this work grew out of the M\"unster workshop on classification and model theory held in July  2014.  We would also like to thank the Fields Institute for their hospitality during the creation of this document.

This work was partially supported by the Deutsche Forschungsgemeinschaft through SFB 878 and by the NSERC. M.\ Lupini is supported by NSF Grant DMS-1600186. A.\ Tikuisis is supported by EPSRC first grant EP/N00874X/1. A.\ Vignati was supported by a Susan Mann Dissertation Scholarship and by the NSERC.

\newpage
\tableofcontents
\newpage

\chapter{Continuous model theory}
\section{Preliminaries}\label{S.Preliminaries} 
In this section we present the formalities of the logic that we will use to study \cstar-algebras and related structures in a model theoretic manner.  Let us begin by looking at \cstar-algebras themselves and seeing how they are viewed.  Fix a \cstar-algebra $A$.  We first consider $\cS^A$, the collection of balls $B_n = \{ a \in A \colon \| a \| \leq n \}$ for $n \in \bbN$.  This is a collection of bounded complete metric spaces with the metric on each ball given by $\|x - y\|$.  We let $\cF^A$ be the collection of functions obtained by considering the ordinary algebraic operations on $A$ restricted to the balls in $\cS^A$ i.e., we have the restrictions of $+, \cdot, {}^*$ as well as scalar multiplication by any $\lambda \in \bbC$.  For good measure, we also have the constants $0$ and $1$ (if the algebra is unital)  as $0$-ary functions and standard inclusion maps from $B_m$ to $B_n$ when $m < n$.  It is important to notice that all of these functions are uniformly continuous on their domains.  It is clear that any $A$ yields a structure consisting of $\cS^A$ and $\cF^A$ as described.  A little less clear (the details are worked out in Proposition 3.2 of \cite{FaHaSh:Model2}, see also Example~\ref{Ex.cstaralg}) is that there are axioms that one can write down such that if one has a pair $(\cS,\cF)$ satisfying those axioms then there is a unique \cstar-algebra from which it came.  

In general, a \emph{metric structure}\index{metric structure} is a structure $\langle \cS,\cF,\cR \rangle$ which consists of
\begin{enumerate}[leftmargin=*]
\item $\cS$, the \emph{sorts},\index{sort} which is an indexed family of metric spaces $(S,d_S)$ where $d_S$ is a complete, bounded metric on $S$.
\item $\cF$, the \emph{functions},\index{function} is a set of uniformly continuous functions such that for $f \in \cF$, $\dom(f)$ is a finite product of sorts and $\rng(f)$ is a sort. This includes the possibility that the domain is an empty product and in this case, $f$ is a constant in some sort.
\item $\cR$, the \emph{relations},\index{relation} is a set of uniformly continuous functions such that for $R \in \cR$, $\dom(R)$ is a finite product of sorts and $\rng(R)$ is a bounded interval in $\bbR$.
\pushcounter
\end{enumerate}
Clearly \cstar-algebras as described above fit this format. In the case of \cstar-algebras, there are no relations (or one can think that the metrics on the sorts are relations).  A non-trivial example of a relation would be a \cstar-algebra together with a trace.  The formalism would again be as in the first paragraph: for a \cstar-algebra $A$ together with a trace $\tau$ the structure would be as described above with the addition of relations arising from the trace restricted to the balls (formally you would need to add both the real and imaginary part of the trace restricted to each ball as a relation since we want our relations to be real-valued but we won't labour this point now).

To capture the notion of a metric structure precisely, we need to introduce the syntax of continuous logic.  This begins with a formalism that mimics the essential elements of a metric structure.  A language~$\calL$ is a formal object $\langle \ccS, \ccF, \ccR \rangle$ which contains the following data:\index{language}
\begin{enumerate}[leftmargin=*]
\popcounter
\item $\ccS$ is the set of sorts; for each sort $S \in \ccS$, there is a symbol $d_S$ meant to be interpreted as a metric together with a positive number $M_S$ meant to be the bound on $d_S$;
\item $\ccF$ is the set of function symbols; for each function symbol\index{function symbol} $f \in \ccF$, we formally specify $\dom(f)$ as a sequence $(S_1,\ldots,S_n)$ from $\ccS$ and $\rng(f) = S$ for some $S \in \ccS$.  We will want $f$ to be interpreted as a uniformly continuous function.  To this effect we additionally specify, as part of the language, functions $\delta^f_i:\bbR^+ \rightarrow \bbR^+$, for $i \leq n$. These 
functions are  called \emph{uniform continuity moduli}.\index{uniform continuity modulus}
\item $\ccR$ is the set of relation symbols; for each relation symbol\index{relation symbol} $R \in \ccR$ we formally specify $\dom(R)$ as a sequence $(S_1,\ldots,S_n)$ from $\ccS$ and $\rng(R) = K_R$ for some compact interval $K_R$ in $\bbR$.  As with function symbols, we additionally specify, as part of the language, functions $\delta^f_i:\bbR^+ \rightarrow \bbR^+$, for $i \leq n$, called \emph{uniform continuity moduli}.
\pushcounter
\end{enumerate}
In order to express formulas in this language we will need some auxiliary logical notions.  For each sort $S \in \ccS$, we have infinitely many variables~$x^S_i$ for which we will almost always omit the superscript.  \emph{Terms} in the language $\calL$ are defined inductively:\index{term}
\begin{enumerate}[leftmargin=*]
\popcounter
\item\label{I.Terms.1} A variable $x^S_i$ is a term with domain and range $S$;
\item\label{I.Terms.2}  If $f \in \ccF$, $\dom(f) = (S_1,\ldots,S_n)$ and $\tau_1,\ldots,\tau_n$ are terms with $\rng(\tau_i) = S_i$ then $f(\tau_1,\ldots,\tau_n)$ is a term with range the same as $f$ and domain determined by the $\tau_i$'s.
\pushcounter
\end{enumerate}
All terms inherit uniform continuity moduli inductively by their construction.

In the case of \cstar-algebras, the language contains sorts for each ball~$B_n$ and a function symbol for all the functions specified at the beginning of the section.  In the case of a \cstar-algebra with a trace there are additional relation symbols corresponding to the restriction of the trace to each ball.  In both cases, terms are formed in the usual manner and correspond formally to $^*$-polynomials in many variables over $\bbC$ restricted to a product of balls.

Now in order to say anything about metric structures in some language $\calL$, we need to define the collection of basic formulas of $\calL$.  As with relation symbols, formulas will have domains and uniform continuity moduli which we also define inductively along the way.

\begin{definition}\label{formula}
\begin{enumerate}[leftmargin=*]
\popcounter
\item \index{formula}
If $R$ is a relation symbol in $\calL$ (possibly a metric symbol) with domain $(S_1,\ldots,S_n)$ and $\tau_1,\ldots,\tau_n$ are terms with ranges $S_1,\ldots,S_n$ respectively then $R(\tau_1,\ldots,\tau_n)$ is a formula.  Both the domain and uniform continuity moduli of $R(\tau_1,\ldots,\tau_n)$ can be determined naturally from $R$ and $\tau_1,\ldots,\tau_n$.
These are the \emph{atomic formulas}.\index{formula!atomic}
\item (Connectives) If $f:\bbR^n \rightarrow \bbR$ is a continuous function and $\varphi_1,\ldots,\varphi_n$ are formulas then $f(\varphi_1,\ldots,\varphi_n)$ is a formula.  Again, the domain and uniform continuity moduli are determined naturally from $f$ and $\varphi_1,\ldots,\varphi_n$.
\item\label{I.Quantifiers}  (Quantifiers) If $\varphi$ is a formula and $x=x^S$ is a variable then both $\inf_{x\in S} \varphi$ and $\sup_{x\in S} \varphi$ are formulas.  The domain of both is the same as that of $\varphi$ except that the sort of $x$ is removed and the uniform continuity moduli of all other variables is the same as that of $\varphi$.
\end{enumerate}
\end{definition}

We denote the collection of all formulas of $\calL$ by $\ccF_\calL$;\index{F@$\ccF_\calL$}
 if we wish to highlight the free variables $\bar x = (x_1,\ldots,x_n)$ we will write $\ccF_\calL^{\bar x}$. In particular,  $\ccF_\calL=\bigcup_n \ccF_\calL^{(x_1,\ldots,x_n)}$.
  We will often omit the subscript if there is no confusion about which language we are using. Suppose $\varphi(x_1,\ldots,x_n)$ is a formula with free variables $x_1,\ldots, x_n$ and $M$, an $\calL$-structure.  Moreover if $m_1,\ldots,m_n$ are in $M$ and $m_i$ is in the sort associated with $x_i$ then the \emph{interpretation}\index{interpretation (of a formula), $\varphi^M$} of $\varphi$ in $M$ at $m_1,\ldots,m_n$, is written $\varphi^M(m_1,\ldots,m_n)$ and is the number defined naturally and inductively according to the construction of $\varphi$.  Quantification as in \eqref{I.Quantifiers} is interpreted as taking suprema and infima over the sort associated with the variable being quantified.
The following is straightforward (see \cite[Theorem 3.5]{BYBHU}). 

\begin{prop} If $\varphi(\bar x)$ is a formula with domain $(S_1,\ldots,S_n)$ then there is a constant 
$K<\infty$ such that if $M$ is an $\calL$-structure then
\begin{enumerate}
\item $\varphi^M:S_1(M)\times \ldots \times S_n(M) \rightarrow \bbR$ is uniformly continuous with uniform continuity modulus given by $\varphi$, i.e., for $\e>0$, if $\bar a=(a_1,\dots,a_n), \bar b=(b_1,\dots,b_n) \in S_1(M) \times \ldots \times S_n(M)$ and $d_{S_i(M)}(a_i,b_i)<\delta_i^\phi(\e)$ then $|\phi^M(\bar a) - \phi^M(\bar b)| < \e$, and
\item $\varphi^M$ is bounded by $K$. \qed
\end{enumerate}
\end{prop}

We will almost always write $\|x\| \leq k$ for ``$x$ is in the sort corresponding to the ball of operator norm $\leq k$''.  Here are some examples of formulas in the language of \cstar-algebras together with their interpretations.
\begin{example}
\begin{enumerate}
\item All of $x$, $x^2$ and $x^*$ are terms in the language of \cstar-algebras.  If $d$ is the metric on the sort for the unit ball then $\varphi(x) := \max(d(x^2,x),d(x^*,x))$ is a formula with one free variable.  If $A$ is a \cstar-algebra then for $p \in B_1(A)$, $\varphi^A(p) = 0$ if and only if $p$ is a projection. We will often write $\| x - y \|$ for $d(x,y)$ when the metric structure under discussion is a \cstar-algebra, and likewise, $\|x\|$ for $d(x,0)$.
\item Let $\psi(x) := \inf_{\|y\| \leq 1} \| y^*y - x \|$.  $\psi(x)$ is a formula in the language of \cstar-algebras and it is reasonably easy to see (using that the set of positive elements is closed and equal to $\{y^*y \mid y \in A\}$) that for any \cstar-algebra $A$ and $a \in B_1(A)$, $\psi^A(a) = 0$ if and only if $a$ is positive.
\end{enumerate}
\end{example}

We will be concerned with particular subclasses of formulas.  
\begin{definition}
\begin{enumerate}
\item We say that a formula $\varphi$ 
is \emph{$\bbR^+$-valued}\index{formula!$\bbR^+$-valued} if in all interpretations, the value of $\varphi$ is greater than or equal to 0.  
\item We say that a formula $\varphi$ is \emph{$[0,1]$-valued}\index{formula!$[0,1]$-valued} if in all interpretations, the value of $\varphi$ lies in $[0,1]$.  
\end{enumerate}
\end{definition}
The simplest way of obtaining an $\bbR^+$-valued formula is by taking the absolute value of a formula.  It is clear that every formula is a difference of two $\bbR^+$-valued formulas.  Since evaluation is a linear functional, the value of all formulas is determined by the value on the $\bbR^+$-valued formulas.  The $[0,1]$-valued formulas are used in a similar way.  Since the range of any formula is bounded in all interpretations, we can compose with a linear function to obtain a formula which is $[0,1]$-valued.  Again, since evaluation is linear, the value of all formulas is determined by the values on $[0,1]$-valued formulas.

$\bbR^+$-valued formulas allow the interpretations of the quantifiers to approximate classical first order quantifiers.  For instance, if we are considering an $\bbR^+$-valued formula $\varphi(x,\bar y)$ and we know that in some structure $M$ and some $\bar a$ in $M$, $\sup_x \varphi(x,\bar a) = 0$ then we know that for all $b \in M$, $\varphi(b,\bar a) = 0$.  So $\sup$ acts like a universal quantifier.  On the other hand, $\inf$ only tells you about $\varphi$ approximately.  That is, if $\inf_x \varphi(x,\bar a) = 0$ in $M$, then we know that for every $\e > 0$ there is a $b \in M$ such that $\varphi^M(b,\bar a) < \e$.

The syntactic form of a formula often has semantic content.  There are several classes of formulas whose form we will consider.
\begin{definition}\label{def:complexity}
\begin{enumerate}
\item A \emph{quantifier-free} formula\index{formula!quantifier-free}
 is obtained by invoking only the first two clauses i.e., we do not use quantifiers.
\item A \emph{$\sup$-formula}\index{formula!$\sup$-formula}
 is a formula of the form $\sup_{\bar x} \varphi$ where $\varphi$ is an $\bbR^+$-valued quantifier-free formula.  We similarly define \emph{$\inf$-formulas}.\index{formula!$\inf$-formula}
 \item An \emph{$\forall\exists$-formula}\index{formula!$\forall\exists$-formula} is a  formula of the form $\sup_{\bar x}\inf_{\bar y}\varphi$ where $\varphi$ is an $\bbR^+$-valued quantifier-free formula.
\item A \emph{positive}\index{positive!formula}  formula is one formed using Definition \ref{formula} but in the second clause one only composes with increasing functions i.e., with functions $f:\bbR^n \rightarrow \bbR$ such that if $\bar r \leq \bar s$ coordinate-wise then $f(\bar r) \leq f(\bar s)$.
\end{enumerate}
$\sup$-formulas are also known as \emph{universal formulas}\index{formula!universal}
and 
$\inf$-formulas are also known as \emph{existential formulas}\index{existential!universal}. 
\end{definition}

Here is how these different syntactic forms get used in practice. 

\begin{prop} Suppose that $M \subseteq N$ are $\calL$-structures.
\begin{enumerate}
\item If $\varphi(\bar x)$ is a quantifier-free formula then for any $\bar a$ in $ M$, $\varphi^M(\bar a) = \varphi^N(\bar a)$.
\item If $\varphi(\bar x)$ is a $\sup$-formula then for any $\bar a$ in $ M$, if $\varphi^N(\bar a) = 0$ then $\varphi^M(\bar a) = 0$.
\item If $\varphi(\bar x)$ is an $\inf$-formula then for any $\bar a$ in $M$, if $\varphi^M(\bar a) = 0$ then $\varphi^N(\bar a) = 0$.

\end{enumerate}
\end{prop}

\begin{proof}This is straightforward from the definitions. \end{proof}

A function $f:M \rightarrow N$ between two structures $M$ and $N$ 
is a \emph{homomorphism}\index{homomorphism} if  whenever $\varphi$ is an atomic and $\bar m$ in $M$ then $\varphi^M(\bar m) \geq \varphi^N(f(\bar m))$.  In the case of \cstar-algebras this gives the usual definition of $^*$-homomorphism.

\begin{prop}\label{P.Positive.Preservation}
If $f:M \rightarrow N$ is a surjective homomorphism and $\varphi(\bar x)$ is a positive formula then for any $\bar a$ in $M$, $\varphi^M(\bar a) \geq \varphi^N(f(\bar a))$.
\end{prop}

\begin{proof} This is proved by induction on complexity of the formula, following Definition~\ref{formula}. 
The case for atomic formulas is just the definition. Since $f$ is a surjection, if the assertion is true for $\varphi$ then it is true for 
$\sup_{\bar x} \varphi$ and $\inf_{\bar x}\varphi$. Finally, if the assertion is true for $\varphi_1, \dots \varphi_n$ 
and $g\colon \bbR^n\to \bbR$ is such that $\bar r\geq \bar s$ coordinate-wise implies $g(\bar r)\geq g(\bar s)$, 
then the assertion is clearly true for $g(\varphi_1, \dots, \varphi_n)$.  This completes the induction. 
\end{proof} 

\section{Theories}

A formula with no free variables is called a \emph{sentence}\index{sentence} and the set of sentences of $\calL$ is 
denoted $\mbox{Sent}_\calL$ (that is, 
 denoting 
the 0-tuple by $()$, we have $\mbox{Sent}_\calL:= \ccF_\calL^{()}$). 
\index{Sent@$\mbox{Sent}_\calL$}  Given an $\calL$-structure $M$, the 
theory\index{theory}\index{Th@$\Th(M)$}
 of $M$ is the functional $\Th(M): \mbox{Sent}_\calL \rightarrow \bbR$ given by 
\[
\Th(M)(\varphi) := \varphi^M.
\]
  This functional is determined by its kernel and so, for any given $M$, it suffices to know those $\varphi$ such that $\varphi^M = 0$.  If we identify $\Th(M)$ with its kernel, we often write $\varphi \in \Th(M)$ to mean $\varphi^M = 0$.  Any set of sentences is called a theory. If $T$ is a set of sentences then we say $M$ \emph{satisfies}\index{satisfies} $T$, 
  \[
  M \models T,
  \]
   if $T \subseteq \Th(M)$ i.e., if for all $\varphi \in T$, $\varphi^M = 0$.
   The \emph{satisfaction relation $\models$}\index{satisfaction relation $\models$}
   can be used in a purely syntactical context. If $T$ is a theory and  
   a \emph{condition}\index{condition} (i.e., restriction in the language of $T$ 
    of the form $r \leq \varphi$) is given then 
   \[
   T\models r\leq \varphi
   \]
   stands for the assertion `If $M\models T$ then $M\models r\leq \varphi$.'
   
      A theory $T$ is said to be \emph{consistent}\index{theory!consistent} if there is some $M$ which satisfies $T$. Finally, we say two $\calL$-structures $A$ and $B$ are \emph{elementarily equivalent},\index{elementarily equivalent} $A \equiv B$, if $\Th(A) = \Th(B)$.  A theory $T$ is said to be \emph{complete}\index{theory!complete} if whenever $A$ and $B$ satisfy $T$ they are elementarily equivalent.

\begin{example}\label{Ex.cstaralg}
To get used to some of the formalism, let's explicitly write out some of the axioms for the theory of \cstar-algebras.  Much of this is taken from Section 3 of \cite{FaHaSh:Model2}.  To make things easier, we adopt the following conventions: if $\tau$ and $\sigma$ are terms in the language of \cstar-algebras then to say that the equation $\tau = \sigma$ holds in all \cstar-algebras is the same as saying that the sentence $\sup_{\|\bar x\| \leq k}\| \tau - \sigma \|$ evaluates to 0 for all $k$; formally this is infinitely many sentences to express the fact that this one equation holds.  In this way, one can explicitly write out the axioms which say that we have an algebra over $\mathbb{C}$ with an involution $^\ast$.  Moreover, the axioms contain the sentences $(xy)^* = y^*x^*$ and $(\lambda x)^* = \bar \lambda x^*$ for every $\lambda \in \mathbb{C}$ (remember that we have unary functions symbols for every complex scalar).

We need sentences that guarantee that the operator norm will behave correctly. This allows us to introduce some notation for inequalities.  We define the function $\dotminus$ by  
\[
r \dotminus s := \max\{r - s, 0\} = (r-s)_+.
\]
  Notice that $r \leq s$ if and only if $r \dotminus s = 0$.  For two formulas $\varphi$ and $\psi$, to know that in all \cstar-algebras, $\varphi \leq \psi$ is the same as knowing that $\sup_{\|\bar x\| \leq k}( \varphi \dotminus \psi)$ evaluates to 0 for all $k$.  So \cstar-algebras satisfy:
\begin{enumerate}
\item $\|xy\| \leq \|x\|\|y\|$, 
\item  for all $\lambda \in \mathbb{C}$, $\|(\lambda x)\| =  |\lambda| \|x\|$,
\item (the \cstar-identity) $\|x^*x\| = \|x\|^2$.
\pushcounter
\end{enumerate}
To anyone who has seen the axioms of a \cstar-algebra, this already looks like a complete list.
However, two more axioms are needed, 
because the sentences listed so far are still too weak to tell us that the formal unit ball $B_1$ (meant to correspond to the ball of operator norm 1) is in fact the unit ball, and likewise for the norm-$n$ balls $B_n$. 
In other words, we need to enforce a compatibility between the norm structure and the sorts $B_n$.
Here are sentences that make the two notions coincide: remember that we have inclusion maps $i_n$ from $B_1$ into $B_n$ that are isometries which preserve all the algebraic operations ($+,\cdot,*$ and scalar multiplication); these statements are easy to translate into sentences.  The additional sentences that we need are, for every $n$:
\begin{enumerate}
\popcounter
\item $\sup_{x \in B_1} \|x\| \leq 1$.
\item\label{inclusion} $\sup_{x \in B_n} \inf_{y \in B_1} (\|x - i_n(y)\| \dotminus (\|x\| \dotminus 1))$, for each $n \in \bbN$.
\end{enumerate}
To understand when this last sentence evaluates to 0, consider some element~$a$ of the ball of operator norm $n$ which in fact has operator norm less than or equal to 1.  This sentence will guarantee that there is something in the unit ball which maps onto $a$ via the inclusion $i_n$ i.e., everything of operator norm 1 will lie in the unit ball.  The axioms in~(\ref{inclusion}) are a series of $\forall\exists$-sentences whereas all the other axioms are universal.  It is sometimes desirable to have a universal axiomatization. This can be achieved by slightly increasing the language. See Section 3 of~\cite{FaHaSh:Model2} to see how this can be done.
\end{example}

\begin{example}\label{Ex.projectionless}
We provide another, at first glance random, sentence.
Let  $t(x) := (x + x^*)/2$, $f(y) := 1 - \sqrt{1 - 4y}$,  and let $\varphi$ be defined as  
\begin{multline*}
\displaystyle\sup_{\|x\| \leq 1} \min\{ \|t(x)^2 - t(x)\| \dotminus \frac{1}{8},\\
 \max\{ \|x\|, \|1-x\|\} \dotminus \frac{\|x - x^*\| + f(\|t(x)^2 - t(x)\|)}{2}\}. 
\end{multline*}
It is a nice exercise in functional calculus (see \S\ref{S.cfc})
to show that $A$ satisfies $\varphi$ if and only if $A$ is unital projectionless (see \S\ref{S.UP}).  Although it is often possible and fruitful to find a sentence to express a particular notion, we will see below a variety of semantic tools that can be used to determine if a class of structures can be captured by certain kinds of theories.

\end{example}

A theory is a rather coarse invariant.   The following was proved  in~\cite{Mitacs2012}.

\begin{lemma} \label{L.ee}
\begin{enumerate}[leftmargin=*]
\item\label{uhf-ee}  Unital separable UHF algebras 
  are isomorphic if and only if they are elementarily equivalent. 
\item \label{non-unit-ee}There are nonisomorphic separable nonunital 
UHF  algebras with the same theory.   
\item \label{af-ee}There are nonisomorphic unital separable AF algebras 
with the same theory. 
\item \label{kirch-ee}There are nonisomorphic unital Kirchberg algebras (satisfying the Universal Coefficient Theorem) with the same theory.
\end{enumerate}  
\end{lemma}

\begin{proof} This was proved in \cite[Theorem 3(2)]{Mitacs2012} but we include the proof for the reader's convenience. 
By a result of Glimm, two unital separable UHF algebras $A$ and $B$ are isomorphic if and only if the following holds for every $n\in \bbN$:
 there is a unital embedding of $M_n(\bbC)$ into $A$ if and only if there is a unital embedding of $M_n(\bbC)$ into~$B$. 
Since for every fixed $n$ the existence of unital embedding of $M_n(\bbC)$ is axiomatizable 
(and even existentially axiomatizable, as well 
as co-axiomatizable; see \S\ref{S.Mn}), \eqref{uhf-ee} follows.

In \cite[Proposition~5.1]{FaToTo:Descriptive} 
it was proved that the computation of the theory of a \cstar-algebra 
is a Borel function from the Borel space of separable
\cstar-algebras into the Borel space of their theories. Therefore $\Th(A)$ is in the terminology of \cite{FaToTo:Descriptive} 
a smooth invariant. However, the isomorphism relation of nonunital separable UHF algebras 
  is (in a very natural way) Borel-equireducible with the isomorphism relation of rank one torsion-free abelian groups. 
  Since the latter relation is not smooth, \eqref{non-unit-ee} follows. 
 \eqref{af-ee} follows from \eqref{non-unit-ee} by taking unitizations.  \eqref{kirch-ee} follows by observing that the relation of isomorphism of (UCT) Kirchberg algebras is not smooth by the analogous result for abelian groups and the Kirchberg--Phillips classification theorem.
  \end{proof} 
  
\section{Ultraproducts}\label{S.Ultraproducts} 
The definition of the ultraproduct of metric structures is given in \cite[\S 5]{BYBHU} and for multi-sorted structures in
 \cite[\S 4.1]{FaHaSh:Model2}. We repeat it here for the convenience of the reader.
 
 Fix a set $I$, an ultrafilter $\cU$ on $I$ and a language $\calL$.  For each $i \in I$, fix an $\calL$-structure $M_i$.  We wish to define $M = \prod_\cU M_i$, 
 the \emph{ultraproduct}\index{ultraproduct} of the $M_i$'s with respect to $\cU$.
 
 For each sort $S \in \calL$, suppose that $d^S_i$ is the metric on $S(M_i)$.  Let 
 \[
 (S(M),d) := \prod_\cU (S(M_i),d^S_i)
 \]
  as a metric space ultraproduct. More precisely,  we take $\prod_I S(M_i)$ together with the pseudo-metric 
  \[
  d = \lim_{i \rightarrow \cU} d^S_i
  \]
   $S(M)$ is the quotient of $\prod_I S(M_i)$ by $d$.
 
 For each function symbol $f \in \calL$, we define $f^M$ coordinate-wise on the appropriate sorts.  The uniform continuity requirements on $f$ are used critically to see both that this is well-defined and that $f^M$ has the necessary continuity modulus.  Similarly, for each relation symbol $R \in \calL$, we define $R^M := \lim_{i \rightarrow \cU} R^{M_i}$.  That this is well-defined and uniformly continuous follows from the uniform continuity requirements of the language.
 
 In the case of \cstar-algebras, this definition is equivalent to the 
 standard definition of ultraproduct. Suppose $I$ is an index set, $\cU$ is an ultrafilter on $I$, 
  and $A_i$, for $i\in I$, are \cstar-algebras. Elements $a$ of  $\prod_{i\in I} A_i$ are norm-bounded 
  indexed families  $(a_i: i\in I)$. On $\prod_i A_i$ define  
  \[
\textstyle  c_{\cU}=\{a\in \prod_i A_i: \lim_{i\to \cU} \|a_i\|=0\}. 
  \]
This   is a two-sided, self-adjoint, norm-closed ideal, and the quotient algebra
\[
\prod_{\cU}   A_i:=\prod_i A_i/c_{\cU}
\]
is the ultraproduct associated to $\cU$.  To see that the \cstar-algebra definition coincides with the metric structure definition, it suffices to see that the unit ball of $\prod_\cU A_i$ is the ultraproduct with respect to $\cU$ of the unit balls of $A_i$ for $i \in I$.
If all algebras $A_i$ are equal to some~$A$ then the ultraproduct is 
called \emph{ultrapower}\index{ultrapower} and denoted $A^{\cU}$. 
One identifies $A$ with its diagonal image in the ultrapower and 
often considers the \emph{relative commutant}\index{relative commutant} of $A$ in its ultrapower, 
\[
A'\cap A^{\cU}:=\{b\in A^{\cU}: ab=ba\text{ for all }a\in A\}. 
\]
Ultraproducts and relative commutants 
play a key role in the study of the theory of \cstar-algebras.  Model theoretically this is aided by the following theorem:

 \begin{thm}[{\protect \L}o\'s' Theorem]\label{Los}\index{L@\L o\'s' Theorem}
Suppose $M_i$ are $\calL$-structures for all $i \in I$, $\cU$ is an ultrafilter on $I$, $\varphi(\bar x)$ is a formula and $\bar m$ in $M = \displaystyle\prod_\cU M_i$ then
$$
\varphi^M(\bar m) = \lim_{i \rightarrow \cU} \varphi^{M_i}(\bar m_i)
$$
\end{thm}

\begin{proof} This is proved by a straightforward induction on the complexity of $\varphi$, following Definition~\ref{formula}. 
\end{proof}

Suppose $\Sigma$ is a set of sentences.  
Then we say $\Sigma$ is \emph{satisfiable}\index{satisfiable} if it is \emph{consistent}\index{consistent} i.e., satisfied by some model. 
It is \emph{finitely satisfied}\index{finitely satisfied}  if every finite subset of $\Sigma$ is consistent.  The approximation of $\Sigma$ is the set of sentences $\{ |\varphi| \!\!\dotminus \e : \varphi \in \Sigma, \e > 0 \}$.  $\Sigma$ 
is  \emph{finitely approximately satisfiable}\index{finitely approximately satisfiable} if the approximation of $\Sigma$ is finitely satisfied (i.e., every finite subset of the  approximation of  $\Sigma$ is consistent).  As a consequence of \L o\'s' Theorem, we have the compactness theorem in its exact and approximate versions.

\begin{thm}[Compactness Theorem]\index{Compactness Theorem}
 The following statements are equivalent for a set of sentences~$\Sigma$ \label{T.Compactness}
\begin{enumerate}
\item $\Sigma$ is satisfiable.
\item $\Sigma$ is finitely satisfiable.
\item $\Sigma$ is finitely approximately satisfiable.
\end{enumerate}
\end{thm}

\begin{proof} Only the implication from (3) to (1) requires a proof. 
We reproduce this standard compactness argument. 
Assuming $\Sigma$ is finitely approximately satisfiable, let $I$ be the set 
of all pairs $(F,\e)$ such that $F\subseteq \Sigma$ is finite and $\e>0$. 
By (3) for each $(F,\e)$ we can fix $A_{F,\e}$ such that $|\varphi^A|<\e$ for all $\varphi\in F$. 
The family of all sets of the form 
\[
X_{(F,\e)}:=\{(G,\delta)\in I: F\subseteq G\text{ and } \delta<\e\}
\]
can be extended to an ultrafilter $\cU$ on $I$. By Theorem~\ref{Los}
the ultraproduct $\prod_{\cU} A_{(F,\e)}$ satisfies $\Sigma$. 
\end{proof} 

\subsection{Atomic and Elementary Diagrams}\label{S.Diagram}
The following definition is central.
\begin{definition}
Suppose that $M$ is a substructure of $N$ in some language $\calL$.  We say that $M$ is an \emph{elementary submodel}\index{elementary submodel} of $N$, write $M \prec N$, and say that $N$ is an \emph{elementary extension}\index{elementary extension} of $M$, if for every formula 
$\varphi \in \cF_{\calL}$ and every $\bar m$ in $M$, $\varphi^M(\bar m) = \varphi^N(\bar m)$.   
If $f:M \rightarrow N$ is an embedding of $M$ into $N$ and the image of $M$ is an elementary submodel of $M$ then we say that $f$ is an \emph{elementary embedding}.\index{elementary embedding}
\end{definition}

The notion of elementary submodel is much stronger than $A$ being a subalgebra of $B$. For example, it implies that the center of $A$ is equal to the intersection of 
the center of $B$ with $A$. By \L o\'s' theorem (Theorem~\ref{Los}), we have
\[
A\prec A^{\cU}
\]
for every $A$ and every ultrafilter $\cU$.  

Suppose that $M$ is an $\calL$-structure.  We introduce an expansion of the language, $\calL_M$, which has a constant $c_m$ for every $m \in M$.  In this language, we define the theory $\Elem(M)$,  
the \emph{elementary diagram}\index{elementary diagram} of~$M$, as
\[
\{\varphi(c_{\bar m}) \colon \bar m\text{ in }M,  \varphi\text{ is any formula, and } \varphi^M(\bar m) = 0 \}.
\]
Here $c_{\bar m}$ is the sequence of constants corresponding to $\bar m$.

Elementary extensions of $M$ correspond to $\calL_M$-structures satisfying $\Elem(M)$, as follows.
If $f:M \rightarrow N$ is an elementary map between two $\calL$-structures then~$f$ can be used to expand $N$ to an $\calL_M$-structure satisfying $\Elem(M)$ by defining $c^N_m := f(m)$ for all $m \in M$.  Conversely, if $N$ is a model of $\Elem(M)$ then the map $f:M \rightarrow N$ defined by $f(m) := c_m^N$ is an elementary embedding of $M$ into the reduct of $N$ to the language $\calL$ (the reduct is the structure obtained by retaining only the functions and relations from $\calL$).  We often say that this reduct satisfies $\Elem(M)$ by which we mean that it can be expanded to an $\calL_M$-structure which satisfies $\Elem(M)$.  It is typical to use $m$ for the constant $c_m$ when the meaning is clear.  With these provisos, we state the following result:

\begin{thm}\label{T.Elem} 
Suppose that $M$ and $N$ are $\calL$-structures. The following are equivalent:
\begin{enumerate}
\item There is an $\calL$-structure $N'$ such that $N \prec N'$ and $M \prec N'$.
\item There is an $\calL$-structure $N'$ such that $N \prec N'$ and $N'$ satisfies $\Elem(M)$.
\item For every formula $\varphi(\bar x)$, $\bar m$ in $M$ and $\e > 0$ with $\varphi^M(\bar m) = 0$ there is
$\bar n$ in $N$ such that $|\varphi^N(\bar n)| \leq \e$.
\end{enumerate}
\end{thm}

We note that (1) is equivalent to 
 the assertion that $M$ and $N$ are elementarily equivalent by the Keisler--Shelah
 Theorem, \cite[Theorem~5.7]{BYBHU}; 
 this is not needed in the proof. It may also be worth noting that  conditions (2) and (3) are, unlike (1),  not obviously symmetric.

\begin{proof} 
(2) implies (1) is clear.   To see that (1) implies (3), note that from (3), $N'$ satisfies $\inf_{\bar x} |\varphi(\bar x)| = 0$ so the same holds in $N$.  To show that (3) implies (2), we follow a proof very similar to the proof of Theorem \ref{T.Compactness}.  Let $I$ be the set of pairs $(F,\e)$ where $F$ is a finite subset of $\Elem(M)$ and $\e > 0$.  The family of all sets of the form
\[
X_{(F,\e)} := \{ (G,\delta) \in I : F \subseteq G, \delta < \e \}
\] 
can be extended to an ultrafilter $\cU$ on $I$.  Using (3), for every $(F,\e) \in I$, we can expand $N$ to an $\calL_M$-structure $N_{(F,\e)}$ which satisfies 
\[
\max\{ \varphi^{N_{(F,\e)}}(\bar m) : \varphi(\bar m) \in F \} \leq \e
\]
Let $N' := \prod_\cU N_{(F,\e)}$; the reduct of $N'$ to $\calL$ is an elementary extension of~$N$ and by construction, $N'$ satisfies $\Elem(M)$.
\end{proof} 

The theory $\Diag(M)$, known as the \emph{atomic diagram}\index{atomic diagram}\index{Diag@$\Diag(M)$} of $M$, is similar to the elementary diagram of $M$.  It is again defined 
in the language~$\calL_M$ as
\[
\Diag(M) := 
\{\varphi(\bar m) : \bar m\text{ in }M,  \varphi\text{ is quantifier-free, and } \varphi^M(\bar m) = 0 \}.
\]
In the same way as for the elementary diagram, we have that $N$ satisfies $\Diag(M)$ if and only if there is an embedding of $M$ into $N$.
The following Theorem will be important later (\S\ref{S.Tensorial}) when we talk about $D$-stability, for a strongly self-absorbing $D$; it has a similar proof to Theorem \ref{T.Elem}

\begin{thm}\label{T.Diag} 
Suppose that $M$ and $N$ are $\calL$-structures. The following are equivalent:
\begin{enumerate}
\item There is an $\calL$-structure $N'$ such that $N \prec N'$ and $M$ embeds in $N'$.
\item There is an $\calL$-structure $N'$ such that $N \prec N'$ and $N'$ satisfies $\Diag(M)$.
\item For every quantifier-free formula $\varphi(\bar x)$, $\bar m$ in $M$ and $\e > 0$ with $\varphi^M(\bar m) = 0$ there is
$\bar n$ in $N$ such that $|\varphi^N(\bar n)| \leq \e$. 
\end{enumerate}
\end{thm}

\section{Elementary classes and preservation theorems}\label{S.Elementary}
For a theory $T$, $\Mod(T)$\index{Mod@$\Mod(T)$} is the class of all models of $T$. 
The \emph{theory}\index{theory} of a class $\cC$ of $\calL$-structures, denoted $\Th(\cC)$ is the intersection of theories of its elements.
We call a class $\cC$ of $\calL$-structures \emph{elementary}\index{elementary class}
(or \emph{axiomatizable})\index{axiomatizable} if $\cC = \Mod(T)$ for some theory~$T$.  
This is equivalent to being equal to $\Mod(\Th(\cC))$.
The following theorem is often helpful in determining if a class is elementary.  We say that $\cC$ is \emph{closed under ultraroots} if whenever $A^\cU \in \cC$ for some structure $A$ and ultrafilter $\cU$ then $A \in \cC$.  

\begin{thm} \label{T.Ax} For a class~$\cC$ of $\calL$-structures, the following are equivalent:
\begin{enumerate}[leftmargin=*]
\item $\cC$ is an elementary class.
\item $\cC$ is closed under isomorphisms, ultraproducts and elementary submodels.
\item $\cC$ is closed under isomorphisms, ultraproducts and ultraroots.
\end{enumerate}
\end{thm}

\begin{proof}
The directions (1) implies (2) implies (3) are straightforward so in principle we only need to show (3) implies (1).  Nevertheless we will give a proof of the apparently stronger implication (2) implies (1) to highlight the use of the elementary diagram.  Suppose that we have a class $\cC$. Set $T := \Th(\cC)$, that is, the set of $\calL$-sentences $\varphi$ such that $\varphi^M = 0$ for all $M \in \cC$.   Fix a model $M$ of $T$.  We need to produce $N \in \cC$ such that $M$ embeds elementarily into $N$.  To this end we fix a finite set $S$ in the elementary diagram of $M$ and $\e > 0$.  By considering the maximum of the formulas involved in $S$, we can assume that $S$ contains a single formula $\varphi(\bar m)$ where $\varphi^M(\bar m) = 0$.  Hence $M$ does not satisfy $\inf_{\bar x} |\varphi(\bar x)| \geq \e$, so that the sentence $\e\ \dotminus inf_{\bar x} |\varphi(\bar x)|$ is not in $T$.  This means that there must be some $N^S_\e \in \cC$ and $\bar n \in N^S_\e$ such that $|\varphi(\bar n)| < \e$ in $N^S_\e$.  Let's expand $N^S_\e$ to an $\calL_M$-structure by interpreting $c_{\bar m}$ as $\bar n$ and letting $c_{m'}$ for any other $m' \in M$ be anything we want.  As in the proof of Theorem \ref{T.Compactness}, let $\cU$ be a suitably chosen ultrafilter over $I$, the set of all pairs $(S,\e)$.  If $N = \prod_\cU N^S_\e$ then we can
define $\iota: M \rightarrow N$ by $\iota(m) := \langle c_m^{N^S_\e} : (S,\e) \in I \rangle/\cU$.  By construction and \L o\'s Theorem, we see that $N$ satisfies $\Elem(M)$ which is what we want.
 In order to see that (3) implies (1), repeat the previous proof and note that $M \equiv N$.  By the Keisler--Shelah theorem (see \cite[Theorem~5.7]{BYBHU}), 
 there is some ultrafilter $\cU$ such that $M^\cU \cong N^\cU$ and since $N^\cU \in \cC$, $M \in \cC$ by closure under ultraroots.
\end{proof} 

It should be noted that in Theorem \ref{T.Ax} one has to consider ultrafilters on arbitrary sets (\cite{shelah1992vive}). 
Here is a good example of how this theorem can be used to show a class of \cstar-algebras is elementary.

\begin{example}\label{Ex.rr0} 
The class of algebras of real rank zero is elementary.  A \cstar-algebra $A$ 
has \emph{real rank zero}\index{real rank zero} if the set of all invertible self-adjoint operators is dense among all self-adjoint operators; equivalently the set of self-adjoint elements with finite spectrum is dense among all self-adjoint operators. In \cite{brown1991c}, it is proved that $A$ is real rank zero if whenever $x,y \in A$ are two positive elements and $\e > 0$ such that $\|xy\| < \e^2$ then there is a projection $p$ such that 
\[
\|px\| < \e \mbox{ and } \|(1-p)y\| < \e.
\]  This statement can be written as a single  sentence (see \S\ref{S.rr0.revisited}) even though it seems to involve $\e$'s over which we cannot quantify.   
Even if one cannot see how to write the previous fact as a sentence, it is clear that any class satisfying this property is closed under ultraproducts.  Moreover, from the definition of real rank zero, if $A^\cU$ is real rank zero then so is $A$.  These two facts are enough to see that the class of real rank zero \cstar-algebras is elementary.  
\end{example}

\begin{definition} Let $\Psi$ be a set of sentences as in Definition~\ref{def:complexity}. A class $\cC$ of \cstar-algebras is said to be \emph{$\Psi$-axiomatizable} 
\index{axiomatizable!P@$\Psi$-axiomatizable} if there exists $\Phi\subseteq\Psi$ such that $A\in\cC$ if and only if $\varphi^A=0$ for all $\varphi\in\Phi$. If $\Psi$ is the set of all $\sup$-sentences, we say that $A$ is universally axiomatizable\index{axiomatizable!universally axiomatizable}.
We similarly define existentially axiomatizable (corresponding to $\inf$-sentences) and \aea.\index{axiomatizable!existentially axiomatizable}\index{axiomatizable!\aea}
\end{definition}

The following standard facts will be useful.

\begin{prop} \label{P.Ax}
\begin{enumerate}
\item A class $\cC$ is universally axiomatizable if and only if it is elementary and closed under taking substructures.  
\item A class $\cC$ is existentially axiomatizable if and only if it is elementary and closed under taking superstructures.  
\item 
A class $\cC$ is \aea{} if and only if it is elementary and closed under taking inductive limits. 
\end{enumerate}
\end{prop} 

\begin{proof}
(1) If $A$ is a substructure of $B$  and $\phi$ is a universal sentence then clearly $\phi^A\leq \phi^B$. It is  then straightforward to see that if $\cC$ is universally axiomatizable 
then it is closed under taking substructures. 
In order to prove the converse suppose $\cC$ is an elementary class closed under 
taking substructures. 
Let 
\begin{align*}
T&:=\{\phi: B\models \phi\text{ for all $B\in \cC$}\},\text{ and} \\
T_\forall&:=\{\phi: \text{ $\phi$ is universal and $B\models \phi$ for all $B\in \cC$}\}.
\end{align*}
We shall prove that every model of $T_\forall$ is a substructure of a model of~$T$. 
Suppose $A\models T_\forall$. 
With the atomic diagram  as in  \S\ref{S.Diagram} consider the theory 
$T_1:=T\cup \Diag(A)$ in the language expanded by constants for elements of $A$. 
To see that  $T_1$ is consistent,  fix a  finite fragment $T_0$ of~$T_1$. 
Let $\phi_j(\bar a)$, for $1\leq j\leq n$, be all the sentences in the intersection of $T_0$ with $\Diag(A)$
(by adding dummy variables if necessary we may assume that these sentences use the same tuple $\bar a$ from A). 
Let $r$ be greater than the maximal value allowed for $\max_j \phi_j(\bar x)$, and define the sentence
\[
\psi:= \sup_{\bar x} ( r \dotminus \max_{j\leq n}\phi_j(\bar x))
\]
which evidently has value $r$ in $A$.
Let $\e>0$. As $A\models T_\forall$, there exists $B \in \cC$ such that $\psi^B >r-\e$ (or else, $\psi^B\dotminus (r-\e)$ would be a universal sentence satisfied by all $B \in \cC$ but not by $A$).
This means that there exists a tuple~$\bar b$ of the appropriate sort 
 such that $\phi_j(\bar b)^B<\e$. 
This shows that $T_1$ is finitely approximately satisfiable, and therefore by the 
Compactness Theorem (Theorem~\ref{T.Compactness}), consistent.
 If $C$ is a model of $T_1$ then, being a model of~$T$, it belongs to~$\cC$ and, being 
 a model of 
  $\Diag(A)$,  it contains an isomorphic copy of~$A$. 
  
  We have proved that    an arbitrary model of $T_\forall$ is a substructure of a model of $T$, 
  and therefore, since $\cC$ is closed under substructures, it is universally axiomatizable.

Clause (2) can be proved along the similar lines as (1)  or (with some  care) deduced from it.

We now sketch a proof of (3). 
It is straightforward to see that if $\cC$ is \aea{} then it is closed under inductive limits.  In the other direction, suppose that $\cC = \Mod(T)$ for some theory $T$.  Let $T_{\forall\exists}$ be the set of $\forall\exists$-sentences $\varphi$ such that $\varphi^M = 0$ for all $M \in \cC$.  We would like to see that any model of $T_{\forall\exists}$ is a model of $T$.

\begin{claim}
For any $M_0 \models T_{\forall\exists}$ there is $N \in \cC$ and $M_1$ such that $M_0 \subseteq N \subseteq M_1$ and $M_0 \prec M_1$.
\end{claim}

With this claim we finish the proof.  For any $M \models T_{\forall\exists}$, let $M_0 := M$ and construct a sequence $M_0 \subseteq N_0 \subseteq M_1 \subseteq N_1 \subseteq \ldots$ such that $N_i \in \cC$ and $M_i \prec M_{i+1}$ for all $i$.  Since $\cC$ is closed under inductive limits, $\overline{\bigcup N_i} \in \cC$.
Using induction on the length of a given formula, one can show that $M \prec \overline{\bigcup M_i}$.
Since $\overline{\bigcup M_i} = \overline{\bigcup N_i}$, it follows that $M \in \cC$.

It remains to prove the claim.  We construct $N\models T$ so that
$M_0 \subseteq N$ 
 and whenever $\varphi(\bar x)$ is a sup-formula then for all $\bar m$ in $M_0$,
$\varphi^{M_0}(\bar m) = \varphi^N(\bar m)$.  To this end, define $\Diag_\forall(M_0)$\index{Diag@$\Diag_\forall(M)$}, the
 \emph{sup-diagram}\index{sup-diagram} of $M_0$, 
 in the language $L_{M_0}$, to be
\[
\{\varphi(\bar m) : \varphi \text{ is a sup-formula}, \bar m\text{ in }M_0, \text{ and } \varphi^{M_0}(\bar m) = 0 \}.
\]

Notice that $\Diag(M_0) \subseteq \Diag_\forall(M_0)$ so if $T \cup \Diag_\forall(M_0)$ is consistent then we will be able to find the necessary $N$.  If not then (by Theorem \ref{T.Compactness}) after some rearranging we can find a sup-formula $\varphi$, $\bar m$ in $M_0$ and $\e > 0$ such that 
\[
T \models \varphi(\bar m) \geq \e \text{ but } \varphi^{M_0}(\bar m) = 0
\]
 With a little more rearranging this says that $T \models \sup_{\bar x} (\e \dotminus \varphi(\bar x))$ and since~$\varphi$ is a sup-formula, $\e \dotminus \varphi$ is an inf-formula.  This says that
$\sup_{\bar x} (\e \dotminus \varphi(\bar x))$ is in $T_{\forall\exists}$ which contradicts the fact that $M_0 \models T_{\forall\exists}$ and $\varphi^{M_0}(\bar m) = 0$.

So choose $N \models T \cup \Diag_\forall(M_0)$.  To construct $M_1$, we need to see that $\Elem(M_0) \cup \Diag(N)$ is consistent.  If not then for some quantifier-free formula $\varphi(\bar x,\bar y)$, $\bar n$ in $N$, $\bar m$ in $M_0$ and $\e > 0$ we have
\[
\Elem(M_0) \models \varphi(\bar n,\bar m) \geq \e \text{ but } \varphi^N(\bar n,\bar m) = 0
\]
But then $\Elem(M_0) \models \sup_{\bar x}(\e \dotminus \varphi(\bar x,\bar m))$ and since $N \models \Diag_\forall(M_0)$, this is a contradiction.
\end{proof}

There are corresponding results at the level of formulas called \emph{preservation theorems}. 
To say that a formula $\varphi$ is \emph{equivalent}\index{formula!equivalent} relative to a class~$\cC$ to a particular kind of formula from some set $\Phi$ ($\sup$, $\inf$ or positive formulas for instance) means that $\varphi$ can be uniformly approximated by formulas from $\Phi$; see \S\ref{S.DefinablePred} for a justification of this terminology.  That is, for every $\e$ there is a formula $\varphi_\e \in \Phi$ such that $\|\varphi - \varphi_\e\| < \e$ in all models in $\cC$.  The proofs of these results are routine generalizations of the preservation theorems from classical model theory; see for instance \cite[\S 6.5]{Hodg:Model}

\begin{prop}\label{prop:preservation} For 
a formula~$\varphi(\bar x)$,
relative to an elementary class $\cC$, 
\begin{enumerate}[leftmargin=*]
\item $\varphi$ is equivalent to a $\sup$-formula if and only if for all $M \subseteq N$ with $M, N \in \cC$ and $\bar m$ in $M$, $\varphi^M(\bar m) \leq \varphi^N(\bar m)$.
\item $\varphi$ is equivalent to a $\inf$-formula if and only if for all $M \subseteq N$ with $M, N \in \cC$ and $\bar m$ in $M$, $\varphi^M(\bar m) \geq \varphi^N(\bar m)$.
\item $\varphi$ is equivalent to a positive formula if and only if for all $M, N \in \cC$, surjective homomorphisms $f:M \rightarrow N$ and $\bar m$ in $M$,\\ $\varphi^M(\bar m) \geq \varphi^N(f(\bar m))$.
\end{enumerate}
\end{prop}

\begin{proof} We will sketch a proof of (3).  By introducing constants for the free variables in $\varphi$ (and working in a language which includes these constants), we can assume that $\varphi$ is a sentence and we would like to see that if its value decreases under surjective homomorphisms then it is equivalent to a positive sentence.
  We first prove a useful lemma:
\begin{lemma}
Suppose that $M, N \in \cC$ and whenever $\varphi$ is a positive sentence, if $\varphi^N = 0$ then $\varphi^M = 0$. Then there are $M^*,N^*$ such that $M \prec M^*$, $N \prec N^*$ and a surjective homomorphism $f:N^* \rightarrow M^*$.
\end{lemma}

The proof of the lemma is an argument involving increasing chains of models which relies on the following two facts: If $M,N$ are as in the hypothesis of the lemma then
\begin{enumerate}
\item   for every $a\in N$, there is $M_1$ such that $M \prec M_1$ and $b \in M_1$ such that for every positive formula $\varphi(x)$, if $\varphi^N(a) = 0$ then $\varphi^{M_1}(b) = 0$.
\item for every $a\in M$, there is $N_1$ such that $N \prec N_1$ and $b \in N_1$ such that for every positive formula $\varphi(x)$, if $\varphi^{N_1}(b) = 0$ then $\varphi^M(a) = 0$.
\end{enumerate}

To prove the first, let 
\[
t^+(a) = \{ \varphi(x): \varphi \text{ is positive}, \varphi^N(a) = 0\}.
\]
We will show that $\Elem(M) \cup t^+(a)$ is consistent.  Otherwise (using Theorem \ref{T.Compactness}) there is a positive formula $\varphi(x)$ and $\e > 0$ such that $\varphi^N(a) = 0$ but
$\Elem(M) \models \inf_x \varphi(x) \geq \e$.  The sentence $\inf_x \varphi(x)$ is a positive sentence which is 0 in $N$ but is not 0 in $M$, in contradiction to the hypothesis of the lemma.
Taking $M_1$ to be a model of $\Elem(M) \cup  t^+(a)$ we get the necessary elementary extension of~$M$.

Similarly, for the second fact, let 
\[
t^-(a) := \{ \varphi(x) \dotminus \e : \varphi \text{ positive}, \varphi^M(a) > \e > 0 \}.
\]
We wish to show that $\Elem(N) \cup t^-(a)$ is consistent.  Otherwise, we can produce a positive formula $\varphi(x)$ such that
$\Elem(N) \models \sup_x \varphi(x) \leq \e$ but $\varphi(a) > \e$.
The sentence $\sup_x \varphi(x) \dotminus \e$ is a positive, is 0 in $N$ but is not 0 in $M$, in contradiction to the hypothesis of the lemma.
Letting $N_1$ be a model of $\Elem(N) \cup t^-(a)$ produces the necessary elementary extension of $N$.

To prove the lemma, we go back and forth, growing $M$ so that every element of $N$ has an image, and growing $N$ so that every element of $M$ has a preimage.
The use of positive formulas encode whether some $b \in M$ can be the image of an element $a \in N$, and vice versa.
This concludes the proof of the lemma.

Now to prove the main proposition, suppose that $\varphi$ is a sentence whose value decreases under surjective homomorphisms between models from $\cC$.  Fix a theory $T$ such that $\cC = \Mod(T)$ and consider the set
\[
\Delta := \{ \psi : T \cup \varphi \models \psi, \psi \text{ positive}\}.
\]
Now suppose that for every $\e > 0$, $\Gamma_\e := T \cup \Delta \cup \{\varphi \geq \e\}$ is inconsistent. Then for every $\e > 0$ there is a positive sentence $\psi_\e$ such that modulo $T$,
$\varphi \models \psi_\e$ and $\psi_\e \models \varphi \leq \e$.  From this we conclude that $\varphi$ is uniformly approximated by positive sentences.

To show that the above conclusion holds, we now suppose for a contradiction that for some $\e$, $\Gamma_\e$ is consistent.  Let $M$ be a model of~$\Gamma_\e$.  As above, we 
define~$\Thminus(M)$\index{Th@$\Thminus(M)$} to be
\[
\{ \psi \geq \delta : \psi \text { positive}, \psi^M > \delta > 0 \}.
\]
We would like to show that $T \cup \Thminus(M) \cup \{\varphi\}$ is consistent.  Otherwise, there would be a positive sentence $\psi$ and $\delta > 0$ such that 
$T \cup \varphi \models \psi \leq \delta$ but $\psi^M > \delta$.  But the sentence $\psi \dotminus \delta$ is positive and in $\Delta$ which contradicts the choice of $M$.

Now, since we have established that $T \cup \Thminus(M) \cup \{\varphi\}$ is consistent, so choose a model $N$ of $T \cup \Thminus(M) \cup \{\varphi\}$.  If $\psi$ is a positive sentence and $\psi^N = 0$ then $\psi \geq \delta$ is not in $\Thminus(M)$ for any $\delta > 0$ and so we must have $\psi^M = 0$.  This puts us in the position of the hypothesis of the lemma and so we can produce $M^*, N^*$ and a surjective homomorphism $f:N^* \rightarrow M^*$ as in the conclusion of the lemma.  But $\varphi^N = \varphi^{N^*} = 0$ and $\varphi^M = \varphi^{M^*} \geq \e$ which contradicts $\varphi$ decreases under surjective homomorphisms.
\end{proof}

\section{Elementary classes of \cstar-algebras} \label{S.Axiomatizable}
We now list axiomatizable classes of \cstar-algebras. In many instances proofs of axiomatizability 
depend on more advanced material; the section in which the proof appears is listed immediately after the result. 
Some of the axiomatizations are given in the language of unital \cstar-algebras. 
As in classical model theory, the complement of an axiomatizable class is not always axiomatizable.  However, it is often useful to note
when it is and we will take this up in \S\ref{Co-elementarity}.

\begin{theorem}\label{Summary}
The following classes of \cstar-algebras
are universally axiomatizable.
\begin{enumerate}[leftmargin=*]
\item Abelian algebras (see \S\ref{S.Abelian}). 
\item $n$-subhomogeneous algebras, for every $n\geq 1$ (see \S\ref{S.SubHom}). 
\item Unital finite algebras  and unital stably finite algebras (see \S\ref{S.SF}). 
\item Unital tracial algebras (see \S\ref{S.tracial.0}). 
\item Unital algebras with a character (see \S\ref{S.character}). 
\item Projectionless algebras (see \S\ref{S.UP}).  
\item Unital projectionless algebras (see \S\ref{S.UP}).  
\item Unital algebras not containing a unital copy of $M_n(\bbC)$ for given $n\geq 2$ (see \S\ref{S.Mn}). 
\item MF algebras (see \S\ref{S.MF}). 
\pushcounter
\end{enumerate}

The following classes of \cstar-algebras
are existentially axiomatizable. 
\begin{enumerate}[leftmargin=*]
\popcounter
\item Nonabelian algebras (see \S\ref{S.nonabelian}). 
\item Algebras that are not $n$-subhomogeneous, for $n\geq 1$ (see \S\ref{S.not.SubHom}). 
\item Unital algebras containing a unital copy of $M_n(\bbC)$ for given $n\geq 2$ (see \S\ref{S.Mn}). 
\item Infinite algebras  (see \S\ref{S.Infinite}). 
\pushcounter
\end{enumerate}

The following classes of \cstar-algebras
are  \aea. 
\begin{enumerate}[leftmargin=*]
\popcounter
\item Algebras with real rank zero 
(see \S\ref{Ex.rr0}, \S\ref{S.rr0}, \S\ref{S.rr0.revisited},  and \S\ref{S.rr0.1}) and abelian algebras with real rank $\leq n$ (see \S\ref{S.rr0.1}). 
\item Unital algebras with stable rank $\leq n$ for $n\geq 1$ (see \S\ref{S.sr}). 
\item Simple, purely infinite algebras (see \S\ref{S.PI}). 
\item Unital \cstar-algebras with strict comparison of positive elements by traces or 2-quasitraces (see \S\ref{S.elfunsc}).
\item Unital \cstar-algebras with the $\bar k$-uniform strong Dixmier property for a fixed $\bar k$ (see Lemma~\ref{L.usDP}). 
\pushcounter
\end{enumerate}

The following classes of \cstar-algebras are also elementary. 
\begin{enumerate}[leftmargin=*]
\popcounter
\item \cstar-algebras with real rank greater than zero (see  \S\ref{S.rr0.1}). 
\item  Unital \cstar-algebras with stable  rank $>n$, for every $n\geq 1$ (see \S\ref{S.sr} and 
 \cite{farah2016axiomatizability}). 

\pushcounter
\end{enumerate}
\end{theorem}

Tensorial absorption of distinguished \cstar-algebras (such as the algebra of compact operators, 
or strongly self-absorbing algebras) is an important property of \cstar-algebras. 
By the main result of \cite{Gha:SAW*} a nontrivial ultrapower is \emph{tensorially indecomposable},\index{tensorially indecomposable} 
i.e.,  it cannot be isomorphic to a nontrivial 
tensor product. (A different proof was given in \cite{kania2015c}.)
 Therefore every \cstar-algebra $A$ has an elementarily equivalent
\cstar-algebra, $A^{\cU}$, which does not tensorially absorb any infinite-dimensional \cstar-algebra (this is 
trivially true if $A$ is finite-dimensional itself).

We say that a class $\cC$ of \cstar-algebras is \emph{separably axiomatizable}\index{axiomatizable!separably axiomatizable} if 
there is an axiomatizable class $\cC_0$ such that for separable \cstar-algebras $A$ one 
has $A\in \cC$ if and only if $A\in \cC_0$. 
If $\cC_0$ is \aea{} we say that $A$ is separably \aea. 

\begin{theorem}\label{T.ssa}
The following classes of \cstar-algebras are separably \aea. 
\begin{enumerate}
\popcounter
\item $D$-stable algebras, where $D$ is a strongly self-absorbing algebra (see \S\ref{Ex.ssa-stable}) .
\item Approximately divisible algebras (see \S\ref{S.AD}). 
\item Stable algebras (see \S\ref{S.Stable}). 
\end{enumerate}
\end{theorem}

We now proceed with some of the more straightforward proofs of elementarity.  In order to improve the readability of many formulas in the language of \cstar-algebras,  we will write $\sup_x$ and $\inf_x$ when the intended meaning is to quantify over the unit ball.
 
\subsection{Abelian algebras} \label{S.Abelian} 
The class of abelian \cstar-algebras is clearly closed under ultraproducts and subalgebras and so is a universally axiomatizable class.  It is also clearly axiomatized by the single sentence, 
 $\sup_x \sup_y \|[x,y]\|$. 
 
 \subsection{Non-abelian algebras}\label{S.nonabelian}The class of non-abelian \cstar-algebras 
is axiomatizable by 
\[
1\dotminus\sup_x (\|x\|^2 - \|x^2\|). 
\]
This is because a \cstar-algebra is nonabelian if and only if it contains a nilpotent element (\cite[II.6.4.14]{Black:Operator}). 
If $a\neq 0$ is a nilpotent contraction and $m$ is minimal such that $a^m=0$, then $\frac 1{\|a\|^{m-1}} a^{m-1}$
satisfies the above sentence. 
On the other hand, every normal element satisfies $\|a^n\|=\|a\|^n$ and therefore 
the above sentence has value 1 in every abelian \cstar-algebra. 

\subsection{Real rank zero again}\label{S.rr0}
Real rank zero algebras were introduced in Example~\ref{Ex.rr0}. 
It is clear from the definition of real rank zero that the class of real rank zero algebras is closed under inductive limits.  Since it is an elementary class, it is \aea{}. See also \S\ref{S.rr0.revisited} for an explicit formula. 

\subsection{$n$-subhomogeneous} (Pointed out by Bruce Blackadar.)\label{S.SubHom}
 For $n\geq 1$ a \cstar-algebra $A$ 
 is \emph{$n$-homogeneous}\index{n@$n$-homogeneous}\index{C@\cstar-algebra!homogeneous}  if each  of its irreducible representations is $n$-dimensional. 
 It is \emph{$n$-subhomogeneous}\index{n@$n$-subhomogeneous}\index{C@\cstar-algebra!subhomogeneous}  if each  of its irreducible 
 representations is $k$-dimensional for some $k\leq n$. A \cstar-algebra is \emph{homogeneous} (subhomogeneous) if it is $n$-homogeneous ($n$-subhomogeneous) 
 for some $n$.\index{C@\cstar-algebra!homogeneous}  
 Therefore $A$ is $n$-subhomogeneous if and only if it is isomorphic to a closed, though not necessarily unital, subalgebra
 of a direct product  of copies of $M_n(\bbC)$. 
  
   The class of $n$-subhomogeneous 
  \cstar-algebras is axiomatized by a single universal sentence $\alpha$. 
  This is a consequence of the \emph{Amitsur--Levitzki theorem}, 
  stating that the $^*$-polynomial in $2n$ variables,
with $\Sym(2n)$ denoting the symmetric group on a $2n$-element set
and $\sgn(\sigma)$ denoting the parity of a permutation $\sigma$,
\[
\sum_{\sigma\in \Sym(2n)}(-1)^{\sgn(\sigma)}\prod_{j=1}^{2n} x_{\sigma(j)} . 
\]
  is identically 0 in $M_n(\bbC)$ 
  but not in $M_k(\bbC)$ for any larger $k$  (see also \cite[IV.1.4.5]{Black:Operator}). 
We claim that a \cstar-algebra being $n$-subhomogeneous is axiomatized by  
\[
\vphiALn:=\sup_{x} 
\bigg\|\sum_{\sigma\in \Sym(2n)}(-1)^{\sgn(\sigma)}\prod_{j=1}^{2n} x_{\sigma(j)} \bigg\|. 
\]
An algebra $A$ is not $n$-subhomogeneous if and only if it has   an irreducible representation 
on a Hilbert space of dimension $\geq n+1$. 
This is by the above equivalent to having a $2n$-tuple  $\bar a$ in 
 $A$ such that 
 \[
 \bigg\|\sum_{\sigma\in \Sym(2n)}(-1)^{\sgn(\sigma)}\prod_{j=1}^{2n} x_{\sigma(j)} \bigg\|=r>0
 \]

\subsection{Non-$n$-subhomogeneous algebras}\label{S.not.SubHom}
For every   $n\geq 1$ the class of algebras that are not $n$-subhomogeneous is 
existentially axiomatizable. 
Fix $m>n$. With the notation from \S\ref{S.SubHom} we have
that $\vphiALn^{M_m(\bbC)}=r_{m,n}>0$. By embedding $M_m(\bbC)$ into the top left corner of $M_k(\bbC)$ for $k>m$
we see that $r_{m,n}\leq r_{k,n}$. Therefore there exists $r>0$ such that $A$ is not $n$-subhomogeneous 
if and only if $\vphiALn^A\geq r$ and being not $n$-subhomogeneous is axiomatized by 
$r\dotminus \vphiALn$.

Alternatively,  one can prove that 
 $A$ is not $n$-subhomogeneous if and only if there is an element $x\in A$ 
 such that $\|x^k\|=1$ if $k\leq n$ and $x^{n+1}=0$. 
 That is, non-$n$-subhomogeneity is axiomatized by the following 
 sentence
\[
\varphi:=\inf_{x}( \|x^{n+1}\|+\max_{1\leq i\leq n}(1 -\|x^i\|))
\]

\subsection{Tracial \cstar-algebras} \label{S.tracial.0}The class of unital, tracial \cstar-algebras (that is, unital \cstar-algebras which have a trace) is universally axiomatizable in the language of unital \cstar-algebras. 
Since this class is closed under taking ultraproducts and subalgebras, 
this is a consequence of Theorem \ref{T.Ax} and Proposition \ref{P.Ax}. 
To see that an ultraproduct $\prod_{\cU} A_i$ 
 of tracial \cstar-algebras $A_i$ for $i\in J$ is tracial note that, if $\tau_i$ is a trace on $A_i$ for each $i$, then
 \[
 \tau(a):=\lim_{i\to \cU} \tau_i(a_i)
 \]
 (where $a:=(a_i)/\cU$) is a trace on the ultraproduct.
 Having a constant for the multiplicative unit assures that subalgebras are unital. 
 
 We give an explicit axiomatization in \S\ref{S.CP}.

\subsection{\cstar-algebras with a character}\label{S.character}
A \emph{character}\index{character}
 is a  $^*$-homomorph-ism from a \cstar-algebra $A$ into the complex numbers. 
Although our main preoccupation is with simple algebras, the existence of a character for the relative 
commutant of $A$ in its ultrapower is an important property (see e.g.,  \cite{kirchberg2014central}). 

 The class of unital \cstar-algebras with a character is universally axiomatizable in the language of unital \cstar-algebras. 
Just like in the case of \S\ref{S.tracial.0}, since this class is closed under taking ultraproducts and subalgebras, 
this is a consequence of Theorem \ref{T.Ax} and Proposition \ref{P.Ax}.   See \S\ref{S.CP} for an explicit axiomatization.

\section{Downward L\"owenheim-Skolem}
Elementary submodels
(see \S\ref{S.Diagram}) 
   provide separable counterexamples to statements to which nonseparable 
counterexamples exist (see \cite[II.8.5]{Black:Operator} for an overview). 
 A handy test for checking when $A$ is an elementary submodel of $B$ is the following:

\begin{thm}[Tarski--Vaught]\label{TV}\index{Tarski--Vaught test}
Suppose that $A$ is a submodel of $B$ and that for every $r$ and every formula $\varphi(x,\bar b)$ with $\bar b$ in $A$, if $(\inf_x \varphi(x,\bar b))^B < r$ then there is $a \in A$ such that $\varphi^B(a,\bar b) < r$. Then $A \prec B$.
\end{thm}

\begin{proof} The proof is by induction on the definition of formulas.  The cases of atomic formulas are automatic since $A$ is a submodel of $B$ and the case of connectives follows easily.  The only issue is quantifiers.  Since the quantifiers are dual to one another, it suffices to check only $\inf$.  So if our formula is $\inf_x \varphi(x,\bar y)$ and $\bar a$ in $A$ then by induction $\inf_x \varphi^A(x,\bar a) \geq \inf_x \varphi^B(x,\bar a)$.  But the operative condition guarantees that if $\inf_x \varphi^B(x,\bar a) < r$ then $\varphi^B(a,\bar a) < r$ for some $a \in A$ which means that by induction, $\varphi^A(a,\bar a) < r$ and so $\inf _x\varphi^A(x,\bar a) < r$.  So $\inf_x \varphi^A(x,\bar a) \leq \inf_x\varphi^B(x,\bar a)$ and we are done.
\end{proof}

A \emph{density character}\index{density character}
 of a metric structure is the minimal cardinality of a dense subset. In particular 
 a structure is separable if and only if its density character is at most $\aleph_0$. 

For a fixed theory $T$ in a language $\calL$ we define a seminorm on the set $\ccF^{\bar x}_{\calL}$ of formulas in the free variables $\bar x$, as follows:
\[
\| \varphi \|_T := \sup \{ |\varphi^M(\bar a)| \colon M \mbox{ satisfies $T, \bar a$ }\in M \}.
\]
Since polynomials with rational coefficients are dense among all continuous functions on any compact subset of $\mathbb{R}^n$  by the Stone--Weierstrass theorem, it is clear that the density character of the quotient of $\ccF^{\bar x}_{\calL}$ by this seminorm is at most $|\calL| + \aleph_0$ (remember that the range of any $\calL$-formula is bounded).  If $\calL$ has only countably many sorts and $\ccF^{\bar x}_{\calL}$ is separable for all $\bar x$ then we say that $\calL$ is separable. We will return to this quotient space in \S\ref{S.DefinablePred}.
The following theorem is of general importance;
its use in the context of \cstar-algebras was noticed by Blackadar  (\cite{blackadar1978weak}).

\begin{thm}[Downward L\"owenheim--Skolem]\label{DLS}
Suppose that $\calL$ is separable and $X$ is a subset of an $\calL$-structure $B$ of density character $\lambda$.  Then there is $A \prec B$ with $X \subseteq A$ such that $A$ has density character $\lambda \geq \aleph_0$.
In particular, if $X\subseteq B$ and $X$ is separable, then 
there exists a separable $A$ such that $X\subseteq A$ and $A\prec B$.
\end{thm}

\begin{proof} The proof is a closure argument. 
Fix a countable dense set of formulas $\ccF_0$ of $\ccF^x_\calL$ with respect to the seminorm described above. 
Now construct a sequence $X=X_0\subseteq X_1\subseteq X_2\dots$ of subsets of 
$B$, all of the same cardinality, such that for every $n$, every  $\varphi(x,\bar y)\in \ccF_0$ and 
every $\bar b$ in $X_n$ of the appropriate sort, one has $a \in X_{n+1}$ such that   
\[
\varphi(a, \bar b)^B \leq \inf_x\varphi(x,\bar b)^B + \frac{1}{n}
\]
Since $\ccF_0$ is countable and we may assume that $X_0$ is of size $\lambda$, one can guarantee that all $X_n$ are of cardinality $\lambda$. 
With the sequence $X_n$ constructed in this way,  the closure of their union is an 
elementary submodel of $B$ as required. 
\end{proof}

 The following is an immediate consequence of Theorem~\ref{DLS}. 
 
 \begin{corollary} If $B$ is nonseparable then $B$ is the direct limit of $\{B_\lambda\}_{\lambda\in \Lambda}$, where $B_\lambda\prec B$ is separable and $\Lambda$ is some index set. If $B=A^\cU$ for a separable $A$ then each $B_\lambda$ can be chosen containing~$A$. \qed
 \end{corollary} 
 
 If in addition $B$ has density character $\aleph_1$ (the least uncountable cardinal) then 
 $\Lambda$ can be chosen to be $\aleph_1$ with its natural ordering and  
$B$ is an inductive limit of an increasing sequence of its separable elementary submodels. 
 Even better the proof of Theorem~\ref{DLS} shows that whenever $B$ is 
represented by a continuous increasing $\aleph_1$-sequence of its separable submodels then 
the set of indices corresponding to elementary submodels of $B$ is closed in the order topology and 
unbounded in $\aleph_1$. 
Ultrapowers of separable algebras associated with nonprincipal ultrafilters on $\bbN$ comprise a 
particularly important class of algebras which, assuming the Continuum Hypothesis, are of density character~$\aleph_1$.

\section{Tensorial absorption and elementary submodels} \label{S.Tensorial} 

\subsection{Strongly self-absorbing \cstar-algebras}
Strongly self-absorbing \cstar-algebras play a particularly important role in the Elliott classification programme. 
Two $^*$-homomorphisms $\Phi_1$ and $\Phi_2$ from $A$ into a unital \cstar-algebra  $B$ 
 are \emph{approximately unitarily equivalent}\index{approximately unitarily equivalent}  if there is a net of unitaries $u_\lambda$ in $B$, for $\lambda\in \Lambda$,
 such that  $\Ad u_\lambda\circ \Phi_2$ converges to $\Phi_1$ in the point-norm topology. 
 This is equivalent to stating that the set of conditions $\|xx^*-1\|=0$, $\|x^*x-1\|=0$
and $\|\Phi_1(a)-x\Phi_2(a)x^*\|=0$, for all $a\in A$,  is consistent.

A unital \cstar-algebra $A$ is \emph{strongly self-absorbing}\index{strongly self-absorbing} (\cite{ToWi:Strongly}) if $A\cong A\otimes A$ and 
the $^*$-homomorphism $a\mapsto a\otimes 1_A$ from $A$ into $A\otimes A$ is approximately unitarily equivalent 
to an isomorphism between $A$ and $A\otimes A$.
There are only a few kinds of known strongly self-absorbing \cstar-algebras (see \cite{ToWi:Strongly}): the 
Jiang--Su algebra  $\cZ$, UHF algebras of infinite type, Kirchberg algebras $\cO_\infty$, $\cO_\infty\otimes \mbox{UHF}$ (again 
for UHF algebras of infinite type), and $\cO_2$.    
These algebras 
have remarkable 
model-theoretic properties (see
\cite{FaHaRoTi:Relative} and, for the II$_1$ factor variant, \cite{Bro:Topological}).

 For unital \cstar-algebras $A$ and  $D$
we say that $A$ is \emph{$D$-stable}\index{D@$D$-stable} if $A \cong A \otimes D$, 
where for definiteness $\otimes$ denotes the minimal (also referred to as the spatial) tensor product.

\begin{lemma}\label{L.tensor} 
If $D=C^{\otimes \bbN}$ and $A$ is $D$-stable 
then $A\prec A\otimes C$ via the map sending $a$ to $a \otimes 1$.
\end{lemma} 

\begin{proof} Since $A \cong A \otimes D$, it suffices to show that $A \otimes D \prec A \otimes D \otimes C$ via the map sending $a \otimes d$ to $a \otimes d \otimes 1$.
We use the Tarski--Vaught test, Theorem \ref{TV}. 
Let $\alpha(x, \bar y)$ be a formula and let $\bar a$ be a parameter in $A \otimes D$. 
We need to check that if $(\inf_{x} \alpha(x, \bar a))^{A\otimes D \otimes C} < r$
then for some $b \in A \otimes D$, $\alpha^{A \otimes D \otimes C}(b,\bar a) < r$. 
 Choose
 $c\in A\otimes D \otimes C$ such that $\alpha(c,\bar a)< r$.   Let $\e := r - \alpha(c,\bar a)$ and pick $\delta$ to correspond to the continuity modulus of $\alpha$ associated to $\e/3$. By choosing suitable approximations, we can
 assume that there is an $n$ such that if we write $D$ as $C^{\otimes n} \otimes C \otimes C^{\otimes \bbN}$ then there is
\[
\bar a'\in A\otimes C^{\otimes n} \otimes 1 \otimes 1 \text{ with } \|\bar a - \bar a'\| < \delta
\]
 and 
 \[
 c' \in A \otimes C^{\otimes n} \otimes 1 \otimes 1 \otimes C \text{ with } \|c - c'\| < \delta. 
 \]

Let $\gamma$ be the automorphism of $A\otimes C^{\otimes n} \otimes C \otimes C^{\otimes \bbN} \otimes C$ which interchanges the two individual copies of $C$ and fixes $A$, $C^{\otimes n}$ and $C^{\otimes \bbN}$. Then $\gamma$ does not move $\bar a'$, $\gamma(c')$ is in the original copy of $A \otimes D$ and 
 $\alpha(\gamma(c'),\bar a') = \alpha(c',\bar a')$.   It follows by the choice of $\delta$ that $\alpha(\gamma(c'),\bar a) < r$.
 \end{proof} 

Since a strongly self-absorbing algebra $D$ satisfies 
 $D \cong D^{\otimes\bbN}$, we have the following. 
 
\begin{coro} \label{C.ssa}
If $D$ is strongly self-absorbing and $A$ is $D$-stable 
then $A\prec A\otimes D$ for all $A$ via the map sending $a$ to $a \otimes 1$ for all $a \in A$. 
\end{coro}

Here is another nice consequence of Lemma \ref{L.tensor}. 
In the following lemma $p$ ranges over the prime numbers. 

\begin{lemma} If $A=\bigotimes_p M_{p^{k(p)}}(\bbC)$ and 
 $B=\bigotimes_p M_{p^{l(p)}}(\bbC)$ then the following are equivalent
 \begin{enumerate}
 \item $A\prec A\otimes B$. 
 \item $A\equiv A\otimes B$. 
 \item For every $p$ such that $l(p)>0$ we have that $k(p)=\infty$; 
 that is, $B\cong A\otimes B$. 
\end{enumerate}
\end{lemma} 

\begin{proof} Clearly (1) implies (2). If (3) fails for some $p$ 
then in $A$ there is no unital copy of 
$M_{p^{k(p)+1}}(\bbC)$ 
but in $A\otimes B$ there is a unital copy of $M_{p^{k(p)+1}}(\bbC)$. 
As having a unital copy of $M_n(\bbC)$ is axiomatizable for every $n$ 
by \S\ref{S.Mn}, (2) fails as well. 

If (3) holds then the assertion easily follows from Lemma \ref{L.tensor}. 
\end{proof} 

Here is another useful lemma.
Recall that $M_n(A)$ is identified with $A\otimes  M_n(\bbC)$ and $A$
is identified with $A\otimes 1_n$. 

 \begin{lemma} If 
 $A\prec M_n(A)$ 
 then $A\prec M_{n^\infty}(A)$. 
 \end{lemma} 
 
 \begin{proof} By Lemma \ref{L.Mn} we have that $M_{n^k}(A)\prec M_{n^{k+1}}(A)$
for all $k$. Therefore the direct limit $\lim_n M_{n^k}(A)$ (with unital connecting maps) 
is an elementary chain and $A$ is its elementary submodel. 
\end{proof} 

\begin{prop} If $A$ and $B$ are \cstar-algebras, then the space of all $\Phi\colon A\to B$
that are elementary embeddings is closed under point-norm limits
and unitary conjugacy. 
\end{prop}

\begin{proof} We only need to prove that a point-norm limit of elementary embeddings is 
an elementary embedding.   Fix a net $\Phi_\lambda\colon A\to B$ of
elementary embeddings and let $\Phi$ be their limit. 

If $\alpha(\bar x)$ is a formula and $\bar a$ is a tuple in $A$ we need to check 
(with the obvious interpretation of $\Phi(\bar a)$) that we have
 $\alpha(\bar a)^A=\alpha(\Phi(\bar a))^B$.  
 By the uniform continuity of $\alpha^A$, we 
 have that $\alpha(\Phi(\bar a))^B=\lim_\lambda\alpha(\Phi_\lambda(\bar a))^B=\alpha(\bar a)^A$ 
\end{proof} 

\begin{coro} Assume $D$ is a strongly self-absorbing algebra. Then every endomorphism of  $D$ is an elementary embedding. 
\end{coro} 

\begin{proof} This is because all endomorphisms of $D$ are approximately unitarily 
equivalent to the identity (see e.g.,  \cite[Corollary 1.12]{ToWi:Strongly}). 
\end{proof}

\subsection{Stable algebras} \label{S.Stable} 
Recall that algebra $A$ is \emph{stable}\index{stable} if it tensorially absorbs the 
algebra of compact operators $\cK$, $A\otimes \cK\cong A$.\index{K@$\cK$} 
We show that the class of separable stable algebras is \aea.  
  This is an immediate consequence of a result of Hjelmborg and R\o rdam (\cite{HjRo:On}). 
 
 \begin{prop} \label{T.HjR} 
 For a separable \cstar-algebra $A$ the following are equivalent. 
\begin{enumerate} 
\item The map $a\mapsto \begin{pmatrix} a & 0 \\ 0 & 0 \end{pmatrix}$ 
is an elementary embedding of $A$ into $M_2(A)$. 
\item $A$ is stable. 
\item The theory of $A$ includes the sentence
\[
\sup_x \inf_y \|x^*x y^*y\|+ \|x^*x-yy^*\|. 
\]
\end{enumerate}
\end{prop} 

\begin{proof} 
The fact that (2) is equivalent to (3) is in \cite[Theorem~2.1]{HjRo:On}. 

Note that the sentence in (3) states that for every positive element $a=x^*x$ of norm $\leq 1$ 
there exists a positive 
element $b=yy^*$ and $y$ of norm $\leq 1$ such that both $ab$  and $a-y^*y$ are 
 as small as possible. 
 
(2) implies (1) is similar to the proof of Lemma \ref{L.tensor}. 
To see that (1) implies (3), pick $x \in A$. Then $a:=x^*x$ is mapped to 
$\begin{pmatrix} a & 0 \\ 0 & 0 \end{pmatrix}$, and so 
$y:=\begin{pmatrix} 0 & a^{1/2}  \\ 0 & 0 \end{pmatrix}$ witnesses 
the sentence from (3) in $M_2(A)$. 
Since $A$ is an elementary submodel, we can find an approximate copy of $y$ in $A$. 
\end{proof} 

\chapter{Definability and $A^{\eq}$}

\section{Expanding the definition of formula: definable predicates and functions} \label{S.DefinablePred}
In the present subsection we shall demonstrate that the language can be significantly enriched without losing desirable properties 
of the logic. We shall introduce definable predicates, definable functions,  and definable sets and show that allowing quantification over
definable sets is but a handy abbreviation. Our treatment is essentially the  one given in \cite[Definition 9.16]{BYBHU} except that we define
definability via quantification and not distance functions and
we do not require the underlying theory to be complete. 
Particularly important examples of definable sets are sets corresponding to approximation properties 
(see \S\ref{S.AP}).

\subsection{Definable predicates} \label{S.Definable.Predicates}
Fix a theory $T$ in a language $\calL$. Remember that we defined a seminorm on $\ccF^{\bar x}_{\calL}$ (see \S\ref{formula} for the definition of~$\ccF^{\bar x}_{\calL}$) by:
$$
\| \varphi \|_T := \sup \{ |\varphi^M(\bar a)| \colon M \mbox{ satisfies  $T, \bar a$  in }M \}.
$$
Let $\ccW^{\bar x}_{T}$ (or just $\ccW^{\bar x}$) be the Banach algebra obtained by quotienting and completing~$\ccF^{\bar x}_{T}$ with respect to $\| \cdot \|_T$.  On every model $M$ of $T$, every $P \in \ccW^{\bar x}_{T}$ 
is a uniform limit of formulas.
 More precisely, it 
can be interpreted as a uniformly continuous function $P^M$ such that there are formulas $\varphi_n \in \ccF^{\bar x}_{\calL}$ for which 
\[
\|P - \varphi_n\|_T \leq 1/n
\]
 for all $n$. 
 In \cite[Definition 9.1]{BYBHU}, such $P$ are called {\em definable predicates}\index{definable predicate}; when we wish to emphasize the uniformity across all models of $T$ or some elementary class $\cC$, we will say that $P$ is a definable predicate relative to $T$ or $\cC$. It follows from this definition that a definable predicate is uniformly continuous and every definable predicate is a uniform limit of functions given by formulas on every model of $T$. We will frequently make no distinction between definable predicates and formulas.

It will be useful to extend the terminology of quantifier-free, sup, inf, $\forall\exists$, and positive formulas 
(see Definition~\ref{def:complexity}) 
to definable predicates.\index{definable predicate!quantifier-free, sup, inf, $\forall\exists$,\dots}
  In all cases, the class of definable predicates consists of uniform limits from the corresponding class of formulas.

\section{Expanding the definition of formula: definable sets}\label{S.Definable}
One of the key elements of applying continuous logic to \cstar-algebras is recognizing which properties expressed by formulas can be quantified over.  This leads to the following definition 
(made relative to an elementary class~$\cC$) 
and its relationship with weakly stable relations.  We will use the notation $\Met$ for the category of bounded metric spaces with isometries as morphisms.

\begin{definition} \label{D.Assignment}
Suppose $\cC$ is an elementary class with theory $T$. Let $m\geq 1$  and $D_j$, for $1\leq j\leq m$, be sorts of the language of $\cC$.
A functor $X:\cC \rightarrow \Met$ is called a {\em uniform assignment}\index{uniform assignment of closed sets} relative to the class $\cC$ 
(or relative to the theory $T$) if in addition to being a functor it also satisfies the following: 
\begin{enumerate}
\item for each $A \in \cC$, $X(A)$ is a closed subset of $\prod_{j=1}^m D_j^A$, and
\item for each $f:A \rightarrow B$, $X(f) = f \restriction X(A)$.
\end{enumerate}
 
Such a functor $X$ is called a \emph{definable set}\index{definable set} if, for all formulas $\psi(\bar x, \bar y)$, the functions defined for all $A \in \cC$ 
by
\[ 
\sup_{\bar x \in X(A)} \psi^A(\bar x, \bar y) \quad \text{and} \quad \inf_{\bar x \in X(A)} \psi^A(\bar x, \bar y) 
\]
are definable predicates relative to $\cC$.
\end{definition}

An important example of a uniform assignment of closed sets relative to a class $\cC$ arises from the zero-sets of definable predicates.  If $\varphi(\bar x)$ is a definable predicate then, for $A \in \cC$, the zero-set of $\varphi$ in $A$ is\index{Z@$Z^A(\varphi)$}
\[
Z^A(\varphi) = \{ \bar a \in A : \varphi^A(\bar a) = 0 \}.
\]
Written without the superscript, $Z(\varphi)$ is the assignment of the zero-set of $\varphi$ to the structures in a given class.
Although for any definable predicate this is a uniform assignment, much of the work of continuous model theory goes into determining which such assignments are definable.  For instance, although the set of normal elements in a \cstar-algebra forms a zero-set of this form (the zero-set of the formula $\|xx^* - x^*x\|$) this assignment is not a definable set; see Proposition \ref{P.NormalNotDefinable}.

For a closed set $X$ in a structure $A$, we write, for the distance from $a\in A$   of the appropriate sort 
to $X$,\index{d@$d(\cdot ,X)$!distance to $X$}
 \[
 d(a,X):=\inf\{d(a,x) \colon x \in X\}.
 \]
To be clear, we will always use the sup metric on the finite product of sorts.  In particular, in the case of \cstar-algebras we will write $\|\bar a\|$ for the sup of $\|a_i\|$ when $\bar a$ is an $n$-tuple. An important equivalent definition of definable set is captured by the following Theorem.

\begin{thm}\label{T.def-dist} Given an elementary class $\cC$ and a uniform assignment $Z$,
$Z$ is a definable set if and only if $d(\bar x,Z(A))$ is a definable predicate relative to $\cC$.
\end{thm}

\begin{proof} This is proved  in \cite[Theorem~9.17]{BYBHU} 
so we only sketch the proof.  For the direction from left to right, 
$d(\bar x,Z(A))$ is given by $\inf_{\bar y \in Z(A)} d(\bar x,\bar y)$.  In the other direction, suppose $d(\bar x, Z(A))$ is a definable predicate. Fix a formula $\psi(\bar x,\bar y)$ and let  
 $\alpha$ be a modulus of continuity for $\psi$ in the variable $\bar y$. 
 That is $\alpha$ is a continuous function such that
\begin{enumerate}
\item $\lim_{t \rightarrow 0^+} \alpha(t) = 0$, and
\item for all $\bar x$, $|\psi(\bar x,\bar y) - \psi(\bar x,\bar z)| \leq \alpha(d(\bar y,\bar z))$ holds in $\cC$.
\end{enumerate}
Since $d(\bar y,Z(A))$ is a definable predicate, the 
 formula 
 \[
\inf_{\bar y} (\psi(\bar x,\bar y) + \alpha(d(\bar y,Z(A))))
\] 
is one as well. 

Fix $A\in \cC$. 
Clearly 
$\inf_{\bar y}(\psi(\bar x,\bar y) + \alpha(d(\bar y,Z(A))))\leq 
\inf_{\bar y \in Z(A)} \psi(\bar x,\bar y)$. 
In order to prove the converse inequality 
fix  $\bar x\in A$
and $r$ such that 
$\inf_{\bar y} (\psi(\bar x,\bar y) + \alpha(d(\bar y,Z(A))))<r$. 
If $\bar b\in A$ and $\bar c\in Z(A)$ are such that  
$\psi(\bar x,\bar b) + \alpha(d(\bar b,\bar c))<r$ then 
$\psi(\bar x,\bar c)\leq\psi(\bar x,\bar b)+\alpha(d(\bar b,\bar c))<r$.  
Since $\bar x\in A$ was arbitrary we have 
\[
\inf_{\bar y\in Z(A)}\psi(\bar x,\bar y)=\inf_{\bar y} (\psi(\bar x,\bar y) + \alpha(d(\bar y,Z(A)))).
\]
 The case of $\sup$ is handled similarly.
\end{proof}

Suppose we are given a uniform assignment $X$ which is a definable set relative to some elementary class $\cC$ with $m$ and $D_j$, for $j\leq m$, 
as in Definition~\ref{D.Assignment}. If we consider the formula $\psi(\bar x,\bar y) := \max_{i\leq m} \{d_j(x_j,y_j)\}$ where $d_j$ is the metric symbol on~$D_j$ then for each $A \in \cC$ we have that
\[
X(A) = \{ \bar a \in \prod_j D_j^A : \inf_{\bar x \in X(A)} \psi^A(\bar x, \bar a) = 0 \}.
\]
That is, $X(A)$ is the zero-set of a definable predicate.  The advantage of the definition of definable set given above is that we need not a priori recognize the definable set as a zero-set or even related to the formulas of the given language.
It should be noted that a definable set $X$ is often the zero-set of a natural definable predicate other than $d(\bar x, X(A))$. 
If this predicate is quantifier-free definable then we say that $X$ itself is \emph{quantifier-free definable}.\index{quantifier-free definable}

In most of our applications $\cC$ will be the class of all \cstar-algebras.  
By the definition, definable sets are bounded. In the case of \cstar-algebras we extend the definition to allow unbounded subsets as
follows (for  a general treatment of unbounded definable sets see \cite{Lut:PhD}). Let $\Met^*$ be the category of metric spaces (that are not necessarily bounded).

\begin{definition}
Let $\cC$ be an elementary class of \cstar-algebras, $m\geq 1$ and $n\geq 0$, 
and suppose we are given a functor $X:\cC \rightarrow \Met^*$ such that
\begin{enumerate}
 \item for each $A \in \cC$, $X(A)$ is a closed subset of $A^m\times \bbC^n$, and
 \item  for any $^*$-homomorphism $\Psi:A \rightarrow B$, $X(f) = \Psi \restriction X(A)$.
 \end{enumerate}
Then $X$ is called a \emph{definable set}\index{definable set!unbounded} if, for all formulas $\psi(\bar x, \bar y)$, the functions defined for all $A \in \cC$ and $n \in \bbN$ 
by 
\[ 
\sup_{\bar x \in X(A), \|\bar x\|\leq n} \psi^A(\bar x, \bar y) \quad \text{and} \quad \inf_{\bar x \in X(A),\|\bar x\|\leq n} \psi^A(\bar x, \bar y) 
\]
are definable predicates relative to $\cC$.
\end{definition}

One standard trick for recognizing that an assignment is a definable set is to see that it is the range of a term; if it is, 
then the assignment is definable since it can be quantified over.  
More generally, for \cstar-algebras, if the set is the image of a definable set under a continuous function $f$  (in the sense of functional calculus, see \S\ref{S.cfc}) then it is also definable. 
 This is clear in the case when $f$ is a $^*$-polynomial. In  the case when $f$ is an arbitrary continuous function the statement follows from the complex Stone--Weierstrass theorem.

Following \cite{Mitacs2012} we say what it means for 
a definable predicate to be weakly stable.
\begin{definition}
A definable predicate $\psi(\bar x)$ is \emph{weakly stable}\index{weakly stable} relative to an elementary class $\cC$  if for every  
$\e>0$ there exists $\delta>0$ such that
for every  $A\in \cC$ and every 
$\bar a\in A^n$ with   $|\psi(\bar a)|<\delta$ there exists $\bar b\in A^n$ 
such that $\|\bar a - \bar b\|<\e$ and $\psi(\bar b)=0$. 
\end{definition}
This notion corresponds to \emph{weakly stable relations}\index{weakly stable!relations} (\cite[Chapter 4]{Lor:Lifting}).  
Equivalently, there is a continuous increasing function $u:[0,\infty) \rightarrow [0,\infty)$ such that $u(0) = 0$ and for all $A \in \cC$,
$d(\bar a,Z(\psi)) \leq u(\psi(\bar a))$ for $\bar a \in A^n$.
(The proof that these definitions are equivalent is similar to the proof of Theorem~\ref{T.def-dist}.)

Just like in case of definable predicates, functions and sets, certain formulas can be weakly stable only when relativized to 
members of a specific class of \cstar-algebras. For example, several predicates were shown to be weakly stable 
only when relativized to the class of \cstar-algebras with stable rank 1 in \cite{eilers1998stability}.

In \cite[\S 2]{Mitacs2012} `weakly stable relations' were referred to as `stable relations' and definable relations
were referred to as `uniformly definable relations.' This granted, we summarize the relationship between weakly stable relations and definable sets.  The equivalence of (2) and (3) was proved as \cite[Lemma~2.1]{Mitacs2012}.

\begin{lemma} \label{L.weaklystable} For an elementary class $\cC$ and a uniform assignment of closed sets $Z$ relative to $\cC$, the following are equivalent:
\begin{enumerate}
\item $Z$ is a definable set.
\item $d(\bar x,Z(A))$ is a definable predicate relative to $\cC$.
\item $d(\bar x,Z(A))$ is a weakly stable predicate relative to $\cC$.
\item $Z(A)$ is the zero-set of a weakly stable predicate relative to $\cC$.
\end{enumerate}
\end{lemma}

Here is a semantic approach to recognizing when you have a definable set.  It is a consequence of the Beth definability theorem and it is Corollary~\ref{C.sem.def} discussed in \S\ref{S.Beth} and so we will postpone the proof until then.

\begin{thm}\label{Th.sem.def}
Suppose that $\cC$ is an elementary class and $A \mapsto S^A$ is a uniform assignment of closed sets.  Then this assignment is a definable set if and only if for all ultrafilters $\cU$ on $I$, $A_i \in \cC$ for $i \in I$ and $A = \prod_\cU A_i$,  $S^A = \prod_\cU S^{A_i}$.
\end{thm}

\begin{example} 
\label{Ex.1} 
Using Lemma \ref{L.weaklystable} and standard results (\cite[\S 4.1]{Lor:Lifting})
one can show that each of the following subsets is definable---and even 
quantifier-free-definable---in any \cstar-algebra $A$.  
\begin{enumerate}[leftmargin=*]
\item The set $A_{\mathrm{sa}}$ of all self-adjoint elements of $A$ is an unbounded definable set. 
 It is the zero-set of  $\alpha(x):=\|x-x^*\|$.  The set of self-adjoint elements is the range of $(a + a^*)/2$ in any \cstar-algebra.
\item The set of all positive elements in $A$ is an 
unbounded definable set. 
 It is the range of $aa^*$ in any \cstar-algebra.
 \item The unit ball of $A$, as the zero-set of $\|x\|\dotminus 1$. 
 
\item The unit sphere  of $A$, as the zero-set of $|\|x\|- 1|$. 
 \item \label{Ex.1.projection} The set $\cP(A)$\index{P@$\cP(A)$} of all projections of $A$,  as the 
 zero-set of 
 \[
 \pi (x):=\|x^2-x\|+\|x^*-x\|.
 \] 
 To see a variety of sets of projections which are also definable, see \S\ref{S.sets.proj}.
\item (In a unital \cstar-algebra.) The set of all isometries, as the zero-set of $\|x^*x-1\|$. 
\item  (In a unital \cstar-algebra.) The set $U(A)$\index{U@$U(A)$} of all unitaries 
as the zero-set of $\|x^*x-1\|+\|xx^*-1\|$.   
\item \label{Ex1.PI} The set of all partial isometries, as the zero-set of $\pi(xx^*)$ (equivalently, 
of $\pi(x^*x)$), where $\pi $ is the formula defining projections as in~(\ref{Ex.1.projection}). 
\item \label{L.alpha-n} (See \cite[Lemma~2.2]{Mitacs2012} or \cite[Theorem 4.9]{Loring:stable-rel}) 
For every $n\geq 2$, the set of all $n^2$-tuples $x_{ij}$ for $1\leq i,j\leq n$ that
are matrix units of a copy of $M_n(\bbC)$, as the zero-set of 
\[
\alpha_n(\bar x):=\max_{1\leq i,j,k,l\leq n}\|\delta_{kl}x_{ij}-x_{ik} x_{lj}\|+\|x_{ij}-x_{ji}^*\|+|\|x_{11}\|-1|. 
\]
\item \label{L.alpha-n-u} (See \cite[Lemma~2.2]{Mitacs2012} or \cite[Theorem 4.9]{Loring:stable-rel}) 
For every $n\geq 2$, the set of all $n^2$-tuples $x_{ij}$ for $1\leq i,j\leq n$ that
are matrix units of a unital copy of $M_n(\bbC)$, as the zero-set of 
\[
 \alpha_n^u(\bar x):=\max\left (\alpha_n(\bar x), \bigg\|1-\sum_{i\leq n} x_{ii}\bigg\|\right ). 
\]
\item \label{Ex.1.alpha-F} Suppose $F$ is the finite-dimensional algebra $\bigoplus_{k=1}^n M_{m_k}(\bbC)$.  Let $J$ be the set
of all triples $(i,j,m_k)$ where $0 \leq i,j \leq m_k$ and $1 \leq k \leq n$.  $J$ indexes the matrix units of $F$.
There is a weakly stable formula $\alpha_F$ in the tuple of variables $x^{m_k}_{i,j}$ for $(i,j,m_k)\in J$ whose zero-set contains tuples that represent the matrix units
of a copy of $F$.
 
\item \label{Ex.1.orthogonal} The set of pairs $(p,q)$ of orthogonal projections is the zero-set of 
the formula $\pi_\perp(x,y)=\pi(x)+\pi(y)+\|xy\|$ (with $\pi$ as in  \eqref{Ex.1.projection}). 
\pushcounter
\end{enumerate}
By the standard methods 
the following closed sets are also definable. However, 
 their definitions are $\inf$-formulas instead of  quantifier-free formulas. 
\begin{enumerate}[leftmargin=*]
\popcounter
 \item \label{Ex.1.MvN} The set of pairs $(p,q)$ of Murray--von Neumann equivalent 
projections is the zero-set of 
the formula $\piMVN(x,y):=\pi(x)+\pi(y)+\inf_z\|x-zz^*\|+\|y-z^*z\|$ 
(with $\pi$ as in  \eqref{Ex.1.projection}). 
\item \label{Ex.1.pi-m} For every $m\geq 1$ the set of projections $p$ such that
there are $m$  projections $p_j$, for $1\leq j\leq m$,  such that $p=p_1$ is Murray--von Neumann 
equivalent to each $p_j$ and all   $p_j$ are orthogonal. 

\item \label{Ex.1.beta-F} (See \cite[Lemma~2.3]{Mitacs2012}.) 
For every finite-dimensional algebra $F$, the set of $(m+1)$-tuples 
$\bar a,b$ such that $\bar a$ belongs to a copy of $F$ whose unit is $b$ 
is the zero-set of a formula $\beta_{F,m}$. 
Adopting the notation from (\ref{Ex.1.alpha-F}) and using tuples of variables
$x^{m_k}_{i,j}$ for $(i,j,m_k) \in J$
then $\beta_{F,m}(\bar y,z)$
reads as follows (using $\alpha_F$ as above)
\begin{gather*}
\inf_{\bar x}\left ( \alpha_F(\bar x)+   \inf_{\bar \lambda}
\max_{1\leq l\leq m} \left(\left\|y_l-\sum_{(i,j,m_k)\in J} \lambda^{l,m_k}_{i,j} x^{m_k}_{i,j}\right\|, 
\left\|z - \sum_{i\leq m_k, k\leq n} x^{m_k}_{ii}\right\|\right) \right )
 \end{gather*}
 where $\bar \lambda$ ranges over tuples formed by
 $\lambda^{l,m_k}_{i,j}$ in the unit disk~$\bbD$ for $(i,j,m_k) \in J$ and $l \leq m$. (See Remark \ref{rmk.numbers-as-scalars} for an explanation of this formula).
  
\item Similarly to \eqref{Ex.1.beta-F} (and more simply),  for every finite-dimensional algebra $F$ and every~$m$, 
the set of $m$-tuples  that belong
to a unital copy of $F$ is the zero-set of a formula $\beta_{F,m}^u$. 
\pushcounter
\end{enumerate}
\end{example} 

\begin{prop}\label{P.NormalNotDefinable}
The set of normal elements,
the zero-set of\\ $\|xx^*-x^*x\|$, 
is not 
definable in the class of all \cstar-algebras, but it is definable in the theory of  AF algebras (see \S\ref{S.Elementary}).
\end{prop}

\begin{proof}
Let $H$ be an infinite-dimensional Hilbert space and let $s$ be the unilateral shift associated to a distinguished basis of $H$. 
We shall prove that the set of normal elements of $B(H)$ is not definable. 
 There is a sequence $a_n$ of operators in $B(H)$ of finite-rank perturbations of $s$ such that $\lim_n \|a_n a_n^*-a_n^*a_n\|=0$
 (see e.g.,  \cite{Lor:Lifting}). Since each $a_n$ has Fredholm index equal to $-1$ and norm $1$, no normal operator $b$ 
 satisfies $\|a_n-b\|<1$. 

We turn to  the second statement, that the set of normal elements is definable in a restricted class of \cstar-algebras. 
Let $T$ be the theory of AF algebras i.e., the set of sentences true in all AF algebras and let $\cC$ be the class of all models of $T$.

Lin proved that for every $\e>0$ there exists $\delta>0$ such that 
if $a$ is contained in a finite dimensional \cstar-algebra $F$, $\|a\| \leq 1$ and $\|aa^*-a^*a\|<\delta$ then there is a normal 
element $b$ satisfying $\|b\| \leq 1$ and $\|a-b\|<\e$  (\cite{lin1994almost},  see also \cite{FR}). In other words, the predicate $\|xx^*-x^*x\|$ is weakly stable relative to the 
class of finite dimensional \cstar-algebras. Lin's theorem clearly extends to AF algebras and by \L o\'s' theorem
to the class axiomatized by the theory of AF algebras.
\end{proof}

\section{Expanding the language: imaginaries}\index{imaginaries} 

It is convenient to introduce additional sorts to our language 
called \emph{imaginary sorts}\index{sort!imaginary} over which we are allowed to quantify.  This will simplify our ability to express various notions but will not increase the expressive power of the language we are working with as this expanded language will be a \emph{strongly conservative extension}: every model in the original language can be expanded uniquely to a model of the new language.  We do this in three steps:

\subsection*{Countable products}  If $\calL$ is a language and $\bar S = \langle S_i : i \in \bbN \rangle$ is a countable sequence of sorts, we introduce a new sort $U := U_{\bar S}$, a metric symbol $d_U$ and unary functions $\pi_i:U \rightarrow S_i$.  Now if $M$ is an $\calL$-structure, we expand $M$ by interpreting $U$ as $\displaystyle\prod_{i \in \bbN} S_i^M$.  The choice of metric $d_U$ on $U$ is not canonical but for definiteness, let 
\[
d_U(\bar m,\bar n) := \sum_{i \in \bbN} \frac{d_i(m_i,n_i)}{B_i2^i}
\]
where $B_i$ is a bound on the metric $d_i$.  The function symbols $\pi_i$ are interpreted as projections i.e., for $\bar m$ in $U$, $\pi_i(\bar m) := m_i$.  The point of this construction is seen by recording two statements, each of which can easily be translated into sentences of continuous logic.  First of all, if $\bar m, \bar n \in U$ then for any $k$, 
\[
d_U(\bar m,\bar n) \leq  \frac{1}{2^k} + \sum_{i \leq k} \frac{d_i(m_i,n_i)}{B_i2^i} 
\]
This sentence guarantees that any $\calL$-structure $M$, expanded to include the sort $U$, will at least satisfy that the map $\Pi$ sending $U$ into $\prod_i S_i$ via the projections will be an injection.  Additionally, if we have $a_i \in S_i^M$ for $i \leq k$ then there is some $\bar m \in U$ with $\pi_i(\bar m) = a_i$ for all $i \leq k$.  Again, this can be expressed as a sentence in continuous logic and for any structure that satisfies both this and the previous sentence, since our $\calL$-structures are complete, $\Pi$ is a bijection; that is, $U$ will be interpreted as the product of the sorts $S_i$.

\begin{example} One way that the product of sorts is helpful is in interpreting continuous functions, 
 for instance, in trying to understand what model theoretic information is available in $K_1(A)$ for a \cstar-algebra $A$.    We could add the sort $S$ which is the product of $A_1$ countably many times; we index this product by $\bbQ' =\bbQ \cap [0,1]$.  Let's temporarily adopt the notation that $\check A$ for a \cstar-algebra $A$ is the structure $A$ expanded by the sort $A_1^{\bbQ}$ and that $\cC$ is the elementary class of all $\check A$.  Suppose that we have a continuous function $f:[0,1] \rightarrow A_1$. $f$ could then be interpreted as the sequence with $q$-entry $f(q)$.  The following proposition shows that the collection of all continuous functions from $[0,1]$ into $S$ is not a definable set but the set of continuous functions which are $N$-Lipschitz is definable by Theorem \ref{Th.sem.def}.  
 
 \begin{prop}
 If $X:\cC \rightarrow \Met$ is defined so that $X(\check A)$ is the set of all $N$-Lipschitz functions from $[0,1]$ to $A_1$ as described above then~$X$ is definable.
\end{prop} 

\begin{proof}
Fix an enumeration $q_j$, for $j\in \bbN$, of $\bbQ'$
and consider the associated weighted sum metric as defined above.  
 It is clear that the set of $N$-Lipschitz functions with this metric forms a closed subset of $A_1^{\bbQ'}$ for any \cstar-algebra and that the ultralimit of $N$-Lipschitz functions is again $N$-Lipschitz.  It remains to show that if we fix a ultrafilter $\cU$ on $I$, \cstar-algebras $A_i$ for all $i \in I$ and an $N$-Lipschitz function $f$ from $[0,1]$ into the unit ball of $\prod_\cU A_i$ then $f$ is an ultralimit of $N$-Lipschitz functions $f_i\colon\bbQ'\to A_i$ for all $i \in I$.  Toward this end, write $\bbQ'$ as an increasing union of finite sets $F_n$ and for each $s \in \bbQ'$, fix a representative $I$-tuple $\bar a^s\in\prod A_i$ so that $f(s) = \bar a^s/\cU$.  For each $n$, choose $X_n \in \cU$ such that if $s,t \in F_n$ then for all $ i \in X_n$, $\|a^s_i - a^t_i \| \leq \mu_n |s - t|$ where $\mu_n := N + 1/n$.  There are two cases for any $i \in I$: If $i \in X_n$ for all $n$ then we can define $f_i(t) := a^t_i$ for all $t \in \bbQ'$. If there are only finitely many $n$ such that $i \in X_n$ choose $n$ maximally so that $i \in X_n$ and for every $t \in F_n$, let 
 \[
 f_i(t) := \frac{N}{\mu_n} a^t_i
 \]
 and complete $f_i$ to a path by connecting the defined points linearly.  One now checks that $f_i$ is $N$-Lipschitz for all $i$ and $\lim_{i \rightarrow \cU} f_i = f$. \end{proof}
\end{example}

\subsection*{Definable sets} 
We already saw in \S\ref{S.Definable} that one can allow quantification over definable sets so it will make sense, and be convenient, to allow definable sets themselves as additional imaginary sorts.  Formally then, if we are working with models of some theory $T$ and $T$ knows that the zero-set of $\varphi(\bar x)$ is a definable set then we will add a sort $U_\varphi$, a metric $d$ and function symbols $\rho_i:U_\varphi \rightarrow S_i$ where $S_i$ is the sort of the variable $x_i$.  The intended interpretation will be that $\rho := \langle \rho_i \rangle$ is an injection from $U_\varphi$ into $\prod_i S_i$ and that $\varphi(\rho(x)) = 0$ for all $x \in U_\varphi$. The metric $d$ will be induced by the restriction of the metric on $\prod_i S_i$ via $\rho$.  All of this is easily expressed by sentences in continuous logic.

\subsection*{Quotients}
This is the most interesting of the imaginary sorts that we can add.  Suppose that $\varphi(\bar x,\bar y)$ is a formula where the variables $\bar x$ and $\bar y$ are distinct tuples of variables coming from a product of sorts $\bar S = \prod_i S_i$.  Further assume that in all $\calL$-structures, this formula is interpreted as a pseudo-metric on $\bar S$.  In this case, we add a sort $U$ intended to be the quotient of $\bar S$ by $\varphi$, a metric $d$ for $U$ and the function $\pi:\bar S \rightarrow U$.  Any $\calL$-structure $M$ is expanded to this new language by interpreting $U$ as the quotient of $\bar S^M$ by $\varphi$, $d$ as the metric on $U$ induced by $\varphi$ and $\pi$ is the projection map. The relevant details of this construction are captured by the statements that the range of $\pi$ is dense in $U$ and that for all $a, b \in \bar S, \varphi(a,b) = d(\pi(a),\pi(b))$, both of which can be expressed in continuous logic.  Here are two ways in which this construction is used.

\begin{example} Suppose that $\rho(\bar x,\bar z)$ is any formula; let $\varphi(\bar x,\bar y)$ be the formula $\sup_{\bar z} | \rho(\bar x,\bar z) - \rho(\bar y,\bar z)|$.  It is easy to see that in all $\calL$-structures, $\varphi$ is a pseudo-metric and the associated quotient is often called the set of \emph{canonical parameters}\index{canonical parameters}  for the formula $\rho$.
\end{example}

\begin{example} \label{Ex.MvN} This example will be useful when we discuss $K$-theoretic invariants in \S\ref{S.elfun}.  Suppose that $A$ is a \cstar-algebra.  We saw that the projections in $A$ form a definable set and so could be treated as an imaginary sort.  On this sort, we can think of the formula $\varphi(x,y)$ whose value is the infimum over all partial isometries $u$ of $\|x -u^*u\| + \|y - uu^*\|$.  
 By \cite[Lemma 2.1.3]{RoLaLa:Introduction}
 for projections $p$ and $q$, if $\|p - q\| < 2$ then $p \sim q$. 
 Therefore $\varphi(p,q)=0$ if and only if $p\sim q$ and $\varphi(p,q)\in \{0,2\}$ 
 for all projections $p$ and $q$ and  
  we conclude that $\rho$ is a pseudo-metric. 
  It follows that the canonical parameters of $\rho$ are exactly the equivalence classes modulo Murray-von Neumann equivalence and of course the quotient sort has the discrete metric.  This allows us at least to talk about the first part of the construction of $K_0(A)$ in model theoretic terms. 
  
One should keep in mind that   
in case of \cstar-algebras in which the Murray--von Neumann semigroup 
   fails the  cancellation property (see \cite[p. 36]{RoLaLa:Introduction}) 
   the canonical parameters of $\rho$ do not correspond to the   
elements of the positive part of $K_0$. 
\end{example}

\subsection*{$M^{\eq}$ and $T^{\eq}$}
We now put all of these imaginary sorts together.  If~$T$ is any theory, not necessarily complete, we could iteratively add all the sorts generated in the three ways described above: take products of countable sorts, add definable sets known to be definable by the theory and quotient by definable pseudo-metrics.
 The full expanded language will be called $\calL^{\eq}$ to match the tradition in discrete first order logic and the theory will be $T^{\eq}$.  If we begin with an $\calL$-structure $M$ and then expand by adding all products, definable sets and quotients, we call the structure $M^{\eq}$.  We will rarely need the entirety of $T^{\eq}$ or $M^{\eq}$ but we will need to recognize instances when the concept we are dealing with is captured by imaginaries or is expressible in~$T^{\eq}$. In general, if $\calL'$ contains $\calL$ and $T'$ is a theory in $\calL'$ which contains $T$ then we say that $T'$ is a strongly conservative extension of $T$ if the forgetful functor from $\Mod(T')$ to $\Mod(T)$ is an equivalence of categories. We record here the main abstract facts about $T^{\eq}$---essentially, that $T^{\eq}$ is the universal strongly conservative extension of $T$.  A proof can be found in \cite{Albert-thesis}.
 
 \begin{theorem}\label{T.conceptual-completeness} 
 Suppose that $T$ is a theory in the language $\calL$.  Then~$T^{\eq}$ is a strongly conservative extension of $T$ via the forgetful functor $F$.  Moreover, if $T'$ is any strongly conservative extension of $T$ via the forgetful functor $F'$ then there is a functor $G:\Mod(T') \rightarrow \Mod(T^{eq})$ such that $F \circ G = F'$.
 \end{theorem}

Although we shall need only a special case of Theorem~\ref{T.conceptual-completeness}, the Beth Definability Theorem, 
proved in \S\ref{S.Beth}, we add a remark regarding scalars and our choice of language for \cstar-algebras to demonstrate the usefulness of Theorem \ref{T.conceptual-completeness}.  One could imagine adding a new sort $D$ to the language of \cstar-algebras meant to interpret the complex unit disk.  Additionally, we would add a function symbol $\lambda:D \times A_1 \rightarrow A_1$ whose intended meaning is $\lambda(r,a) = ra$.  If we consider the class of all pairs $(A,\bbD,\lambda)$ where $A$ is a \cstar-algebra, $\bbD$ is the unit disk and $\lambda$ is as given, then the theory of this class is a strongly conservative extension of the theory of \cstar-algebras.  It follows by Theorem \ref{T.conceptual-completeness} that $D$ and $\lambda$ can be found in~$T^{\eq}$.  Moreover, it is benign to act as if scalar multiplication in the sense presented here is part of the language of \cstar-algebras (see also Remark \ref{rmk.numbers-as-scalars}).  
We note that the unitization of $A$ is also in $A^{\eq}$ by the same argument or by leveraging off what we just said about understanding how scalars are present in $A^{\eq}$.

\
\section{The use of continuous functional calculus}\label{S.cfc}\index{continuous functional calculus}  
 If $C$ is a unital \cstar-algebra and $a\in C$, then the \emph{spectrum}\index{spectrum} of $a$ is 
defined as\index{Sp@$\Sp(a)$}
 \[
 \Sp(a)=\{\lambda\in \bbC\colon a-\lambda 1\text{ is not invertible}\}. 
 \]
If $A$ is not unital then $\Sp(a)$ is defined to be $\Sp(a)$ as computed in the unitization of $A$. 
The spectrum is always a nonempty compact subset of $\bbC$ and for a normal element $a$, 
the \cstar-algebra $\mathrm{C}^*(a,1)$ generated by $a$ and the identity is isomorphic to $\mathrm{C}(\Sp(a))$.  
The isomorphism sends the identity function on $\Sp(a)$ to $a$. The Stone--Weierstrass theorem implies that  
$^*$-polynomials in the variable $z$ are dense in $\mathrm{C}(\Sp(a))$ and    therefore the isomorphism is well-defined and 
automatically an isometry. 
In the not-necessarily-unital case we have  $\mathrm{C}_0(\Sp(a)\setminus \{0\})\cong \mathrm{C}^*(a)$. 

Therefore, for  a given normal element $a$ in a  \cstar-algebra $A$ 
 and any continuous function $f$ on $\Sp(a)$ (satisfying  $f(0)=0$, if $A$ is not unital) 
 we have a well-defined element $f(a)$ in $A$, uniformly approximated by polynomials. That is, the map $a \mapsto f(a)$ is a definable predicate at least when restricted to definable subsets of normal elements relative to the class of \cstar-algebras.
We can therefore use these definable predicates on definable classes of normal elements, such as unitaries, self-adjoints, or positive elements.\footnote{The question whether the class of all normal elements is definable, and in particular whether the continuous functional calculus on all normal elements can be added to the language, 
  is nontrivial (see 
Proposition~\ref{P.NormalNotDefinable})}
We will freely use expressions 
such as $|a|$,   $a^{1/2}$ for $a\geq 0$, or $\exp(ia)$,  $a_+$ and $a_-$ 
for a self-adjoint $a$.

Although we shall not need  multi-variable continuous functional calculus in the present document, the following is worth noting:
continuous functional calculus  can be extended to an $n$-tuple of commuting normal elements~$\bar a$. 
Then $\mathrm{C}^*(\bar a,1)$ is isomorphic to $\mathrm{C}(X)$ where $X$ is the
 \emph{joint spectrum}\index{joint spectrum} of $\bar a=(a_1,\dots, a_n)$ defined
 to be the set of all $n$-tuples 
 $(f(a_1), \dots , f(a_n))$
 when $f$ ranges over all 
 characters of $ \mathrm{C}^*(\bar a)$. 
  We need the following Lemma which is a good example of the use of continuous functional calculus.

\begin{lemma}\label{L.fc-tech-lemma}
Suppose that $a$ is a self-adjoint element of a \cstar-algebra $A$, $b \in A$ and $A \subset B(H)$ for some Hilbert space $H$.  Further assume that $b = va$ for some $v \in B(H)$.  Then for any continuous function $f$ with $f(0) = 0$ we have $vf(a) \in A$.  In particular, if $b = v|b|$ is the polar decomposition of $b$ in $B(H)$ then $vf(|b|) \in A$.
\end{lemma}

\begin{proof} Choose polynomials $p_n$ for $n \in \bbN$ with the property that $p_n(0) = 0$ for all $n$ and that $p_n$ tends to $f$ uniformly on the spectrum of $a$.  Write $p_n(x) = xq_n(x)$ and we have 
\[
vf(a) = \lim_{n \rightarrow \infty} vp_n(a) = \lim_{n \rightarrow \infty} vaq_n(a) = \lim_{n\rightarrow \infty} bq_n(a).
\]
Since $b \in A$ and $q_n(a) \in A$ for all $n$, we have $vf(a) \in A$.
\end{proof}

\begin{prop} \label{P.spectrum} 
The spectrum of an element in a unital \cstar-algebra~$A$ is a
quantifier-free definable set in $A^{\eq}$.
More precisely, there is a quantifier-free definable predicate $F\colon A_1\times \bbD \to [0,1]$ whose zero-set is  
\[\{(a, \lambda): \lambda\in \Sp(a)\}\] where $\bbD$ is the unit disk in $\bbC$.
This predicate is weakly stable.
\end{prop} 

\begin{proof} 
This proposition (as well as its multi-variate version) 
was proved in case of abelian \cstar-algebras in     \cite[Proposition~5.25]{EaVi:Saturation}.  

Consider the formula on $A_1$ (see remarks immediately after the proof)
\[
G(a) = \|a\| - \| (\|a\|\cdot 1 - |a|) \|
\]  
which evaluates to 0 if and only if $a$ is not left-invertible.

The formula $G$ is weakly stable since if $G(a) < \epsilon$ for some $a \in A_1$ then $G(b) = 0$ where $b = v(|a|-\epsilon)_+$ and
$a = v|a|$ is the polar decomposition in $A^{**}$.  Note that $b \in A$ by Lemma \ref{L.fc-tech-lemma}.
So the formula $H(a) = \min\{G(a), G(a^*)\}$ evaluates to 0 if and only if $a$ is not invertible and is still weakly stable.

The desired formula is then  $F(a,\lambda) = H(a - \lambda\cdot 1)$ which evaluates to 0 if and only if  $\lambda$ is in the spectrum of $a$.  Remember that by the remarks immediately after Theorem \ref{T.conceptual-completeness}, $\bbD$ is a sort in $A^{\eq}$ and so the zero-set of $F$ is a definable set in $A^{\eq}$.

It is also relatively easy to see that $\{(a, \lambda) \in A_1 \times \bbD: \lambda\in \Sp(a)\}$ is a definable set by applying Theorem \ref{Th.sem.def} but this does not yield an explicit formula.
\end{proof} 

\begin{remark}\label{rmk.numbers-as-scalars}
The formula $G$ above is not formally a formula or definable predicate in the language of \cstar-algebras as we have defined it.  The problem is that when we write $\|a\|\cdot 1$ we are evaluating the predicate $\|a\|$, obtaining a number and then using the corresponding scalar to multiply by the identity.  This issue comes up from time to time  in what follows (and even in the definition of $F$ above) so we will address this here and refer back to this explanation whenever we use this device.

So suppose that $P(x)$ and $\varphi(x)$ are predicates and $\varphi$ takes only positive values.  We want to make sense of $\varphi(P(x)y)$ i.e., the formula $\varphi$ evaluated at $P(x)y$ where the number  $P(x)$ is being thought of as the corresponding scalar.  Consider the predicate 
\[
\psi_\lambda(x,y) = \max\{|P(x) - \lambda|, \varphi({f_\lambda}y)\}
\]
 where $\lambda$ is a number and $f_\lambda$ is the corresponding scalar - notice that this is a definable predicate in the language of \cstar-algebras.   Now consider the predicate 
 \[
 \theta(x,y) = \inf_{\lambda \in K} \psi_\lambda(x,y)
 \]
  where $K$ is a compact real interval containing the range of $P$.  Leaving aside for the moment a proof that this is actually a definable predicate, let's understand when it evaluates to 0. If $\theta(x,y) = 0$ then there is a Cauchy sequence $\lambda_n$ tending to $\lambda = P(x)$ and $\varphi(\lambda_n y)$ tends to $\varphi(\lambda y)$ which is what we want.  
Now back to the proof that  $\theta$ is a definable predicate.  For a finite subset of $K$, say~$J$, we define
\[
\theta_J(x,y) = \min_{j \in J} \psi_j (x,y)
\]
If one considers finite subsets of $K$, $K_n$ which are $1/n$-dense in $K$ then~$\theta_{K_n}$ tends to $\theta$ uniformly and so $\theta$ is a definable predicate.

It is this way that we understand that $G$ in Proposition \ref{P.spectrum} is a definable predicate in the language of \cstar-algebras.  See also the comments immediately after Theorem \ref{T.conceptual-completeness}.
\end{remark}

\section{Definability of traces} \label{S.Def.Tau.1} 
Assume $\cC$ is an axiomatizable class of \cstar-algebras. 
A \emph{definable trace}\index{trace!definable} relative to $\cC$ is a definable predicate  (in the sense of \S\ref{S.DefinablePred}) which is a trace on every $A \in \cC$.
This means that, if~$\tau$ denotes this definable predicate, then 
  for every $\e>0$ there exists a unary formula $\varphi_\e$ such that 
$\|\varphi_\e(a)-\tau(a)\|\leq \e$ for all $a\in A_1$ for every $A\in \cC$. 
If $A$ is a unital \cstar-algebra and $\tau$ a trace on $A$, we say that~$\tau$ is \emph{definable}  
if there is a definable trace~$\tau_0$ relative to the class of all \cstar-algebras elementarily equivalent to $A$, 
 $\cC:=\Mod(\Th(A))$, such that $\tau_0^A=\tau$. If $A$ is monotracial (i.e., has a 
unique trace)\index{C@\cstar-algebra!monotracial} 
and this trace is   definable, then the  ultraproduct of~$A$
 (as well as any other algebra elementarily equivalent to $A$) is also monotracial.

To every trace $\tau$ of $A$ one associates $N_{\tau,A}$,\index{N@$N_{\tau,A}$} 
the weak closure of the image of $A$ in the GNS representation 
of $A$ corresponding to $\tau$. This is a tracial von Neumann algebra. 
Tracial von Neumann algebras, and II$_1$ factors in particular, 
play an increasingly important role in the Elliott programme (see e.g., the introduction to \cite{SaWhWi:Nuclear}). 
When equipped  with a distinguished trace $\tau$
and the metric corresponding to the  
 $\ell_2$-norm, 
\[
\|a\|_{\tau,2}:=(\tau(a^*a))^{1/2},
\]
tracial von Neumann algebras also form an axiomatizable class (\cite[\S 3.2]{FaHaSh:Model2}). 
It is well-known that 
 $N_{\tau,A}$ is isomorphic to the algebra obtained from~$A$ by taking the $\|\cdot\|_{\tau,2}$-completion
 of each bounded ball of $A$. 

The following proposition and its consequence, Lemma~\ref{L.tau-McDuff}, 
will be used in \S\ref{S.Not-ee-nuclear}. 

\begin{prop}\label{L.tau-N}
Assume that  an elementary class of \cstar-algebras $\cC=\Mod(T)$ 
  has a definable trace $\tau$. Then there is a sort in $\calL^{\eq}$ such that 
  for every $A\in \cC$ the interpretation of this sort is  $N_{\tau^A,A}$. 
  In particular, if $A\in \cC$ 
 then $N_{\tau^A,A}$ and $N_{\tau^B,B}$ are elementarily equivalent 
whenever $B$ is elementarily equivalent to $A$.  
\end{prop} 

\begin{proof} 
Fix $A\in \cC$. 
The idea is to view 
points of the operator norm unit ball $B_1(N_{\tau^A,A})$ as limit points of $\|\cdot\|_{\tau^A,2}$-Cauchy sequences in~$A_1$ as follows. 
Let
\[ 
C^A:=\{(x_i) \in \prod_{\bbN} B_1(A) \mid \|x_i-x_j\|_{\tau^A,2} \leq \frac1i,\text{ for all }  j \geq i\}. 
\]
Then $C^A$ is a definable set. This can be proved directly, but it  clearly follows from  
 Theorem~\ref{Th.sem.def}. 

The predicate  
$\phi\colon (C^A)^2 \to \bbR$  defined by 
\[
\phi((x_i),(y_i)) := \lim_{i\to\infty} \|x_i-y_i\|_{\tau^A,2}
\]
 is also  definable by Theorem~\ref{Th.sem.def}.
Therefore, the quotient is in $A^{\eq}$. 
 The algebraic operations on this quotient are inherited from $A$. 
By the  Kaplansky Density Theorem (see \cite[Theorem~I.9.1.3]{Black:Operator}), 
this quotient can be identified with $B_1(N_{\tau^A,A})$. 
The entire algebra can now be recovered from its unit ball. 
\end{proof}

Proposition~\ref{L.tau-N} implies that if a trace $\tau$ of $A$ is definable then the theory of 
the corresponding tracial von Neumann algebra  is in $A^{\eq}$. We shall make use of some terminology from the theory of von Neumann algebras.   
A~II$_1$ factor~$N$ has \emph{property $\Gamma$}\index{property $\Gamma$} if the relative commutant of~$N$ in its tracial ultrapower is nontrivial. 
It is \emph{McDuff}\index{McDuff factor} if it tensorially absorbs the hyperfinite~II$_1$ factor $R$. 
By a result of McDuff in the separable case, 
this is equivalent to the assertion that the relative commutant of~$N$ in its
tracial ultrapower is nonabelian, and also to the assertion that $R$ is isomorphic to a subalgebra of this relative commutant.\footnote{This may be a good moment to point out that 
$R$ is the unique  strongly self-absorbing II$_1$ factor with separable predual.}  
Remarkably, until very recently these were the only known axiomatizable properties of II$_1$ factors (cf.\  \cite{boutonnet2015ii_1}).

\begin{lemma} \label{L.tau-McDuff} Assume $B$ is a unital 
\cstar-algebra with trace $\tau$.
Assume moreover  that $N_{B,\tau}$ is a factor which is  not McDuff. 
Then there is no nuclear \cstar-algebra with trace $(C,\sigma)$ which is 
elementarily equivalent to the  expanded structure $(B,\tau)$. 

In particular, if $\tau$ is a definable trace and $N_{B,\tau}$ is a factor which is not McDuff 
then $B$ is not elementarily equivalent to any nuclear \cstar-algebra.  
\end{lemma} 

\begin{proof} 
Assume $(C,\sigma)$ is a tracial \cstar-algebra elementarily equivalent to $(B,\tau)$. 
Since being a McDuff factor is axiomatizable in the language of tracial von Neumann algebras (\cite[\S 3.2.3]{FaHaSh:Model3}), 
Lemma~\ref{L.tau-N} implies that $N_{C,\sigma }$ is a II$_1$ factor which is not McDuff. 
By two results of Connes (see e.g.,  \cite{Black:Operator}) the weak closure of a nuclear \cstar-algebra~$C$ in each of its representations is an injective factor
and  the hyperfinite~II$_{1}$ factor~$R$ is the unique injective factor with a separable predual. 
Since $R$ is McDuff,~$C$ is not nuclear. 

Now assume  $\tau$ is definable. By Lemma~\ref{L.tau-N} the defining formulas of $\tau$ define a trace $\sigma$ such that 
$(B, \tau)$ and $(C,\sigma)$ are elementarily equivalent. By the first part of the proof, $C$ is not nuclear. 
\end{proof} 

We have two tools to use to prove definability of traces, Cuntz--Pedersen equivalence (\S\ref{S.CP}) and
the uniform strong Dixmier property (\S\ref{S.Def.Tau.2}). 

\subsection{Definability of Cuntz--Pedersen equivalence} \label{S.CP}
The \emph{Cuntz--Ped\-er\-sen nullset}~$A_0$\index{Cuntz--Pedersen nullset} is the 
norm-closure of the linear span of self-adjoint commutators $[a,a^*]$.
Two self-adjoint elements are called \emph{Cuntz--Pedersen equivalent}\index{Cuntz--Pedersen equivalent} if their difference is in $A_0$.
The set $A_0$ is a closed subspace of the real Banach space of self-adjoint elements of $A$. 
By~\cite{CuPe:Equivalence}, the tracial simplex of $A$ is affinely homeomorphic to the dual unit sphere
of~$A/A_0$.  One consequence of this is that we can provide an explicit axiomatization of the class of unital tracial \cstar-algebras as promised in~\S\ref{S.tracial.0}.  If $A$ is a unital, tracial \cstar-algebra then $A \neq A_0$ and in fact,~$1_A$ is at distance 1 from $A_0$.  So the class is axiomatized, for instance, by the following set of sentences: for each $n \in \bbN$,
\[
\sup_{\bar x} \left(1\dotminus \bigg\| 1 - \sum_{i=1}^n [x_i,x_i^*]\bigg\| \right).
\]
Similarly, if we consider the class of unital \cstar-algebras with a character, an algebra $A$ is in this class if and only if the distance from $1_A$ to the ideal generated by the commutators is 1.  This shows that the class of 
unital \cstar-algebras with a character can be axiomatized by the following set of sentences: for each $n \in \bbN$,
\[
\sup_{\bar b,\bar c,\bar x,\bar y} \left(1\dotminus  \bigg\|1 - \sum_{i=1}^n b_i[x_i,y_i]c_i  \bigg\| \right).
\]

We now examine the conditions under which the Cuntz--Pedersen nullset $A_0$ is definable.
\begin{lemma} \label{L.CP-trace} 
If $A$  is monotracial then its unique trace is  definable if and only if $A_0$ is definable. 
\end{lemma} 

\begin{proof}  
We know that if $\tau$ is the unique trace of $A$ and $a$ is self-adjoint then 
$\tau(a)$ is equal to the distance from $a$ to $A_0$ (or equivalently, to the norm of $a$ in the 
one-dimensional quotient space $A/A_0$).  This shows the direction from left to right immediately.

For the converse, let  $\alpha(x)$  be a formula
which measures the distance to the Cuntz--Pedersen nullset.
Then
\[
\tau(x) = \alpha((x+x^*)/2) + i\alpha((x-x^*)/(2i)).
\]
\end{proof} 

For $k\geq 1$ let 
\[
\alpha_k(a):=\inf_{x_j, y_j, 1\leq j\leq k} \bigg\|a-\sum_{i=1}^{k} [x_j, y_j]\bigg\|. 
\]
Note that $\dist(a,A_0)=\inf_k \alpha_k(a)=\lim_k \alpha_k(a)$. 

\begin{lemma} \label{L.CP.1} Assume $A$ is a unital \cstar-algebra. Then $A_0$ is definable 
if and only if there is a sequence $\bar k=(k(j): j\in \bbN)$ such that for every self-adjoint contraction $a$   
and every $j\geq 1$ we have 
\[
|\dist(a,A_0)-\alpha_{k(j)}(a) |\leq \frac 1j. 
\]
\end{lemma} 

\begin{proof} The converse implication follows from the definition. 
For the forward  implication, assume that a sequence $\bar k$ does not exist. 
This means that for some $j\geq 1$ and for every $k$ there exists a self-adjoint $a_k$ of norm $\leq 1$ such that 
$\dist(a_k,A_0)\geq 
\alpha_k(a)+\frac 1j$. 
Choose $b_k\in A_0$ such that $\|a_k-b_k\|<\dist(a_k,A_0)+1/(2j)$. Then for every 
$k$ we have  $\alpha_k(b_k)>\frac 1{2j}$
and in particular 
the element with the representing sequence $(b_k)$ in the ultrapower $A^{\cU}$ 
does not belong to $(A^{\cU})_0$.
Therefore the Cuntz--Pedersen nullset is not definable. 
\end{proof}

If a sequence $\bar k$ as in Lemma~\ref{L.CP.1} exists we say that the Cuntz--Pedersen nullset is $\bar k$-uniformly definable.

\begin{theorem} \label{T.tau-definable.0} 
 If a unital \cstar-algebra $A$ satisfies any of the following conditions then $A_0$ is definable. 
\begin{enumerate}
\item\label{I.tau-def.1}  $A$ is AF. 
\item\label{I.tau-def.2} $A$ has finite nuclear dimension (see \S\ref{S.dimnuc}).
\item\label{I.tau-def.3} $A$ is exact and $\cZ$-stable.
\item\label{I.tau-def.4} $A$ is an ultraproduct of algebras with uniformly bounded nuclear dimension. 
\item\label{I.tau-def.5} $A$ is an ultraproduct of exact $\cZ$-stable algebras. 
\item\label{I.tau-def.6} $A$ is $\mathrm{C}^*_r(F_\infty)$,  
the reduced group \cstar-algebra associated with the free group. 
\item\label{l.tau-def.7} $A$ has strict comparison of positive elements by traces (see \S\ref{S.elfunsc}).
\end{enumerate}
If $A$ is in addition monotracial, then its trace is definable. 
\end{theorem}

\begin{proof} 
\eqref{I.tau-def.1} 
is an immediate consequence of \cite[Theorem~3.1]{Fack1982} where it was shown that every element of $A_0$ in an AF algebra is a sum of seven commutators.  

\eqref{I.tau-def.2} 
  In \cite[Theorem~1.1]{robert2013nuclear} it was proved that if $\nucdim(A)\leq m$ then 
 every $a\in A_0$ of  norm $1$ is a limit of sums of $m+1$ commutators of the form $[x,x^*]$ for $\|x\|\leq \sqrt 2$. 
 Therefore in this case $A_0$ is $\bar k$-uniformly definable with $k(j):= 2 (m+1)$. 
  
 \eqref{I.tau-def.3} 
       \cite[Theorem 6]{Oza:Dixmier} implies that there exists  a universal constant~$C$ with the 
     following property: if $A$ is $\cZ$-stable and exact and  $a\in A_0$ has norm $\leq 1$  then for every $R\geq 1$ there are 
     $b_j$, $c_j$ for $j\leq R$  such that 
     \[
     \sum_{j=1}^R \|b_j\|\|c_j\|\leq C\qquad\text{and}\qquad 
         \bigg\|a-\sum_{j=1}^R [b_j, c_j]\bigg\|  < C R^{-1/2}.
     \] 
By replacing $C$ with $\lceil C\rceil$ we may assume $C$ is an integer. 
     Fix $j\geq 1$ and let $R:=C^2j^2$. With $b_j$ and $c_j$ as guaranteed by \cite[Theorem 6]{Oza:Dixmier}
 let          $b_j':=b_j C^{-1/2}$ and $c_j':=c_j C^{-1/2}$.   Repeating each of the commutators $[b_j',c_j']$ 
           $C$ times we obtain a sum of at most $jC$ commutators  $1/j$-approximating $a$. 
           Therefore with $k(j):=j^2C^3$ we have that $A_0$ is $\bar k$-uniformly definable.

\eqref{I.tau-def.4} 
 and \eqref{I.tau-def.5} 
 follow from the uniform estimates given in  \eqref{I.tau-def.2} 
 and \eqref{I.tau-def.3} 
and \L o\'s' theorem.

 \eqref{I.tau-def.6} 
 follows from  \cite[Theorem~1.4]{Ng:Commutators} where it was proved that 
 every element of 
 the Cuntz--Pedersen nullset in~$\mathrm{C}^*_r(F_\infty)$ 
 is the sum of three commutators.

  \eqref{l.tau-def.7} By  
    \cite[Corollary 4.3]{NgRobert} and straightforward induction 
    one can prove that 
  for every $k\geq 1$ there exists  $n(k)$ 
  such that for every $h\in A_0$ of norm $\leq 1$ 
  there are $x_i$ and $y_i$ for $i<n(k)$ of norm $\leq 1$
  such that $\|h-\sum_{i<n(k)} [x_i,y_i]\|<2^{-k}$. 
  See more information in \S\ref{S.elfunsc}.

If $A$ is in addition monotracial then its trace is definable by the above and  Lemma~\ref{L.CP-trace}. 
\end{proof}

\begin{prop} \label{P.non-def-trace} 
There exists  a simple, unital, infinite-dimensional, monotracial \cstar-algebra whose trace is not definable. 
 \end{prop} 

\begin{proof}   In \cite[Theorem 1.4]{robert2013nuclear} Robert constructed 
 a simple, unital, infinite-dimensional, monotracial 
 \cstar-algebra $A$ with the property that for every $n$ there exists a unit element $a\in A_0$ 
 such that $\alpha_n(a)=1$. The ultrapower of $A$ therefore does not have a unique trace, and the trace of $A$ is not definable.  
\end{proof}

The following was pointed out to the authors by Ilan Hirshberg. 

\begin{lemma} 
Every state definable with respect to some axiomatizable class $\cC$ is a trace. 
\end{lemma} 

\begin{proof} Suppose that $s$ is a definable state on some \cstar-algebra $A$. The assumption implies that $s(a)=s(\alpha(a))$ for every automorphism $\alpha$ of $A$. 
In particular, for every unitary $u$ we have $s(uau^*)=s(a)$. Since every element of $A$ is a linear combination of 
unitaries, $s(ab)=s(ba)$ for all $a$ and~$b$. 
\end{proof}

\section{Axiomatizability via definable sets} \label{S.Ax.Definable}
Armed with the expanded language (\S\ref{S.DefinablePred} and \S\ref{S.Definable}) we continue proving assertions on axiomatizability of classes of \cstar-algebras made in  \S\ref{S.Axiomatizable}. 

 \subsection{Projectionless and unital projectionless} \label{S.UP} Here is an improvement on our earlier axiomatization. Using the formula 
$\pi(x):=\|x^2-x\|+\|x^*-x\|$ introduced in Example \ref{Ex.1}(\ref{Ex.1.projection})
we have that 
\[
\varphi:=\sup_{x \text{ proj}} \min\{\|x\|,\|1 - x\|\}
\]
 satisfies $\varphi^A=0$ if and only if $A$ has no nontrivial projections (and $\varphi^A=1$ otherwise). 
 Therefore $\varphi$ is a universal formula axiomatizing unital projectionless algebras (cf.\ Example~\ref{Ex.projectionless}). 

A similar, and easier proof shows that being projectionless is axiomatizable. 

\subsection{Real rank zero revisited} \label{S.rr0.revisited} 
Since we can quantify over the set of  positive elements and over the set of 
projections (Example~\ref{Ex.1}), 
the axiomatization of real rank zero given in  Example~\ref{Ex.rr0}  can be succinctly written as follows: 
\[
\sup_{x,y\text{ self-adj}} \inf_{z \text{ proj}}
\max(\|zx\|, \|(1-z)y\|)^2 \dotminus \|xy\|
\]
\subsection{Infinite \cstar-algebras} \label{S.Infinite} A \cstar-algebra is \emph{infinite}\index{C@\cstar-algebra! infinite} if 
it has nonzero orthogonal projections $p$ and $q$ such that $p+q$ is Murray--von Neumann equivalent to $p$. 
In unital case this is equivalent to having a proper isometry. As in Example~\ref{Ex.1}, the standard methods 
show that the set of triples $(p,q,v)$ such that $p$ and $q$ are orthogonal projections, $q\neq 0$, 
$v^*v=p+q$, and $vv^*=p$ is definable. 
Therefore $A$ is infinite if and only if this definable set is nonempty in $A$, 
and being infinite is axiomatizable.  

\subsection{Finite and stably finite algebras}\label{S.SF} A \cstar-algebra $A$ 
is \emph{finite}\index{C@\cstar-algebra!finite} if it is not 
infinite (see \S\ref{S.Infinite}) and it is \emph{stably finite}\index{C@\cstar-algebra!stably finite} if $M_n(A)$ is finite for all $n$. 
Among the unital algebras, the class of finite \cstar-algebras is universally axiomatizable. It is clearly closed under subalgebras so it is enough to see that the class of finite \cstar-algebras is elementary.  Since
the set of isometries in a unital \cstar-algebra, $\{s\colon s^*s=1\}$, 
is quantifier-free definable (Example \ref{Ex.1}) by the weakly stable formula 
$\alpha(x):=\|x^*x-1\|$,
we have the fact that $A$ is finite if and only $\sup_{\alpha(x) = 0}\|xx^*-1\|$ 
is 0 in $A$. 

Note that a similar argument easily shows that the theory of unital 
finite \cstar-algebras form a clopen set.

 By the above and the fact that 
  the norm in $M_n(A)$ is definable  (Lemma \ref{L.Norm})
  $M_n(A)$ being finite is also axiomatized by a single universal sentence~$\beta_n$.  
  Therefore $A$ being stably finite is universally axiomatized by $\{\beta_n\colon n\geq 1\}$.

\section{Invertible and non-invertible elements} \label{S.invertible} 
Apart from proving Proposition~\ref{P.invertible} below, 
the present section introduces some of the ideas required in the analysis of  
stable rank  (\S\ref{S.sr})   and real rank  (\S\ref{S.rrn}).

\begin{prop} \label{P.invertible}
\begin{enumerate}
\item The set of non-invertible elements in the unit ball is quantifier-free definable in any unital \cstar-algebra. 
\item The closure of the set of invertible elements in the unit ball of any \cstar-algebra is a definable set.
\item  The closure of the set of self-adjoint, invertible elements in the unit ball of any \cstar-algebra is a definable set.  
\end{enumerate}
\end{prop}

\begin{proof}
By Lemma~\ref{L.weaklystable} a set is definable if and only if it is a zero-set of a 
weakly stable predicate. 
Therefore (2)  is an 
immediate consequence of Lemma~\ref{L.invertible.2} below 
and (3) is an immediate consequence of Lemma~\ref{L.invertible.3} below. 

As in the proof of Proposition~\ref{P.spectrum} 
let 
\[
G(a) := \|a\| - \| (\|a\|\cdot 1 - |a|) \|\text{ and }H(a) :=\min\{G(a), G(a^*)\}
\]
As shown in Proposition~\ref{P.spectrum},  $H$ is weakly stable  and it 
has the set of non-invertible elements as 
its zero-set, hence (1) follows. 
\end{proof}

\begin{lemma} \label{L.invertible.2} 
In a unital \cstar-algebra, the predicate defined on the unit ball  
\[
\varphi(a):=\inf_{u \text{ unit.} }\|a- u|a|\|
\]
is weakly stable and its zero-set is equal to the closure of the invertible elements 
in the unit ball. 
\end{lemma}

\begin{proof}  Since the set of unitaries is definable by Example \ref{Ex.1},
$\varphi$ is a definable predicate (Definition~\ref{D.Assignment}).

If $a$ is an invertible element, then $|a|=(a^*a)^{1/2}$ is also invertible
and $u=a|a|^{-1}$ is a unitary element such that $a=u|a|$.  
 Thus invertible elements in a \cstar-algebra have polar decompositions with the partial 
 isometry being a unitary. 
 Moreover, elements with a polar decomposition 
 with partial isometry as a unitary
 are in the closure of the invertibles because if $a=u|a|$ and $u$ is a unitary 
 then $u(|a|+\e)$ is invertible for all $\e>0$. Therefore  $\varphi(a)<\e$ implies 
 that for a unitary $u$ satisfying $\|a-u|a|\|<\e$ 
 we have that $b:=u(|a|+\e)$ is in the zero-set of $\varphi$ and within $2\e$ of $a$. 
 Since $\e>0$ was arbitrary, this proves 
  the weak stability of $\varphi$. 
\end{proof}

\begin{lemma} \label{L.invertible.3} 
In a unital \cstar-algebra, the closure of the set of invertible self-adjoint  elements of the unit ball is a definable set.
\end{lemma} 

\begin{proof} 
We prove this by using the ultraproduct characterization of definable sets, Theorem \ref{Th.sem.def}.  Fix unital \cstar-algebras $A_i$ for $i \in I$ and an ultrafilter $\cU$ on $I$.  Let $X_i$ be the closure of the set of  invertible self-adjoint elements of the unit ball in $A_i$ and $X$ be the same in $\prod_\cU A_i$.  We would like to show that $\prod_\cU X_i = X$.  We show the inclusion right to left first.  Fix $\bar a \in X$, $n \in \bbN$ and an invertible self-adjoint element $\bar a^n$ in the unit ball such that $\displaystyle\|\bar a - \bar a^n\| < \frac{1}{n}$.  Using \L o\' s' Theorem, we can obtain a decreasing sequence of sets $U_m \in \cU$ such that
\[
U_m = \{ i \in I : \| a_i - a_i^k \| < \frac{1}{k} \text{ and } a_i^k \text{ is invertible for all } k \leq m \}.
\]
Now we define an element $\bar b \in \prod_\cU A_i$ as follows: if $i \in U_n \setminus U_{n+1}$, let $b_i = a_i^n$.  If $i \in U_n$ for all $n$ then let $b_i = a_i$.  In all other cases, define $b_i$ to be any element of $X_i$.  It is clear that $\bar b$ is another representation of $\bar a$ in $\prod_\cU A_i$ and that $\bar b \in \prod_\cU X_i$.

%

To prove the other inclusion, suppose that $a_i \in X_i$ for all $i \in I$ and $\e > 0$.  Choose an invertible self-adjoint $b_i$ in the unit ball of $A_i$ such that $\|a_i - b_i\| < \e$.  Since $b_i$ is invertible, $0 \not\in\text{sp}(a)$ and so the function
\[
f(x) = \begin{cases}
		\max\{\e,x\} & \text{if } x > 0\\
		\min\{-\e,x\} & \text{if } x < 0
		\end{cases}.
\]
is continuous on $\text{sp}(a)$.  By replacing $b_i$ with $f(b_i)$, we can assume that $b_i^{-1}$ has norm at most $1/\e$.
  Then $\bar b = \langle b_i : i \in I \rangle$ is an invertible self-adjoint element in the unit ball of $\prod_\cU A_i$ and $\|\bar a - \bar b\| \leq \e$.  Since $\e$ was arbitrary, $\bar a \in X$.
%
%
%
%
\end{proof}

The proofs of the previous two lemmas demonstrate that
 for every invertible $a$ and $\e>0$ 
there exists invertible $b$ such that $\|a-b\|\leq \e$ and $\|b^{-1}\|\leq \e^{-1}$. 
The following multivariable variant will be needed in \S\ref{S.rrn}. 

\begin{lemma}\label{L.rrn}
Suppose that $A$ is an abelian unital \cstar-algebra, $\bar a \in A$ is an n-tuple of self-adjoint elements such that $\sum a^2_i$ is invertible.  Then for every $\e > 0$ there is an n-tuple of self-adjoint elements $\bar b \in A$ such that $\|\bar a - \bar b\| \leq \e$, $\sum b^2_i$ is invertible and the norm of this inverse is at most $\e^{-2}$.
\end{lemma}

\begin{proof} 
By the Gelfand--Naimark Theorem, since $A$ is unital, we can assume that $A$ is $C(X)$ for some compact Hausdorff space $X$.  Since the $a_i$'s are self-adjoint, as functions, they are real-valued.  The assumption that $\sum a_i^2$ is invertible means that $\sum a_i^2(x) > 0$ for all $x \in X$.  Fix $\e > 0$ and let $h(x) = \sqrt{\sum a_i^2(x)}$.  For $i \leq n$ define continuous functions on $X$ as follows:
\[
b_i(x) =
\begin{cases}
\e a_i(x) h(x)^{-1}& \text{if } h(x) \leq \e\\
a_i(x) & \text{if } h(x) > \e
\end{cases}
\]
It is easy to see that $\sum b_i^2 \geq \e^2$ and that $|a_i(x) - b_i(x)| \leq \e$ for all $x\in X$ and $i \leq n$.
\end{proof}

 \section{Stable rank} \label{S.sr} 
 \cstar-algebras are considered as noncommutative topological spaces and 
two noncommutative analogues of  dimension are stable rank and real rank, 
introduced in \cite{Rie:Dimension} and \cite{brown1991c} respectively 
(see~\cite[V.3]{Black:Operator}).

A unital \cstar-algebra  $A$ has \emph{stable rank $\leq n$}\index{stable rank} if and only if the set of tuples $(a_1,\dots, a_n)\in A^n$ which generate
an improper right ideal is dense in $A^n$.   If $A$ is not unital, we say that $A$ has stable rank $\leq n$ if its unitization does (recall that the unitization of $A$ lies in~$A^{\eq}$).
The \emph{stable rank} of $A$, denoted $\sr(A)$,\index{s@$\sr(A)$} is defined to be 
the least $n$ such that the stable rank of $A$ is $\leq n$
(stable rank was originally called \emph{topological stable rank} and denoted $\tsr(A)$ in \cite{Rie:Dimension}).    
It is not difficult to see that $\sr(\mathrm{C}(X))=\lfloor \dim(X)/2\rfloor+1$ 
for any compact Hausdorff space~$X$
(see~\cite[V.3.1.3]{Black:Operator}).

Results of the present subsection on the axiomatizability of  
stable rank are summarized in the following. 

\begin{proposition} \label{P.SR} 
\begin{enumerate}
\item \label{P.SR.1} 
For every $n\geq 1$, having stable rank $\leq n$  is axiomatizable. 
\item \label{P.SR.1.5} Having stable rank $>n$ is axiomatizable. 
\item \label{P.SR.2} Having stable rank $\leq n$ is \aea, but  
neither universally nor existentially axiomatizable. 
\end{enumerate}
\end{proposition}

 
Since it is the most important case, and its proof is simpler and more enlightening, we first  
treat the case of stable rank one; a \cstar-algebra has stable rank one if and only if its invertible elements are dense. 

\begin{lemma}\label{L.SR.1} 
Both $\sr(A)=1$ and $\sr(A)>1$ 
are axiomatizable. 
\end{lemma} 

\begin{proof} 
It will suffice to find a sentence $\theta$ such that $\theta^A=0$ if $\sr(A)=1$
and $\theta^A=1$ if $\sr(A)>1$ (cf.  Proposition~\ref{P.el-co-el}). 
By Lemma~\ref{L.invertible.2}, 
has stable rank 1 if and only if 
  \[
\theta:= \sup_x\inf_{y \text{ unit.}} \|x-y|x|\|
 \]
 is zero in $A$. 
By \cite[Theorem~2.5]{rordam1988advances} (where $\alpha(a)$ was used to denote the distance 
of $a$ to the group of invertible elements), the value of  $\theta$ in any \cstar-algebra is either $0$ or $1$, 
and this completes the proof. 
\end{proof}

We proceed to   prove  that having 
stable rank $\leq n$ is  also axiomatizable for all $n\geq 2$.  

\begin{lemma}\label{L.stable-rank}
Let $A$ be a unital \cstar-algebra, $n\geq 1$,  and $(a_1,\dots,a_n)\in A^n$ a tuple of contractions. The following are equivalent:
\begin{enumerate}[leftmargin=*]
\item
$x_1a_1+\dots x_n a_n=1$ for some $x_i\in A$,
\item
$a_1^*a_1+\dots+a_n^*a_n$ is invertible  in $A$,
\item
there exists $(v_1,\dots,v_n)\in A^n$ and   $a\in A_+$ such that $a_i=v_ia$
for all~$i$ and $\sum_{i=1}^n v_i^*v_i=1$, $a$ is invertible, and $\|a\|\leq \sqrt n$, 

\item
there exist $(v_1,\dots,v_n)\in A^n$ and  $a\in A_+$ such that $a_i=v_ia$
for all~$i$, $\|\sum_{i=1}^n v_i^*v_i-1\|<1$, $a$ is invertible, and $\|a\|\leq \sqrt n$. 
\end{enumerate}
\end{lemma}
\begin{proof}
(1)$\Rightarrow$(2): Let $\bar x$ and $\bar a$ be as in (1). 
Consider $A^n$ as a Hilbert $A$-module  with the inner product
\[
(\bar x| \bar y) =\sum_i x_i^* y_i. 
\]
By the 
  Cauchy--Schwarz inequality  (see \cite[Proposition~1.1]{lance1995hilbert})
  applied to $\bar a$ and $\bar x^*$ we have 
  \[
   \sum_i x_i a_i  \bigg(\sum_i x_i a_i\bigg )^*\leq \bigg\|\sum_i x_i x_i^* \bigg\|\sum_i a_i^* a_i. 
  \]
  Since the left-hand side is equal to 1, we conclude that $\sum_i a_i^* a_i$ is bounded below and therefore 
  invertible. 

(2)$\Rightarrow$(3): Assume (2) holds for $\bar a$. 
Let  $a:=(a_1^*a_1+\dots+a_n^*a_n)^{1/2}$ and 
$v_i=a_ia^{-1}$ for all $i$. Then $a\geq 0$ and it is invertible by (2). 
Clearly $\|a\|\leq \sqrt n$, and $\sum_{i=1}^n v_i^*v_i=a^{-1}(\sum_{i=1}^n a_i^*a_i)a^{-1}=1$. 

(3)$\Rightarrow$(4): Obvious.

(4)$\Rightarrow$(1): Assume $\bar a$, $\bar v$ and $a$ are as in (4).  Then 
$\sum_{i=1}^n v_i^* a_i=(\sum_{i=1}^n v_i^*v_i)a$ is a product of two invertibles and 
there is $x$ such that $x\sum_{i=1}^n v_i^* a_i=1$. Therefore $x_i=xv_i^*$ are as required in (1). 
\end{proof}

By continuity the following  is an immediate consequence of the equivalence of (1), (3) and (4) in Lemma~\ref{L.stable-rank}. 

\begin{lemma}\label{L.sr.3} 
Let $A$ be a unital \cstar-algebra and let $(a_1, \dots, a_n) \in A^n$ be a tuple of contractions. The following are equivalent:
\begin{enumerate}[leftmargin=*]
\item $(a_1, \dots, a_n)$ is a limit of a sequence of tuples $(b_1,\dots, b_n)$, such that $(b_1,\dots,  b_n)$ 
generates $A$ as a right ideal.
\item For every $\e > 0$, there exist contractions $v_1,\dots,v_n$ and  $a\geq 0$ of norm $\leq\sqrt{n}$ such that $\|a_i-v_ia\| < \e$ for all $i$ and $\|\sum_i v_i^*v_i - 1\| < \e$.
\item For every $\e > 0$, there exist contractions $v_1,\dots,v_n$ and  $a\geq 0$ of norm  $\leq\sqrt{n}$ such that $\|a_i-v_ia\| < \e$ for all $i$ and $\|\sum_i v_i^*v_i - 1\| <1$. \qed
\end{enumerate}
\end{lemma}

\begin{proof}[Proof of Proposition~\ref{P.SR}] 
\eqref{P.SR.1}: 
 The equivalence of conditions (1) and (2) in Lemma~\ref{L.sr.3}   implies that  the formula 
\begin{equation}\label{Eq.SR}
\sup_{a_i}\inf_{\|a\|\leq \sqrt n} \inf_{v_i}\sum_{i=1}^n \|a_i-v_ia\|+\|1-\sum_{i=1}^n v_i^*v_i\|
\end{equation}
evaluates to 0 if and only if $A$ has stable rank at most $n$.
Therefore having $\sr(A)\leq n$ is axiomatizable for all $n$. 

\eqref{P.SR.1.5} was proved in \cite{farah2016axiomatizability}. The case $n=1$ follows from Lemma~\ref{L.SR.1}.


\eqref{P.SR.2}  (Cf.\ \cite[Fact 1.3]{eagle2015pseudoarc}.) 
The formula in \eqref{Eq.SR} is clearly $\forall\exists$. 
It remains to prove  that for every $n$ having stable rank $\leq n$ is neither universally nor existentially axiomatizable. 
 For a compact Hausdorff space $X$ 
  the stable rank of $\mathrm{C}(X)$  is equal to $\lceil \dim(X)/2\rceil +1$ (see \cite[V3.1.3]{Black:Operator})
    Since $[0,1]^{2m-1}$ is a continuous image of $[0,1]^{2n-1}$ for all $m\geq 1$ and all $n\geq 1$, we 
have an example of \cstar-algebra of stable rank $m$ with a subalgebra
of stable rank $n$ for all such $m$ and $n$. 
Since universally axiomatizable classes are closed under taking submodels and
 existentially axiomatizable classes are closed under taking supermodels 
   (Proposition~\ref{P.Ax}), 
we conclude that stable rank $\leq n$ is neither universally nor existentially axiomatizable for any $n\geq 1$. 
\end{proof} 

 \section{Real rank}\label{S.rrn}
A unital \cstar-algebra $A$ has real rank at most $n$ if and only if the set
of $(n+1)$-tuples of self-adjoint contractions $\bar g$ such that 
$\sum_i g_i^2$ is invertible is dense
(see \cite{brown1991c}, \cite[V.3.2]{Black:Operator}). 
The \emph{real rank}\index{real rank $\rr(A)$}  of~$A$, denoted~$\rr(A)$, is defined to be 
the least $n$ such that the real rank of $A$ is~$\leq n$.  
It is not difficult to see that $\rr(\mathrm{C}(X))=\dim(X)$ for any compact Hausdorff space $X$.

\begin{proposition} \label{P.RR} 
\begin{enumerate}
\item The class of abelian \cstar-algebras with
 real rank $\leq n$  is elementary.  In fact, it is $\forall\exists$-axiomatizable but neither universally or existentially axiomatizable.
\item The class of \cstar-algebras with real rank greater than zero is elementary. 

\end{enumerate}
\end{proposition}

A simple proof that having real rank equal to zero is axiomatizable  was given in \S\ref{Ex.rr0}.  The abelian case essentially follows from the results in  \cite[Theorem~2.2.2]{Bank}  and 
\cite{Rie:Dimension} but we include an elementary proof in order to ask a question about the general case after the proof.

\begin{proof}[Proof of Proposition~\ref{P.RR} (1)]

We check this semantically.  Suppose that $A \prec B$ and $B$ is a \cstar-algebra with real rank $\leq n$.  Fix an $(n+1)$-tuple of self-adjoint elements $\bar a \in A_1$. For any $\e > 0$, there is an $(n+1)$-tuple of self-adjoints $\bar b \in B_1$ such that $\|\bar a - \bar b\| < \e/2$ and $\sum b_i^2$ is invertible.  Suppose the norm of the inverse is at most $M$.  By elementarity  we can find an $(n+1)$-tuple of 
self-adjoints $\bar b' \in A_1$ such that $\|\bar a - \bar b'\| < \e$ and~$\sum {b'}_i^2$ is invertible. 
This shows that $A$ has real rank $\leq n$.

Now suppose that $A_i$ is abelian and has real rank $\leq n$ for all $i \in I$.  Fix an ultrafilter $\cU$ on $I$ and we wish to show that $\prod_\cU A_i$ has real rank $\leq n$. 
Suppose that $\bar a \in \prod_\cU A_i$ is an $(n+1)$-tuple of self-adjoint elements and $\e > 0$.  Choose a representing sequence $\langle \overline{ {}_{i}a }: i \in I \rangle$.  Since each $A_i$ has real rank $\leq n$, there is some $\overline{{}_{i}b} \in A_i$ such that $\|\overline{{}_{i}a} - \overline{ {}_{i}b}\| < \e/2$ and $\sum {}_{i}b_j^2$ is invertible.  By Lemma \ref{L.rrn}, there is some $\overline{ {}_{i}c} \in A_i$ such that $\|\overline{ {}_{i}c} - \overline{ {}_{i}b}\| < \e/2$, $\sum {}_{i}c_j^2$ is invertible and the norm of the inverse is bounded by $4/\e^2$.  If we let $\bar c$ be the element of $\prod_\cU A_i$ with representing sequence $\langle \overline{ {}_{i}c} : i \in I\rangle$ then we see that $\|\bar a - \bar c\| < \e$ and $\sum \bar c_j^2$ is invertible.  We conclude then that $\prod_\cU A_i$ has real rank $\leq n$ and so the class of \cstar-algebras with real rank $\leq n$ is elementary.

The property of having real rank $\leq n$ is clearly closed under unions of chains 
 and therefore by Proposition~\ref{P.Ax} axiomatizable if and only if it is 
 \aea.  It therefore suffices to show
 that the class of abelian \cstar-algebras of real rank $\leq n$ is neither universally axiomatizable nor existentially axiomatizable. 
 This is similar to the proof of Proposition~\ref{P.SR}(\ref{P.SR.2}). 
For abelian algebras $\mathrm{C}(X)$ the real rank is equal to the covering dimension of the spectrum $X$ (\cite[V.3.2.2]{Black:Operator}). 
Since $[0,1]^m$ is a continuous image of $[0,1]^n$ for all $m\geq 1$ and all $n\geq 1$, we 
have an example of \cstar-algebra of real rank $m$ with a subalgebra
of real rank $n$ for all such $m$ and $n$. 
Since universally axiomatizable classes are closed under taking submodels  (Proposition~\ref{P.Ax}), 
we conclude that real rank $\leq n$ is neither universally nor existentially axiomatizable for any $n > 0$. 

(2)  We prove that the  class of algebras of real rank greater than 0  is elementary.  From Theorem 2.6 of \cite{brown1991c}, one knows that $A$ has real rank zero if and only if every hereditary subalgebra has an approximate identity consisting of projections. 
\begin{lemma} The set 
\[
\cX:=\{(a,e): \|a\|\leq 1, \|e\|\leq 1, e\geq 0,\text{ and }ea=a\}
\]
is definable. 
\end{lemma}

\begin{proof}  
Fix sufficiently small $0<\delta<\e^2$. 
We shall prove that for all $0\leq e\leq 1$ and contraction $a$ 
such that $\|(1-e)a\|<\delta$ there are $0\leq e'\leq 1$ and 
a contraction $a'$ such that $\|a-a'\|<\e$, $\|e-e'\|\leq 2\e$ and $e'a=a'$. 
This is a continuous functional calculus argument. 
We shall use the following fact. 
If $0\leq x\leq y$, then $\|xz\|\leq \|yz\|$ for all $z$. 
This is because then $z^*x^2 z\leq z^* y^2z$ and 
\[
\|xz\|^2=\|z^*x^2z\| \leq \|z^*y^2z\|=\|yz\|^2. 
\]
Let $f\colon [0,1]\to [0,1]$ be such that $f(0)=0$, $f(t)=1$ for $1-2\e\leq t\leq 1$ and is linear on $[0,1-2\e]$. Let $g\colon [0,1]\to [0,1]$ be such that $g(t)=0$ for $0\leq t\leq 1-2\e$, $g(t)=1$
for $1-\e\leq t\leq 1$,  
and is linear on $[1-2\e,1-\e]$. In particular  $fg=g$ and $h:=\chi_{[0,1-\e]}$ satisfies 
$h\geq 1-g$ and $1-f\geq \e h$. 

Let $e':=f(e)$ and $a':=g(e)a$. Then $\|e-e'\|=\|(1-f)e\|\leq 2\e$ and 
 $e'a'=f(e)g(e)a=g(e)a=a'$.  It remains to check that $\|a-a'\|=\|(1-g)(e)a\|<\e$. Assume otherwise. 
 
Fix a faithful representation of our (hitherto unnamed) \cstar-algebra and let $p$ 
(in its weak closure)  be the spectral projection of $e$ corresponding to the interval $[0,1-\e]$. 
By the above we have $(1-e)\geq \e p\geq \e (1-g)(e)$ and therefore 
$\|(1-e)a\|\geq \e \|(1-g)(e)a\|\geq \e^2$. This contradicts our choice of $\e$ and $\delta$
and proves our claim. 

Since both $\{e: 0\leq e\leq 1\}$ and the unit ball are definable, 
by Theorem~\ref{Th.sem.def}, this implies definability of $\cX$.  
\end{proof}

Let  
\[
\alpha:= \sup_{(a,e)\in \cX} \inf_{p\text{ proj.}}\|ep-p\|+ \|(1-p)a\|. 
\]
Then $\alpha^A=0$ if and only if 
 in  every hereditary subalgebra $\overline{eAe}$ for every  $a\in \overline{eAe}$
 and every $\e>0$ there exists a projection $p\in \overline{eAe}$ satisfying $\|a-pa\|<\e$.
 (To get $p\in \overline{eAe}$ we use the definability of the set of projections.) 

 Since for every $n$, all  $a_1, \dots, a_n$ and every projection $p$ for all $j$ we have 
 \begin{multline*}
 \|(1-p)\sum_j (a_ja_j^*)^{1/2}\|^2=
 \|(1-p)\sum_j (a_ja_j^*)(1-p)\|\\
 \geq 
 \|(1-p)a_ja_j^*(1-p)\|=\|(1-p)a_j\|^2, 
 \end{multline*}
 in order to find an approximate unit consisting of projections it suffices to show that for every 
 element $a$ of the algebra and $\e > 0$ there is a projections $p$ such that $\|(1-p)a\| < \e$. 
 Therefore $\alpha^A=0$ implies that $A$ has an approximate unit consisting of projections. 
 
 We claim that $\alpha^A=r<1$ implies $\alpha^A=0$. Fix $\e>0$. 
 Fix $(a,e)\in \cX$. Since $\alpha^A=r$ 
 we can find a projection $p_1$ such that $\|p_1-e p_1\|<2^{-1}\e$
 and $a_1:=r(1-p_1)a$ has norm $<1$. 
 Let $e_1:=(1-p_1)e(1-p_1)$. Then $(a_1,e_1)\in \cX$ and we can find projection $p_2$ such that 
 $\|p_2-e_1 p_2\|<2^{-2}\e$ and $a_2:=r(1-p_2)a_1$ has norm $<1$. 
 Note that $p_1p_2=0$ and that $\|(1-p_1-p_2)a\|=r^{-2}a_2$ has norm $<r^{-2}$. 
 
 Proceeding in this manner we find $(a_j,e_j)\in \cX$ and pairwise orthogonal 
 projections $p_j$, for $j\in \bbN$, such that 
 $\|p_j -e p_j\|<2^{-j} \e$ and $\|(1-p_j)a\|<r^{-j}$. 
 For a large enough $n$ the projection $q=\sum_{j=1}^n p_j$ belongs to $\overline{eAe}$ 
 and satisfies $\|(1-q)a\|<\e$. 

Therefore $\alpha^A=0$ if $A$ has real rank zero and $\alpha^A=1$ if $A$ has real rank greater than zero. 
\end{proof} 

\begin{question}
Is the class of \cstar-algebras of real rank $\leq n$ elementary?  By the proof given above, this question is equivalent to asking if the word `abelian' can be removed from Lemma \ref{L.rrn}.
\end{question}

\section{Tensor products} \label{S.Tensor} 
 An old question of Feferman and Vaught was whether for `algebraic systems' 
 free products or 
 tensor products (where applicable) preserve elementary 
 equivalence, in the sense that $A_1\equiv A_2$ and $B_1\equiv B_2$ 
 implies $A_1\otimes B_1\equiv A_2\otimes B_2$ (\cite[footnote on p. 76]{feferman1959first}). 
 This question in case of tensor products of modules was answered in the negative 
 in \cite{Olin:Direct}. We shall consider a \cstar-algebraic version of this question. 
 
Does  tensoring with a fixed \cstar-algebra $C$
preserves elementary equivalence? More precisely, is it true that 
for every pair of elementarily equivalent \cstar-algebras $A$ and $B$ the algebras
$A\otimes C$ and $B\otimes C$ are elementarily equivalent? 
(Our convention is that $\otimes$ denotes the minimal tensor product.)

\begin{lemma} Suppose a \cstar-algebra $C$ is 
such that for all $A$ and $B$, if $A\prec B$ then $A\otimes C\prec B\otimes C$. 
Then  for all $A$ and $B$,  $A\equiv B$ implies $A\otimes C\equiv B\otimes C$. 
\end{lemma} 

\begin{proof} Fix $A$ and $B$ which are elementarily equivalent. 
If $D$ is a sufficiently saturated model of theory of $A$, then both $A$ and $B$ are elementarily embeddable into $D$. 
Our assumption implies $A\otimes C\prec D\otimes C$ and $B\otimes C\prec D\otimes C$, and (1) follows. 
\end{proof}

We have both a positive result (Lemma \ref{L.Mn}) and a negative result (Proposition \ref{P.C01}) related to this question. 
  By the standard methods one sees that it suffices to consider the case when $C$ is separable. 
  This is because every \cstar-algebra is an inductive limit of separable \cstar-algebras which are elementary submodels.
 If $C$ is a nuclear \cstar-algebra  and  $A$ is a subalgebra of $B$ then we can canonically 
identify $A\otimes C$ with a subalgebra of $B\otimes C$. 
Our question has a positive answer when $C$ is finite-dimensional.

\begin{lemma} \label{L.Mn} For every finite-dimensional algebra $F$ and all $A$ and $B$ we have that 
$A\prec B$ implies $F\otimes A\prec F\otimes B$. 
\end{lemma} 

\begin{proof} 
It suffices to prove the most interesting case, when $F$ is $M_n(\bbC)$ for some $n$. 
In this case we need to show that $M_n(A)\prec M_n(B)$. 

By Lemma \ref{L.Norm}, the norm in $M_n(A)$ is a definable predicate, and it belongs to $A^{\eq}$. 
The conclusion follows by general theory of definability
 (\S\ref{S.Definable}) or by a straightforward induction on complexity of formulas. 
\end{proof}

By   $U(A)$\index{U@$U_0(A)$} we  denote the unitary group of 
$A$ and by $U_0(A)$ we denote  the connected component of the identity in $U(A)$. 
Proposition \ref{P.C01} below is related to 
 the observation made in   \cite[Example 4.7]{Li:Ultraproducts} that there exists a 
simple \cstar-algebra $A$ such that $U(A)$ is path-connected but $U(A^{\cU})$ is not. 
Without the requirement that $A$ be simple already the algebra $B=\mathrm{C}([0,1])$ has this property. 
It is an exercise in topology to show that every unitary in $B$ is of the form $\exp(i a)$ for a self-adjoint $a$, 
hence $U(B)$ is path-connected. However, the unitary in $B^{\cU}$ whose representing sequence 
is $(t\mapsto \exp(i n t))_{n\in \bbN}$ does not belong to $U_0(B^{\cU})$.

 \begin{prop} \label{P.C01} 
 There are \cstar-algebras $A$ and $B$ such that $A\prec B$ but
 $\mathrm{C}([0,1])\otimes A$ and $\mathrm{C}([0,1])\otimes B$ are not elementarily equivalent. 
 \end{prop} 
 
 \begin{proof} We show that there exists $A$ such that $\mathrm{C}([0,1])\otimes A$ 
 has stable rank one but $\mathrm{C}([0,1])\otimes A^{\cU}$ does not. 
 This suffices since having stable rank one is axiomatizable (\S\ref{S.sr}). 
 
 If $A$ is simple then $\mathrm{C}([0,1])\otimes A$ has stable rank one if and only if $U(A)$ 
 is path-connected (\cite{AnBoPePe:Geometric}), and in general the assumption of 
 $U(A)$ being path connected is necessary 
 for the former algebra having stable rank one.  

In order to have $U(A^{\cU})$ not path connected it suffices that 
for every $n\in \bbN$ there exists a unitary $u\in U(A)$ such that $u$ cannot be written 
as a product of fewer than $n$ factors of the form $\exp(ia)$ for $a$ self-adjoint. 
This is because if $\|u-v\|<2$  then the spectrum of $uv^*$ has a gap and therefore 
$u=v\exp(ia)$ for a self-adjoint $a$. 
 In \cite[Theorem 3.1]{Phi:Exponential} N.C.\ Phillips has proved that any simple \cstar-algebra $A$ 
 with two distinct
 traces which agree on projections has this property. By \cite{Vil:Range} 
 there exist simple nuclear unital separable (and even classifiable) \cstar-algebras 
with trivial $K_0$ and many traces. These algebras clearly satisfy the assumptions of Phillips' theorem. 
 \end{proof} 
 
 In the case of \cstar-algebras one could talk in general about minimal, maximal, or other 
 tensor products. 
 Since for a nuclear \cstar-algebra (and in particular for an abelian \cstar-algebra) 
 $C$ and  arbitrary \cstar-algebra $B$ there is a unique tensor product 
 $B\otimes C$,   Proposition \ref{P.C01} implies a negative answer to all of these questions.

 \begin{coro} \label{C.C01} Tensor products in the category of \cstar-algebras, and even nuclear \cstar-algebras, 
 do not preserve elementary equivalence. In particular, there are \cstar-algebras $A\equiv B$  such that $A$ is nuclear, 
 but $A\otimes \mathrm{C}([0,1])\not\equiv B\otimes \mathrm{C}([0,1])$. \qed
\end{coro}

\begin{question} \label{Q.elem} 
Can one characterize \cstar-algebras $C$ such that tensoring with $C$ preserves elementary equivalence?  
In particular, what if $C$ is an (infinite-dimensional) AF algebra, UHF algebra, $\cK$ or $\cZ$? 
Is there any infinite-dimensional $C$ such that tensoring with $C$ preserves elementary equivalence? 
\end{question} 

The `obvious' proof of a positive answer to the AF case of  Question \ref{Q.elem}  does not work, as the following example shows. 

\begin{example} There are two discrete directed systems,  $(A_n\colon n\in \bbN)$ with   $f_{mn}\colon A_m\to A_n$ 
and $(B_n\colon n\in \bbN)$ with $g_{mn}\colon B_m\to B_n$ such that $A_n\prec B_n$ for all $n$ 
but $\lim_n A_n\not\equiv \lim_n B_n$. 

Let $A$ be a random graph. Partition the set of its vertices into infinitely many infinite pieces, $V_n$ for $n\in \bbN$. 
Now let $A_n$ be the subgraph of $A$ consisting of the induced graph on $\bigcup_{j\leq n} V_n$ together 
with $V_{n+1}$ taken as a set of isolated vertices. 
This $\lim_n A_n$ is $A$. 

Now let $v$ be a vertex not in $A$, and let $B_n$ be $A_n$ with $v$ added to its set of vertices. 
Then $A_n\prec B_n$ for all $n$, but $\lim_n B$ has an isolated vertex $v$ and $A$ does not. 
\end{example}

\section{$K_0(A)$ and $A^{\eq}$}
\label{S.K0} 

The $K_0$-group\index{K@$K_0$} 
of a unital \cstar-algebra $A$ is constructed in two stages (see \cite[Chapter 3]{RoLaLa:Introduction}). 
The first stage is to construct the Murray--von Neumann semigroup, as follows.
In the first 
stage one tensors $A$ with the algebra $\cK$ of compact operators on $H$. 
The resulting algebra $A\otimes \cK$ is the \emph{stabilization} of $A$\index{stabilization}.
The set $\cP(A\otimes \cK)$ of projections of the stabilization of $A$ is quotiented by Murray--von Neumann 
equivalence. (Projections $p$ and 
$q$ are \emph{Murray--von Neumann equivalent}\index{Murray--von Neumann equivalence} 
  if there exists $v$ such that $v^*v=p$ and $vv^*=q$; see also  Example~\ref{Ex.1} \eqref{Ex.1.MvN}). By  Example~\ref{Ex.MvN}, the obtained quotient denoted $V(A)$\index{V@$V(A)$} 
belongs to $(A\otimes \cK)^{\eq}$. 
This set is equipped with the addition defined as follows. 
Fix an isomorphism between $\cK$  and $M_2(\cK)$; this gives an isomorphism 
between $A\otimes \cK$ and $M_2(A\otimes \cK)$. For projections $p$ and $q$ one defines $[p]+[q]$ to be the 
equivalence class of block-diagonal $\diag(p,q)$.\index{diag@$\diag(p,q)$} This is a well-defined operation on $V(A)$ and it turns it into an abelian 
semigroup, the so-called \emph{Murray--von Neumann semigroup}.\index{Murray--von Neumann semigroup}  

Another way to describe $V(A)$ and its relationship with $A^{\eq}$ is to use the fact that any projection $p \in A \otimes \cK$ is actually Murray--von Neumann equivalent 
to a projection in $M_n(A)$ for some $n$.  Moreover, if $p \in M_m(A)$ and $q \in M_n(A)$ then $[p] + [q]$ is the class of the projection $r \in M_{n+m}(A)$ which has $p$ and $q$ on the diagonal.  $V(A)$ is then the natural inductive limit of $\cP(M_n(A))$ modulo Murray--von Neumann equivalence.  Each part of this inductive limit exists in $A^{\eq}$ although the limit itself may not.

In either presentation, $K_0(A)$ is the Groethendieck group associated with $V(A)$. It is considered as an ordered group with distinguished 
 order unit, with $(K_0(A), K_0^+(A), 1):=(K_0(A), \overline{V(A)}, [1_A])$ where $\overline{V(A)}$ is the image of $V(A)$ in $K_0(A)$.
In the following both $V(A)$ and $K_0(A)$ are considered as discrete structures.  
\begin{thm}\label{T.K0}  
Assume $A\prec B$ are unital \cstar-algebras. 
 \begin{enumerate}
\item Then 
$V(A)$ is a subsemigroup of $V(B)$ and $K_0(A)$ is a subgroup of  $K_0(B)$. 
\item  If in addition  $A\otimes \cK\prec B\otimes \cK$, then 
 $(V(A),+)\prec (V(B),+)$  and 
 $(K_0(A), K_0^+(A), 1)\prec 
(K_0(B), K_0^+(B), 1)$.  
\end{enumerate}
\end{thm}

 \begin{proof} By Lemma~\ref{L.Norm} we have $M_n(A)\prec M_n(B)$ for all $n$. 
 Since every projection in $A\otimes\cK$ is Murray--von Neumann equivalent to a projection in some $M_n(A)$, by the definability of Murray--von Neumann equivalence 
 (1)  follows.  

(2):
It is an immediate consequence of Example~\ref{Ex.MvN} and the fact that addition in the Murray--von Neumann 
semigroup is definable that $A \otimes \cK \prec B \otimes \cK$ implies that $(V(A),+)\prec (V(B),+)$.
Given a semigroup $S$, one can easily see that its enveloping ordered group is in $S^{\eq}$; from this, we see that $(K_0(A), K_0^+(A), 1)\prec 
(K_0(B), K_0^+(B), 1)$ is also implied.
\end{proof}

In general we do not know whether $A\prec B$ implies $K_0(A)\prec K_0(B)$. 

\begin{prop} \label{P.K0.eq}
If $A$ is a purely infinite, simple, unital \cstar-algebra then $K_0(A)$ considered as an ordered group with order unit
 is in $A^{\eq}$. 
In particular, if  $A\prec B$ then  
$K_0(A)$ is an elementary submodel of $K_0(B)$, considered as ordered groups with order unit. 
\end{prop} 

\begin{proof} 
This is a consequence of a result of Cuntz, who proved that 
$K_0(A)$ is isomorphic to the quotient of $\cP(A)$  modulo the Murray--von Neumann 
equivalence and that the addition is defined in the natural way: $[p]=[q]+[q]$ if and only 
if there are partial isometries $v$ and $w$ such that $p=vv^*+ww^*$, $v^*v=q$ and $w^*w=q$. 
(see \cite[p. 188]{cuntz1981k}). 
By the proof of Theorem~\ref{T.K0}  
this shows that $K_0(A)$ belongs to $A^{\eq}$. 

Since being purely infinite and simple is axiomatizable (\S\ref{S.PI}), $A\prec B$ implies that $B$ is purely infinite, simple, and unital and 
the second claim is an immediate consequence of the above. 
\end{proof}

We don't know whether the additional assumption that $A\otimes \cK\prec B\otimes \cK$ 
 in (2) in Theorem~\ref{T.K0} is in general necessary 
 (cf.\ Question~\ref{Q.elem}); 
note however  the following. 

\begin{prop} For any \cstar-algebra $A$, $A\otimes\cK$ does not belong to~$A^{\eq}$. 
\end{prop} 

\begin{proof} This is a consequence of the fact that tensoring with $\cK$ does not commute with 
the operation of taking an ultrapower. As a matter of fact,  an ultrapower of a 
\cstar-algebra is never a nontrivial tensor product by \cite{Gha:SAW*}
(an analogous result for II$_1$ factors was proved in \cite{Li:Ultraproducts} using different methods). 
\end{proof}

\begin{lemma} \label{L.K0.AF} 
The property that a group $(G, G^+)$ is isomorphic to 
the~$K_0$ group  of some  AF algebra is axiomatizable
among ordered groups.
\end{lemma} 

\begin{proof} 
 An abelian ordered group  is isomorphic to $K_0(A)$ for an AF algebra if and only if it is
a dimension group.  
By \cite[Theorem~2.2]{EfHaShe}
an ordered group is a dimension group 
 if and only if it is   unperforated and satisfies the   
\emph{Riesz interpolation property}\index{Riesz interpolation property}  
 stating that if $a\leq c$, $a\leq d$, $b\leq c$ and $b\leq d$ then there is a 
 single element $e$ such that $a\leq e$, $b\leq e$, $e\leq c$ and $e\leq d$. 
 Moreover, a moment of introspection shows that if $A$ is AF and 
 $a,b,c,d$ are all projections in  $M_{2^n}(A)$ then $e$ can be chosen in $M_{2^n}(A)$ 
 as well. Therefore having Riesz property is (infinitely) axiomatizable, with one axiom 
 stating that $M_{2^n}(A)$  satisfies the Riesz property for every $n$. 
 \end{proof}

Related  model-theoretic results  on the relation between $A$ and $K_0(A)$ for AF algebras 
were independently obtained in  \cite{Scow:Some}.

\section{$K_1(A)$ and $A^{\eq}$} 
For the definition of $K_1(A)$ see \cite{RoLaLa:Introduction}. 
By   \cite[Example 4.7]{Li:Ultraproducts} (see also 
Proposition \ref{P.C01}) there is a \cstar-algebra $A$ such that $K_1(A)$ is trivial but 
its ultrapower $A^{\cU}$ has a nontrivial $K_1$. Therefore 
$A\prec B$ does not imply 
that $K_1(A)\prec K_1(B)$ in general.

\begin{prop}\label{P.K1.sr}  Assume   $A$ is a unital \cstar-algebra  of finite stable rank. 
\begin{enumerate} 
\item  If   $U_0(M_n(A))$ is definable for all $n$
then $K_1(A)$ belongs to $A^{\eq}$. 
\item If $A$ has real rank zero then $K_1(A)$ belongs to $A^{\eq}$.
\item If $B$ is such that $A\prec B$ then $K_1(A)$ is a subgroup of $K_1(B)$. 
\item If $B$ is such that $A\prec B$ and $U_0(M_n(B))$ is definable for all $n$ then $K_1(A)\prec K_1(B)$. 
\end{enumerate}
\end{prop} 

\begin{proof}  
(1) By \cite[Theorems 2.3 and 10.12]{Rie:Dimension}, for $n \geq \sr(A) + 2$, 
\[ K_1(A)\cong U(M_n(A))/U_0(M_n(A)). \]
The definability $U_0(M_n(A))$ implies that $K_1(A)$ is in $A^{\eq}$. 

(2) If $A$ has real rank zero then by \cite[Theorem~5]{lin1993exponential} 
the set of unitaries of the form $\exp(ia)$ for $a$ self-adjoint with finite spectrum is dense in $U_0(A)$. 
Therefore the set of unitaries of the form $\exp(ia)$ for $a$ self-adjoint of norm $\leq\pi$ is dense in $U_0(A)$, and $U_0(A)$ is definable. As this applies to $M_n(A)$ as well, 
the conclusion follows by (1).

(3) Since having stable rank $\leq n$ is axiomatizable by \S\ref{S.sr}, $A\prec B$ implies $B$ has finite stable rank. We shall prove that $U_0(M_n(A))=U_0(M_n(B))\cap M_n(A)$ for all $n$. 
For $u\in U(M_n(B))$ we have that $u\in U_0(M_n(B))$ if and only if 
there exists $k$ 
and  self-adjoint elements $a_1, \dots, a_k$ of norm $\leq 1$ such that 
$\|u -\prod_{j=1}^k \exp(i a_j \pi)\|$ is arbitrarily small. By the elementarity and Lemma~\ref{L.Norm}
for $u\in U(M_n(A))$ for any fixed $k$ 
such self-adjoint elements  exist in  $M_n(A)$ if and only if they exist in $M_n(B)$. 
Therefore $U_0(M_n(A))=U_0(M_n(B))\cap M_n(A)$ and 
by using \cite[Theorems 2.3 and 10.12]{Rie:Dimension}
as in (1) we conclude that  $K_1(A)$ is a subgroup of  $K_1(B)$. 

(4)  follows from (1) and (3). 
\end{proof}

The definability of  $U_0(A)$ depends on whether the exponential length of~$A$ is bounded (see \cite{Phi:Exponential}). 

\begin{prop} \label{P.K1.eq}
If $A$ is a purely infinite, simple, unital \cstar-algebra then $K_1(A)$ 
 is in $A^{\eq}$. 
In particular, if  $A\prec B$ then  
$K_1(A)$ is an elementary submodel of $K_1(B)$. 
\end{prop}

\begin{proof} 
In  \cite[p. 188]{cuntz1981k} it was  proved that for purely infinite and simple $A$ 
one has $K_1(A)\cong U(A)/U_0(A)$.
Since Phillips  (\cite{Phi:Exponential}) has proved that 
for a purely infinite simple $A$ we have 
\[
U_0(A)=\{\exp(ia): a=a^* \text{ and }\|a\|\leq 2\pi\},
\]
 $U_0(A)$  is in this case  definable. As the addition is definable as well, 
we conclude that $K_1(A)$ belongs to $A^{\eq}$ when $A$ is purely infinite and simple. 
\end{proof}

The following counterexample to the version of Question~\ref{Q1} in which 
the algebras are not required to be nuclear
was pointed out by N.C. Phillips. 

\begin{example} \label{Ex.Counterexample} 
There are nonisomorphic, elementarily equivalent, simple, separable,  unital \cstar-algebras
with the same Elliott invariant. More precisely, there are 
 nonisomorphic,  simple, separable,  unital \cstar-algebras
with the same Elliott invariant and the same theory as the Calkin algebra 
$\cQ:=\cB(H)/\cK(H)$. 

We have  $K_1(\cQ)\cong \bbZ$ (\cite[Example~9.4.3]{RoLaLa:Introduction}
and $K_0(\cQ)=0$ (\cite[Corollary~6.4.2]{RoLaLa:Introduction}), 
and neither of these groups has a nontrivial elementary submodel.  
If $A$ is an elementary submodel of $\cQ$, then 
Proposition~\ref{P.K0.eq}
and 
Proposition~\ref{P.K1.eq} 
imply that $A$ has the same $K$-theory as the Calkin algebra. Since the Calkin algebra is purely infinite, 
it is traceless, and  all elementary submodels of $\cQ$ have the same Elliott invariant. 

It remains to prove that $\cQ$ has nonisomorphic elementary submodels. 
The space of all complete isometry classes of 
finite-dimensional operator systems
is equipped with a natural topology. 
By \cite[Proposition~2.6(a)]{JunPis}, this space is nonseparable  
while for every separable \cstar-algebra $A$, the space of operator systems completely isometric
to a subsystem of $A$ is separable. Therefore we can find an uncountable family of 
nonisomorphic and separable elementary submodels of $\cQ$, as required. 
\end{example} 

By adapting the proof of  
\cite[Theorem~4.3.11]{Phi:Classification}, one can vary the $K$-theory of 
the algebras in  Example~\ref{Ex.Counterexample}. More interestingly, 
in \cite[Theorem~4.3.8]{Phi:Classification} Phillips constructed an infinite 
family of  nonisomorphic, simple, separable, exact, 
 unital \cstar-algebras
with the same Elliott invariant. We do not know whether Phillips's examples are elementarily equivalent, 
or whether exact counterexamples to Question~\ref{Q1} can be found.

\section{Co-elementarity} \label{Co-elementarity}
In classical first-order logic, an elementary  property is finitely axiomatizable (i.e.,  axiomatizable by a finite theory~$T$) 
if and only its negation is axiomatizable. This is an easy consequence of the compactness theorem. 
In the logic of metric structures this equivalence is not true. We will say that a class is 
\emph{co-elementary}\index{co-elementary} (or \emph{co-axiomatizable})\index{co-axiomatizable} if its complement is elementary.

There is no implication in the logic of metric structures, but since $\min$, $\max$  acts as the disjunction,  
conjunction, respectively,  we have the following. 
 
\begin{lemma} \label{L.implication} 
Assume the property $P$ is both elementary and co-elementary, and the property $R$ is elementary. 
Then both the properties $P\Rightarrow R$ and $\lnot P\Rightarrow R$ are elementary. \qed
\end{lemma} 

A useful characterization of properties that are both elementary and co-elementary is given in the following. 

\begin{prop} 
\label{P.el-co-el}
For a class $\cC$ of models of a separable language~$\calL$ the following 
are equivalent. 
\begin{enumerate}
\item \label{I.el-co-el.1} Both $\cC$ and its complement are elementary. 
\item \label{I.el-co-el.2} There is  a sentence $\phi$ such that $\cC=\{A\mid \phi^A=0\}$ 
and $\cC^{\complement}=\{A\mid \phi^A=1\}$. 
\item \label{I.el-co-el.3} There are sentence $\phi$ and $s<r$ in $\bbR$  such that 
\[
\cC=\{A\mid \phi^A\leq s\}\text{ 
and }\cC^{\complement}=\{A\mid \phi^A\geq r\}.
\] 
\item \label{I.el-co-el.4} The class $\cC$ is elementary, and for any sentence $\phi$ such that $\cC = \{A \mid \phi^A = 0\}$, there exists $r>0$ such that 
\[ 
\cC^{\complement}=\{A \mid \phi^A \geq r\}. 
\]
\end{enumerate}
In fact, in (2), $\phi$ is a sentence in the original sense of that term and not a uniform limit of such. 
\end{prop}

\begin{proof} It is clear that 
\eqref{I.el-co-el.3} implies \eqref{I.el-co-el.2} and
\eqref{I.el-co-el.2} implies \eqref{I.el-co-el.1}.

\eqref{I.el-co-el.4} $\Rightarrow$ \eqref{I.el-co-el.3}:
Since $\mathcal C$ is axiomatizable, we can find a sequence $(\psi_k)_{k=1}^\infty$ of formulas which axiomatize $\mathcal C$.
Thus $\psi := \sum_{k=1}^\infty 2^{-k}\psi_k$ is a definable predicate which axiomatizes $\mathcal C$, so by (4), there exists some $r>0$ such that $\psi^A \geq 2r$ for all $A \not\in \mathcal C$.
Choosing $n$ such that $2^{-n+1} < r$, we see that for all $A \not\in \mathcal C$,
\[ \sum_{k=1}^n \psi_k^A \geq r. \]
Thus (3) holds with 
\[ \phi:=\sum_{k=1}^n \psi_k, \]
$s:=0$, and $r$ as already defined.

\eqref{I.el-co-el.1} $\Rightarrow$ \eqref{I.el-co-el.4}:
We assume that both $\mathcal C$ and its complement are axiomatizable. As in the proof that 
 (4) implies (3) above let $\phi$ be a definable predicate which axiomatizes $\cC$, and suppose for a contradiction that there does not exist $r>0$ such that $\phi^A\geq r$ for all $A \not\in \cC$.

Thus for each $n$, there exists $A_n$ such that $0<\phi^{A_n}<1/n$. 
Thus $A_n \not\in \cC$, so that for any nonprincipal ultrafilter $\cU$,
\[ A:=\prod_\cU A_n \not\in \cC. \]
However, by \L o\'s' theorem $\phi^A=0$, which means that $A \in \cC$, a contradiction. 
\end{proof} 

\subsection{Abelian algebras} 
In section \S\ref{S.Axiomatizable}, we saw that the property of being abelian was elementary and co-elementary.  Somewhat less trivially, we also saw that being $n$-subhomogeneous was elementary and co-elementary.

 \subsection{Infinite algebras} By \S\ref{S.Infinite} and \S\ref{S.SF} being infinite is elementary and co-elementary. 
 
 \subsection{Algebras containing a unital copy of $M_n(\bbC)$} 
  \label{S.Mn} 
Fix $n\geq 2$. 
In the class of unital \cstar-algebras, the class of algebras containing a unital copy of $M_n(\bbC)$ 
is elementary and co-elementary.

Let  $\alpha_n^u$ be the $n^2$-ary formula in Example~\ref{Ex.1}  \eqref{L.alpha-n-u}
whose zero-set is the set of all matrix units of unital copies of $M_n(\bbC)$. 
This formula is  weakly stable, and therefore there exists $\e>0$ such that in every unital \cstar-algebra $A$,
$\alpha^u_n(\bar a)^A\leq \e$ for some $\bar a$ implies the existence of $\bar b$ such that  $\alpha^u_n(\bar b)^A=0$.
Therefore having a unital copy of $M_n(\bbC)$ is axiomatized by 
$\phi:=\inf_{\bar x}\alpha^u_n(\bar x)$, and we see that this formula satisfies Proposition \ref{P.el-co-el} (3) (with $s:=0,r:=\e$).

 The class of algebras containing a (not necessarily unital) copy of $M_n(\bbC)$ for $n\geq 2$  
is also elementary and co-elementary. This is proved as above, using formula
$\alpha_n$ in place of $\alpha_n^u$ (see Example~\ref{Ex.1} \eqref{L.alpha-n}).

The argument from \S\ref{S.Mn} easily generalizes to show the following. 

 \begin{lemma} \label{L.PnotP} Assume that property $P$ is axiomatizable by $\inf_{\bar x}\varphi(\bar x)$ where $\varphi(\bar x)$ is weakly stable. 
 Then $\lnot P$ is also axiomatizable. \qed
 \end{lemma}

\subsection{Definability of sets of projections}\label{S.sets.proj}
It was noted in Example~\ref{Ex.1} that the set of all projections is definable. 
We now use Theorem~\ref{Th.sem.def} and the notion of co-elementarity to
show that  some distinguished sets of projections are also definable. 

\begin{lemma} \label{L.def.proj} Assume $P$ is a property of \cstar-algebras that is both elementary and co-elementary. 
Then the set of projections $q$ such that $qAq$ has property $P$ is definable. 
\end{lemma} 

\begin{proof} Choose a
projection $q$ in an  ultrapower $\prod_\cU A_i$. 
By the definability of the set of all projections (Example~\ref{Ex.1}), we can choose a representing sequence 
$(q_i)$ so that $X:=\{i: q_i$ is a projection$\}$ belongs to $\cU$. 
Since the property $P$ is elementary,   
if the set  $Y:=\{i\in X: q_iA_i q_i$ satisfies $P \}$ belongs to $\cU$ then  $q(\prod_\cU A_i )q=\prod_{\cU} (q_i A_i q_i)$ 
has property $P$. Similarly, since the negation of $P$ is elementary, 
if $X\setminus Y$ belongs to $\cU$ then $q(\prod_\cU A_i) q$ satisfies the negation of $P$. 
This shows that 
\[ \{q \in \prod_\cU A_i \mid q\text{ satisfies }P\} = \prod_\cU \{q_i \in A_i \mid q_i \text{ satisfies }P\}, \]
and therefore the conclusion follows by Theorem~\ref{Th.sem.def}. 
\end{proof}

A projection $p$ in a \cstar-algebra $A$ is \emph{abelian}\index{projection!abelian} if $pAp$ is abelian, 
\emph{finite}\index{projection!finite} if $pAp$ is finite (see \S\ref{S.SF}), and \emph{infinite}\index{projection!infinite} if $pAp$ 
is not finite. 

\begin{proposition} \label{P.definable.projection}Each of the sets of abelian, finite and infinite projections is definable. 
Therefore the class of \cstar-algebras having an abelian, respectively finite, or infinite projection, is elementary. 
\end{proposition}

\begin{proof} This follows by Lemma~\ref{L.def.proj} and 
the elementarity and co-elementarity of being abelian (\S\ref{S.Abelian} and 
\S\ref{S.nonabelian}) and  being finite (\S\ref{S.SF}  and  \S\ref{S.Infinite}). 
\end{proof} 

  \subsection{Stable rank one} \label{S.SR.1}
The class of algebras of stable rank one is  elementary and co-ele\-men\-ta\-ry
by Proposition~\ref{P.SR}. 

  \subsection{Real rank zero} \label{S.rr0.1}
The class of algebras of real rank  zero is  elementary and co-ele\-men\-ta\-ry
by Proposition~\ref{P.RR} (2).

  \subsection{Purely infinite simple \cstar-algebras }\label{S.PI} 
  A \cstar-algebra is \emph{purely infinite and simple}\index{C@\cstar-algebra!purely infinite and simple} if it has dimension greater than 1 and for every two nonzero positive elements $a$ and $b$ there is a sequence $x_n$, for $n\in \bbN$, such that $\|x_nax_n^*-b\|\to 0$ as $n\to \infty$. 
  (In other words, $a$ Cuntz-dominates $b$---see \S\ref{S.Cuntz}.)
  
  Being of dimension greater than 1 (i.e.,  not being isomorphic to $\bbC$) is clearly elementary and co-elementary.    
Consider the following sentence:
    \[
\alpha := \sup_{x\text{ pos.}} \sup_{y\text{ pos.}}  \min(\|x\| \dotminus 1/2, \inf_z\|2zxz^* - y\|)  
\]
 Roughly, $\alpha^A=0$ says that every positive element $a$ of norm at least 1/2 
 is Cuntz-above every positive contraction $b$  \emph{with the witnesses
 of Cuntz-ordering taken to be of norm $\leq \sqrt 2$}.
It is therefore clear that $\alpha^A=0$ implies $A$ is purely infinite and simple.
 
 
 On the other hand, assume $A$ is purely infinite and simple. Suppose $a\in A_+$ and $b\in A_+$ are such that $\|a\|\geq 1/2$ and $\|b\|\leq 1$ as in $\alpha$. 
 If $a$ is positive of norm $>1/2$   then $b\precsim (a-1/2)_+$. 
 By Lemma \ref{L.Cuntz} we can find $z$ such that $\|z\|\leq \sqrt 2$ 
 and $b=zaz^*$, and therefore $z/\sqrt 2$ witnesses that the given instance of $\alpha$ evaluates to 0. 
 
 Finally, suppose that $\alpha^A < 1$.  Fix $a \in A$ positive, a projection $p$ and, by working in $pAp$, we may assume without loss  of generality that $pap = a$ and $A$ is unital.  Then, if $\| a \| > 1/2$ and we choose $z$ witnessing that $\| 2zaz^* - 1\| < 1$, we see that $zaz^*$ is positive and invertible so in fact, $\alpha^A = 0$.  We conclude then that by considering the sentence $\alpha \dotminus 1$, one sees that being purely infinite and simple is co-elementary.

\section{Some non-elementary classes of \cstar-algebras} \label{S.non-elementary} 

 By Theorem \ref{T.Ax} every axiomatizable class of \cstar-algebras is closed under taking ultraproducts. 
 This is not a sufficient condition for the axiomatizability. The class of nonseparable \cstar-algebras is not axiomatizable due to
 the downward L\"owenheim--Skolem theorem (Theorem \ref{DLS}) and the class of non-nuclear \cstar-algebras is not axiomatizable
 because it not closed under ultraroots. However, both of these classes are easily seen to be closed under ultraproducts.

Modulo  \S\ref{S.SubHom}, the following is a reformulation of an unpublished result of Kirchberg.

\begin{prop} \label{P.nonax} For a \cstar-algebra $A$ the following are equivalent. 
\begin{enumerate}
\item \label{I.nonax.1} Every \cstar-algebra elementarily equivalent to $A$ is nuclear. 
\item \label{I.nonax.2} Every \cstar-algebra elementarily equivalent to $A$ is exact. 
 \item \label{I.nonax.3} $A$ is subhomogeneous. 
 \end{enumerate}
 \end{prop}

Our proof of Proposition~\ref{P.nonax} relies on two  lemmas. 

\begin{lemma} \label{L.prod-M-n} 
There exists an injective $^*$-homomorphism  
\[
\Phi\colon \prod_{n\in \bbN} M_n(\bbC)\to \ASA.
\] 
It can be chosen so that  for every nonprincipal ultrafilter $\cU$ on $\bbN$ the composition $\pi_{\cU}\circ \Phi$ is injective, with 
\[
\pi_{\cU}\colon \ASA\to \prod_{\cU} M_n(\bbC)
\]
denoting the quotient map. 
\end{lemma} 

\begin{proof} For every $n$ the algebra $\prod_{j<n} M_j(\bbC)$ is clearly isomorphic to a (not necessarily unital) 
subalgebra of $M_k(\bbC)$ for all $k\geq n(n+1)/2$; let $\Psi_{n,k}$ denote this embedding. 
Let $J_n:=[n(n+1)/2, (n+1)(n+2)/2)$.  
Define $\Psi\colon  \prod_n M_n(\bbC)\to \prod_n M_n(\bbC)$ by 
\[
\Psi(\bar a)(k):=\Psi_{n,k}(\bar a\rs n)
\]
if $k\in J_n$. The composition of $\Psi$ with the quotient map $\pi\colon \prod_n M_n(\bbC)\to \ASA$ clearly satisfies the requirements. 
\end{proof} 

\begin{lemma} \label{L.nonexact} If $\cU$ is a nonprincipal ultrafilter on $\bbN$ then $\prod_{\cU} M_n(\bbC)$ is not exact. 
\end{lemma} 

\begin{proof} 
The full group algebra $\cst(F_n)$ is not exact by  \cite[Corollary~3.7.12]{BrOz:cstar}. 
It is  also isomorphic to a subalgebra of $\prod_n M_n(\bbC)$ by  \cite[Theorem 7.4.1]{BrOz:cstar}. 
 Lemma~\ref{L.prod-M-n} now implies that $\prod_{\cU}M_n(\bbC)$ has a subalgebra isomorphic to the nonexact algebra $\cst(F_n)$. 
Since  exactness passes to subalgebras (\cite[Proposition~10.2.3]{BrOz:cstar}), this completes the proof. 
 \end{proof} 

\begin{proof}[Proof of Proposition~\ref{P.nonax}]  Clearly \eqref{I.nonax.1} implies \eqref{I.nonax.2}. 
If \eqref{I.nonax.3} holds and $n$ is such that $A$ is $n$-subhomogeneous then 
by \S\ref{S.SubHom} every $B$ elementarily equivalent to~$A$ is $n$-subhomogeneous and therefore nuclear, 
hence \eqref{I.nonax.1} follows. 

Now assume \eqref{I.nonax.3} fails and $A$ is not $n$-subhomogeneous for any $n$. 
Let~$\cU$ be a nonprincipal ultrafilter on $\bbN$; we'll prove that $A^{\cU}$ (which is elementarily equivalent to $A$) is not exact, by proving more generally that if $A_n$ is not $n$-subhomogeneous then $\prod_{\cU} A_n$ is not exact.
First,~$A_n$ has an irreducible representation on a Hilbert space 
of dimension $>n$ and (as in   the proof of \S\ref{S.not.SubHom}) 
for every $n$ there exists a (not necessarily unital, even if $A_n$ is unital) 
subalgebra $B_n$ of $A_n$  and an ideal $J_n$ of $B_n$ such that $B_n/J_n\cong M_n(\bbC)$. 
Then $\prod_{\cU} A_n$ has  $\prod_{\cU} B_n$ 
as a subalgebra, and the latter has $\prod_{\cU}B_n/\prod_{\cU} J_n\cong \prod_{\cU} M_n(\bbC)$ as a quotient. 
This algebra is not exact by  Lemma~\ref{L.nonexact}. 
Since exactness passes to quotients (this is a deep result, see 
\cite[Corollary~9.4.3]{BrOz:cstar}), 
$\prod_{\cU} B_n$ is not  exact. Finally, since exactness passes to subalgebras by \cite[Proposition~10.2.3]{BrOz:cstar}, 
this completes the proof. 
\end{proof} 

\begin{remark} The anonymous referee suggests the following proof of Proposition~\ref{P.nonax} which relies on conditional expectation: if $A_n$ is not n-subhomogeneous for each $n \in \bbN$ then there is a u.c.p.\ embedding of $\prod_\cU M_n(\bbC)$ into $\prod_\cU A_n$ with conditional expectation (this is essentially Lemma 4.13 in \cite{goldbring-2015}). It is clear then, using the finite-dimensional approximation version of exactness, that exactness of $\prod_\cU A_n$ would confer exactness on $\prod_\cU M_n(\bbC)$.
\end{remark}
 
 The following immediate consequence of Proposition~\ref{P.nonax} 
  amounts to a strengthening of a result implicit in  \cite{FaHaSh:Model2}. 
  
 \begin{coro} \label{C.nonax} 
 Assume $\cC$ is a class of \cstar-algebras that contains an infinite-dimensional and simple algebra and 
 all algebras in $\cC$ are exact. Then $\cC$ is not axiomatizable. \qed
 \end{coro}

   Some   classes of \cstar-algebras are non-elementary but have some semblance of elementarity.

  \begin{definition}
  We say that a property P is {\em local}\index{local} if whenever a structure $A$ has property P and $A \equiv B$ then $B$ has property P.
  \end{definition}

Item \eqref{I.S.simple} of the following theorem was  first observed in \cite{GeHa}.

  \begin{theorem}  \label{T.nonaxiomatizable}  None of the following classes of \cstar-algebras is axiomatizable:   
  \begin{enumerate}
  \item\label{I.nonax.1.x}  UHF, AF, AI, AT, nuclear, and exact \cstar-algebras. 
  \item \label{I.S.simple}  Simple \cstar-algebras 
  \item \label{I.S.Traceless} 
Traceless \cstar-algebras. 
  \item\label{I.S.characterless}  \cstar-algebras without a character. 
  \pushcounter
\end{enumerate}
The following classes of \cstar-algebras are local, but not elementary. 
\begin{enumerate}
\popcounter
\item\label{I.nonax.finite}    Finite-dimensional \cstar-algebras. 
\item\label{I.nonax.Pop}  Unital \cstar-algebras without a trace. 
\item \label{I.nonax.characterless} Unital \cstar-algebra without  a character. 
\end{enumerate}
  \end{theorem} 

\begin{proof}  
 \eqref{I.nonax.1.x} 
and \eqref{I.S.simple} are consequences of  Corollary~\ref{C.nonax} (take e.g. the CAR algebra).

\eqref{I.S.Traceless} 
For any algebra $A$, recall that the Cuntz--Pedersen
ideal  $A_0$ is the closure of the span of the set of commutators in $A$.  
  $A$ has a trace if and only if $A \neq A_0$ (\S\ref{S.Def.Tau.1}).
By \cite[Example~4.11]{robert2015lie}  for every $n$ there exist an algebra $B_n$ and $a_n \in B_n$ 
 such that $B_n$ has no trace, $a_n$ has norm at most 1, and $a_n$ has distance at least 1 from any  
linear combination of $n$ commutators.  
The element in  $\prod_{\cU}B_n$ with representing sequence~$(a_n)$ does not belong to 
$(\prod_{\cU} B_n)_0$, and therefore $\prod_{\cU} B_n$ has a trace.  
This shows that the class of \cstar-algebras without a trace is not closed under ultraproducts and that the class of tracial \cstar-algebras is not co-elementary.

 \eqref{I.S.characterless}
  follows immediately from the fact that there is a sequence $A_n$, for $n\in \bbN$, of 
 \cstar-algebras without a character such that $\prod_{\cU} A_n$ has a character
  \cite[Corollary~8.5]{RobRor:Divisibility}.   
  
\eqref{I.nonax.finite} 
Since for any finite-dimensional \cstar-algebra $F$  the unit ball of $F$  is compact, its theory is categorical. 
 
\eqref{I.nonax.Pop} 
By a result of Pop, \cite{Pop}, the property of a unital \cstar-algebra $A$ being without a trace is local.
\end{proof}

\chapter{Types}\label{S.Types}
In this section we introduce model-theoretic tools capable of handling  most important properties of \cstar-algebras
 (such as nuclearity---see Corollary~\ref{C.nonax}).

\section{Types: the definition} 
For a fixed theory $T$, a \emph{complete type}\index{type!complete} $p$ in the free variables $\bar x$ is a functional from $\ccW^{\bar x}$, the Banach algebra of definable predicates in the variables $\bar x$ (see \S\ref{S.Definable.Predicates}), to $\bbR$ such that for some $M$ satisfying $T$ and $\bar a$ in $ M$, for all $\varphi \in \ccW^{\bar x}$, $p(\varphi) = \varphi^M(\bar a)$.  
We say that such a tuple $\bar a$ realizes $p$.
A \emph{partial type}\index{type!partial} is a partial function arising by restricting a complete type to a subset of $\ccW^{\bar x}$; we may always assume that the domain of a partial type is a subspace of $\ccW^{\bar x}$
by extending it to the linear span of its domain.

Types are very special functionals on $\ccW^{\bar x}$.  The space of complete types in the variables $\bar x$ relative to a given theory $T$ will be denoted $S_T^{\bar x}$.  If we endow the dual space with the weak$^*$ topology then by considering ultraproducts, we see that the space of types is closed under weak$^*$-limits and is in fact compact. The weak$^*$ topology when restricted to the space of types is also known as the logic topology.  By the definition of the logic topology, for any definable predicate $\varphi$ the function $f_\varphi$ defined by $f_\varphi(p) = p(\varphi)$ is continuous.  Since the type space is compact and distinct types are separated by formulas, the Stone-Weierstrass theorem guarantees the following:
\begin{proposition}\label{P.type.SW}
Any continuous function $f:S_T^{\bar x} \rightarrow \bbR$ is of the form $f_\varphi$ for some definable predicate $\varphi$.
\end{proposition}

Types also satisfy various permanence properties arising from the formation of formulas.  For instance, if $\varphi_1,\ldots,\varphi_n$ are all definable predicates in the variables $\bar x$ and $f:\bbR^n \rightarrow \bbR$ is a continuous function then for any complete type $p$, $p(f(\varphi_1,\ldots,\varphi_n)) = f(p(\varphi_1),\ldots,p(\varphi_n))$.  Moreover, any complete type is continuous with respect to the norm on $\ccW^{\bar x}$. To summarize, we have the following. 

\begin{prop} 
Complete types are continuous Banach algebra homomorphisms from $\ccW^{\bar x}$ to $\bbR$. \qed
\end{prop} 

\subsection{Types as sets of conditions}

There are a number of different ways to view partial types which will be used later.  
As with complete theories, complete types are determined by their kernels and so we will often write $\varphi \in p$ to mean $p(\varphi) = 0$.  By assuming that the domain of a partial type is 
a subspace, it can also be identified with its kernel. The domain of a partial type is 
the linear span of its kernel and the scalars. 

Using this point of view, a partial type $q$ corresponds to a  non-empty weak$^*$-closed subspace of  $S_T^{\bar x}$: 
\[ 
X_q :=\{ p \in S_T^{\bar x} : q \subseteq p \}. 
\]
  In fact, any non-empty weak$^*$-closed subspace corresponds to a partial type.
For instance, suppose $r \leq s$ are in $\bbR$ and $\varphi \in \ccW^{\bar x}$. If we consider a basic closed set of the form $C := \{ p \in S_T^{\bar x} : r \leq p(\varphi) \leq s \}$ then $C = X_q$ where $q$ is the partial type determined by 
\[
q(r \dotminus \varphi) = q(\varphi \dotminus s) = 0.
\] We sometimes specify a partial type by giving a list of closed {\em conditions}\index{condition} i.e., restrictions of the form $r \leq \varphi \leq s$, which are meant to determine a closed subspace of~$S_T^{\bar x}$.  When giving a list of conditions, the main question which needs to be answered is whether the list of conditions can be completed to a type at all; this is often referred to as being consistent i.e., whether the basic closed sets determined by these conditions
have non-empty intersection.

\begin{definition}
A type $p$ is \emph{omitted}\index{type!omitted} 
in a model $A$ if no tuple in~$A$ realizes $p$. 
\end{definition}

For a complete type in a separable theory, there is an omitting types theorem (see e.g.,  \cite[Theorem~12.6]{BYBHU}). 
\begin{thm}
For a complete type $p$ in a separable complete theory~$T$, the following are equivalent:
\begin{enumerate}
\item $T$ has a model which omits $p$.
\item The zero-set of $p$ is not definable.
\end{enumerate}
\end{thm}
Unfortunately, there is no corresponding result for partial types (see~\cite{BY:Definability}) and the general situation is complicated. By  \cite[Theorem~1]{FaMa:Omitting} there is no simple general criterion for when a given 
type is omissible in a model of a given theory.

 \section{Beth definability} 
\label{S.Beth}
Here is another very useful tool in determining what is definable in a given setting; it is called the \emph{Beth definability theorem}.\index{Beth definability theorem}
\begin{thm}\label{Th.Beth}
Suppose that $T$ is a theory in a language $\calL$ and $\calL_0 \subseteq \calL$ is a sublanguage with the same sorts.  Assume that whenever $M \in \Mod(T)$ and $\sigma$ is an $\calL_0$-automorphism of $M\!\!\restriction_{\calL_0}$ then $\sigma$ is an $\calL$-automorphism of $M$.  Then every $\calL$-formula is $T$-equivalent to a definable predicate in $\calL_0$.  Moreover, $\{ M\!\!\restriction_{\calL_0} : M \in \Mod(T) \}$ is an $\calL_0$-elementary class.

\end{thm}

\begin{proof} Suppose that $\varphi(\bar x)$ is a formula in $\calL$.  We will prove that $\varphi$ induces a well-defined, continuous function on the space $\calL_0$-types in the variables $\bar x$ and use Proposition \ref{P.type.SW} to conclude that $\varphi$ is $T$-equivalent to a definable predicate in $\calL_0$.

To this end we 
expand $\calL$ by adding two tuples of constants $\bar c$ and $\bar d$
of the same sort as $\bar x$ in $\varphi(\bar x)$. 
For every $\e > 0$, consider the theory 
\begin{multline*}
\Gamma_\e := T \cup \{ | \varphi(\bar c) - \varphi(\bar d)| \geq \e \} \cup \\
 \{ |\psi(\bar c) - \psi(\bar d)| \leq \delta : \delta \in \mathbb{R}^+, \psi \text{ is an $\calL_0$-formula of the same sort as $\varphi$}\}.
\end{multline*}
Let us for a moment assume that the
theory $\Gamma_\e$ is inconsistent for every $\e > 0$. 
Then,  by compactness,  for every $\e > 0$ there are $\calL_0$-formulas $\psi_i$ for $i \leq k$ and a $\delta > 0$ such that if $A$ is any model of $T$ and $\bar a,\bar b \in A$ with $|\psi_i(\bar a) - \psi_i(\bar b)| < \delta$ then $|\varphi(\bar a) - \varphi(\bar b)| < \e$.  From this we conclude that $\varphi$ induces a well-defined continuous function on the $\calL_0$-type space and hence is $T$-equivalent to an $\calL_0$-definable predicate.

Toward a contradiction, suppose that $\Gamma_\e$ is consistent for some $\e > 0$ and $M$ satisfies $\Gamma_\e$.  In particular, we have $\bar a$ and $\bar b$ in $M$ with the same $\calL_0$-type.  For simplicity, assume that there is a saturated model $N \in \Mod(T)$ such that $M \prec N$.  Then $N\!\!\restriction_{\calL_0}$ is saturated as an $\calL_0$-structure and there is an automorphism $\sigma$ of $N\!\!\restriction_{\calL_0}$ sending $\bar a$ to $\bar b$.  By assumption this is an $\calL$-automorphism of $N$ which contradicts that $\bar a$ and $\bar b$ disagree on $\varphi$.  If you don't want to assume the existence of a saturated model, $N$ can be chosen to be a special model of $T$ and the same proof goes through.  The moreover clause follows immediately from Theorem \ref{T.Ax}.
\end{proof}

We include the following corollary to match the formalism of Theorem \ref{T.conceptual-completeness}.
\begin{corollary}
Suppose that $T$ is a theory in a language $\calL$ and $T \subseteq T'$ where $T'$ is a theory in a language $\calL'$ with no new sorts.  Further suppose that the forgetful functor $F:\Mod(T') \rightarrow \Mod(T)$ is an equivalence of categories.  Then every predicate in $\calL'$ is $T'$-equivalent to a definable predicate in $\calL$.
\end{corollary}

Here is a practical way to use the Beth definability theorem.  
First a bit of notation: if $A$ is metric structure and $P$ is a bounded, uniformly continuous, 
real-valued function on a product of sorts from $A$ then by the structure $(A,P)$ we will mean the structure in the language of $A$ together with one new predicate in the appropriate product of sorts, interpreted by $P$.  We say that $A$ has been expanded by $P$.

\begin{corollary}\label{C.sem.def}
Suppose that $\cC$ is an elementary class of structures in a language $\calL$ and, for every $A \in \cC$, the structure $A$ is expanded by a predicate $P^A$ which is uniformly continuous with uniform continuity modulus independent of our choice of $A$.  Let $\cC' := \{ (A,P^A) \colon A \in \cC \}$ be a class of structures for an expanded language $\calL'$ with a predicate for $P$.  If $\cC'$ is an elementary class for a theory $T'$ in the language $\calL'$ then $P$ is $T'$-equivalent to a definable predicate in $\calL$.
\end{corollary}

\begin{proof} It suffices to see that if $\sigma$ is an automorphism of $A \in \cC$ then $\sigma$ preserves $P^A$.  But if $P^A$ differs from $\sigma(P^A)$ then $(A,\sigma(P^A))$ is a model of $T'$ which is not in $\cC'$.
\end{proof}

We give two examples of how the Beth definability theorem can be applied.  
First of all, we give a proof of Theorem \ref{Th.sem.def}.
Suppose we have an elementary class $\cC$, and $A \mapsto S^A$ is a uniform assignment of closed sets as in Theorem \ref{Th.sem.def}.  For $A \in \cC$, define a relation $R^A$ on $A$ by $R^A(\bar a) = d(\bar a,S^A)$.  These functions are uniformly continuous and so let $\cC' := \{ (A,R^A) : A \in \cC \}$.  It is straightforward to check that this is an elementary class using Theorem \ref{T.Ax}. In fact, by using the assumptions of 
 Theorem~\ref{Th.sem.def} one sees that $\cC$ is closed under ultraproducts and ultraroots.  
If $A \in \cC$ and $\sigma$ is an automorphism of $A$, then by the functoriality of the uniform assignment, $\sigma$ induces a permutation of $S^A$.  This implies that $\sigma$ preserves the predicate $R^A$.
 We conclude from Theorem \ref{Th.Beth} that $R^A$ is equivalent to a definable predicate in the language of $\cC$ which shows that our assignment is a definable set.
 
For a second example, we give a Beth definability argument of the following result of Lupini  which appears in \cite[Appendix C]{goldbring2014kirchberg}. 
\begin{lemma} \label{L.Norm} The function $F_n\colon A_1^{n^2}\to [0,\infty)$ given by 
(when $\bar a\in A^{n^2}$ is identified with an element of $M_n(A)$) 
\[
F_n(\bar a)=\|\bar a\|
\]
is defined by a sup-formula. Moreover, the unit ball of $M_n(A)$ is a definable set.
\end{lemma} 

\begin{proof} 
Assume that $\cC'$ is the class of all pairs of structures $(A,F_n^A)$ where $A$ is a \cstar-algebra.  Now if $(A_i,F_n^{A_i}) \in \cC'$ for all $i \in I$ and $\cU$ is an ultrafilter on $I$ then if $A = \prod_\cU A_i$ and $\bar a \in A_1^{n^2}$, as an operator, $\bar a$ is just the ultralimit of $\bar a_i$ as $i \rightarrow \cU$ and so $F_n^A(\bar a) = \lim_{i \rightarrow \cU} F_n^{A_i}(\bar a_i)$ which shows that $F_n$ is a definable predicate in the language of \cstar-algebras.  To see that this predicate is universal just involves a preservation theorem (Proposition \ref{prop:preservation} (1)).  The unit ball is a definable set as it is the zero-set of $F_n \dotminus 1$.
\end{proof} 

We remark that as in the case of scalars, discussed after Theorem \ref{T.conceptual-completeness}, by a very similar argument, $M_n(\bbC)$ lives in $A^{\eq}$ for any \cstar-algebra $A$ as does its action on $M_n(A)$.  We will use this in \S\ref{S.2nd-proof}.
  
    \section{Saturated models}  \label{S.Saturated}
    
Remember that when we say a formula has parameters in a set $B$ we mean that we are working in the language $\calL_B$ with constants for the elements of $B$. We say that a \cstar-algebra $A$ is \emph{countably saturated}\index{countably saturated} if every type $p$ relative to the theory of $A$ whose formulas have parameters over a separable subalgebra $A_0 \subseteq A$ is realized in $A$.  Important instances of countable saturation are $A^\cU$ when  $\cU$ is a non-principal ultrafilter on~$\bbN$ and $\prod_n A_n/\bigoplus_n A_n$
(the first result is folklore and the second was proved in \cite{farah2014rigidity}).  For more 
information about countable saturation and its use in operator algebra, see
\cite{FaHa:Countable},  \cite{farah2014rigidity}, or \cite{FaHaRoTi:Relative}.

\subsection{} \label{Ex.ssa-stable} 

A unital \cstar-algebra $A$ has  \emph{approximately inner half-flip}\index{approximately inner half-flip}
 if the
$^*$-homomorphisms from $A$ into $A\otimes A$, where $\otimes$ denotes the minimal tensor product,
defined by $a\mapsto a\otimes 1_A$ and $a\mapsto 1_A\otimes a$
are approximately unitarily equivalent. In \cite[Lemma 3.9]{KircPhi:Embedding} it  
was proved that algebras with approximately inner half-flip 
are nuclear, simple and have at most one trace. 
The following result is taken from \cite{Ror:Classification} and also appears in \cite{KircPhi:Embedding}.  We include the proof in order to highlight its model theoretic nature.

\begin{thm}\label{T.aihf}
If $A$ is a separable unital \cstar-algebra and $D$ has approximately inner half-flip then:
\begin{enumerate}[leftmargin=*]
\item If $D$ unitally embeds into $A' \cap A^\cU$ then $A \cong A \otimes D$, and the isomorphism is approximately unitarily equivalent to the map $id_A \otimes 1_D$; 
\item If $D$ is strongly self-absorbing then the converse holds, so that $D$ unitally embeds into $A' \cap A^{\cU}$ if and only if $A \cong A \otimes D$.
\end{enumerate}
\end{thm}

\begin{proof}
As the forward direction of (2) follows from (1), it is sufficient to prove (1) and the reverse implication in (2). 
We start with the latter.

If $D$ is strongly self-absorbing and $A \cong A \otimes D$ then by Lemma \ref{L.tensor}, the embedding $a \mapsto a \otimes 1_D$ is an elementary embedding and so we can assume that $A \prec A \otimes D \prec A^\cU$ which means $D$ embeds into $A' \cap A^\cU$.

We now prove (1). To do so, we introduce some notation. Suppose 
that $D$ has approximately inner half-flip and let  $B := A \otimes D_1$ where $D_1 \cong D$ via an isomorphism $\psi_1:D \rightarrow D_1$ and consider $A$ embedded into $B$ by $a \mapsto a \otimes 1_{D_1}$.  This induces an embedding from $A^\cU$ into $B^\cU$.  Fix a copy of $D$ embedded in $A' \cap A^\cU$,  
call it $D_2$, and an isomorphism $\psi_2:D \rightarrow D_2$.  
As subalgebras of $B^\cU$, $D_1$ and $D_2$ commute and so $\mathrm{C}^*(D_1,D_2) \cong D \otimes D$ (since $D$ is nuclear, this is necessarily 
the minimal tensor product).  Since $D$ has approximately inner half-flip, the maps $\psi_1$ and $\psi_2$ are approximately unitarily equivalent.  In $B^\cU$ this means that there is a unitary $u \in A' \cap B^\cU$ such that 
for all $d \in D$, $u^*\psi_1(d)u = \psi_2(d)$.

If $u = \langle u_n : n \in \bbN \rangle/\cU$ where the $u_n$'s are a representing sequence of unitaries, then 
\[
\lim_{n \rightarrow \cU} \| [u_n,a] \| = 0 \text{ for all } a \in A
\]
and since $u^*bu \in A^\cU$ for all $b \in B$, we have
\[
\lim_{n \rightarrow \cU} d(u^*_nbu_n,A) = 0 \text{ for all } b \in B.
\]
From these two facts, one deduces that if $a_1,\ldots,a_m \in A$ and $b_1,\ldots,b_n \in B$ and $\e > 0$ then there is a unitary $v$ and $c_j \in A$ for all $j \leq n$ such that for all $i \leq m$ and $j \leq n$
\[
\|v^*a_iv - a\| < \e \text{ and } \|v^*b_jv - c_j \| < \e.
\]
Now if we fix a dense subset $\langle a_n : n \in \bbN\rangle$ of $A$ and $\langle b_n : n \in \bbN \rangle$ of $B$, we can pick unitaries $\langle v_n : n \in \bbN \rangle$ and $a^j_m \in A$ for all $j \leq m$ and $m \in \bbN$ inductively such that
\begin{enumerate}
\item $\| v^*_n a_i v_n - a_i \| < 1/2^n$ for all $i \leq n$,
\item $\|v^*_n a^j_m v_n - a^j_m \| < 1/2^n$ for all $j \leq m < n$ and 
\item $\|w_n^* b_i w_n - a^i_n\| < 1/2^n$ for all $i \leq n$ where $w_n := v_1v_2\ldots v_n$.
\end{enumerate}
From this one concludes that for all $a_i$, the sequence $\langle w_n a_i w^*_n : n \in \bbN \rangle$ is a Cauchy sequence and, since the $a_i$'s are dense in $A$,  that
the map $\varphi(a) = \lim_{n \rightarrow \infty} w_n a w_n^*$ is a well-defined $^*$-homomorphism ($\varphi$ is a uniform limit of inner automorphisms).  Moreover, the sequence $\langle w^*_n b_i w_n : n \in \bbN \rangle$ is also Cauchy for all $i$ and by the choices of the $a^j_m$'s converges to something in $A$.  This means that $b_i$ is in the range of $\varphi$ for all $i$ and so $\varphi$ is surjective.  We conclude that $A \cong B$ and that $\varphi$ is approximately unitarily equivalent to $\mathrm{id}_A \otimes 1_{D_1}$ by construction.
\end{proof}

We now discuss several examples which relate axiomatizability and the relative commutant.  Suppose that $D$ is a separable, unital \cstar-algebra and $A$ is any separable \cstar-algebra.  $D$ unitally embeds into $A^{\cU} \cap A'$ if and only if the following set of conditions over $A$ is a type realized in $A^\cU$
(recall the definition of $\Diag(A)$ in \S\ref{S.Diagram})
\begin{align*}
\Sigma := \{ \varphi(\bar d) \leq 1/n \colon n \in \bbN, \bar d\text{ in }  D, \varphi(\bar d) \in \Diag(D) \} \\
\cup \{ \|[a,d]\| \leq 1/n \colon n \in \bbN, a \in A, d \in D\}.
\end{align*}
Since $A^\cU$ is countably saturated, this is equivalent to the assertion that $\Sigma$ is a type in the theory of $A$. This says that the property of $D$ embedding into  $A'\cap A^{\cU}$ is \aea{}  by sentences of the form:
$$\sup_{\bar y} \inf_{\bar x} (\| [\bar y,\bar x]\| + \varphi(\bar x))$$
where $\varphi(\bar x)$ range over the atomic diagram of $D$.

By Theorem \ref{T.aihf}, 
this has an immediate consequence for the axiomatizability of $D$-stability when $D$ is strongly self-absorbing
(see \S\ref{S.Tensorial} and Theorem~\ref{T.aihf}).  

\begin{corollary} Suppose $D$ is strongly self-absorbing. Then $D$-stability is \aea{} 
among separable \cstar-algebras. 
In particular, $\cZ$-stability is an elementary property among separable \cstar-algebras. \qed
\end{corollary}

\subsection{} \label{Ex.MFstable}  
A frequently asked question about  massive \cstar-algebras such as ultrapowers or relative commutants 
of its distinguished separable subalgebra  is to characterize the separable \cstar-algebras that arise (up to isomorphism) as subalgebras.
Particular instances of this problem are the Connes Embedding Problem (see \cite{Oz:About})
and Kirchberg Embedding problem 
(see \cite{goldbring2014kirchberg}). In addition, classes 
of \cstar-algebras of central importance such as MF algebras 
(\S\ref{S.MF}, cf.\  also their more prominent relatives,  
quasidiagonal  algebras, \S\ref{S.QD}) are defined in terms of being a subalgebra of a fixed massive \cstar-algebra.

The set of universal sentences (i.e., universal formulas with no free variables; see Definition~\ref{def:complexity}) 
in the theory of $D$ is called the \emph{universal theory} of $D$\index{theory!universal}
 and denoted    $\ThA(D)$.\index{T@$\ThA(D)$}

\begin{lemma} \label{L.MF.1} If a \cstar-algebra $C$ is countably saturated then both
\begin{align*}
 \cC_1&=\{A\colon A
\text{ is separable and isomorphic to a subalgebra of }C\} \text{ and}\\
\cC_2&=\{A\colon A
\text{ is separable and isomorphic to a unital subalgebra of }C\}
\end{align*}
 are universally axiomatizable in the class of separable \cstar-algebras.
\end{lemma} 

\begin{proof} We prove that $A\in \cC_1$ if and only if $A\models \ThA(C)$. 
From Proposition~\ref{P.Ax}, we see that $A\in \cC_1$ implies $A\models\ThA(C)$. 
For the other direction, assume $A\models \ThA(C)$,
fix a dense sequence of the unit ball of $A$, $a_n$, for $n\in \bbN$, 
and let $\alpha_k(\bar x)$, for $k\in \bbN$, 
 be an enumeration of a $\|\cdot\|$-dense subset of 
 all quantifier-free formulas such that 
 $\alpha_k(\bar a)^A=0$. Since $A\models \ThA(C)$ we have that 
 $C\models \sup_{\bar x} \alpha_k(\bar x)$ for all $k$. 
 By countable saturation, we can find $c_n$, for $n\in \bbN$, such that 
 $\alpha_k(\bar c)^C=0$ for all $k$. This implies that the map $a_n\mapsto c_n$ 
 extends to a $^*$-isomorphism of $A$ into a separable subalgebra of $C$. 
 
 The proof of the second claim is very similar, with one small difference. 
 Using the above notation, assure that $a_1=1_A$ and proceed with the 
 same construction. Since the unit of a \cstar-algebra is definable, we must have~$c_1=1$. 
 \end{proof} 

Although the isomorphism type of an ultrapower and relative commutant may depend on the choice of the ultrafilter 
(\cite[Theorem~5.1]{FaHaSh:Model1}), this does not affect what separable subalgebras embed into them. 
This follows by the last lemma and the fact that by \L o\'s' Theorem (Theorem~\ref{Los})
the theory of $A^{\cU}$ does not depend on the choice of $\cU$.

\begin{corollary} For any algebra $B$  the class of separable algebras that embed into an ultrapower $B^{\cU}$ of $B$ 
does not depend on the choice of a nonprincipal ultrafilter  $\cU$. \qed
\end{corollary} 

In contrast, which algebras embed into an ultraproduct can depend on the ultrafilter, for example $\prod_{\cU}M_n(\bbC)$ may or may not have a unital copy of $M_2(\bbC)$, 
depending on whether or not $\cU$ concentrates on even natural numbers.

\section{MF algebras} 
\label{S.MF} 
The class of MF  (or `matricial field') 
algebras was introduced in  \cite{BlaKir} and the following lemma is well-known.  
 
\begin{lemma}\label{L.MF.def}  
The following are equivalent for every  separable \cstar-algebra $A$ 
($Q$ denotes the universal UHF algebra). 
\begin{enumerate}
\item $A$ is isomorphic to a subalgebra of $\prod_{\cU} M_n(\bbC)$. 
\item $A$ is isomorphic to a subalgebra of  $Q^{\cU}$.  
\item $A$ is isomorphic to a subalgebra of  $\ASA$. 
\item $A$ is isomorphic to a subalgebra of $\ell_\infty(Q)/c_0(Q)$. 
\end{enumerate}
\end{lemma} 

\begin{proof} 
The algebras $\ASA$ and $\ell_\infty(Q)/c_0(Q)$ are countably saturated by 
 \cite[Theorem 1]{farah2014rigidity}, and all ultraproducts associated by nonprincipal ultrafilters on $\bbN$ 
 are countably saturated (see \S\ref{S.Saturated}). 
By Lemma \ref{L.MF.1} it suffices to prove that all four algebras have the 
same universal theory. 

\end{proof} 

A separable \cstar-algebra is 
\emph{MF}\index{C@\cstar-algebra!MF}
if it satisfies any of the four equivalent 
conditions in Lemma~\ref{L.MF.def}. 
A prominent  subclass of MF algebras, the quasidiagonal (or QD) algebras, is introduced in~\S\ref{S.QD}. 

\begin{remark} All four massive algebras appearing in Lemma~\ref{L.MF.def} are nonisomorphic. 
This is least obvious for the two ultraproducts. 
We shall however see that $\prod_{\cU} M_n(\bbC)$ is not even elementarily equivalent to $Q$, or to 
any nuclear \cstar-algebra (Proposition~\ref{P.not-nuclear}). 
\end{remark} 

\section{Approximately divisible algebras} \label{S.AD} A unital \cstar-algebra $A$ is \emph{approximately divisible}\index{C@\cstar-algebra!approximately divisible} 
(\cite{blackadar1992approximately}) 
if for every finite subset $F$ of $A$
and every~$\e>0$ there is a unital finite-dimensional subalgebra $M$ of $A$ such that (i) each element of its unit ball is $\e$-commuting 
with each element of $F$ and (ii) no direct summand of $M$ is abelian. 
At the first sight this looks like an omitting types property. However, 
it was proved in~\cite{blackadar1992approximately} that one can always choose $M$ to be $M_2(\bbC)$, $M_3(\bbC)$ or
$M_2(\bbC)\oplus M_3(\bbC)$. 
Therefore the argument from \S\ref{Ex.ssa-stable} 
  also shows that being approximately divisible is \aea.

  \chapter{Approximation properties} 
\label{S.AP} 
A map $\varphi\colon A\to B$ is \emph{completely positive}\index{completely positive} if 
\[
\varphi\otimes \id_n\colon A\otimes M_n(\bbC)\to B\otimes M_n(\bbC)
\] 
is positive for all $n$. 
The acronym \emph{c.p.c.}\index{c.p.c.}\index{linear map!completely positive and contractive (c.p.c.)} stands for completely positive and 
contractive, where \emph{contractive}\index{contractive ($1$-Lipshitz)}	 means $1$-Lipschitz. 

\section{Nuclearity} 
\label{S.Nuclearity}
A \cstar-algebra $A$ is \emph{nuclear}\index{C@\cstar-algebra!nuclear} if for every tuple $\bar a$  in the unit ball $A_1$ and every $\e>0$
there are a finite-dimensional \cstar-algebra $F$ and  c.p.c.\  maps $\varphi\colon A\to F$ and
$\psi\colon F\to A$ such that the diagram   

\begin{center}
\begin{tikzpicture} 
\matrix[row sep=1cm, column sep=1cm]{
\node (A1){$A$}; &  &   \node (A2) {$A$}; \\
 & \node (F) {$F$};
\\
};
\path (A1) edge [->] node [above] {$\id$} (A2); 
\path (A1) edge [->]  node [below] {$\varphi$} (F);
\path (F) edge [->] node [below] {$\psi$} (A2); 
\end{tikzpicture} 
\end{center}
\noindent $\e$-commutes on $\bar a$.

For every $n$, define a function $\nuc\colon \bigsqcup_{n\geq 1} A^n \to \bbR$ by (recall that we write $\|\bar a-\bar b\|$ for $\max_i \|a_i-b_i\|$ whenever both 
$\bar a$ and $\bar b$ are $n$-tuples) 
\[
 \nuc(\overline a) := \inf_{F}\inf_{\varphi,\psi} \|(\psi\circ \varphi)(\bar a)-\bar a\| 
 \]
where $F$ ranges over all finite-dimensional \cstar-algebras, $\varphi$ ranges over all c.p.c.\  
maps $\varphi\colon F\to A$ and $\psi$ ranges over all c.p.c.\  maps $\psi\colon A\to F$. 
By \cite[Theorem 1.4]{HirKirWhi}, $A$ is nuclear if and only if $\nuc$ is identically $0$ on $\bigsqcup_{n\geq 1} A^n$.
The zero-set of this function is the set of all tuples that can be well-approximated by c.p.c.\  maps factoring through 
finite-dimensional \cstar-algebras. 
It essentially measures the distance of a tuple to this zero-set. 
As we shall see later, once $F$ is fixed, the expression 
\[
\inf_{\varphi,\psi} \|(\psi\circ \varphi)(\bar a)-\bar a\| 
 \]
can be expressed by a formula. An analogous remark applies to $\nucdim_m$ and $\drank_m$ defined in \S\ref{S.dimnuc} 
 and \S\ref{S.DecompositionRank} below.

\section{Completely positive contractive order zero maps}
\label{S.cpc} 
Maps of order zero were defined in \cite{Winter:covering} and \cite{WinterZacharias:NucDim} as 
the  completely positive maps that preserve orthogonality, in the sense that a c.p.c.\  map $\varphi$ has order zero 
if $ab=0$ 
implies $\varphi(a)\varphi(b)=0$ for all positive $a$ and $b$. 
A linear map  $\Phi\colon M_k(\bbC)\to B$ 
is \emph{c.p.c.\  order zero}\index{c.p.c.\ order zero} if and only if the 
 images of the matrix units, $b_{ij}$ for $i,j,m,n\leq k$,  satisfy the relations
 \[
b_{ij}=b_{ji}^*, \qquad b_{ij} b_{mn}=\delta_{jm} b_{ii} b_{in},\qquad
0\leq  b_{ii}\leq 1. 
\]
This is an immediate consequence of the structure theorem for c.p.c.\  order zero maps
given in \cite[Theorem~3.3]{WiZa:Completely}. 
These relations, and more generally the relations for a c.p.c.\ order zero map from any fixed finite dimensional \cstar-algebra, are weakly stable by \cite[Theorem 4.9]{Loring:stable-rel}.
 
\begin{lemma} \label{L.cpco0.1} 
For every finite-dimensional \cstar-algebra $F$ of the form $\bigoplus_{l\leq n} M_{k(l)}(\bbC)$, there is a weakly stable formula $\alpha_F(\bar x)$ with free variables $x_{ij}^l$ for $i,j\leq k(l)$ and $l \leq n$ such that
a linear map $\Phi\colon F\to B$ is c.p.c.\  order zero if and only if,
with $x^l_{ij}=\Phi(e^l_{ij})$ for $i,j\leq k(l)$ and $l \leq n$, we have $\alpha_F(\bar x)^B=0$. 
\end{lemma} 

\begin{proof} 
In the case when $F$ is $M_k(\bbC)$ for some $k$ by the above we can take $\alpha_F(\bar x)$ to be
\[
\sum_{i,j\leq k} \|x_{ij}-x_{ji}^*\|+\sum_{i,j,m,n\leq k} \|x_{ij}x_{mn}-\delta_{jm} x_{ii} x_{in}\| + (\|\sum_i x_{ii} \| \dotminus 1)
\]
The last term is added in order to guarantee that $\Phi$ is contractive.

For the general case, 
fix $F=\bigoplus_{l\leq n} M_{k(l)}(\bbC)$. 
Then let 
\[
\alpha_F(\bar x):=\sum_{l\leq n} \alpha_{M_{k(l)}}(\bar x_l) + \sum_{\begin{array}{c} \scriptstyle l < l', \\ \scriptstyle i\leq k(l), \\ \scriptstyle j \leq k(l') \end{array}} \|x_{ii}^lx_{jj}^{l'}\| 
\] 
where $\bar x_l$ is $(x_{ij}^l)_{i,j\leq k(l);\, l\leq n}$. 
Weak stability follows from \cite[Theorem 4.9]{Loring:stable-rel}.
\end{proof}

\section{Nuclear dimension} \label{S.dimnuc} 
 
 For $m\in \bbN$, a \cstar-algebra has \emph{nuclear dimension}\index{nuclear dimension}
  at most $m$ 
if for every 
 tuple $\bar a$ in  $A_1$ and every $\e>0$
there are  finite-dimensional \cstar-algebras $F_i$ for $0\leq i\leq m$ 
 and  c.p.c.\  maps $\varphi_i\colon A\to F_i$ and
$\psi_i\colon F_i\to A$ such that in addition $\psi_i$ has order zero and 
the diagram 

\begin{center}
\begin{tikzpicture} 
\matrix[row sep=1cm, column sep=1cm]{
\node (A1){$A$}; &  &   \node (A2) {$A$}; \\
 & \node (F) {$\bigoplus_{i\leq m} F_i$};
\\
};
\path (A1) edge [->] node [above] {$\id$} (A2); 
\path (A1) edge [->]  node [below] {$\varphi = \sum_{i \leq m} \varphi_{i}\qquad\qquad$} (F);
\path (F) edge [->] node [below] {\qquad\qquad$\qquad\psi = \sum_{i\leq m} \psi_i$} (A2); 
\end{tikzpicture} 
\end{center}
\noindent $\e$-commutes on $\bar a$. 

Fix $m$. 
We define a function 
$\nucdim_ m\colon \bigsqcup_{n\geq 1} A^n\to \bbR^+$\index{n@$\nucdim(A)$} 
by 
\[
\nucdim_m(\bar a)=\inf_{F}\inf_{\varphi,\psi} \|(\psi\circ \varphi)(\bar a)-\bar a\|
\]
where $F$ ranges over all finite-dimensional \cstar-algebras, $\varphi$ ranges over all c.p.c.\  
maps $\varphi\colon A\to F$ and $\psi$ ranges over all maps $\psi\colon F\to A$
such that $F$ is a direct sum of at most $(m+1)$ \cstar-algebras $F_i$ and the restriction of $\psi$ to each $F_i$ is c.p.c.\  of order zero. 
Note that $\psi$ itself is not required to be contractive (cf.\ the definition of decomposition rank below).  
A \cstar-algebra $A$ has nuclear dimension $\leq m$
 if and only if $\nucdim_m$ is identically zero on $\bigsqcup_{n\geq 1}A^n$.

\section{Decomposition rank} \label{S.DecompositionRank} 
A \cstar-algebra $A$ has \emph{decomposition rank}\index{decomposition rank}  at most $m$ 
if its nuclear dimension is at most $m$ \emph{and}  the map $\psi=\sum_{i\leq m} \psi_i$ 
as in \S\ref{S.dimnuc} can be chosen to be c.p.c.

Fix $m$. 
We can define a function $\drank_m\colon \bigsqcup_{n\geq 1} A^n\to \bbR^+$ 
by\index{dr@$\drank(A)$}
\[
\drank_m(\bar a):=\inf_{F}\inf_{\varphi,\psi} \|(\psi\circ \varphi)(\bar a)-\bar a\|
\]
where $F$ ranges over all finite-dimensional \cstar-algebras, $\varphi$ ranges over all c.p.c.\  
maps $\varphi\colon A\to F$ and $\psi$ ranges over all c.p.c.\  maps $\psi\colon F\to A$
such that $F$ is a direct sum of at most $m+1$ \cstar-algebras $F_i$ and the restriction of $\psi$ to each $F_i$ is of order zero. 

Then $A$ has \emph{decomposition rank $\leq m$} if and only if $\drank_m$ is identically zero on $\bigsqcup_{n\geq 1} A^n$. 

\section{Quasidiagonal algebras} \label{S.QD} A  \cstar-algebra $A$ 
is \emph{quasidiagonal}\index{C@\cstar-algebra!quasidiagonal (QD)}
 (frequently abbreviated as QD)
if for every finite $F\subset A$ and every $\e>0$ there are $n=n(F,\e)$ and a c.p.c.\  map $\phi=\phi_{F,\e}\colon A\to M_n(\bbC)$ 
such that 
\begin{enumerate}
\item $\|\phi(a)\|\geq \|a\|-\e$ for all $a\in F$, and 
\item $\|\phi(ab)-\phi(a)\phi(b)\|<\e$ for all $a$ and $b$ in $F$. 
\end{enumerate}
Such an $A$ is clearly isomorphic to a subalgebra of a suitable ultraproduct of matrix algebras. 
In the case when $A$ is separable it is also isomorphic to a subalgebra of $\ASA$ and is therefore MF (\S\ref{S.MF}). 
By the Choi--Effros lifting theorem (\cite[Chapter 7]{choi1976completely} or \cite{Arv:Notes}), quasidiagonality and being MF agree for separable nuclear \cstar-algebras. 
See \cite{brown2000quasidiagonal} and \cite[Chapter~7]{BrOz:cstar} for more information on QD algebras.

\section{Approximation properties and definability} \label{S.approx}
Let $F$ be the finite-dimensional algebra $\bigoplus_{k=1}^n M_{m_k}(\bbC)$.  
As in Example~\ref{Ex.1} \eqref{Ex.1.alpha-F}, let $J$ be the set
of all triples $(i,j,k)$ where $0 \leq i,j \leq m_k$ and $1 \leq k \leq n$. $J$ indexes the matrix units $e^{k}_{i,j}$ of $F$.
With this notation, we have

\begin{lemma} \label{L.def.cpc} 
Fix a finite-dimensional \cstar-algebra $F$ and index matrix units of its direct summands
as $\bar e = (e^{k}_{ij})_{(i,j,k) \in J}$ as above. For every \cstar-algebra $A$ the set 
\[
\{\bar a = (a^k_{ij})_{J} \colon \bar a=\psi(\bar e) \text{ for a c.p.c.\   map }\psi\colon F\to A\}
\]
is definable. 
\end{lemma} 

\begin{proof} 
If $F$ is $M_k(\bbC)$ for some $k$, then 
by \cite[Prop. 1.5.12]{BrOz:cstar}
a map 
$\psi\colon F\to A$ is completely positive if and only if $[\psi(e^k_{ij})]$ is positive in $M_k(A)$.  $\psi$ is contractive if and only if $\psi(1) \leq 1$ in $M_k(A)$. As $M_k(A)$ lies in $A^{\eq}$ by Lemma \ref{L.Norm} and being positive is definable, we are finished in this case.

Now suppose that $F=\bigoplus_{k=1}^n M_{m_k}(\bbC)$ and $\psi\colon F\to A$ is a function.
From the argument above, $\psi$ is completely positive if and only if
the restriction of $\psi$ to each $M_{m_k}(\bbC)$ is completely positive for each $k$, and these properties are definable separately.
In addition, $\psi$ is contractive if and only if $\psi(1)\leq 1$ which is a definable property in~$F(A)$. 
\end{proof} 

For the following two lemmas, we assume that $F = \bigoplus_{k \leq n} F_k$ is finite dimensional,   $J(k)$ indexes the matrix units  of $F_k$, and $\bar e$ stands for the matrix units of $F$.

\begin{lemma} \label{L.def.dimnuc} 
For every \cstar-algebra $A$, the set  
\begin{align*}
\{(a_j^k)_{k\leq n, j\in J(k)} \colon & \bar a=\psi(\bar e) \text{ for   }\psi\colon F\to A\\
&\text{ with $\psi\rs F_k$ c.p.c.\  order zero for all $k$}\}
\end{align*}
is definable. 
\end{lemma} 

\begin{proof} 
By Lemma \ref{L.cpco0.1} 
 the above set is the zero-set of a formula of the form $\sum_{k\leq n} \alpha_{k}(\bar x^k)$. Being a sum of $n+1$ weakly stable formulas with disjoint sets of variables, this formula is easily seen to be weakly stable as well. 
\end{proof} 

\begin{lemma} \label{L.def.dr} 
For every \cstar-algebra $A$, the set  
\begin{align*}
Y:=\{(a_j^k)_{k\leq n, j\in J(k)} \colon&  \bar a=\psi(\bar e) \text{ for a c.p.c.\  }\psi\colon F\to A\\
&\text{ with $\psi\rs F_k$ c.p.c.\  order zero for all $k$}\}
\end{align*}
is definable. 
\end{lemma} 

\begin{proof} Let us do this semantically leveraging off the previous lemma.  It suffices to assume we have \cstar-algebras $A_i$ for $i \in I$, an ultrafilter $\cU$ on $I$ and $A = \prod_\cU A_i$.  Further assume that $\psi\colon F \rightarrow A$ is a c.p.c.\  map such that $\psi = \sum_k \psi_k$ where $\psi_k\restriction F_k$ is c.p.c.\  order zero.  From Lemma \ref{L.def.dimnuc} we can find c.p.c.\  order zero maps $\psi_k^i \colon F_k \rightarrow A_i$ such that $\lim_{i \rightarrow \cU} \psi_k^i = \psi_k$, $\psi^i = \sum_k \psi^i_k$ is completely positive and $\lim_{i \rightarrow\cU} \psi^i = \psi$.  We just need to tweak the $\psi^i$'s so they are contractive.  We know that $\|\psi\| \leq 1$ so we have $\lim_{i \rightarrow \cU} \| \psi^i \| \leq 1$.  Let 
\[
X_n := \{ i \in I \colon \|\psi^i\| \leq 1 + 1/n \}.
\]
  If $i \in X_n$ for all $n \in \bbN$, then $\psi_i$ is contractive and there is nothing to do.  If there is a maximum $n$ such that $i \in X_n$, let $n_i := \max\{n \colon i \in X_n \}$ and replace $\psi^i$ and all of its components by scaling them by $n_i/(n_i + 1)$.   $\cU$-often now we have that $\psi^i$ is contractive and all of the ultralimits still tend to the correct functions.  We conclude that $Y$ is definable.
\end{proof}

\section{Approximation properties and uniform families of formulas} 
 \label{Proof.T1} 
 In this section we state and prove one of our main results. 

\subsection{Uniform families of formulas}\label{Subsection: uniform types}
Suppose that $T$ is a theory and $\mathcal{F}$ is a countable set of $\bbR^+$-formulas in the free variables $\bar x$; that is, formulas that take values in $[0,\infty)$.  Moreover, assume there is a uniform continuity modulus $u$ such that for all $\varphi \in \mathcal{F}$
\[
T \models |\varphi(\bar x) - \varphi(\bar y)| \leq u(d(\bar x,\bar y)).
\]
We will call such an $\mathcal{F}$ 
a \emph{uniform family of formulas with respect to $T$}\index{uniform family of formulas}  and the function defined by 
\begin{equation}
\label{Eq.UnifFamilyInf}
p_{\mathcal{F}}(\bar x) := \inf_{\varphi \in \mathcal{F}} \varphi(\bar x)
\end{equation}
is uniformly continuous with uniform continuity modulus $u$ for all models of~$T$.
When $T$ is clear from the context we shall drop reference to it and say that  $\cF$ is a uniform family of formulas. 

\begin{definition}
A class $\fC$ of metric structures   is \emph{\udt}\index{definable by uniform families of formulas} if there are a theory $T$ and  uniform families of formulas~$\mathcal{F}_n$ for $n \in \bbN$ such that $A \in \fC$ if and only if
$A\models T$ and  $p^A_{\mathcal{F}_n} \equiv 0$ for all $n \in \bbN$. 
\end{definition}

The functions $p^A_{\mathcal{F}_n}$ above are particular instances of $\calL_{\omega_1,\omega}$-formulas in the sense of continuous logic.  See \cite{ben2009model}.  We will say more about the consequences of this in \S\ref{S.pert} and \S\ref{S.borel}.

Let's say a word about where we are headed. If $\mathcal{F}$ is a uniform family of formulas then for an algebra $A$ to satisfy $p^A_{\mathcal{F}} \equiv 0$ it must be the case that $p^A_{\mathcal{F}}  \geq 1/n$ is not satisfied in $A$ for all $n$.  Unrolling this, we see that either the set of formulas $\Sigma_n := \{\varphi(\bar x) \geq 1/n : \varphi \in \mathcal{F}\}$ is inconsistent with $T$ or this set represents a partial type which is omitted in $A$.  This more specialized notion was somewhat inaccurately named `uniformly definable by types' in an early version of the present manuscript. 
We will see in \S\ref{S.Henkin} that if $\mathcal{F}$ is a uniform family of formulas for a theory $T$ and $\Sigma_n$ is a partial type then there is a relatively simple criterion for determining if $\Sigma_n$ can be omitted for each $n$. 

Note that we can combine properties (even countably many) which are definable by uniform families of formulas: if $\fC_n$ is a class which is definable by uniform families of formulas for each $n \in \bbN$ then $\bigcap \fC_n$ is also definable by uniform families of formulas.

We record a lemma on which this criterion hinges. 

\begin{lemma} \label{L.udt}
If  $\cF$ is a uniform with respect to $T$ family of formulas and~$A$ is a model of $T$, 
then the zero-set of $p^A_{\cF}$ is a closed subset of $A$. 
\end{lemma} 
\begin{proof} The assumptions imply that  the interpretation of $p^A_{\cF}$ in $A$ is continuous. 
\end{proof}

\begin{thm} \label{T1} The following classes of \cstar-algebras
are \udt 
\begin{enumerate}
\item \label{T1.UHF} UHF algebras, but only among separable \cstar-algebras. 
\item \label{T1.AF} AF algebras, but only among separable \cstar-algebras. 
\item \label{T1.Nuclear} Nuclear \cstar-algebras. 
\item \label{T1.NucDim} For every $n$, \cstar-algebras of nuclear dimension $\leq n$ (see \S\ref{S.dimnuc}). 
\item \label{T1.DR} For every $n$, \cstar-algebras of decomposition rank $\leq n$ (see \S\ref{S.DecompositionRank}). 
\item \label{T1.Simple} Simple \cstar-algebras.
\item \label{T1.Popa} Popa algebras (see \S\ref{Subsection: Popa}).
\item \label{T1.TAF} Simple tracially AF algebras (see \S\ref{Subsection: TAF}). 
\item\label{T1.QD} Quasidiagonal algebras (see \S\ref{S.QD}). 
\end{enumerate} 
Moreover, in \eqref{T1.UHF}--\eqref{T1.Popa} each 
uniform family of formulas is composed of existential formulas.
\end{thm} 

\begin{proof}
Proofs of  (\ref{T1.UHF}) and (\ref{T1.AF})  are contained in  \cite{Mitacs2012}. 
Proofs of \eqref{T1.Nuclear}--\ref{T1.DR} are given in~\S\ref{S.Beth-proof}. 
 Proofs of (\ref{T1.Simple}), (\ref{T1.Popa}), and (\ref{T1.TAF}) are in Proposition~\ref{P.Simple}, 
 \S \ref{Subsection: Popa.2}, and \S \ref{Subsection: TAF.2} respectively.
 \eqref{T1.QD} is proved in \S\ref{S.T1.QD}. 
\end{proof}

 The reason why in \eqref{T1.UHF} and \eqref{T1.AF} of Theorem \ref{T1} separability is needed is 
an historical accident. The proof shows that the `locally finite' and `locally matricial' (or \emph{matroid})\index{C@\cstar-algebra!matroid} algebras (i.e.,  ones in which every finite set can be approximately 
well approximated by a finite-dimensional subalgebra or a full matrix subalgebra)  of arbitrary density character are axiomatizable.  
 In \cite{FaKa:Nonseparable}, some nonseparable \cstar-algebras were constructed which are not UHF or even AF although every finite subset can be arbitrarily well approximated by elements of a matrix subalgebra.

\section{Nuclearity, nuclear dimension and decomposition rank: First proof}\label{S.Beth-proof}

The main goal of this section is to prove Theorem~\ref{T1}(\ref{T1.Nuclear})--(\ref{T1.DR}) but first we need some preliminaries on operator systems and c.p.c.\ maps.

\begin{definition}
If $A$ is a unital \cstar-algebra and $S \subseteq A$, then $S$ is an \emph{operator system} if it is a vector subspace of $A$ which contains the unit and is closed under adjoints.
\end{definition}
We are going to be concerned with c.p.c.\  maps from finite dimensional operator systems $S$ to $M_n(\bbC)$.  The following correspondence between linear maps from $S$ to $M_n(\bbC)$ and linear functions on $M_n(S)$ reduces such questions to questions about linear functionals.
Suppose that $S$ is an operator system inside some \cstar-algebra $B$ and $\varphi$ is a linear map from $S$ to~$M_n(\bbC)$.  Define a linear functional $s_\varphi$ on $M_n(S)$ as follows:
$$
s_\varphi(a) := \frac{1}{n} \sum_{1\leq i,j \leq n} \langle \varphi(a_{ij})e_j,e_i \rangle
$$
where the matrix $a$ has entries $a_{ij}$, the $e_i$'s are the standard basis for $\bbC^n$ and the inner product is the usual one on $\bbC^n$.  In the other direction, given a linear functional on $M_n(S)$, we define a linear map $\varphi_s$ from $S$ to $M_n(\bbC)$ as follows: if $a \in S$ then the $(i,j)$ entry of $\varphi_s(a)$ is determined by
$$\langle \varphi_s(a)e_j,e_i \rangle = ns(a\otimes e_{ij})$$ where $e_{ij}$ is $(i,j)$ matrix unit. The following proposition summarizes the relationship between these two maps (see \cite[Proposition~1.5.12]{BrOz:cstar} or \cite[Theorem 4.10]{Paul:Completely}). 

\begin{prop}\label{Prop-OpSys}
Suppose $S$ is an operator system in a \cstar-algebra $B$, $s$ is a linear functional on $M_n(S)$ and $\varphi$ is a linear map from $S$ to $M_n(\bbC)$. Then we have the following. 
\begin{enumerate}
\item $s = s_{\varphi_s}$
\item $\varphi = \varphi_{s_\varphi}$
\item $\varphi$ is a c.p.\ map if and only if $s_\varphi$ is positive.
\end{enumerate}
\end{prop}

We also record the following proposition.

\begin{prop}\label{Prop-Ult-Pos}
\begin{enumerate}
\item Suppose that we have a sequence of operator systems $S_i$ and linear functionals $\psi_i:S_i \to \bbC$, such that for every $\bar s\in \prod_\cU S_i$,
\[ \varphi(\bar s) := \lim_{i\to\cU} \psi_i(s_i) \]
is defined, and defines a positive functional on $\prod_\cU S_i$.
Then there exists a sequence of positive functionals $\varphi_i$ on $S_i$ such that $\varphi(\bar s) = \lim_{i\to\cU} \varphi_i(s_i)$.

\item Suppose that $A = \prod_\cU A_i$ for an ultrafilter $\cU$ on an index set $I$, $S$ is a finite-dimensional operator system in $A$ and $\varphi$ is a positive linear functional on $S$.  Then there are finite-dimensional operator systems $S_i$ in $A_i$ and positive linear functionals $\varphi_i$ on $S_i$ such that 
\[
S = 
\prod_\cU S_i \text{ and } \varphi = \lim_{i \rightarrow \cU} \varphi_i.
\]
\end{enumerate}
\end{prop}

\begin{proof}
For (1), since $\varphi$ is positive, it is bounded.
Let us first verify that $\lim_{i\to\cU} \|\psi_i\| = \|\varphi\|$.

If not, then there exists $\gamma>\|\varphi\|$ and some set $J \in \cU$ such that $\|\psi_j\| > \gamma$ for all $j \in J$.
Thus there exists a contraction $x_j \in S_j$ such that $|\psi_i(x_j)| \geq \gamma$ for $j \in J$.
But then $(x_i)$ (let $x_i = 0$ for $i \not\in J$) defines a contraction  in $\prod_\cU S_i$ and $|\varphi((x_i))| \geq \gamma$, which is a contradiction.

We may now assume without loss of generality that $\|\psi_i\| \leq \|\varphi\|$ for each $i$.
Thus it has a Jordan decomposition, $\psi_i=(\psi_i)_+-(\psi_i)_-$, with $\|\psi_i\| = \|(\psi_i)_+\| + \|(\psi_i)_-\|$.
With $\alpha_+$, $\alpha_-$ the functionals on $\prod_\cU S_i$ induced by $((\psi_i)_+)_i$ and $((\psi_i)_-)_i$, we see (by the same argument as above) that $\|\varphi\| = \|\alpha_+\|+\|\alpha_-\|$.
By uniqueness of Jordan decomposition, it follows that $\alpha_+ = \varphi_+ = \varphi$ (since $\varphi$ is positive).
Therefore the sequence $\phi_i:=(\psi_i)_+$ is as required.

For (2), since $S$ is finite dimensional, it has a finite basis $x_1,\dots,x_n$ consisting of self-adjoint elements.
Lifting each $x_p$ to a sequence $(x_{p,i})_{i=1}^\infty$, this enables us to define $S_i := \mathrm{span}\{x_{1,i},\dots,x_{n,i}\}$, and it easily follows that $S = \prod_\cU S_i$.

To obtain $\varphi_i$, by (1) it suffices to lift $\varphi$ to any sequence of linear maps $\psi_i:S_i \to \bbC$.
Note that for $\cU$-almost all $i$, $x_{1,i},\dots,x_{n,i}$ is linearly independent; for such $i$, we simply define $\psi_i$ by
\[ \psi_i(x_{p,i}) := \varphi(x_p). \]
(For the other $i$, we can define $\psi_i:=0$.)
\end{proof}

We now define three predicates in order to state Theorem \ref{T.Rn.def}. We first define the predicate $R_n$. Fix $k,n \in \bbN$
and a \cstar-algebra $A$. 
For $\bar a \in A_1^k$, we define
\[
R_n^A(\bar a) := \inf_{\varphi,\psi} \| \bar a - (\psi\circ\varphi)(\bar a)) \|
\]
where $\varphi$ ranges over c.p.c.\ maps from $A$ to $M_n(\bbC)$ and $\psi$ ranges over c.p.c.\ maps from $M_n(\bbC)$ to $A$. 

For the other two predicates $U_{\bar F}$ and $V_{\bar F}$, fix $k,n \in \bbN$, a finite-dimensional algebra $F = \oplus_{i = 1}^k F_i$ where $\bar F$ is the sequence of $F_i$'s, 
and a \cstar-algebra $A$. For $\bar a \in A_1^k$, we define
\[
U_{\bar F}^A(\bar a) := \inf_{\varphi,\psi} \| \bar a - (\psi\circ\varphi)(\bar a)) \|
\]
where $\varphi = \sum_{i = 1}^k \varphi_i$ and $\varphi_i$ ranges over c.p.c.\ maps from $A$ to $F_i$.  Moreover, $\psi = \sum_{i=1}^k \psi_i$ and $\psi_i$ ranges over order zero c.p.c.\ maps from $F_i$ to $A$.  

$V^A_{\bar F}$ is defined as $U^A_{\bar F}$ 
with the additional requirement that $\psi$ is also a c.p.c.\ map. 

Notice that all of $R^A_n$, $U^A_{\bar F}$ and $V^A_{\bar F}$ are 2-Lipschitz and have bounded range. Hence $(A, R^A_n), (A,U^A_{\bar F})$ and $(A,V^A_{\bar F})$ are metric structures in a language which adds a new $k$-ary predicate symbol to the language of \cstar-algebras.

\begin{theorem}\label{T.Rn.def}
The following expansions of the class of \cstar-algebras are elementary:  
\begin{enumerate}
\item $\cC_{n,k}$ of all structures $(A,R^A_n)$,
\item $\cC_{\bar F,k}$ of all structures $(A,U^A_{\bar F})$ for a fixed $\bar F$, and
\item $\cC'_{\bar F,k}$ of all structures $(A,V^A_{\bar F})$ for a fixed $\bar F$.
\end{enumerate}
As a consequence, each predicate
$R_n$, $U_{\bar F}$ and $V_{\bar F}$ are defined by an existential definable predicate in the language of \cstar-algebras.
\end{theorem}

\begin{proof}
The listed consequences are immediate from the elementarity of the given classes and Beth definability. In fact, by the Arveson extension theorem (\cite{Arv:Notes}), all of these predicates are existentially defined. 

We tackle the first part of the proof semantically by showing that each class listed is closed under ultraproducts and ultraroots.  We first concentrate on  $\cC_{n,k}$. It suffices to show that if we have \cstar-algebras $A_i$ for $i \in I$, an ultrafilter $\cU$ on $I$ and if $A = \prod_\cU A_i$ then
\[
R^A_n = \lim_{i \rightarrow \cU} R^{A_i}_n.
\]
We first show that $R^A_n$ is less than or equal to $\displaystyle \lim_{i \rightarrow \cU} R^{A_i}_n$.  Towards this end, suppose that $\displaystyle \lim_{i \rightarrow \cU} R^{A_i}_n < d$.  Fix $\bar a \in A_1^k$. This inequality is witnessed on some $J \in \cU$ by c.p.c.\ maps $\varphi_j:A_j \rightarrow M_n(\bbC)$ and c.p.c.\ maps $\psi_j:M_n(\bbC) \rightarrow A_j$ for $j \in J$ such that 
\[
\|\bar a_j - (\psi_j\circ\varphi_j)(\bar a_j))\| < d
\] 
for all $j \in J$.  This implies that $\| \bar a - (\psi\circ\varphi)(\bar a)) \| < d$ where $\varphi = \prod_\cU \varphi_i$ and $\psi = \prod_\cU \psi_i$ (define $\psi_i$ and $\varphi_i$ to be identically zero for $i \not\in J$).  We conclude that $R^A_n < \displaystyle \lim_{i \rightarrow \cU} R^{A_i}_n$.

We now prove the other direction of the inequality. Fix $\bar a \in A_1^k$ and suppose that $R_n^A(\bar a) < d$.  This means that we can find c.p.c.\ maps $\varphi:A \rightarrow M_n(\bbC)$ and $\psi:M_n(\bbC) \rightarrow A$ such that
\[
\| \bar a - (\psi\circ\varphi)(\bar a)) \| < d.
\]
We concentrate first on $\psi$.  By Lemma \ref{L.def.cpc}, the $n^2$-tuples that are images of the matrix units of c.p.c.\ maps from $M_n(\bbC)$ to $A$ form a definable set.  It follows then that there are c.p.c.\ maps $\psi_i$ from $M_n(\bbC)$ to $A_i$ for all $i \in I$ so that $\psi = \prod_\cU \psi_i$.

%

%
%
%
%
We now need to deal with $\varphi$. 
%
 We can assume that $\bar a$ is the basis of a finite dimensional operator system $S$.  We would like to show that $\varphi\!\!\restriction_S = \prod_\cU \varphi_i\!\!\restriction_{S_i}$ for operator systems $S_i$ in $A_i$ with $\bar a_i$ a basis for $S_i$, $\bar a = (\bar a_i)$ and c.p.c.\ maps $\varphi_i:A_i \rightarrow M_n(\bbC)$ for $i \in I$.
 Using Proposition \ref{Prop-OpSys}, we can reduce the question to whether $s_\varphi$, a positive linear functional on $M_n(S)$, is an ultraproduct of positive linear functionals $s_i$ on $M_n(S_i)$ for operator systems $S_i$ in $A_i$.  But this is the situation handled by Proposition \ref{Prop-Ult-Pos}.  So we conclude that
 \[
 \lim_{i \rightarrow \cU} \| \bar a_i - (\psi_i\circ\varphi_i)(\bar a_i))\| < d
 \]
 which finishes the proof of the first part of the Theorem.
 
 We now turn our attention to $\cC_{\bar F,k}$ for a fixed $\bar F$, the class of structures of the form $(A,U^A_{\bar F})$ for \cstar-algebras $A$.  The proof is very similar to the case of $\cC_{n,k}$ so we will be brief.  The definition of $U_{\bar F}$ involves quantification over maps $\varphi$ and $\psi$.  When considering $\psi$, the proof given above goes through exactly as written substituting Lemma \ref{L.def.dimnuc} for Lemma \ref{L.def.cpc}.  
 
 Now when considering the $\varphi$ mentioned in the definition of $U^A_{\bar F,k}$, let's fix $F = \oplus_i F_i$ and c.p.c.\ maps $\varphi_i:A \rightarrow F$ such that $\varphi = \sum_i \varphi_i$.  It is then clear that we need only demonstrate what to do with a single c.p.c.\ map $\varphi_i$ into a finite-dimensional \cstar-algebra $F_i$.  If we now write $F_i = \oplus_j M_{n_i^j}(\bbC)$ then $\varphi_i$ can be expressed as $\sum_j \varphi^j_i$ where each $\varphi^j_i$ is a c.p.c.\ map from $A$ to $M_{n_i^j}(\bbC)$.  This returns us to the case we handled when we considered $\cC_{n,k}$.
 
 For the class $\cC'_{\bar F,k}$, we substitute Lemma \ref{L.def.dr} for Lemma \ref{L.def.dimnuc} for the treatment of $\psi$.  The case of $\varphi$ is unchanged and this finishes the proof.
\end{proof}

We now give the proof of Theorem~\ref{T1}(\ref{T1.Nuclear})--(\ref{T1.DR}).

\begin{proof}  To do this we introduce certain uniform families of formulas. In the case of nuclearity  (\ref{T1.Nuclear}), for each $k$ we can consider the set of formulas $\mathcal{F}_k := \{ R_n(\bar x) : n \in \bbN \}$ where $\bar x$ is an $k$-tuple of variables.  Requiring that $\inf_{\theta\in \cF_k} \theta(\bar a) \equiv 0$ for all $\bar a$ in $A_1^k$ and all $k$ is the same as saying that the function $\text{nuc}$ introduced in \S\ref{S.Nuclearity} is identically 0 in $A$ which in turn means that $A$ is nuclear. 

For nuclear dimension (\ref{T1.NucDim}) and decomposition rank (\ref{T1.DR}), we need to modify the set of formulas.  
Let $\mathcal{F}^k_m$ is the set of all formulas $U_{\bar F}(\bar x)$ from Theorem \ref{T.Rn.def} where $\bar x$ is an $k$-tuple and $\bar F$ is an $(m+1)$-tuple of finite-dimensional algebras.  
An algebra $A$ has nuclear dimension at most $m$ if for all $k$, 
\[
\inf_{\theta \in \mathcal{F}^k_m} \theta^A(\bar x) \equiv 0.
\]
The case of decomposition rank is handled similarly by using $V_{\bar F}$ from Theorem \ref{T.Rn.def} in place of $U_{\bar F}$. 
\end{proof}

The advantage of using the semantic approach in proofs like the proof of Theorem \ref{T.Rn.def} is that such proofs are usually shorter and they highlight the underlying facts from operator algebra.  The disadvantage is that they do not  usually provide explicit formulas for various notions although we can often use preservation theorems to derive qualitative information.
In fact, the last part of the proof of Theorem \ref{T.Rn.def} above we actually prove slightly more which we record here. Fix $F$, a finite dimensional \cstar-algebra.
\begin{prop}\label{claim:elementaryclass2} Let 
$\cC_F$ be the class of all triples $(A,B,R)$ where $A$ is a \cstar-algebra,  $B \cong F$ and $R(\bar a,\bar b)$ is the $(2k)$-ary relation on 
$A_1^k\times B_1^k$ defined by
\[
\inf_\varphi \| \varphi(\bar a) - \bar b \|
\] where $\varphi$ is a c.p.c.\  map from $A$ to $B$.
Then $\cC_F$ is an elementary class and the zero-set of $R$ is a definable set.
 \end{prop}

It is the definability of this predicate $R$ which allowed us to define the necessary predicates in Theorem \ref{T.Rn.def}.    In the next section we will give a second proof of Theorem \ref{T1}(\ref{T1.Nuclear}) via excision of states by showing that the predicate $R$ from the Proposition can be explicitly given by formulas.

\section{Nuclearity, nuclear dimension and decomposition rank: Second proof}\label{S.2nd-proof}

We shall show that the predicates $\nuc$, $\nucdim_m$ and $\drank_m$   from 
$ \bigsqcup_{n\geq 1} A^n $ to $\bbR$ 
(see \S\ref{S.Nuclearity}, \S\ref{S.dimnuc}  and \S\ref{S.DecompositionRank}, respectively) 
are definable for every $m\geq 1$ by providing explicit formulas. 

By  $P(A)$\index{P@$P(A)$} we denote the set of all pure states on $A$. 
We record some well known facts about pure states.

\begin{lemma} \label{L.App.Pure} 
For every positive element $f\in A$ there exists a pure state $s$ of $A$ such that $s(f)=\|f\|$. 
\end{lemma} 

\begin{proof} Since $\{s: s(f)=\|f\|\}$ is a non-empty face of the state space of $A$, any of its extreme points 
is a pure state as required. 
\end{proof} 

\begin{lemma} \label{L.App.State} 
If $s$ is a state and $s(f)=\|f\|=1$ for a positive $f$ then $s(faf)=s(fa)=s(af)=s(a)$ for all $a\in A$. 
\end{lemma} 

\begin{proof} This is a consequence of the Cauchy--Schwarz inequality. Since $(1-f)^2\leq 1-f$, we have $s((1-f)^2)=0$. 
Also,  
\[
s((1-f)a)\leq s((1-f)^2)s(a^*a)=0
\]
and therefore $s(a)=s(fa)$. The equality $s(a)=s(af)$ is proved by taking the adjoint. 
\end{proof} 

The following was proved in \cite[Proposition~2.2]{AkeAndPed}. 

\begin{lemma}[Excision of pure states]\label{L.App.Excision} 
If $A$ is a \cstar-algebra and $s$ is a state in the closure of $P(A)$
then for every finite set $G\subset A$ and every $\e>0$ there exists a positive element $a\in A$ such that $\|a\|=s(a)=1$ 
and $\|s(x)a^2-a xa\|<\e$ for all $x\in G$. \qed
\end{lemma}

Let $\ell,m \in \bbN$ and let $A$ be a \cstar-algebra.
Define $\Gamma_m(A,M_\ell(\bbC))$ to be the set of all c.p.c.\  maps $A \to M_\ell(\bbC)$ of the form
\[ 
 x \mapsto \frac1m \sum_{i=1}^m (s_i(x \otimes e_{j,k}))_{j,k}, 
\]
where $s_1,\dots,s_m \in P(M_\ell(A))$.
More generally, if $F$ is a finite dimensional algebra, define $\Gamma_m(A,F)$ to be the set of all c.p.c.\  maps $\psi:A \to F$ such that $\pi_{F'} \circ \psi \in \Gamma_m(A,F')$ for every simple direct summand $F'$ of $F$, where $\pi_{F'}:F \to F'$ denotes the projection of $F$ onto~$F'$.
 Define $\alpha_{F,m}:A^n \times F^n \to \bbR$ by
\[ 
\alpha_{F,m}(\bar a, \bar b) := \inf_{\psi \in \Gamma_m(A,F)} \max_{i\leq n} \|b_i - \psi(a_i)\|. 
\]
We will write $\alpha_{\ell,m}$ for $\alpha_{M_\ell(\bbC),m}$.
It should be noted that in many interesting cases we can assume $m=1$. This is because,
if $A$ is simple, then by Glimm's lemma, pure states are weak$^*$-dense in the state space of~$A$
(see~\cite[Lemma 1.4.11]{BrOz:cstar}). 

\begin{lemma}
\label{lem:NuclAlphaFormula}
Suppose $A$ is a \cstar-algebra.
Let $\bar a \in A^n$.
Then for every $\ell$ and $\bar b\in M_\ell^n(\bbC)$ we have that 
$\alpha_{\ell,1}(\bar a, \bar b)$ is equal to
\[
\theta(\bar a,\bar b) := \inf_{\begin{array}{c}\scriptstyle f \in M_\ell(A)_+ \\ 
\scriptstyle \|f\| = 1\end{array}} \max_{i\leq n} \|f^2 \otimes b_i - (f \otimes 1_\ell)(\sum_{j,k} a_i \otimes e_{j,k} \otimes e_{j,k})(f \otimes 1_\ell)\|, 
\]
where the norm is taken in $M_{\ell^2}(A)$.  Moreover, $\alpha_{\ell,1}$ is a definable set as is $\alpha_{\ell,m}(\bar a,\bar y)$.
\end{lemma} 

We remark that the formula $\alpha_{M_\ell,1}$ is a formula where the parameters come from the structure $(A,M_\ell(\bbC))$ which we have mentioned before is part of $A^{\eq}$. 

\begin{proof}  We prove the two-way inequality.

Let $f \in  M_\ell(A)_+$ be such that $\|f\| = 1$, and suppose that
\[ 
\bigg \|f^2 \otimes b_i - (f \otimes 1_\ell)(\sum_{j,k} a \otimes e_{j,k} \otimes e_{j,k})(f \otimes 1_\ell)\bigg\| < \delta. 
\]
By Lemma~\ref{L.App.Pure} there exists pure state $s$ of $M_\ell(A)$ such that $s(f)=1$.
Lemma~\ref{L.App.State}
 implies that $s(fbf)=s(bf)=s(b)$ for all $b\in A \otimes M_\ell$ and in particular we have $s(f^2)=1$. 
Define $\psi $ in 
$\Gamma_1(A,M_\ell(\bbC))$ by (cf.\ Proposition~\ref{Prop-OpSys}) 
\[ 
\psi(x) := (s(x \otimes e_{j,k}))_{j,k}. 
\]
Then
\begin{align*}
 \|b_i - \psi(a_i)\| &= \|b_i - \sum_{j,k} s(a_i \otimes e_{j,k}) \otimes e_{j,k}\| \\
&= \|s(f^2) \otimes b_i - \sum_{j,k} s(f(a_i \otimes e_{j,k})f) \otimes e_{j,k}\| \\
&\leq \|f^2 \otimes b_i - \sum_{j,k} (f(a_i \otimes e_{j,k})f) \otimes e_{j,k}\| \\
& < \delta,
\end{align*}
where on the third line we use the fact that states are completely contractive.  We note that this calculation demonstrates that if $\theta(\bar a,\bar b) < \delta$ then $\alpha_{\ell,1}(\bar a,\psi(\bar a)) = 0$ and $\|\psi(\bar a) - \bar b\| < \delta$.  This proves that $\alpha_{\ell,m}$ and $\alpha_{\ell,m}(\bar a,\bar y)$ are definable sets once we prove the other direction of the inequality.

Now we prove the converse inequality. 
Let $\psi \in \Gamma_1(A,M_\ell(\bbC))$ be such that $\|b_i - \psi(a_i)\| < \delta$ for all $i$, and let $\e > 0$.
Let $s \in P(M_\ell(A))$ such that $\psi(x) = (s(x \otimes e_{j,k}))_{j,k}$.
By excision (Lemma~\ref{L.App.Excision}), we may find $f \in M_\ell(A)_+$ such that
\[ 
 s(a_i \otimes e_{j,k})f^2 \approx_\eta f(a_i \otimes e_{j,k})f 
\]
for all $i,j,k$, where $\eta := \e/\ell^2$ so that
\[ 
f^2 \otimes (s(a_i \otimes e_{j,k}))_{j,k} \approx_\e (f \otimes 1_\ell)(\sum_{j,k} a_i \otimes e_{j,k} \otimes e_{j,k})(f \otimes 1_\ell). \]
Then it follows that
\begin{align*}
&\|f^2 \otimes b_i - (f \otimes 1_\ell)(\sum_{j,k} a \otimes e_{j,k} \otimes e_{j,k})(f \otimes 1_\ell)\| \\
&\quad \leq 
\|f^2 \otimes (b_i - s(a_i \otimes e_{j,k})_{j,k})\| + \\
&\qquad \|f^2 \otimes (s(a_i \otimes e_{j,k})_{j,k} - (f \otimes 1_\ell)(\sum_{j,k} a_i \otimes e_{j,k} \otimes e_{j,k})(f \otimes 1_\ell)\| \\
&\quad < \delta + \e,
\end{align*}
as required.  
\end{proof}

We will need $\alpha_{l,m}$ for arbitrary $m$; the proof that they can be expressed in continuous logic follows immediately from the previous lemma. We record the case for an arbitrary $F$ for completeness; when reading the formula below remember Remark \ref{rmk.numbers-as-scalars} and the fact that we are working in $(A,M_\ell(\bbC))$.

\begin{lemma}  
\label{lem:NuclAlphaFormula.m}
Fix a \cstar-algebra $A$ and a finite-dimensional algebra\\ $F = M_{\ell_1}(\bbC) \oplus \cdots \oplus M_{\ell_r}(\bbC)$.
Let $\bar a \in A^n$ and $\bar b \in F^n$.
Then we have that 
\begin{align*}
 \alpha_{F,m}(\bar a, \bar b) = \inf_{\bar c^{(1)},\dots,\bar c^{(m)} \in F^n} \max_{i\leq n} \bigg\|b_i - \frac1m \sum_j c^{(j)}_i\bigg\| + \\
+\frac1m \sum_{j=1}^m \max_{k \leq r}\  \alpha_{M_{\ell_k},1}(\bar a, \pi_{M_{\ell_k}}(\bar c^{(j)})),
\end{align*}
where $\pi_{M_{\ell_k}}:F^n \to M_{\ell_k}^n$ denotes the projection from $F^n$ onto $M^n_{\ell_k}$. Moreover, $\alpha_{F,m}$ is a definable set as is $\alpha_{F,m}(\bar a,\bar y)$.\qed
\end{lemma}

We now reconsider the proofs of Theorem \ref{T1}(\ref{T1.Nuclear})--(\ref{T1.DR}) given at the end of \S \ref{S.Beth-proof}.  For any fixed $\ell$ and \cstar-algebra $A$, we wish to be able to quantify over c.p.c.\ maps from $A$ to $M_\ell(\bbC)$ and from $M_\ell(\bbC)$ to $A$.  \S \ref{S.approx} provides us with the tools to do the latter.  We previously gave a Beth definability proof that we could do the former.  
We now give a more syntactic proof. We know from Proposition \ref{claim:elementaryclass2} that in the structure $(A,M_\ell(\bbC),R)$ where for $\bar a \in A$ and $\bar b \in M_\ell(\bbC)$
\[
R(\bar a,\bar b) = \inf_\varphi \|\varphi(\bar a)-\bar b\|
\] 
where $\varphi$ ranges over c.p.c.\ maps from $A$ to $M_\ell(\bbC)$ is a definable set and moreover, for any $\bar a \in A$, $R(\bar a,\bar y)$ is also definable. Since c.p.c.\ maps from $A \to F$ correspond to contractive positive linear functionals on $A \otimes F$ 
(see Proposition~\ref{Prop-OpSys})
 and each of these is, by the Krein--Milman theorem, 
a weak$^*$-limit of convex combinations of pure states and zero, it follows that
the point-norm closure of the set 
$ \bigcup_m  \Gamma_m(A,F)$ 
coincides with the set of all c.p.c.\  maps from $A$ into $F$.  
  However, based on the previous two lemmas, we see that 
\[
R(\bar a,\bar b) = \lim_{m \rightarrow \infty} \alpha_{\ell,m}(\bar a,\bar b)
\]
which gives an explicit expression for $R$.

\section{Simple \cstar-algebras} \label{Subsection: simple}
We prove Theorem~\ref{T1} \eqref{T1.Simple}, 
that simple \cstar-algebras are characterized as algebras that are \udt. 
Since an algebra $A$ is simple if and only if for all $a$ and $b$ such that $\|a\|=\|b\|=1$ and every $\e>0$ there exists 
$n$ and $x_j$, $y_j$, for $j\leq n$ in the unit ball of $A$ satisfying $\|b-\sum_{j\leq n} x_j a y_j\|\leq \e$, 
it is easy to see that being simple is characterized by omitting  a sequence of types,
although the obvious sequence of types is not uniform, and this is not a useful property (see \cite{FaMa:Omitting}).
 We need to work harder in order to show that simplicity is \udt. 

\begin{lemma}
\label{L.Simple}
Let $A$ be a \cstar-algebra.
Then $A$ is simple if and only if, for every $a,b \in A_+$ and $\e > 0$, if $\|a\|=1$ and $\|b\| \leq 1$ then there exist $k$ and $x_1,\dots,x_k \in A$ such that
$ \|x_1^*x_1 + \cdots + x_k^*x_k\|  \leq 2 $
and
\[ \|x_1^*ax_1 + \cdots + x_k^*ax_k - b\| < \e. \]
\end{lemma}

\begin{proof}
It is quite clear that the condition stated in this lemma does imply that $A$ is simple.
Therefore, let us assume that $A$ is simple and show the converse.

Let $a,b \in A_+$ and $\e > 0$ be given, with $\|a\| = 1$ and $\|b\| \leq 1$.
Since $A$ is simple, 
 for some $k \in \bbN$ and some $y_1,\dots,y_k \in A$,
\begin{equation}
\label{eq:SimpleyDef}
{ \|b- y_1^*(a-{\textstyle\frac12})_+y_1 + \cdots + y_k^*(a-{\textstyle\frac12})_+y_k\|<\e}.
\end{equation}
Let $g_{1/2} \in C_0((0,1],[0,1])$ be the function which is $0$ on $[0,1/4]$, $1$ on $[1/2,1]$ and linear in between.
Let $f \in C_0((0,1],[0,2])$ be a function such that $f(t)t = g_{1/2}(t)$. 
For $i=1,\dots,k$, define
\[ 
{x_i := ((a-\textstyle\frac12)_+f(a))^{1/2}y_i.} 
\]
Since $t(t-1/2)_+f(t)=(t-1/2)_+$, 
we have 
\[
x_1^*ax_1 + \cdots + x_i^*ax_i = 
y_1^*(a-\tfrac12)_+y_1 + \cdots + y_i^*(a-\tfrac12)_+y_i \]
so that   \eqref{eq:SimpleyDef} implies
$
 \|x_1^*ax_1 + \cdots + x_i^*ax_i - b\| < \e$. 
Setting
\[
 Y := \left(\begin{array}{c} (a-\textstyle\frac12)_+^{1/2}y_1 \\ \vdots \\ (a-\frac12)_+^{1/2}y_i \end{array} \right), 
 \]
we have
\begin{align*}
x_1^*x_1 + \cdots + x_i^*x_i &= Y^*(f(a) \otimes 1_i)Y \\
&\leq 2Y^*Y \\
&= {2(y_1^*(a-\textstyle\frac12)_+y_1 + \cdots + y_i^*(a-\frac12)_+y_i)} \\
&\approx_{2\e} 2(b-\e)_+,
\end{align*}
and therefore (assuming $\e<1$, as we may), $\|x_1^*x_1 + \cdots + x_i^*x_i\| \leq 2\|(b-\e)_+\| +2\e \leq 2$.
\end{proof}

We need the following small lemma before the next Proposition.
\begin{lemma}\label{L.c*c}
The set $\{\bar c\in A^n: \sum_i c_i^*c_i\leq 1\}$ is definable for every~$n$. 
\end{lemma} 

\begin{proof} 
With  $a_+$ denoting the positive part of a self-adjoint operator $a$ we have 
$a_+=f(a)$ where $f(t):=(t+|t|)/2$. Let 
 \[
 \varphi(\bar c):=\bigg\| (\sum_i c_i^*c_i-1)_+\bigg\|
 \]
 Now if $\varphi(\bar a)\leq \e$ then $b_i=a_i/(1+\e)$ satisfies $\|\bar b-\bar a\|\leq \e$
  and $\varphi(\bar b)=0$. By \cite[Lemma 2.10]{BYBHU} one can find a continuous function $f$ such that 
  \[
  \text{d}(\bar a, \{\bar c: \sum_i c_i^*c_i\leq 1\})=f(\varphi(\bar a))
  \] and so, by Theorem \ref{T.def-dist}, the zero-set of $\varphi$ is definable.
\end{proof} 

\begin{prop} \label{P.Simple} Simple algebras are \udt. 
\end{prop}

\begin{proof}
The set
\[ 
\{\bar x\in A^i \mid \|x_1^*x_1 + \cdots + x_i^*x_i\| \leq 2\} 
\]
is definable by Lemma~\ref{L.c*c}
and homogeneity and we can therefore quantify over it. 
Let $\varphi_i(a,b)$ denote the formula
\[ 
\varphi_i(a,b) := \inf_{\bar x\in A^i, \|\sum x_j^*x_j\|\leq 2} \|x_1^*ax_1 + \cdots + x_i^*ax_i - b\|. 
\]
We now prove that $\varphi_i$ is $4$-Lipschitz.
To show this, it suffices to show that, if $a \in A$ is self-adjoint then for $x_1,\dots,x_i \in A_+$ satisfying $\|x_1^*x_1 + \cdots x_i^*x_i\| \leq 2$, we have
\[  \|x_1^*ax_1 + \cdots + x_i^*ax_i\| \leq 4\|a\|. \]
By decomposing $a$ into positive and negative parts, it suffices to show the same inequality, but with $2$ in place of $4$, under the assumption that $a\geq 0$.
Setting
\[
 X := \left(\begin{array}{c} x_1 \\ \vdots \\ x_i \end{array} \right), 
 \]
we have
\begin{align*}
\|x_1^*ax_1 + \cdots + x_i^*ax_i\| &= \|X^*(a \otimes 1_i)X\| \\
&= \|(a^{1/2} \otimes 1_i)XX^*(a^{1/2} \otimes 1_i)\| \\
&\leq \|XX^*\|\,\|a \otimes 1_i\| \\
&= \|X^*X\|\,\|a \otimes 1_i\| \\
&= \|x_1^*x_1 + \cdots + x_i^*x_i\| \|a\| \\
&\leq 2\|a\|,
\end{align*}
as required.

Now consider the single set of formulas $\mathcal{F}$ consisting of $\varphi_i(x,y)$ for $i \in \bbN$.  Lemma \ref{L.Simple} shows that $A$ is simple if and if 
\[
\inf_{i \in \bbN} \varphi_i(a,b) = 0
\]
 for all $a,b \in A$.
\end{proof}

\section{Popa algebras}\label{Subsection: Popa}
\label{Subsection: Popa.2} 

Popa algebras are defined via an internal finite-dimen\-sio\-nal approximation property
which ensures an abundance of projections. They  
were introduced by Popa in \cite{popa1997local} to link \cstar-algebraic quasidiagonality to properties of the  hyperfinite~II$_{1}$ factor $R$.  
A \cstar-algebra is a Popa algebra\index{C@\cstar-algebra!Popa}
 (see \cite[Definition 1.2]{brown2006invariant}, \cite[Definition 2.4.1]{Watson:Structure}) if
 for every $n\in \bbN$ and every finite subset $G$
 of $A$ there is a finite-dimensional subalgebra $F$ of $A$ such that,
denoting by $p$ the unit of $F$,

\begin{enumerate}
\item $\left\Vert px-xp\right\Vert <\frac{1}{n}$ for every $x\in G$, and

\item for every $x\in G$ there is $y\in F$ such that $\left\Vert
pxp-y\right\Vert <\frac{1}{n}$.
\end{enumerate}

Popa algebras are always quasidiagonal, and every simple, unital, and  quasidiagonal algebra of real rank zero is a 
Popa algebra (\cite[Theorem 1.2]{popa1997local}).

We prove Theorem~\ref{T1} \eqref{T1.Popa}, that Popa algebras (\S\ref{Subsection: Popa}) are 
\udt. 

Fix a finite-dimensional \cstar-algebra $F$. By Example~\ref{Ex.1} \eqref{Ex.1.beta-F}, for every $m$  the formula
\[ \alpha_{F,m}(\bar y,z) := \inf_{\phi:F \to A} d(\bar y, \phi(F)) + \|z-\phi(1)\| \]
is weakly stable, and it is evidently $1$-Lipschitz.

For  $m\in\bbN$  consider the set of formulas $\mathcal{F}_{m}$ 
consisting  of all
\begin{equation}\label{Eq.Popa} 
\gamma_{F,m}:=\inf_{z} \max \{ \beta_{F,m}(z\bar yz, z), \max_{j\leq m} \|[z,y_j]\|\}
\end{equation}
where $F$ ranges over all finite-dimensional algebras. 
The formula $\gamma_{F,m}$ in \eqref{Eq.Popa} is still 1-Lipschitz and hence formulas in $\mathcal{F}_{m}$ form a uniform family of formulas. 
By the above, $A$ is a Popa algebra if and only if $\inf_{\varphi \in \mathcal{F}_{m}} \varphi(\bar a) = 0$ for every $m \in \bbN$ and $\bar a$ in $A$.

\section{Simple tracially AF algebras}\label{Subsection: TAF}
\label{Subsection: TAF.2}

A major development in the Elliott programme 
 was  Lin's introduction  (inspired by  Popa algebras, \S\ref{Subsection: Popa}, and also by Kirchberg algebras) 
 of the class of 
 \emph{tracially AF}\index{C@\cstar-algebra!tracially AF (TAF)} 
 (TAF) \cstar-algebras in 
  \cite{lin2001tracially}.      Lin verified, under the additional assumption of the Universal Coefficient Theorem (UCT), 
  that the Elliott conjecture holds for these algebras (\cite{lin_classification_2004}).
 
This modifies Popa's concept, with approximation by finite-di\-men\-sion\-al \cstar-algebras  relaxed and given 
a new twist: finite-dimensional approximation is now achieved not in norm but in terms of `internal  excision'; the additional feature of Lin's notion is it that the excising subalgebras are large in a suitable sense. 
 The class of TAF algebras vastly generalizes the class of AF algebras.  By results of Winter, 
 TAF algebras include all simple \cstar-algebras with real rank zero and finite decomposition rank (the latter condition can be relaxed to locally finite decomposition rank in the presence of $\mathcal{Z}$-stability); cf.\ \cite{winter_classification_2006,winter_simple_2007}.

Lin's original definition of tracially AF algebras appears in  \cite[Definition~2.1]{lin2001tracially}. In the case of simple \cstar-algebras the definition can be simplified  (see \cite[Definition~3.1]{winter_classification_2006}). 
An element $a$ of a \cstar-algebra is \emph{full}\index{full (element of a \cstar-algebra)} if the ideal it generates is improper \cite[II.5.3.10]{Black:Operator}. Lin's definition implies that for every positive element $a\in A$ the hereditary subalgebra $\overline{aAa}$ contains a projection. In case $A$ is simple, every nonzero element is full and therefore every nonzero hereditary subalgebra includes one of the form $qAq$ for some nonzero projection $q$. 
(The latter property is known to be equivalent to the algebra having 
real rank zero, see \cite[Theorem~3.4]{lin2001tracially} but we shall not need this fact.)

\begin{definition}
A unital  simple \cstar-algebra $A$ is tracially AF (TAF) if for any $\varepsilon >0$, nonzero positive element $b\in A_+$, $k\in \mathbb{N}$, finite subset $G$ of
the unit ball of $A$,
there is a finite
dimensional \cstar-algebra $F\subset A$ such that, denoting by $p$ the unit of $%
F $, 
\begin{enumerate}
\item\label{I.TAF.1} $p$ commutes up to $\varepsilon $ with every element of $G$,

\item\label{I.TAF.2}  for every $x\in G$ there is 
$y\in F$ such that $\left\Vert pxp-y\right\Vert <\varepsilon $,

\item \label{I.TAF.5} $1-p$ is Murray--von Neumann equivalent to a projection in the hereditary subalgebra $\overline{bAb}$ generated by $b$.
\pushcounter
\end{enumerate}
\end{definition}

We prove Theorem~\ref{T1} \eqref{T1.TAF}, that simple tracially AF  algebras  are 
\udt. 
Observe that quantification in \eqref{I.TAF.1}  and \eqref{I.TAF.2}  is over definable sets by Example~\ref{Ex.1} \eqref{Ex.1.projection} and 
\eqref{Ex.1.alpha-F}, respectively. 
Also note that, if $g\in C_0((0,1])$ is the function which is $0$ at $0$, $1$ on $[1/2,1]$, and linear on $[0,1/2]$, then
\[ \overline{(b-1/2)_+A(b-1/2)_+} \subseteq \{a \in A \mid g(b)a=ag(b)=a\} \subseteq \overline{bAb}, \]
and therefore the definition of TAF remains the same if the condition \eqref{I.TAF.5} is replaced by
\begin{enumerate}
\popcounter 
\item \label{I.TAF.6} There exists a partial isometry $v$ such that $vv^*=1-p$ and $g(b)v^*v = v^*v$,
\pushcounter
\end{enumerate}
or even (using weak stability of projections and of Murray--von Neumann equivalence), for some fixed sufficiently small $\eta$,
\begin{enumerate}
\popcounter 
\item \label{I.TAF.7} There exists a partial isometry $v$ such that $\|vv^*-(1-p)\|<\eta$ and $\|g(b)v^*v- v^*v\|<\eta$.
\pushcounter
\end{enumerate}

As in the previous section, for a finite-dimensional \cstar-algebra $F$, by Example~\ref{Ex.1} \eqref{Ex.1.beta-F}, the formula
\[ \alpha_{F,m}(\bar y,z) := \inf_{\phi:F \to A} d(\bar y, \phi(F)) + \|z-\phi(1)\| \]
is weakly stable (for all $m$), and evidently $1$-Lipschitz.

Consider the $(m+1)$-ary formula  $\theta_{F,m}(b,\bar a)$ in the variables
$b, \bar a$,\footnote{Here we denote variables by symbols normally used for constants in order to increase the readability.}
  defined to be the infimum over all projections~$z$ of the maximum of the following
\begin{enumerate}
\popcounter
\item $\alpha_{F,m}(z\bar a z, z)$,
\item $\max_{j\leq m} \|[z,a_j]\|$,
\item $\inf_{v\text{ partial isometry}} \max(\|z-vv^*\|, \|(1-g(b))v^*v\|)$,
\end{enumerate}

Since each of the components of $\theta_{F,m}$ is 1-Lipschitz in $\bar a$, all
 formulas $\theta_{F,m}$ are 1-Lipschitz in $\bar a$.
For the variable $b$, there is a uniform continuity modulus $\delta$ such that if $d(b,b')<\delta(\e)$ then $d(g(b),g(b'))<\e$; thus, each $\theta_{F,m}$ has $\delta$ as its uniform continuity modulus for the first variable.
Therefore for every $m\in \bbN$, the set $\Sigma^m$ of formulas in the free variables $\bar x$ and $y$,
 consisting of the  formulas $\theta_{F,m}(y, \bar x)$, where $F$ ranges over all finite-dimensional algebras,  is uniform. 

If a simple \cstar-algebra $A$ is TAF then 
for every positive element $b$ of norm $1$, every $m\in \bbN$, and every $\bar a \in A^m$,
\[ \inf_F \theta_{F,m}(b,\bar a) = 0. \]
On the other hand, if $A$ is simple and
\[ \inf_F \theta_{F,m}(b,\bar a) = 0 \]
for every positive element $b$ of norm $1$, every $m\in \bbN$, and every $\bar a \in A^m$, then using (\ref{I.TAF.7}), we see that $A$ is TAF.

We can therefore use the uniform sets of formulas $\Sigma^m$ for all $m \in \bbN$ 
together with the uniform set of formulas shown in Proposition~\ref{P.Simple} to 
characterize simplicity to show that being simple and TAF is \udt.

\section{Quasidiagonality}\label{S.T1.QD} 
We give a proof of Theorem \ref{T1}~(\ref{T1.QD}) in the style of \S\ref{S.Beth-proof}. 
Fix $k,n \in \bbN$.  For a \cstar-algebra $A$, define $Q^A_n(\bar a)$ for $\bar a \in A_1^k$ 
as
\[
\inf_\phi\max\{\{\|a_j\|-\|\phi(a_j)\|: j\leq k\}\cup \{ \|\phi(a_ia_j)-\phi(a_i)\phi(a_j)\|: i, j\leq k\}\}
\]
where $\phi$ ranges over all c.p.c.\  maps from $A$ into $M_n(\bbC)$. 
Notice that $Q^A_n$ is 1-Lipschitz.  
By Proposition~\ref{claim:elementaryclass2} we can quantify over c.p.c.\ maps from the set of $k$-tuples into $M_n(\bbC)$ and so $Q_n$ is a definable predicate.
$A$ is \emph{quasidiagonal} if for every $k\geq 1$, $\bar a \in A^k$,  and $\e > 0$ there is an $n$ such that $Q_n^A(\bar a) < \e$ so it is clear that being quasidiagonal is \udt.

\section{An application: Preservation by quotients} 
Theorem \ref{T1} (or rather its proof), together with some straightforward calculations, provides a uniform proof of some preservation results.

A uniform family of formulas is positive if each formula in the family is positive 
(see Definition~\ref{def:complexity}).

\begin{lemma} \label{L.decreasing.2} Assume that $\Sigma$
 is a 
 family of positive  formulas and that $\inf_{\varphi \in \Sigma} \varphi^A(\bar a) = 0$ for every  $\bar a$ in $A$. 
 Then the same is true for every homomorphic image of $A$.
 \end{lemma} 
 
 \begin{proof} 
 The salient point of this lemma is that for a positive formula~$\varphi(\bar x)$ 
 and a surjective homomorphism $\Phi\colon A\to B$ one has 
 \[
 \varphi^A(\bar a)\geq\varphi^B(\Phi(\bar a))
 \]  by Proposition~\ref{P.Positive.Preservation}. 
In particular we have 
\[
\inf_{\varphi\in \Sigma}\varphi^A(\bar a)\geq \inf_{\varphi\in \Sigma}\varphi^B(\Phi(\bar a))
\]and the conclusion follows. 
  \end{proof} 
  
  This indicates a strategy for proving that certain classes of \cstar-algebras are closed under quotients.  If the class is characterized by a uniform family of positive formulas then by the previous lemma, that class is closed under quotients.  As an example, we note
 
 \begin{lemma} \label{L.decreasing.1} The uniform families of formulas that characterize being a Popa algebra
 are positive. 
\end{lemma}

\begin{proof} 
The formula \eqref{Eq.Popa} in \S\ref{Subsection: Popa.2}, 
\[
\gamma_{F,m}:=\inf_{z} \max \{ \beta_{F,m}(z\bar yz, z), \max_{j\leq m} \|[z,y_j]\|\}, 
\]
is clearly positive. 
\end{proof}

\begin{coro} \label{C.Quotient} The class of Popa algebras is preserved under quotients. \qed
\end{coro}

The preservation of nuclearity by quotients is a deep result 
with no known elementary proofs  (see \cite[Theorem~10.1.4]{BrOz:cstar})
and it would be of interest  to deduce it from   Lemma~\ref{L.decreasing.2}. 
Predicates $R_n$ were defined at the beginning of \S\ref{S.Beth-proof} as 
\[
R^A_n:=\inf_f d(\bar a, f(\bar a))
\]
where $\bar a\in A^k$ and the infimum is taken over all c.p.c.\ maps 
 $f: A\rightarrow A$ which factor through $M_n(\bbC)$.
 Beth definability theorem implies that~$R_n$ is definable (\S\ref{S.Beth-proof}) 
 and the computations in \S\ref{S.2nd-proof} can be used to give 
 an explicit formula for $R_n$.  
 Neither of these proofs shows that~$R_n$ is positive. 
 If $A$ is a \cstar-algebra and its quotient $A/J$ is nuclear
 then the Choi--Effros lifting theorem  (\cite[Chapter 7]{choi1976completely} or \cite{Arv:Notes})
 implies that the quotient map~$\pi$ has a c.p.c.\  right inverse  $\Phi\colon A/J\to A$. 
 By composing~$\Phi$ with c.p.c.\ maps from $A$ to $M_n(\bbC)$ 
 one obtains $R^A_n\geq R^{A/J}_n\circ \pi$, giving some weak support 
 to a conjecture that $R_n$  decrease in value when one passes to quotients.

\section{An application: Perturbations} \label{S.pert}

The \emph{Kadison--Kastler distance}\index{Kadison--Kastler distance} between  subalgebras of $B(H)$ for a fixed Hilbert space $H$
is defined as the Hausdorff distance between their unit balls, 
\[
\dKK(A,B)=\max(\sup_{x\in A_1}\inf_{y\in B_1} \|x-y\|, \sup_{y\in B_1} \inf_{x\in A_1} \|x-y\|). 
\]

\begin{lemma} \label{L.dkk} 
For a fixed sentence $\varphi$ the map $A\mapsto \varphi^A$ is continuous with respect to $\dKK$. 
\end{lemma}

\begin{proof} First recall that for every formula $\psi(\bar x)$ the interpretation of $\psi$ in $A$ is uniformly continuous on the unit ball, 
and that the modulus of uniform continuity does not depend on $A$. 
Now note that if $f\colon B(H)^{n+1}\to \bbR$ is uniformly continuous on the unit ball then 
the functions ($A$ ranges over subalgebras of $B(H)$ and $\bar a$ ranges over $n$-tuples in the unit ball of $B(H)$)
\[
(A,\bar a)\mapsto \inf_{x}  f(x, \bar a)
\qquad
\text{and}
\qquad 
(A,\bar a)\mapsto \sup_{x}  f(x, \bar a)
\]
are uniformly continuous with respect to the distance 
\[
d((A,\bar a), (B,\bar b))=\dKK(A,B)+\max_{i\leq n} \|a_i-b_i\|. 
\]
If $\varphi(\bar y)$ is a formula 
in the \emph{prenex normal form},\index{prenex normal form (of a formula)} 
\[
\sup_{x_1}\inf_{x_2}\dots \inf_{x_{2n}} \psi(\bar x, \bar y)
\]
where $\psi$ is quantifier-free, then the above observation and induction show that the 
map $(A,\bar a)\mapsto \varphi(\bar a)^A$ is uniformly continuous. Since the 
formulas in prenex normal form are dense in $\|\cdot\|_\infty$ (combine 
\cite[Theorem~6.3, Theorem~6.6 and Theorem~6.9]{BYBHU}), 
this concludes the proof. 
\end{proof}

By putting together the above lemma with the axiomatizability results of \S\ref{S.Axiomatizable}
we obtain the following. (Some of these results, achieved with different techniques, were previously known. References are in brackets).

\begin{corollary}\label{C.dkk.1} 
If $P$  is a property of \cstar-algebras such that both $P$ and its negation are axiomatizable, then the subalgebras of $B(H)$ satisfying $P$ form a $\dKK$-clopen set.  
In particular, the algebras satisfying each of the following form a $\dKK$-clopen set: 
\begin{enumerate}
\item abelian \cstar-algebras (\cite{KK:perturbation}),
\item $n$-subhomogeneous \cstar-algebras, for any given $n$ (\cite{Johnson:subh}),
\item unital, projectionless \cstar-algebras (\cite{KK:perturbation}),
\item algebras containing a unital copy of $M_n(\bbC)$, for any given $n$ (\cite{christensen1980near}), 
\item simple, purely infinite  \cstar-algebras (\cite{christensen2010perturbations}),
\item finite \cstar-algebras (\cite{KK:perturbation}),
\item \cstar-algebras with real rank zero (\cite{christensen2010perturbations}), 
\item\label{C.dkk.1.8} \cstar-algebras with stable rank one. \qed
\end{enumerate} 
\end{corollary} 

\begin{corollary} \label{C.dkk.2} 
If $P$ is a property of \cstar-algebras which is axiomatizable, 
then the subalgebras of $B(H)$ satisfying $P$ form a  $\dKK$-closed set. 
In particular, the algebras satisfying each of the following form a $\dKK$-closed set: 
\begin{enumerate}
\item stably finite \cstar-algebras (\cite{christensen2012perturbations}),
\item separable  MF algebras,
\item unital tracial \cstar-algebras (\cite{PerTomWW:Cuntz}),
\item unital \cstar-algebras with a character, 
\item \label{C.dkk.2.7} \cstar-algebras with stable rank $\leq n$, for any fixed $n\geq 1$, 
\item \label{C.dkk.2.8}abelian \cstar-algebras with real  rank $\leq n$, for any fixed $n\geq 0$, 
\item separable $D$-stable \cstar-algebras, for any strongly self-absorbing algebra $D$ (\cite{christensen2012perturbations}),
\item separable approximately divisible \cstar-algebras, and 
\item separable stable \cstar-algebras. 
\end{enumerate}
\end{corollary} 

\begin{proof} This is an immediate consequence of Lemma \ref{L.dkk},  the results of \S\ref{S.Axiomatizable}, \S\ref{S.sr}, \S\ref{S.rrn}
and (for the last few clauses) the observation that separable \cstar-algebras form a $\dKK$-closed set. 
\end{proof} 

Corollary~\ref{C.dkk.1} \eqref{C.dkk.1.8} and Corollary~\ref{C.dkk.2}\eqref{C.dkk.2.7}-\eqref{C.dkk.2.8} 
answer part of  \cite[Question 7.3]{christensen2010perturbations} where it was asked what happens
with real and stable rank under small perturbations in $\dKK$.

Certain non-axiomatizable classes of \cstar-algebras are also known to be $\dKK$-closed, such as the classes of separable UHF and AF algebras (see Theorem 6.1 in \cite{christensen1980near}), or the class of nuclear separable algebras (Theorem A in \cite{christensen2012perturbations}).
Note that it is not automatic that a class that is \udt{} is $\dKK$-closed.
Given a uniform family of formulas $\mathcal F$, while we do know that for a fixed model~$A$, the function $p_{\mathcal F}$ defined in \eqref{Eq.UnifFamilyInf} is uniformly continuous in~$A$, we do not have any control over the continuity with respect to $\dKK$ as we vary $A$.
This is because the condition that $\mathcal F$ is a uniform family of formulas does not impose any control over the uniform continuity modulus for the quantified variables in the formulas in $\mathcal F$.

One could propose a stronger version of ``uniform family of formulas'' which does require the same uniform continuity modulus for all quantified variables in all the formulas, and classes which are defined by this stronger kind of uniform family of formulas would be $\dKK$-closed.
However, we are unable to prove that the families in Theorem \ref{T1} are definable by this stronger kind of uniform family.

Although we are unaware of a class of \cstar-algebras that is \udt{}  but not $\dKK$-closed, the following weaker example should be kept in mind.
In \cite{BY-N-T}, it is shown that for any language $\mathcal L$ and every isomorphism-invariant, Borel class $\mathcal C$ of models in the language $\mathcal L$, there exists a certain type of infinitary sentence (called an $\calL_{\omega_1\omega}$ sentence) $\phi$, such that
\[ \phi^A = \begin{cases} 1,\quad A \in \mathcal C; \\ 0,\quad A \not\in \mathcal C. \end{cases} \]
The construction of the formula $\phi$ is recursive, like the (finitary) formulas described in \S \ref{S.Preliminaries}; in addition to the constructions there, if $(\phi_n(\bar x))_n$ is a bounded sequence of $\calL_{\omega_1\omega}$ formulas in the same finite tuple $\bar x$ of variables, all with the same uniform continuity moduli for the free variables, then $\sup_n \phi_n(\bar x)$ and $\inf_n \phi_n(\bar x)$ are also $\calL_{\omega_1\omega}$ formulas.
In particular, if $\mathcal F$ is a uniform family of formulas, then $p_{\mathcal F}$ is an $\calL_{\omega_1\omega}$ formula.
The definition of $\calL_{\omega_1\omega}$ ensures that every such formula is uniformly continuous in each model.

In \cite{Kadets}, it was shown that there is a Banach space $X$ with subspaces $H$ and $Y_n$ for $n\in \bbN$, such that $H$ is a Hilbert space, $Y_n$ is not, and the unit balls of $Y_n$ converge to the unit ball of $H$ in the Hausdorff metric.
Using the result mentioned above from \cite{BY-N-T}, it follows that there is an infinitary $\calL_{\omega_1\omega}$ sentence $\phi$ in the language of Banach spaces which evaluates as the characteristic function of the models of Hilbert spaces, so that $\phi^{Y_n} = 0$ and $\phi^H = 1$.
This example shows that the infinitary sentences in $\calL_{\omega_1\omega}$ need not be continuous in $\dKK$ (or the generalization of $\dKK$ to other languages).
Once again, the reason is that there is no control over the uniform continuity modulus of the quantified variables, in the construction of $\calL_{\omega_1\omega}$ formulas.

\section{An application: Preservation by inductive limits} \label{S.inductivelimits}

By the direct (and easier) implication in Proposition~\ref{P.Ax} (3), every 
 \aea{} class of \cstar-algebras is closed taking under inductive limits. By combining this 
 with Theorem~\ref{Summary} we immediately obtain  a uniform proof of the following. 
 
 \begin{proposition} Each of the following classes of \cstar-algebras is closed under taking inductive limits of injective directed systems of its elements. (In \eqref{I.P.aea.sr}--\eqref{I.P.aea.usDP} the inductive limits are required to be unital.)
\begin{enumerate}
\item Algebras of real rank zero. 
\item Simple, purely infinite algebras. 
\item\label{I.P.aea.sr}  Unital algebras with stable rank $\leq n$ for $n\geq 1$. 
\item Unital \cstar-algebras with strict comparison of positive elements by traces or 2-quasitraces.
\item\label{I.P.aea.usDP} Unital \cstar-algebras with the $\bar k$-uniform strong Dixmier property for a fixed $\bar k$. \qed
\end{enumerate}
\end{proposition}

By using  Theorem~\ref{T.ssa} instead of Theorem~\ref{Summary} we obtain the following. 

 \begin{proposition} Each of the following classes of  \cstar-algebras is closed under taking inductive limits of countable injective directed systems of its elements.
\begin{enumerate}
\item Separable $D$-stable algebras, where $D$ is a strongly self-absorbing algebra. 
\item Separable approximately divisible algebras. 
\item Separable stable algebras. \qed
\end{enumerate}
\end{proposition}

The restriction to countable inductive limits of separable algebras cannot be dropped. 
As explained in the paragraph following Theorem~\ref{Summary}, tensorial decomposability of 
the ultrapowers (\cite{Gha:SAW*}) implies that neither of the classes of $D$-stable algebras, approximately divisible algebras, or stable algebras, is elementary.

\section{An application: Borel sets of \cstar-algebras} \label{S.borel}

In \cite{Kec:C*} Kechris introduced a standard Borel space of separable \cstar-algebras and 
computed descriptive complexity of some important classes of \cstar-algebras. 
A variation of this space was used in the proof of Lemma~\ref{L.ee}. 
Although there is a variety of standard Borel spaces of separable \cstar-algebras,  
they are all equivalent (see \cite[Proposition~2.6 and Proposition~2.7]{FaToTo:Turbulence}). 
In particular the question whether a given class of \cstar-algebras is Borel does not depend 
on the choice of the Borel parametrization. 
Kechris proved that some classes of \cstar-algebras (like subhomogeneous,  nuclear, continuous trace algebras, or antiliminal algebras) 
are Borel while some other classes (like  liminal and postliminal algebras) are complete co-analytic. 

\begin{prop} \label{P.Borel}Suppose that $\cC$ is a   class of \cstar-algebras.
 If $\cC$ is separably axiomatizable or \udt{} then~$\cC$ is Borel. 
\end{prop} 

\begin{proof} See
 \cite[Lemma~1.7 and Proposition~6.7]{FaMa:Omitting}.
  \end{proof}

In \cite{Kec:C*} it was  proved that the class of AF algebras is analytic and asked 
whether it is Borel (p. 123). 
The following  gives, among other things, 
 a positive answer to this problem. 
 
 \begin{coro}\label{C.Borel}  If $P$ is any of the properties of \cstar-algebras
 appearing in Corollary~\ref{C.dkk.1} or Corollary~\ref{C.dkk.2} 
 then the class of separable \cstar-algebras with property $P$ is Borel. 

The classes of AF, UHF, nuclear, 
nuclear dimension $\leq n$, decomposition rank $\leq n$,
simple, Popa, simple tracially AF and QD are Borel. 
 \end{coro} 

\begin{proof} This follows by Proposition~\ref{P.Borel}, Theorem~\ref{Summary}, 
Theorem~\ref{T.ssa}, 
 and Theorem~\ref{T1}. 
\end{proof} 

Some of these classes of \cstar-algebras 
were known to be Borel. E.g., for nuclear \cstar-algebras this was proved 
by Effros (\cite[Theorem~1.2]{Kec:C*}) and for $\cZ$-stable \cstar-algebras this was proved
in \cite[A.1]{FaToTo:Descriptive}. 

\chapter{Generic \cstar-algebras} \label{S.Henkin}

We describe a way of constructing \cstar-algebras (and metric structures in general) by Robinson forcing (also known as the Henkin construction). 
Variants of this general method (see Keisler's classic~\cite{keisler1973forcing} and the 
excellent~\cite{hodges2006building})
for the logic of metric structures have  been outlined  
in \cite{ben2009model} and \cite{eagle2013omitting}. 

The notion of model theoretic forcing is very powerful but it takes some getting used to.  We say a few words about how the formalities will play out in the next section.  We start with a theory $T$ in a language~$\calL$ and add countably many new constant symbols to the language.  The theory $T$ will now be incomplete in the new language since it tells us very little about the new constants.  The goal is to create a model abstractly out of the new constants themselves.  Of course, in order to do that, say in the language of \cstar-algebras, we will have to, for instance, decide, for every $^*$-polynomial $f$ and tuple of constants $\bar c$, what the value of $\|f(\bar c)\|$ should be.  More generally, for every formula $\varphi(\bar x)$ and tuple $\bar c$, we will need to decide the value of $\varphi(\bar c)$.  To do this consistently there is a bit of bookkeeping that needs to be done and one can think of the construction inductively.  One thing that saves us is that the space of formulas is separable, so we only have to worry about the value of countably many formulas.  In fact, we will approach this even more slowly and decide, at any given step, only an approximate value for $\varphi(\bar c)$ for finitely many formulas 
$\varphi$.  This is what is called an \emph{open condition}\index{open condition}
 below---a specification of an $r$ and an $\epsilon$ such that  $|\varphi(\bar c) - r| < \epsilon$.  We can even restrict ourselves to rational $r$ and $\epsilon$ in order to accomplish all we need  in countably many steps.

So what is the inductive assumption?  If the open condition we are considering at some step looks like $|\varphi(\bar c) - r| < \epsilon$ then the inductive assumption is that there is a model $A$ of $T$ and $\bar a \in A$ such that $|\varphi^A(\bar a) - r| < \epsilon$---this pair $(A,\bar a)$ is called a \emph{certificate}\index{certificate} in what follows.  Notice that at the start we have no open conditions at all, so that any model of $T$ will work as a certificate.

There are two important points about this process.  First of all, any open condition uses only finitely many constants.  To see why this is important, suppose an open condition we are considering implies that 
$\inf_x \varphi(x,\bar c) <r$.  Going forward we are going to need to make sure that for some constant $a$, 
$\varphi(a,\bar c) <r$.  The fact that we have a certificate for the initial open condition and enough unused constants is going to make it possible to strengthen our condition to enforce the bound on $\inf_x\varphi(x,\bar c)$.

Secondly, because we are going about this so slowly, we also have time along the way to enforce other properties, such as properties that are definable by a uniform family of formulas.  Suppose that $\mathcal{F}$ is a uniform family of formulas and we are trying to guarantee that in the final model we have $\inf_{\varphi \in \mathcal{F}} \varphi(\bar c) = 0$ for all $\bar c$.  The uniformity condition will tell us that if we can do this then $\inf_{\varphi \in \mathcal{F}}\varphi(\bar x)$ will be identically~0.  At any inductive stage then, as long as we have some certificate in which $\inf_{\varphi\in \mathcal{F}} \varphi(\bar x) = 0$, we can, for any $\epsilon > 0$ and $\bar c$, choose a stronger open condition which implies $\varphi(\bar c) < \epsilon$ for some $\varphi \in \mathcal{F}$.

After countably many steps one
has  a sequence of open conditions of increasing strength  deciding 
 values of all relevant sentences in the 
expanded language. In particular, 
 the value of  $\|f(\bar c)\|$ is decided for every $^*$-polynomial in non-commuting variables $f$ and for 
every tuple of constants $\bar c$ of the relevant sort. The  universal \cstar-algebra given by these constants and relations is known 
as the \emph{generic}\index{C@\cstar-algebra!generic} \cstar-algebra.
By induction on the complexity of a formula one proves that the generic \cstar-algebra 
is also a model of the theory $T$. 

This  forcing construction
is flexible enough to allow 
for fine adjustment of some other parameters of the generic \cstar-algebra. 
One of the overriding questions for us is to what extent 
can the Elliott invariant be controlled in these generic constructions? 
The ability to construct a \cstar-algebra with a prescribed theory and  Elliott invariant
 is clearly relevant to our Question~\ref{Q1} from the introduction. Presently
we have some limited information on when the $K$-groups of $A$ belong to $A^{\eq}$ (Theorem~\ref{T.K0}, 
 Propositions \ref{P.K0.eq} and \ref{P.K1.eq}) and the definability of traces appears to be closely related to the
 standard regularity properties of \cstar-algebras (\S\ref{S.Def.Tau.1} 
and~\S\ref{S.Def.Tau.2}).

We look at the formal details in the next section.
In the argument we use, the actual induction process is packaged into an application of the Baire Category Theorem.

\section{Henkin forcing} The results of this section are essentially well-known to logicians and proofs are included for the benefit of the reader. 
We  follow the presentation from \cite[\S 4]{FaMa:Omitting}. 
We first expand the language~$\calL$ defined in  \S\ref{S.Preliminaries}. 
For each  sort $S$, let $c_n^S$
be a new constant symbol for every  $n \in \bbN$.
The language $\calLp$ is obtained
from $\calL$ by adding these new constant symbols. 
 The terms of this new language 
are  $^*$-polynomials in the variables $x_i^S$ and constants $c_i^S$. 
The sentences of $\calLp$ are 
exactly expressions of the form $\varphi(\bar c)$, where $\varphi(\bar x)$ is an $\calL$-formula and $\bar c$ is a 
tuple of new constants of the appropriate sorts.

Although our main interest is in \cstar-algebras we present these results in greater generality.  
The following construction is similar in spirit to \cite{BY:On}. 
Fix an $\calL$-theory $\bfT$ and  a set of $\calLp$ sentences $\Sigma$, such that $\Sigma$ will form a real vector space and is closed under applying continuous functions. More precisely, we require the following 
conditions  taken from  \cite[($\Sigma1$)--($\Sigma3$) in \S4]{FaMa:Omitting}. 
\begin{enumerate}
\item [($\Sigma1$)] $\Sigma$  includes all quantifier-free formulas. 
\item  [($\Sigma 2$)] $\Sigma$ is closed under taking subformulas and  change of variables. 
\item [($\Sigma 3$)] If $k\in \omega$,  $\varphi_i(\bar x)$, for $0\leq i<k$, 
 are in $\Sigma$,  and $f\colon [0,1]^k\to [0,1]$ is a continuous 
function, then  $f(\varphi_0(\bar x),\dots, \varphi_{k-1}(\bar x))$ is in $\Sigma$. 
\end{enumerate}
For our purposes, $\Sigma$ will consist of either  all $\calLp$-sentences
or  all quantifier-free $\calLp$-sentences. 
Consider the following seminorm on $\Sigma$:   
\[
\|\varphi(\bar c)\|_{\bfT}:=\sup_{A, \bar a} |\varphi(\bar a)^A|
\]
where $A$ ranges over models of $\bfT$ and $\bar a$ ranges over tuples in $A$ of the appropriate sort. 
We say that such a pair $(A,\bar a)$ is a $\bfT$-\emph{certifying pair}\index{certifying pair}  for the value of $\|\varphi(\bar c)\|$. 
The  space $\Sigma$ with seminorm $\|\cdot\|_{\bfT}$  will   be denoted by~$\fSST$:
\[ 
\fSST := (\Sigma, \|\cdot\|_{\bfT}). 
\]
Suppose $A$ is a model of a theory $\bfT$ with a distinguished tuple $\bar a$  of elements of the form $a_i^S$ for $i\in \bbN$ and sorts $S$. Then $A$ is an $\calL$-structure and 
it is expanded to an $\calLp$-structure by interpreting each constant 
 $c_i^S$ as $a_i^S$.  
We  define a functional $s_{A,\bar a}$ on $\fSST$ via
\[
s_{A,\bar a} (\varphi(\bar c)):=\varphi(\bar a)^A 
\]
(each constant $c_i^S$ on the left-hand side is replaced by the element $a_i^S$ of $A$ 
on the right-hand side). 
We record a consequence of the definition. 

\begin{lemma}\label{L.Henkin.s}
For all $A$ and $\bar a$, the functional $s_{A,\bar a}$ is a Banach algebra homomorphism 
from  $\fSST$ to $\bbR$ of norm $\leq 1$
     \qed
\end{lemma}

\begin{lemma}\label{L.Henkin.Space}
The set $\fTT:=\{s_{A,\bar a}:A\models \bfT \text{ and }\bar a\text{ is a tuple in }A\}$
is weak$^*$-compact. 
\end{lemma}

\begin{proof} Let  $s_\lambda$, for $\lambda\in \Lambda$, be a net in $\fTT$. 
If $(A_\lambda, \bar a_\lambda)$ corresponds to $s_\lambda$, 
then  \L o\'s' theorem implies that there exists an ultraproduct $(A,\bar a):=\prod_{\cU} (A_\lambda,\bar a_\lambda)$ such that $s_{A,\bar a}$ is a limit of a subnet of $s_\lambda$. 
Therefore the  set of functionals arising in this way is weak$^*$-closed, and  
since it  is a subset of the unit ball of $\fSST^*$ it is weak$^*$-compact. 
\end{proof}

Let $\bbPTS$ be the set of all triples $(\varphi(\bar c), r,\e)$ 
such that $\varphi(\bar c)$ is in $\Sigma$,  
$r\in \bbR$,  $\e>0$, 
and 
for some $\bfT$-certifying pair $(A,\bar a)$ we have 
\[
|\varphi(\bar a)^A-r|<\e. 
\]
Elements of $\bbPTS$ are called \emph{open conditions}.\index{open condition} 
Each open condition $(\varphi(\bar c), r,\e)$ 
 is identified with the set of functionals
\[
U_{(\varphi(\bar c), r,\e)}:=\{s\in \fTT: |s(\varphi(\bar c))-r|<\e\}. 
\]

\begin{lemma} \label{L.Henkin.Up} 
For every $p\in \bbPTS$, the set $U_p$  is an open subset of  $\fTT$ in the relative weak$^*$ topology, 
and these sets form a basis for the weak$^*$ topology restricted to $\fTT$. 
\end{lemma} 

\begin{proof}
It is clear from the definition of $U_p$ that it is the intersection of a weak$^*$-open set with $\fTT$. 
We shall prove that every 
  nonempty weak$^*$-open subset $W$ of $\fSST^*$ includes $U_p$ for some  condition $p$ 
  in $\bbPTS$. 
Since~$W$ is weak$^*$-open, there exist $n\in\bbN$, $\e>0$, and for $j\leq n$,
a   formula $\phi_j(\bar c ) \in \Sigma$ and a real number $r_j$,
 such that 
 \[
 \{s: \max_{j\leq n} |s(\phi_j(\bar c))-r_j|<\e\}\subseteq W.
 \]
Define a formula $\phi$ by
 \[
\phi(\bar c):= \max_{j\leq n} |\phi_j(\bar c)-r_j|.
 \]
Then $\phi$ is in $\Sigma$ by ($\Sigma1$)--($\Sigma3$), so that $p:=(\phi(\bar c), 0, \e)$ is an open condition.
Moreover, for $(A,\bar a)$ such that $s_{A,\bar a}$ is in~$W$ we have 
\begin{align*}
s_{A,\bar a}(\phi(\bar c)) = \phi(\bar a)^A =\max_{j \leq n} |\phi_j(\bar a)-r_j|<\e,
\end{align*}
as required.
  \end{proof}

Using the 
 identification of an open condition $p \in \bbPTS$ with the open set $U_p$, inclusion of open sets provides an order on $\bbPTS$.  That is, for $p,p' \in \bbPTS$, we define $p \leq p'$ to mean that $U_p \subseteq U_{p'}$.
In other words,
\[
(\varphi(\bar c), r,\e)\leq (\varphi'(\bar c'), r',\e')
\]
if and only if $|s(\varphi(\bar c))-r|<\e$ 
implies that $|s(\varphi'(\bar c'))-r'|<\e'$
for every $s\in \fTT$. 
If $p\leq p'$ we say that $p$ \emph{extends}\index{extends} $p'$ or 
that $p$ is \emph{stronger than} $p'$.\index{stronger than}
This is the standard terminology from the theory of forcing, justified by the fact that a stronger condition decides more information about the final model.

\section{Infinite forcing}
We first consider  the case when $\Sigma$ consists of all $\calLp$-sentences.
This is the so-called `infinite forcing'\index{forcing!infinite} and in this case $\bbPTS$ is denoted $\bbPT$, 
  and $\fSST$ is denoted $\fST$. 

\begin{theorem}\label{T.FM1} 
Assume $\bfT$ is a theory in a  separable language $\calL$, possibly incomplete, and $P$ is a property which is \udt{} such that $\bfT$ has a model satisfying $P$.
\begin{enumerate}
\item There exists a dense $G_\delta$ subset $\fG$ of  $\fT$ such that 
every  $s\in \fG$ determines a unique model $A_s$  of $\bfT$ and an interpretation~$\bar a_s$ 
of constants $\bar c$ in $A_s$ such that 
each sort $S$ in $A$ is the closure of $\{a_j^S: j\in \bbN\}$ 
and $s=s_{A_s,\bar a_s}$. 
\item The set
\[
\fG_P := \{ s \in \fG : A_s \text{ has }P\}
\]
 is a nonempty $G_\delta$ subset of $\fG$.
\item Furthermore, if $\bfT$ is a complete theory, then $\fG_P$ is dense in $\fG$. 
\end{enumerate}
\end{theorem}

This relative of the Baire Category Theorem is the standard Henkin construction. 
We shall outline the main ideas and refer the reader to~\cite{FaMa:Omitting} 
for the details. 
Note that the results of \cite{FaMa:Omitting} were stated in set-theoretic forcing terms, referring to   dense open subsets of $\bbPT$ instead of dense open subsets  of $\fST^*$. 
Consequently, in \cite{FaMa:Omitting} a generic filter is required to meet additional dense open sets. 
 This  assures the analogue of  $s_{A,\bar a}\in \fG$, which is  automatic  in our case.

Fix an $\calLp$-formula $\phi(\bar c, x)$ with free variable $x$ in sort $S$, 
an $n\in \bbN$,  and $r\in \bbR$. 
Let
\begin{align*}
\bfE_{\phi(\bar c, x), r,n}:= &
\{p\in \bbPT: \mbox{for all } s\in U_p, s(\inf_x\phi(\bar c, x))> r-1/n\}\\
&\cup 
\{p\in \bbPT:\mbox{for some  } c_j^S \mbox{ and all } s\in U_p, s(\phi(\bar c, c_j^S))<r\}. 
\end{align*}
Informally, we have  $p\in \bfE_{\phi(\bar c, x), r,n}$ if $p$ ``forces'' that 
either
\[
(\inf_x \phi (\bar c, x))^A>r-1/n
\]
or 
\[
\phi(\bar c, c)^A<r
\]
for some ``witness'' $c$. 

\begin{lemma} \label{L.Henkin.E} The set 
\[
\cE_{\phi(\bar c, x), r,n}:=\bigcup\{U_q : q\in \bfE_{\phi(\bar c, x), r,n}\}
\]
 is a 
dense open subset of~$\fT$ for all $\phi(\bar c, x)$, $r$, and $n$. 
\end{lemma} 

\begin{proof}
Since $\{U_p \mid p \in \bbPT\}$ is a basis for the topology on $\fT$, we need to show that each $U_p$ has nonempty intersection with $\cE_{\phi(\bar c, x),r,n}$.  This will follow if we show that each $U_p$ contains  $U_q$ for some  $q \in \bfE_{\phi(\bar c,x),r,n}$.
In other words, it suffices to prove  that for every $p\in \bbPT$ there exists $q\leq p$ in $\bfE_{\phi(\bar c, x), r,n}$.
 
Fix $s\in U_p$ and consider two cases: either $s(\inf_x\phi(\bar c,x))>r-1/n$ or  $s(\inf_x\phi(\bar c,x))\leq r-1/n$.
If $s(\inf_x\phi(\bar c,x))>r-1/n$, then choose a weak$^*$-open neighbourhood 
$W$ of $s$ such that $t(\inf_x\phi(\bar c, x))>r-1/n$ for all $t\in W$. Then Lemma~\ref{L.Henkin.Up} 
provides   $q$  such that 
$U_q\subseteq U_p\cap W$,   
and $q$ belongs to the first set in the definition of 
$\bfE_{\phi(\bar c, x), r,n}$.

In the second case, we have $s(\inf_x \phi(\bar c, x))\leq r-1/n$ and we need to take a closer look at $p$.
Fix $\psi$, $\bar c'$, $r'$ and $\delta$
such  that $p=(\psi(\bar c'), r',\delta)$ and recall that $s \in U_p$. 
Fix $i$ large enough so that $c_i^S$ 
 does  not appear in $\bar c$ or in  $\bar c'$. We claim that  
\[
q_0:=(\phi(\bar c, c_i^S), r-1/n, 1/n)
\]
 is an element of $\bbPT$ satisfying $U_{q_0}\cap U_p\neq \emptyset$. 
 In order to prove $q_0 \in \bbPT$, we need 
 a 
 $\bfT$-certifying pair $(A,\bar a)$ for $q_0$. For this, we first fix $(B,\bar b)$ 
 such that $s=s_{B,\bar b}$. 
Since $s(\inf_x \phi(\bar c,x))\leq r-1/n$, 
we can find 
  $d$ in 
 the unit ball of $B$ satisfying  $ \phi(\bar b, d)^{B}<r$. 
 
 Now let $A:=B$ and 
 define the assignment of constants $\bar a$ by  
 $a_j^S:=b_j^S$ if $j\neq i$ and $a_i^S:=d$. 
  Then $\phi(\bar a, a^S_i)^A=\phi(\bar b, d)^B<r$ and  $s':=s_{A,\bar a}$ belongs to $U_{q_0}\cap U_p$. 
By Lemma~\ref{L.Henkin.Up} we can find a $q$ such that $U_q\subseteq U_{q_0}\cap U_p$, 
and such~$q$ belongs to the second set in the definition of 
$\bfE_{\phi(\bar c, x), r,n}$. 

We have proved that for every condition $p\in \bbPT$ and 
every triple $(\phi(\bar c, x), r,n)$
there exists $q\leq p$ in $\bbPT$ 
such that $q\in \bfE_{\phi(\bar c, x), r,n}$, as required. 
\end{proof}

\begin{proof}[Proof of Theorem~\ref{T.FM1}]
(1) Since $\calL$ (and therefore $\calLp$)  is separable we can fix
 a countable $\|\cdot\|_{\bfT}$-dense set $\Phi$ of sentences in $\calLp$. 
By Lemma~\ref{L.Henkin.E}  (using the notation from Lemma~\ref{L.Henkin.E})
\[
\fG:=\bigcap_{\phi\in \Phi, n\in \bbN, r\in \bbQ} \cE_{\phi(\bar x,c), r,n}
\]
is a dense $G_\delta$ subset of $\fT$. 

Fix $s\in \fG$. By the density of $\Phi$ and the weak$^*$-continuity of $s$, $s$ belongs to 
$\cE_{\phi(\bar x,c), r,n}$  for all $\calLp$-sentences $\phi(\bar x,c)$, $r\in \bbR$, and $n\in\bbN$. 
Define the $\calL$-structure 
$A_s$ generated by the new constant symbols with respect to the conditions
\[
\psi(\bar c)=s(\psi(\bar c))
\]
for every atomic formula $\psi(\bar x)$ and $\bar c$ of 
the appropriate sort.\footnote{In the case when $\calL$ is 
the language of \cstar-algebras, $A_s$ is the universal \cstar-algebra
generated by the new constant symbols with respect to the conditions
$\|f(\bar c)\|=s(\|f(\bar c)\|)$
for all $^*$-polynomials $f$.} 
Remember that $\psi(\bar c)$ is an $\calLp$ sentence and it therefore belongs to the domain of $s$. 
Let $a_i^S\in A$ be the element corresponding to the constant $c_i^S$. 
This defines  assignment   $\bar a_s$ of distinguished elements of $A_s$.

We now prove that for every sort $S$ the set 
$\{a_j^S: j\in \bbN\}$ is dense in the interpretation of $S$ in $A$.\footnote{In the case of \cstar-algebras this reduces to saying that $\bar a$ is dense in $A$.}  Fix a term  $f(\bar x)$.
Then $d(f(\bar c), x)$ is an atomic formula.\footnote{In 
the case when $\calL$ is the language of \cstar-algebras, $f$ is a 
$^*$-polynomial and we consider the atomic formula $\|f(\bar c)-x\|$.} 
Since $s\in \cE_{d(f(\bar c),x), r,n}$ for all $r$ and $n$  we have
\[
\inf_{j\in \bbN} s(d(f(\bar c),c_j^S))=\inf_{x\in A_s} d(f(\bar a_s),x)=0.
\]
 Therefore $\bar a_s$ is dense in $A_s$. 
 
We now need to show that $s=s_{A_s,\bar a_s}$, i.e., that
\[
\phi(\bar a_s)^{A_s}=s(\phi(\bar c))
\]
for every sentence $\phi(\bar c)$ in $\calLp$. 
Since formulas were defined recursively in Definition~\ref{formula}, 
the proof proceeds by induction on the rank of $\phi(\bar c)$. 

For atomic sentences this is a consequence of the definition of $A_s$. 
If the assertion is true for $\phi_j(\bar c)$ for $1\leq j\leq n$ and $f\colon \bbR^n\to \bbR$ is continuous, then the assertion is true for $f(\phi_1(\bar c), \dots \phi_n(\bar c))$ by 
Lemma~\ref{L.Henkin.s}. 

Now fix $\phi(\bar c,x)$ with $x$ of sort $S$ and
assume the assertion is true for $\phi(\bar c, c_j^S)$ for all $j$.  
 Since $s\in  \cE_{\phi, r,n}$ for all $r$ and $n$ 
 and  $\{a_j^S: j\in \bbN\}$ is dense in $S$, we have 
\[
s(\inf_x\phi(\bar a, x))=\inf_j s(\phi(\bar a, a_j^S)) 
=\inf_x\phi(\bar a, x)^{A_s}. 
\]
Since $\sup_x\phi(\bar c, x)=-\inf_x -\phi(\bar c, x)$, this completes the inductive proof that 
$\phi(\bar a)^{A_s}=s(\phi(\bar c))$ for all formulas $\phi(\bar x)$.

(2) Assume $\bfT$ has a model $B$ with property $P$ which is \udt{} (\S\ref{Subsection: uniform types}). 
Fix uniform families $\cF_n$ for $n\in \bbN$ such that a model $A$ of $\bfT$ 
has property $P$ if and only if 
\[
f_n(\bar x):=\inf_{\varphi \in \cF_n} \varphi(\bar x)
\]
vanishes on $A$ for all $n\in \bbN$. 
We shall define a $G_\delta$ subset $\fG_P$ of $\fG$ such that if $s\in \fG_P$
then  $f_n$ vanishes on $A_s$ for all $n$.
 For each 
  $n\in \bbN$,  $m\geq 1$, and each finite sequence $\bar c$ of $\calLp$-constants of appropriate sort, 
  let 
 \[
 \bfD_{\bar c, m,n}:=\{p\in \bbPT: (\forall s\in U_p) \inf_{\psi\in \cF_n} s(\psi(\bar c))<1/m\}
 \]
 and with 
 \[
 \calD_{\bar c, m,n}:=\bigcup\{U_p: p\in \bfD_{\bar c, m,n}\}
 \]
  let 
 \[
 \fG_P:=\fG\cap \bigcap_{m,n, \bar c} \calD_{\bar c, m,n}. 
 \]
 This is a $G_\delta$ set and if $s\in \fG_P$ then for every $n\in \bbN$ the 
 zero-set of~$f_n^{A_s}$ is a dense subset of $A_s$. By Lemma~\ref{L.udt} 
 the zero-set of $f_n^{A_s}$  is closed and therefore $f_n^{A_s}$ vanishes on $A_s$. 
 This shows that $s\in \fG_P$ implies $A_s$ has property $P$. 
  Conversely, it is clear that   
 if $B$ is a model of $\bfT$ with property $P$ then $s_{B,\bar b}\in \fG_P$  
for any   enumeration $\bar b$ of a dense subset of~$B$. 

Finally, the above implies that 
$\fG_P$ is nonempty if and only if some model of $\bfT$ has property $P$. 

(3)  We import the assumptions and notation from (2) 
and in addition  assume $\bfT$ is a complete theory. 
We shall prove that every triple  $\bar c$, $m$, $n$ and every 
$p\in \bbPT$  there is $q\leq p$ such that $q\in \bfD_{\bar c, m,n}$. 
As in (2), let~$B$ be a model of~$\bfT$ with property $P$. 

Fix $p=(\phi(\bar c),r,\e)$ in  $\bbPT$. 
Since the condition 
$\inf_{\bar x} |\phi(\bar x)-r|<\e$ is consistent with $\bfT$ and $\bfT$ is complete,  
 $(\inf_{\bar x} |\phi(\bar x)-r|)^B<\e$ 
holds in every model of $\bfT$, in $B$ in particular. 
We can therefore fix an interpretation $\bar b$ of constants in $B$ so that 
 \[
| \phi(\bar b)^B-r|<\e. 
\]
Since $B$ satisfies $P$, there exists $\psi\in \cF_n$ such that 
$\psi(\bar b)^B<1/m$. By Lemma~\ref{L.Henkin.Up} we can find 
$q\leq p$ in $\bfD_{\bar c, m,n}$ as required. 

This shows that each $\calD_{\bar c,m,n}$ is dense, and so by the Baire Category Theorem,
$\fG_P$ is a dense $G_\delta$ subset of $\fG$. 
\end{proof} 

Theorem~\ref{T.FM1}  takes as an input a theory $\bfT$ and produces 
a generic functional $s$ 
and a \cstar-algebra $A$ 
that satisfies $\bfT$ and has $s$ as the ``type'' of a countable dense subset. 
Depending on $\bfT$, we can assure that $A$ is nuclear or that it satisfies other approximation properties  from \S\ref{S.AP}. 
It should be noted that not every theory of a \cstar-algebra admits a nuclear model; see \S\ref{S.Not-ee-nuclear}.  
However, if $\bfT$ is a complete theory in the language of \cstar-algebras with (for example) a
nuclear model  then 
Theorem~\ref{T.FM1} (3) and Theorem~\ref{T1}  show that every $\bbPT$-generic model of $\bfT$ 
is nuclear. This applies to other properties of \cstar-algebras that are \udt{}, 
 such as being UHF, AF,  simple, Popa, TAF,  QD,  
 having   nuclear dimension $\leq n$,  or having decomposition rank $\leq n$ for $n\geq 1$  
  (Theorem~\ref{T1}). Of course this is also true (and much easier to prove) for axiomatizable properties.

\section{Finite forcing} 
We now consider  the case when $\Sigma$ consists of all quantifier-free $\calLp$-sentences.
This is the so-called `finite forcing'\index{forcing!finite} and in this case $\bbPTS$ is denoted $\bbPTf$ and 
  $\fSST$ 
  (the space of quantifier-free formulas with  seminorm $\|\cdot\|_{\bfT}$) 
  is denoted $\fSTf$. 
As before, we write
\[
\fTT:= \{s_{A,\bar a}:\text{ for some $A\models \bfT$ and $\bar a$}\}.
\]
Recall that a  theory is \emph{\aea}\index{theory!\aea} 
 if it has a set of axioms that are of the form 
\[
\sup_{\bar x} \inf_{\bar y} \varphi(\bar x, \bar y)
\]
where $\varphi$ is a quantifier-free $\bbR^+$-formula. 
Finite forcing applies only to \aea{} theories, as the single alteration between universal and existential  quantifiers occurs in the forcing process.

  Recall that the  properties proved to be \udt{} in Theorem~\ref{T1}, with a possible exception 
of being TAF, are \udut.  
We state and prove the analogue of Theorem~\ref{T.FM1} for finite forcing. 

\begin{theorem}\label{T.FM2} 
Assume $\bfT$ is an  \aea{} theory of \cstar-algebras and $P$ is  \udut{} such that $\bfT$ has a model satisfying $P$. 
\begin{enumerate}
\item There exists a dense $G_\delta$ subset $\fG$ of  $\fTT$ such that 
every  $s\in \fG$ determines a unique model $A_s$  of $\bfT$ and an interpretation~$\bar a_s$ 
of constants $\bar c$ in $A_s$, such that 
each sort $S$ in $A$ is the closure of $\{a_j^S: j\in \bbN\}$ 
and  $s=s_{A_s,\bar a}$. 

\item 
The set
\[
\fG_P := \{ s \in \fG : A_s \text{ has } P\}
\]
 is a $G_\delta$ subset of $\fG$.
\item Furthermore, if $\bfT$ is a complete theory then $\fG_P$ is dense in $\fG$.
\end{enumerate}
\end{theorem}

\begin{proof} 
Fix an axiomatization of $\bfT$, 
\begin{equation}\label{Eq.AE-axiomatization}
\sup_{\bar x} \inf_{\bar y} \varphi_n(\bar x, \bar y)=0
\end{equation}
for $n\in \bbN$ and quantifier-free formulas $\varphi_n$. 

For a quantifier-free formula $\phi(\bar c, x)$ with $x$ of sort $S$ let 
\begin{align*}
\bfE_{\phi(\bar c, x), r,n}:= &
\{p\in \bbPT: \mbox{for all } s\in U_p, s(\inf_x\phi(\bar c, x))> r-1/n\}\\
&\cup 
\{p\in \bbPT:\mbox{for some  } c_j^S \mbox{ and all } s\in U_p, s(\phi(\bar c, c_j^S))<r\}
\end{align*}
and let 
\[
\cE_{\phi, r,n}:=\bigcup\{U_q : q\in \bfE_{\phi(\bar c, x), r,n}\}
\]
as in Lemma~\ref{L.Henkin.E}. 
The proof of Lemma~\ref{L.Henkin.E} (using the case of Lemma~\ref{L.Henkin.Up} 
when $\Sigma$ is the space of all quantifier-free  formulas)
 shows that each $\cE_{\phi,r,n}$, for $\phi$ quantifier-free, 
is dense open in $\fTT$. 

Let 
\[
\fG_0:=\bigcap_{\phi\in \Phi, n\in \bbN, r\in \bbQ} \cE_{\phi, r,n}. 
\]
Since the formula $d(f(\bar c),x)$ (or, in the case of \cstar-algebras, 
the formula $\|f(\bar c)-x\|$) is quantifier-free for every 
term $f$, the proof of Theorem~\ref{T.FM1} shows that 
in the case of  finite forcing 
for $s$ in $\fG_0$ the interpretation $\bar a_s$ is dense in~$A_s$. 

(1) 
We shall find a dense $G_\delta$ subset of $\fG$ such that 
$A_s\models \bfT$ for every~$s$ in this set. By the above we 
may assume that $\bar a_s$ is dense in $A_s$. 

Fix $n$ and let the formula $\varphi_n(\bar x, \bar y)$ be as in \eqref{Eq.AE-axiomatization}.  
 There are only countably many assignments of 
constants $c_i^S$ to variables $\bar x$ in $\varphi_n(\bar x, \bar y)$. 
Let $\varphi_{n,j}(\bar c, \bar y)$, for $j\in \bbN$, 
be an enumeration of all $\calLp$-formulas obtained 
from $\varphi_n(\bar x, \bar y)$ by assigning constants from $\bar c$ of the appropriate sort
 to~$\bar x$.  

For a fixed $n$ and $j$ 
let $\varphi_{n,j,k}$, for $k\in \bbN$, 
be an enumeration of all~$\calLp$-sentences obtained from $\varphi_{n,j}(\bar c,\bar y)$ 
by assigning constants  $c^S_i$ of the appropriate sort to $\bar y$. 
 If $\bar a=\{a_j^S: j\in \bbN\}$  is dense in $S^A$ for every sort~$S$, then $(A,\bar a)$ is a model 
of $\bfT$ if and only if
for all $n$ and $j$ in $\bbN$ we have 
\[
\inf_{k} \varphi_{n,j,k}^A=0. 
\] 
Fix $m,n$, and $j$ in $\bbN$ and let 
\[
\bfF_{m,n,j}:=   \{p\in \bbPTf: (\forall s\in U_p)  (\exists k) s(\varphi_{n,j,k})<1/m\}. 
\]
In order to show that for every $p\in \bbPTf$ there exists $q\leq p$ in $\bfF_{m,n,j}$, let us
fix $p\in \bbPTf$ and a  certifying pair $(A,\bar a)$ for $p$. 
Since $A$ is a model of~$\bfT$, we have
$(\inf_{\bar y}\phi_{n,j}(\bar a, \bar y))^A=0$. Therefore there is an assignment of 
constants  $c^S_i$ of the appropriate sort to $\bar y$, coded by some $k$, 
 such that $\phi_{n,j,k}^A< 1/m$. 
By the case of Lemma~\ref{L.Henkin.Up} when $\Sigma$ consists of quantifier-free formulas 
we can find a condition $q\in \bfF_{m,n,j}$ extending both $p$ and $(\phi_{n,j,k}(\bar c), 0, 1/m)$. 

Therefore 
\[
\fG:=\fG_0\cap \bigcap_{m,n,j}\bigcup\{U_p: p\in \bfF_{m,n,j}\}
\]
  is a dense $G_\delta$
subset of $\fTT$. 

The construction of $A_s$ from $s$ is identical to that in the proof of Theorem~\ref{T.FM1}.

We now prove that $s\in \fG$ implies $A_s\models \bfT$. 
Fix an axiom, say $\sup_{\bar x} \inf_{\bar y}\phi_n(\bar x, \bar y)$, 
of $\bfT$ and fix $j\in \bbN$. 
As for all $m$ and $j$ there exists $p\in \bfF_{m,n,j}$ 
such that $s\in U_p$, we have  
\[
(\inf_{\bar y} \phi_{n,j}(\bar a, \bar y))^{A_s}=0. 
\]
Since  $\bar a_s$ is dense in $A_s$, this implies 
$(\sup_{\bar x}
\inf_{\bar y} \phi_{n}(\bar x, \bar y))^{A_s}=0$. As $n$ was arbitrary, this implies $A_s\models \bfT$. 

It is clear that $s=s_{\bar A_s, \bar a_s}$ and this concludes the proof of (1).

Now suppose 
 the property $P$ is   \udut{}. 
Fix uniform families $\cF_n$ of quantifier-free formulas 
for $n\in \bbN$ such that a model $A$ of $\bfT$ 
has property $P$ if and only if 
\[
f_n(\bar x):=\inf_{\varphi \in \cF_n} \varphi(\bar x)
\]
vanishes on $A$ for all $n\in \bbN$. 
Since each  $\varphi\in \bigcup_n \cF_n$  is quantifier-free, 
the  proofs of (2) and (3) of Theorem~\ref{T.FM1} taken almost verbatim prove (2) and (3)
of Theorem~\ref{T.FM2}. 
\end{proof}

Notably, the  hyperfinite~II$_1$ factor  is a generic model of its theory for 
both finite and infinite forcing    
 (see \cite[Proposition~5.21 and Corollary~6.4]{FaGoHaSh:Existentially}).

\section{$\forall\exists$-axiomatizability and existentially closed  structures} \hfill\\

We begin with the following essential definition.

\begin{definition}\label{D.e.c.} If $\fM$ is a class of \cstar-algebras we say 
that algebra $A\in \fM$ is \emph{existentially closed in $\fM$},\index{C@\cstar-algebra!e.c.}
 e.c.\ for short, if whenever $B\in \fM$ and $\Psi\colon A\to B$ is an injective $^*$-homomorphism
for every quantifier-free formula $\varphi(\bar x, \bar y)$ and every $\bar a$ in $A$ of the appropriate sort,
\[
\inf_{\bar y} \varphi(\bar a, \bar y)^A=
\inf_{\bar y} \varphi(\Psi(\bar a), \bar y)^B.
\]
\end{definition}

Definition~\ref{D.e.c.} comes in two different flavours, depending on whether the
language is equipped with a symbol for the unit or not. 
In the former case the e.c.\ algebra (if it exists) is necessarily unital. In the latter case, 
the e.c.\ algebra (if it exists) need not be unital. This is because if $A$ 
has unit $a$  and it is  isomorphic to a non-unital subalgebra of some $B\in \fM$
then the sentence $\psi:=\inf_y 1\dotminus \|y-ay\|$ satisfies $\psi^A=1$ and $\psi^B=0$.

Existentially closed structures are the abstract analogues of algebraically closed fields. 
In the applications of model theory, the study of e.c.\ structures is often one of the first moves (see \cite[Chapter 8]{Hodg:Model} or \cite{Mac:modelcomplete}).
To see why, the standard model-theoretic approach is to try to learn something about a structure $A$ by studying some elementary class to which $A$ belongs.
If this class happens to be  $\forall\exists$-axiomatizable
 then it will contain e.c.\ models (even one which contains $A$).  In good cases, since the e.c.\ structure
is very rich, one can deduce information there and transfer it back to $A$.  
As an example of this rich structure, one can show that an e.c.\ \cstar-algebra in a class closed under taking crossed products with $\bbZ$ has the property that all of its automorphisms are approximately inner (see the argument of \cite[Proposition~3.1]{FaGoHaSh:Existentially}). The study of e.c.\ structures has worked well in the classical first order case.
For instance, e.c.\ groups are well-studied objects (see e.g.,  \cite{hodges2006building} and \cite{Mac:annals}). 
\cstar-algebras that are e.c.\ in the class of all \cstar-algebras are studied in \cite{goldbring2014kirchberg}. 
Some natural classes allow many e.c.\ objects (see e.g.,  \cite{FaGoHaSh:Existentially}).
In the best case, the class of e.c.\ structures proves to be elementary and thereby provides a particularly good theory, called the model companion. Unfortunately, it is known that there
is no model companion to the theory 
of \cstar-algebras (see \cite[Theorem~3.3]{eagle2015quantifier}).

A straightforward modification of classical result by Robinson yields the following theorem 
(\cite[Corollary~6.6]{FaMa:Omitting}; see  also   \cite[\S 5]{FaGoHaSh:Existentially}).

\begin{thm} \label{T.ec} Assume $\bfT$ is an \aea{} theory and $\fM$  is a  
non-empty class of \cstar-algebras satisfying $\bfT$ that is \udt{} consisting of existential formulas. 
Then there exists a separable algebra $A\in \fM$ 
that is existentially closed in $\fM$. 
Furthermore, $A$ can be chosen to include an isomorphic copy of any given 
 separable algebra $B\in \fM$. 
\qed
\end{thm} 

We add a constant 1 to the language and add (universal) axioms 
stating that the interpretation of this constant is a multiplicative unit. Therefore all \cstar-algebras considered below are unital, unless otherwise specified. 
By Theorem~\ref{T.ec} 
one obtains an e.c.\ algebra (not necessarily unique)  in 
every class of \cstar-algebras listed in  Theorem~\ref{T1}. 
Some of these e.c.\ algebras have other, familiar,  characterizations.

\begin{example} 
\begin{enumerate}
\item \label{I.ec.UHF} The class of UHF algebras contains a 
 unique e.c.\ algebra---the universal UHF algebra $Q$.
It is e.c.\ by Proposition~\ref{I.ec.MF} below.
 Since having a unital copy of $M_n(\bbC)$ 
is existentially axiomatizable for all $n$ (\S\ref{S.Mn}), any algebra $A$ which is e.c.\ for the class of UHF algebras
 must contain a unital copy of $M_n(\bbC)$ for all $n$.  By Glimm's classification theorem for  
 UHF algebras this implies that $A \cong Q$.

\item \label{I.ec.UHF2} Suppose $S$ is a proper subset of the set of all primes. 
The class  of all UHF \cstar-algebras not containing a unital copy of $M_p(\bbC)$ for any 
$p\in S$ contains a unique e.c.\ algebra, 
$\bigotimes_{p \not\in S} M_p(\bbC)^{\otimes \infty}$. 
This is proved by the argument from  \eqref{I.ec.UHF} 
because the class of algebras not containing 
  a unital copy of $M_p(\bbC)$ is universally axiomatizable by 
  Theorem~\ref{Summary}. 
\item\label{I.Q.ec.0}  $Q$ is e.c.\ for the class of all AF algebras, and even for the class of all MF algebras (Proposition~\ref{I.ec.MF}; 
see also Corollary~\ref{C.Q.ec}) but it is not a unique e.c.\ algebra in either of these classes. 
This is because Theorem~\ref{T.ec} implies that every \cstar-algebra in each of these classes $\cC$
is isomorphic to a subalgebra of an algebra in $\cC$ which is  e.c.\ for $\cC$. 
\item 
The unique e.c.\ nuclear \cstar-algebra is $\cO_2$; 
it is also the unique e.c.\ exact \cstar-algebra. 
Both these facts were  proved in 
 \cite{goldbring2014kirchberg}. $\cO_2$ is also one of the e.c.\ algebras for the class of all \cstar-algebras 
 embeddable into $\cO_2^{\cU}$ (cf. \eqref{I.Q.ec.0}). See also    Proposition~\ref{P.ssa-ec} below. 
\item Consider the class of abelian \cstar-algebras of the form $C(X)$, where $X$ is a compact, connected metric space. 
The algebra of continuous functions on the pseudoarc is  an  e.c.\ algebra in this class (\cite{eagle2015pseudoarc}). 
  \end{enumerate}
\end{example}

An injective $^*$-homomorphism is  an \emph{existential embedding}\index{existential embedding}  if 
for every quantifier-free formula $\varphi(\bar x, \bar y)$ and every $\bar a$ in $A$ of the appropriate sort,
\[
\inf_{\bar y} \varphi(\bar a, \bar y)^A=\
\inf_{\bar y} \varphi(\Phi(\bar a), \bar y)^B.
\]
Therefore $A$ is e.c.\ in $\fM$ if every embedding of $A$ into some $B\in \fM$ is
an existential embedding. 
Following \cite{barlak2015sequentially} 
we say that a \cstar-homomorphism $\Phi : A\rightarrow B$ is \emph{sequentially split}\index{sequentially split}  if 
there is a $^*$-homomorphism $\Psi : B\to A^{\cU}$ such that $\Psi\circ\Phi$ 
agrees with the diagonal embedding of $A$ into~$A^{\cU}$.  
Countable saturation of $A^{\cU}$ and \L o\'s' theorem imply that every existential embedding is sequentially split
and that a sequentially split embedding is existential 
if and only if $\Psi$ as above can be chosen to be an injection (see  \cite[\S 4.3]{barlak2015sequentially}
or the proof of Lemma~\ref{L.MF.1}). 
Therefore \cite[Theorem~2.8 and Theorem~2.9]{barlak2015sequentially}
provide various properties of existential embeddings and e.c.\ \cstar-algebras. 
Notably, if there is an existential embedding of $A$ into a nuclear algebra $B$ that satisfies 
the UCT then~$A$ is nuclear and it satisfies the UCT (\cite[Theorem~2.10]{barlak2015sequentially}).

\begin{remark} 
Theorem~\ref{T1} and Theorem~\ref{T.ec} imply the existence of 
e.c.\ AF algebras. By Elliott's classification result for 
 AF algebras  (\cite{Ell:On}, see also \cite{Ror:Classification}) the category 
 of AF algebras is equivalent to the category of dimension groups. 
 The relation between e.c.\ objects in these two categories was  studied in 
 \cite{Scow:Some} and \cite{scowcroft2012existentially}. 
 \end{remark} 

\section{Strongly self-absorbing algebras} 
 A model is \emph{atomic}\index{atomic!model}
  if it is isomorphic to an elementary submodel of every model of its theory. 
Every strongly self-absorbing \cstar-algebra is 
an atomic model of its theory (see
\cite[Proposition~2.15]{FaHaRoTi:Relative}).  An atomic model is always e.c.\  for the class of models of its theory, and so every strongly self-absorbing \cstar-algebra is e.c.\ for the class of models of its theory.
Some instances of the question of the determination of the universal theory of a strongly self-absorbing algebra D are equivalent
to well-known problems (see below).  Recall that $\ThA(D)$ is the set of all universal sentences
in the theory of $D$ and that a separable 
 \cstar-algebra $A$ embeds into the ultrapower $D^{\cU}$ if and only if $A\models\ThA(D)$
 ($\cU$ stands for a nonprincipal ultrafilter on~$\bbN$; see the first line of the proof of Lemma~\ref{L.MF.1}). 
 As being $D^{\cU}$-embeddable  
  is separably 
universally axiomatizable (\S\ref{S.MF}) the existence of an e.c.\ $D^{\cU}$-embeddable
 algebra is a consequence of Theorem~\ref{T.ec}. 
 When $D$ is  $\cO_2$, the question of whether the universal theory of $\cO_2$ is the same as that of all \cstar-algebras is equivalent to the Kirchberg embedding problem.  The question of whether the universal theory of $Q$ is the same as the theory of stably finite \cstar algebras is the quasidiagonality question.
The case when $D$ is $\cZ$ is also interesting.

The analogue  of Proposition~\ref{P.ssa-ec} below  (with the analogous proof) 
was proved for the hyperfinite II$_1$
factor in \cite[Lemma~2.1]{FaGoHaSh:Existentially} and  \eqref{I.ssa.3}   was proved
in \cite{goldbring2014kirchberg}. 

 \begin{prop} \label{P.ssa-ec}\label{I.ec.MF}
 \begin{enumerate}
\item \label{I.ssa.1}  Every strongly self-absorbing \cstar-algebra $D$ is e.c.\ for 
 the class of all unital \cstar-algebras isomorphic to a subalgebra of an ultrapower of $D$. 
\item \label{I.ssa.2} The rational UHF algebra $Q$ is  e.c.\ for the class of unital MF algebras. 
 \item \label{I.ssa.3} $\cO_2$ is e.c.\ for the class of all  \cstar-algebras which are isomorphic to a subalgebra of an ultrapower of $\cO_2$. 
  \end{enumerate}
 \end{prop} 
 
\begin{proof} 
\eqref{I.ssa.1} This is the standard sandwich argument. Suppose $D$ is strongly self-absorbing,  
 $A$ embeds into $D^{\cU}$ 
 and 
$D$ is isomorphic to a subalgebra of $A$. Fix a quantifier-free formula $\varphi$ and $\bar a$ in $D$ and assume that 
$(\inf_{\bar y} \varphi(\bar a, \bar y))^A=r$. 
Then 
$(\inf_{\bar y} \varphi(\bar a, \bar y))^D\geq r$. 
We may identify $A$ with a unital subalgebra of~$D^{\cU}$.  
As all embeddings of $D$ into its ultrapower are unitarily conjugate 
(see e.g., \cite[Theorem~2.15]{FaHaRoTi:Relative}) we can arrange 
the composition of the embeddings of $D$ into $A$ with the embedding of $A$ into $D^{\cU}$ to be the diagonal embedding of 
$D$ into its ultrapower. 
By  \L o\'s' theorem, 
 $(\inf_{\bar y} \varphi(\bar a, \bar y))^{D^{\cU}}=(\inf_{\bar y} \varphi(\bar a, \bar y))^D=r$. 

\eqref{I.ssa.2} is a consequence of \eqref{I.ssa.1} since $Q$ is strongly self-absorbing and 
being unital and MF is equivalent to being isomorphic  to a unital subalgebra of 
$Q^{\cU}$ (Lemma~\ref{L.MF.def}). 

\eqref{I.ssa.3} This  is an immediate  consequence of \eqref{I.ssa.1}. 
\end{proof} 


\begin{corollary} \label{C.Q.ec} Every  stably finite algebra is MF if and only if the 
rational UHF algebra $Q$ is e.c.\ for 
the class of unital, stably finite, \cstar-algebras. 
\end{corollary}

\begin{proof} Only the converse direction requires a proof. 
Assume $Q$ is e.c.\ for the class of unital,  stably finite, \cstar-algebras. Also assume that 
$A$ is a unital, stably finite, \cstar-algebra which is not MF. 
Since $Q$ is MF and nuclear,  the tensor product $A\otimes Q$ is 
stably finite by Lemma~\ref{L.sfMF}. Since being MF is universally axiomatizable, as 
 in the proof of Proposition~\ref{P.6.5.1} we obtain that $Q$ is not MF, a contradiction. 
\end{proof} 

We turn to  $\cZ$ and projectionless, stably finite, unital \cstar-algebras. Let $\cC$ be the class of 
 nuclear, projectionless, stably finite, unital \cstar-algebras.
\begin{proposition} The Jiang--Su algebra $\cZ$ is e.c.\ for $\cC$
if and only 
if every $A\in \cC$ is embeddable into an ultrapower of $\cZ$. 

If $\cZ$ is e.c.\ for $\cC$ then it is the unique projectionless unital strongly self-absorbing algebra. 
\end{proposition} 

\begin{proof} For the first part, assume $\cZ$ is e.c.\ for $\cC$ and $A \in \cC$.  Since $A \otimes \cZ$ is projectionless and stably finite, we may assume that $\cZ$ is contained in $A$.  But then we can find $\cZ'$ elementarily equivalent to $\cZ$ which contains $A$.  $\cZ'$ embeds into an ultrapower of $\cZ$ and so $A$ is embeddable in an ultrapower of $\cZ$.  If on the other hand, every $A \in \cC$ embeds into an ultrapower of $\cZ$ then since $\cZ$ is e.c.\ for the class of subalgebras of ultrapowers of $\cZ$, $\cZ$ is e.c.\ for the smaller class $\cC$ as well.

Now suppose $\cZ$  is e.c.\ for $\cC$ and 
$D$ is projectionless, unital and strongly self-absorbing. Then $D \otimes \cZ \cong D$ and $D \in \cC$.  Evidently $\cZ$ embeds into $D$ and by the argument in the previous paragraph, $D$ embeds into an ultrapower of $\cZ$.  This shows that $\ThA(\cZ)= \ThA(D)$.
 By \cite[Theorem~2.16]{FaHaRoTi:Relative} this implies that $D\cong \cZ$. 
\end{proof}  

It is not known whether $\cZ$ is the unique 
projectionless strongly self-absorbing \cstar-algebra. 
 In the non-unital case, one could also ask whether 
Jacelon's $\cW$ (see~\cite{Jac:W} and also \cite{Win:QDQ}) is a 
stably finite projectionless, nuclear, e.c.\ algebra, or even 
 the unique such algebra.

\section{Stably finite, quasidiagonal, and MF algebras} 
Following~\cite{goldbring2016robinson} we denote the 
class of  stably finite, unital, and separable \cstar-algebras by $\SF$\index{S@$\SF$}, 
the class of algebras in $\SF$ that are in addition nuclear by $\SFn$, 
and the class of algebras in $\SFn$ that are in addition simple by $\SFns$.   
Stably finite,  unital,  \cstar-algebras  are universally axiomatizable (\S\ref{S.SF}), 
and both nuclear and simple \cstar-algebras are definable by uniform families of formulas (Theorem~\ref{T1}).  
Therefore by Theorem~\ref{T.ec} there exists an e.c.\  \cstar-algebra in each of the classes $\SF$, $\SFn$, and $\SFns$, and every \cstar-algebra in either of these classes is isomorphic to a subalgebra of  an e.c.\ algebra for the same class. 
In  \cite[Conjecture~12]{goldbring2016robinson} it was conjectured that 
 the rational 
UHF algebra~$Q$ is a unique nuclear e.c.\ algebra in $\SF$.   
 Also, in 
 \cite[Corollary~16]{goldbring2016robinson} it was proved using results of \cite{TWW}  that  every algebra 
 which is 
 e.c.\  in $\SF$ that is in addition in $\SFns$  
 and in  the UCT class   is AF. 

%
It is not known whether the tensor product of two stably finite \cstar-algebras is stably finite. 
We shall need the following weaker (and well-known)  fact. 

\begin{lemma} \label{L.sfMF} 
Suppose $A$ is stably finite and $B$ is MF. 
 Then some tensor product $A\otimes_\alpha B$ is stably finite. 
 
 If at least one of $A$ and $B$ is nuclear then the minimal tensor product 
 $A\otimes B$ is stably finite. 
 \end{lemma} 

\begin{proof} Since $A$ is stably finite, so is $M_n(A)$ for every $n\in \bbN$. 
Let $\cU$ be a nonprincipal ultrafilter on $\bbN$. 
The ultraproduct  $\prod_{\cU} M_n(A)$ is stably finite by {\L}o\'s' Theorem  and the fact that being stably finite is axiomatizable (Theorem~\ref{Summary}). 
Let $A_0$ denote the (unital) image of $A$ in this ultraproduct under the diagonal map. 
Since $A$ is unital, this ultraproduct also contains a canonical copy of 
$\prod_{\cU} M_n(\bbC)$. 
Let $B_0$ denote its subalgebra isomorphic to  $B$.  
Then  $\cst(A_0,B_0)$ is isomorphic to $A\otimes_\alpha B$ for some tensor 
product $\otimes_\alpha$; this is easily proved by checking the universal property of the tensor product. 
We conclude that $A\otimes_\alpha B$ 
 is isomorphic to a subalgebra of a stably finite \cstar-algebra, 
and therefore stably finite itself. 

If one of $A$ and $B$ is nuclear then $A\otimes_\alpha B$ is the minimal tensor product, and this completes the proof. 
\end{proof} 

 It is not clear that 
one can conclude that $A\otimes B$ embeds into $\prod_{\cU} M_n(A)$, or that it is stably finite, 
without the assumption that one of $A$ or $B$ is nuclear.
It is for example not true that the quotient of a stably finite \cstar-algebra is always stably finite: 
While $C_0((0,1],C)$ is trivially stably finite  for every \cstar-algebra~$C$, 
it always has $C$ as a quotient. 

Quasidiagonality was introduced in \S\ref{S.QD}. 
All stably finite and nuclear \cstar-algebras in the UCT class are quasidiagonal (\cite{TWW}).
It is not known whether all
 stably finite and nuclear \cstar-algebras  are quasidiagonal. The reduced group algebra of the free 
 group with $n\geq 2$ generators is stably finite but not quasidiagonal. 

\begin{prop} \label{P.6.5.1} Some e.c.\ algebra in $\SFn$  
 is quasidiagonal if and only if  
 every \cstar-algebra in $\SFn$ is quasidiagonal. 

If some e.c.\ algebra in $\SFn$ belongs to the UCT class then every \cstar-algebra in $\SFn$ is quasidiagonal. 
\end{prop} 

\begin{proof}  Only the direct implication requires a proof. Suppose $A$ is a quasidiagonal
e.c.\ algebra in $\SFn$.   
Assume there exists a \cstar-algebra~$B$ that is nuclear and stably finite but not quasidiagonal.  
 Then this algebra is not even MF (see \S\ref{S.MF}).  
Since being separable and MF is characterized
 by being a subalgebra of a countably saturated algebra $\ASA$, 
 being MF is   (among separable \cstar-algebras) 
 universally axiomatizable by Lemma~\ref{L.MF.1}. 
 Since $B$ is stably finite and nuclear, and $A$ is MF, the tensor product 
   $A\otimes B$ is nuclear and stably finite by Lemma~\ref{L.sfMF}.   It also satisfies an existential 
 statement assuring that~$B$ is not MF. Since $A$ is a subalgebra of this algebra 
 and it is e.c., $A$ satisfies this statement and is therefore not MF; contradiction.  

If $A\in \SFn$ also belongs to  the UCT class then it is quasidiagonal by \cite[Corollary~D]{TWW}. 
This implies the  second assertion and concludes the proof.  
 \end{proof}


\chapter{\cstar-algebras not elementarily equivalent to nuclear \cstar-algebras} 
\label{S.Not-ee-nuclear} 
\setcounter{thm}{0}

Theorem~\ref{T.FM1}  and Theorem~\ref{T.FM2} comprise a machine that turns theories of \cstar-algebras into 
\cstar-algebras and reduces some open problems on classification and the structure of \cstar-algebras to questions on the 
existence of theories with specific properties. The method of producing a \cstar-algebra with a specified theory was
implicitly used before (see the introduction to \S\ref{S.Elementary.vNa}). 
One of the novelties of our approach is that
 it produces a nuclear \cstar-algebra if it is applied to a 
  theory that allows nuclear \cstar-algebras. 
In this section we shall see that some theories of \cstar-algebras are so peculiar that they do not even have nuclear models. 
As an additional motivation for questions of this sort  we make the following simple observation. 

\section{Exact algebras}\label{S.Exact} 
It is not known whether every stably finite \cstar-algebra $A$ has a tracial state. 
 A result of Haagerup (see \cite{haagerup2014quasitraces}) gives an affirmative answer if in addition one assumes that $A$ is exact. Exactness 
is a weakening of nuclearity: a \cstar-algebra $A$ 
is \emph{exact}\index{C@\cstar-algebra!exact} if 
there exists an injective $^*$-homomorphism $\pi$ of $A$ into a \cstar-algebra $B$ 
such that for every
tuple $\bar a$  in $A_1$ and every $\e>0$
there are a finite-dimensional \cstar-algebra $F$ and  c.p.c.\  maps $\varphi\colon A\to F$ and
$\psi\colon F\to B$ such that the diagram   
\begin{center}
\begin{tikzpicture} 
\matrix[row sep=1cm, column sep=1cm]{
\node (A1){$A$}; &  &   \node (A2) {$B$}; \\
 & \node (F) {$F$};
\\
};
\path (A1) edge [->] node [above] {$\pi$} (A2); 
\path (A1) edge [->]  node [below] {$\varphi$} (F);
\path (F) edge [->] node [below] {$\psi$} (A2); 
\end{tikzpicture} 
\end{center}
\noindent $\e$-commutes on $\bar a$
(cf.\ \S\ref{S.Nuclearity} and see \cite[\S 2.3]{BrOz:cstar} or \cite[IV.3]{Black:Operator}).  

Evidently, nuclear \cstar-algebras are exact and subalgebras of exact \cstar-algebras are exact.  Not all exact \cstar-algebras are nuclear: $\mathrm{C}^*_r(F_\infty)$ and $\prod_{\cU} M_n(\bbC)$ are exact but not nuclear.
 
 \begin{prop} \label{P.Exact} 
 Every unital stably finite \cstar-algebra $A$ elementarily equivalent to an exact \cstar-algebra has 
 a tracial state. 
 \end{prop} 
 
 \begin{proof} Fix an exact $B$ such that $A\equiv B$. 
 Since being stably finite is axiomatizable by \S\ref{S.SF}, $B$ is stably finite. 
 By \cite{haagerup2014quasitraces} $B$ has a tracial state. 
 Since having a tracial state is axiomatizable by \S\ref{S.tracial.0}, 
 $A$ has a tracial state. 
\end{proof} 

\begin{question} Is every \cstar-algebra elementarily equivalent to an exact \cstar-algebra? 
Is every stably finite \cstar-algebra elementarily equivalent to an exact \cstar-algebra? 
\end{question}

Notice that a positive answer to the first part of this question would imply a positive answer to the Kirchberg embedding problem: if $A$ is a separable \cstar-algebra elementarily equivalent to a separable exact \cstar-algebra $B$ then $A$ embeds into $B^\cU$. Since $B$ embeds into $\cO_2$, $A$ would embed into $\cO_2^\cU$.

We shall answer the analogous questions for nuclear \cstar-algebras by first producing a family of monotracial \cstar-algebras 
not elementarily equivalent to nuclear \cstar-algebras in Propositions~\ref{P.not-nuclear.1} and Proposition~\ref{P.not-nuclear}  and then producing a purely infinite and 
simple \cstar-algebra that is not even a model of the theory of nuclear \cstar-algebras in Proposition~\ref{P.not-nuclear.2}.

\section{Definability of traces: the uniform strong Dixmier property}  
\label{S.Def.Tau.2} 
We continue the discussion of \S\ref{S.Def.Tau.1}. 
A unital \cstar-algebra $A$ has the \emph{Dixmier property}\index{Dixmier property} if 
for every $a \in A,$ the norm-closed convex hull $K_a$ of the unitary orbit of $a$ ($\{uau^*: u\in U(A)\}$), 
has nonempty intersection with the center of $A$, $Z(A)$ (see \cite[III.2.5.15]{Black:Operator}). 
If $Z(A)$ is trivial (for example, if $A$ is simple) then the Dixmier property is equivalent to the \emph{strong Dixmier property}\index{Dixmier property!strong}
asserting $K_a$ has nonempty intersection with the scalars for every $a \in A$.

\begin{lemma} \label{L.Dixmier.0} 
An algebra  $A$ with the strong  Dixmier property has  at most one tracial state.  
\end{lemma} 

\begin{proof} Suppose $a \in A$ and $\mu\in \bbC$ is in $K_a$. 
For every $\epsilon > 0$, we have an $n \in \bbN$, unitaries $u_i$ and scalars $\lambda_i \geq 0$ for $i \leq n$ such that
\[
\biggl\|\mu - \sum_{i=1}^n \lambda_i u_i^*au_i\biggr\| < \epsilon \mbox{ and } \sum_{i=1}^n \lambda_i = 1
\]
By letting $\epsilon$ tend to 0, we see that for every trace $\tau$ of $A$ 
we have $\tau(a)=\mu$. This proves that $A$ has at most one trace and 
it also shows that $A$ is tracial if and only if the intersection of $K_a$ with the 
scalars has exactly one point for every $a\in A$. 
\end{proof} 

In \cite{HaZs:Sur} it was proved that every unital simple \cstar-algebra with at most one trace has the strong Dixmier property. 
Thus the simple, unital and monotracial \cstar-algebra whose trace is not definable from 
Proposition~\ref{P.non-def-trace} 
 has the  strong Dixmier property, but its ultrapower does not. This shows that the strong Dixmier property is not axiomatizable. 
 In  \cite[Theorem~1]{Oza:Dixmier} Ozawa 
 proved that a variant of the Dixmier property is equivalent to every non-zero quotient of $A$ having a tracial state. 
 The strong Dixmier property implies that $A$ is simple if it has a faithful trace (\cite{harpe-skandalis}). 
More generally, if $A$ has the strong Dixmier property and $\tau$ is a trace then every proper ideal is included in the ideal 
\[
J:=\{a: \tau(a^*a)=0\}. 
\]
However, there are non-simple tracial \cstar-algebras with the Dixmier property, such as the unitization of the compact operators. 
This algebra even has the uniform strong Dixmier property, defined below, with $k(j)=j$.

We shall consider a uniform variant of the Dixmier property, analogous to the uniform definability of Cuntz--Pedersen nullset considered
in \S\ref{S.CP}. 
We note that if $A$ has the strong Dixmier property then for every $a \in A$ and $\epsilon > 0$, there exist
 $\mu \in \bbC$ and $n$, dependent on $a$ and~$\epsilon$,
together with unitaries $u_i$ and rational numbers $\lambda_i \geq 0$ such that 
\[
\biggl\|\mu - \sum_{i=1}^n \lambda_i u_i^*au_i\biggr\| < \epsilon \mbox{ and } \sum_{i=1}^n \lambda_i = 1
\]
By taking common denominators, possibly repeating unitaries and increasing $n$, we can assume that $\lambda_i = 1/n$ for all $i$.
This leads us to consider,
for $n\in \bbN$, the following sentence (remember Remark \ref{rmk.numbers-as-scalars}): 
\[
\varphi_n = \sup_{a}\inf_{\bar u\in U(A)^n}\inf_{\mu \in \bbC, \|\mu\| \leq 1} \biggl\|\mu-\frac{1}{n}\sum_{i=1}^n u_i^*au_i\biggr\|. 
\]
For a sequence 
$\bar k=(k(j): j\in \bbN)$, we say that $A$ has the $\bar k$-\emph{uniform strong Dixmier property}\index{Dixmier strong property!uniform}\index{strong Dixmier property!$\bar k$-uniform}
if $\varphi_{k(j)}^A < 1/j$ for all $j\geq 1$.    We say that $A$ has the \emph{uniform strong Dixmier property} if it has the $\bar k$-uniform strong Dixmier property for some $\bar k$.
A related concept, the uniform Dixmier property, is studied in \cite{ArchboldRobertTikuisis}.

\begin{lemma} \label{L.usDP} The following are equivalent for a unital \cstar-algebra $A$. 
\begin{enumerate}
\item $A$ has the $\bar k$-uniform strong Dixmier property for some $\bar k$. 
\item There exists $n\geq 2$ such that $\varphi_n^A<1$. 
\end{enumerate}
Also, having the $\bar k$-uniform strong Dixmier property for a fixed $\bar k$ is \aea. 
\end{lemma} 

\begin{proof} It is clear that  (1) implies (2) and that $\varphi_n$ is \aea{} for all $n$. 

Suppose (2) holds for $n$ and let $r:=\varphi_n^A$. 
Then for every $a\in A$ we can find $\bar u\in U(A)^n$ and a scalar $\mu$ such that $|\mu|\leq \|a\|$ and 
$\|\mu-\frac 1n \sum_{i=1}^n u_i^* a u_i\|\leq r\|a\|$. By applying this to 
$a_1:=\mu-\frac 1n \sum_{i=1}^n u_i^* a u_i$ we obtain $\varphi_{n^2}^A\leq r^2$, 
and by induction $A$ has the $\bar k$-uniform strong Dixmier property with $k(j)\leq n^{m(j)}$ with   $m(j)<  -\log_r j$, for all $j\geq 2$. 
\end{proof}

In particular, 
an ultraproduct  of \cstar-algebras with the $\bar k$-uniform strong Dixmier property (for a fixed $\bar k$)  
has the $\bar k$-uniform strong Dixmier property and 
as in the proof of Theorem~\ref{T.tau-definable.0} one can prove that the Cuntz--Pedersen nullset  $A_0$ is definable. 
A stronger statement is true. 

\begin{lemma} \label{L.Dixmier.1} 
If $A$ has the uniform strong Dixmier property then it has the strong Dixmier property.
If moreover $A$ has a trace then the trace is unique and definable. 
\end{lemma} 

\begin{proof}  Only the last point requires a proof.
Assume $A$ has trace $\tau$. 
 For each $j \geq 1$ we have 
\[
\sup_{x}\left|\tau(x)\cdot 1_A-\inf_{\bar u\in U(A)^{k(j)}}\|\frac{1}{k(j)}\sum_{i=1}^{k(j)} u_i^*xu_i\|\right|<\frac 1j
\]
and this shows the definability of $\tau$. 
\end{proof}

Discrete groups provide a rich source of tracial \cstar-algebras. 
If $G$ is a discrete group then the \emph{reduced \cstar-algebra}\index{C@$\mathrm{C}^*_r(G)$}
$\mathrm{C}^*_r(G)$ is generated by the left regular representation of $G$ on $\ell_2(G)$ (see \cite[\S 2.5]{BrOz:cstar}). The weak closure of $\mathrm{C}^*_r(G)$ in this representation is the group von Neumann algebra $L(G)$. It is equal to the tracial von Neumann algebra
 $N_{\tau,\mathrm{C}^*_r(G)}$ (as defined in \S\ref{S.Def.Tau.1}) 
 where $\tau$ is the canonical trace on $\mathrm{C}^*_r(G)$. 
The algebra $L(G)$ is a factor if and only if every nontrivial conjugacy class in $G$ is infinite. 

A discrete group $G$ is a \emph{Powers group}\index{Powers group}
 (\cite[p. 244]{harpe-skandalis}) if for every 
nonempty finite $F\subseteq G\setminus \{e\}$ 
and every $n\geq 1$ there exists a partition $G=X\sqcup Y$ and $g_1, \dots, g_n$ in $G$ such that 
\begin{enumerate}
\item $fX\cap X=\emptyset$ for all $f\in F$, and 
\item $g_jY\cap g_kY=\emptyset$ for all $j<k\leq n$. 
\end{enumerate}
For example, as pointed out in \cite[p. 244]{harpe-skandalis}, $F_\infty$ is a Powers group, which follows by arguments in \cite{Powers}.

 Since the reduced group \cstar-algebra $\mathrm{C}^*_r(G)$ is nuclear whenever $G$ is amenable 
 the following proposition strengthens the fact that
 Powers groups are not amenable (\cite{harpe-simplicity}).  
  McDuff factors were defined before Lemma~\ref{L.tau-McDuff}.

\begin{prop} \label{P.not-nuclear.1} 
If $G$ is a Powers group with infinite conjugacy classes such that the II$_1$ factor 
$L(G)$ is not McDuff 
then $\mathrm{C}^*_r(G)$ is not elementarily equivalent to a nuclear \cstar-algebra. 

In particular, $\mathrm{C}^*_r(F_\infty)$ is not elementarily equivalent to a nuclear \cstar-algebra. 
\end{prop} 

\begin{proof}
 For every Powers group $G$ the reduced group algebra  
$\mathrm{C}^*_r(G)$  satisfies the uniform strong Dixmier property. 
This follows from results of de la Harpe and Skandalis, as follows.
In   \cite[Proposition 3]{harpe-skandalis} it was proved that the pair $(\mathrm{C}^*_r(G),\tau)$ (where $\tau$ is the 
canonical trace)  satisfies a certain property $P_3$. 
 In \cite[Lemma~1]{harpe-skandalis} it was proved that there exists a universal constant $r<1$ 
 such that if $A$ is a \cstar-algebra with trace $\tau$ which  
 satisfies $P_3$ then for every self-adjoint $a\in A$ with $\tau(a)=0$ there exist
 unitaries $u_1$, $u_2$ and $u_3$ in $A$ such that $\|3^{-1}\sum_{j=1}^3 u_j a u_j^*\|\leq r$ 
 (see also  \cite{harpe-simplicity}). 
 Using our notation from before Lemma~\ref{L.usDP}, we have  $\varphi_3^A\leq r<1$ and 
Lemma~\ref{L.usDP} now implies that $A$ has the  uniform strong Dixmier property.
 Therefore the  canonical trace on $A$  is 
 definable by Lemma~\ref{L.Dixmier.1}.  

 Since $G$ 
 has infinite conjugacy classes, $L(G)$ is a II$_1$ factor. 
 By our assumption it is not McDuff and therefore by Lemma~\ref{L.tau-McDuff}
 the conclusion follows. 
 
 Since the II$_1$ factor corresponding to $\mathrm{C}^*_r(F_\infty)$ is isomorphic to the group factor $L(F_\infty)$, we are done since $L(F_\infty)$ does not have property~$\Gamma$ by a well-known result of  Murray and von Neumann. 
\end{proof}

Other \cstar-algebras with the uniform strong Dixmier property are 
 the reduced free products $A*B$ that satisfy the Avitzour condition (see \cite[Proposition 3.1]{avitzour}), 
 certain amalgamated reduced free products, and crossed products by Powers groups (e.g.,  see \cite{harpe-skandalis}).

\begin{prop} \label{P.not-nuclear}
An ultraproduct of full matrix algebras with respect to a nonprincipal ultrafilter, 
$\prod_{\cU} M_n(\bbC)$,  is not elementarily equivalent to a nuclear \cstar-algebra.
\end{prop} 

\begin{proof}
This algebra has a unique definable trace by Theorem \ref{T.tau-definable.0} \eqref{I.tau-def.2} since 
$\nucdim(M_n(\bbC))=0$ for all $n$.
By Lemma~\ref{L.tau-McDuff} it will suffice to show that
the corresponding II$_1$ factor is not McDuff. 
In fact, the~II$_1$ factor corresponding to $\prod_{\cU}M_n(\bbC)$
does not have property $\Gamma$ which follows from a result of von Neumann (\cite{vN:Approximative}); see \cite[\S4]{fang2011note} or \cite[Theorem~5.1]{FaHaSh:Model3}). 
\end{proof}

\section{Elementary submodels of von Neumann algebras} 
\label{S.Elementary.vNa}
Consider a  von Neumann algebra $N$ as a \cstar-algebra. Unless it is finite-dimensional, $N$ is not separable, 
and a separable elementary submodel $A$ is a \cstar-algebra that inherits all elementary properties of $N$, and, as it turns out, certain non-elementary properties.
For example, in 
the proof of  \cite[Theorem~3.2]{Phi:Simple} it was shown (among other things) 
that if $N$ is a II$_1$ factor not isomorphic to its opposite algebra then 
$A$ is a \cstar-algebra with a unique trace not isomorphic to its opposite algebra. 
For another example, an elementary submodel of the free group von Neumann algebra~$L(F_2)$ (or any other II$_1$ factor without property $\Gamma$) is a simple, 
monotracial \cstar-algebra which is not approximately divisible (see \S\ref{S.AD} and  \cite[Example~4.8]{blackadar1992approximately}).

The theory of nuclear \cstar-algebras 
is the largest theory $\bfT_N$ in the language of \cstar-algebras such that every nuclear \cstar-algebra satisfies~$\bfT_N$. 
Not being a model of $\bfT_N$ implies not being elementarily equivalent to a nuclear \cstar-algebra.  
Since any ultraproduct of nuclear \cstar-algebras is a model of~$\bfT_N$, 
Proposition~\ref{P.not-nuclear} implies that 
$\prod_{\cU} M_n(\bbC)$ is a model of $\bfT_N$ 
that is not elementarily equivalent to any nuclear \cstar-algebra. 
The example given by the following is even further removed from nuclear \cstar-algebras. 

\begin{prop} \label{P.not-nuclear.2} 
There exists a unital, purely infinite, simple, \cstar-algebra 
$M$  with the same $K$-theory as $\cO_2$ that is not elementarily equivalent to a nuclear \cstar-algebra
and that is not even a model of $\bfT_N$. 
\end{prop}

This proof is very similar to   
the argument given in \cite[Example~4.8]{blackadar1992approximately}. 

\begin{proof}
Let $M$ be a type III factor without an almost periodic state, as constructed in \cite{connes1974almost}. 
All of its projections are Murray-von Neumann equivalent and it is a purely infinite and simple \cstar-algebra. 
It is also not approximately divisible. Since being purely infinite and simple   is axiomatizable
(\S\ref{S.PI}) and being approximately divisible is axiomatizable (\S\ref{S.AD}), every algebra elementarily equivalent to~$M$ 
is purely infinite and simple but not approximately divisible. However, every nuclear, purely infinite, and simple \cstar-algebra is $\cO_\infty$-absorbing 
(\cite{KircPhi:Embedding}) and therefore approximately divisible. 
In a type III factor all projections are Murray--von Neumann
equivalent. Since $M_n(M)$ is a type III factor, we can conclude 
that $K_0(M)=0$. Similarly, as $M$ is a von Neumann algebra every unitary is equal to $\exp(i a)$ for a self-adjoint 
operator $a$ of norm $\leq \pi$. This property is \aea{} and it implies that $K_1(M)=0$.   

It remains to find a sentence in $\bfT_N$ not satisfied by $M$. 
By the above it suffices to show that the property ``If $A$ is purely infinite and simple then it is $\cO_\infty$-absorbing'' 
is elementary. 
Since purely infinite and simple \cstar-algebras form an elementary and co-elementary class (\S\ref{S.PI})
and  being $\cO_\infty$-absorbing is also elementary (Theorem~\ref{T.ssa}), the conclusion follows  
by Lemma~\ref{L.implication}. 
\end{proof}

All examples of \cstar-algebras not elementarily equivalent to a nuclear algebra
provided here are not $\cZ$-stable. 
 A family of $2^{\aleph_0}$ distinct theories of simple, monotracial, $\cZ$-stable (and even UHF-stable)  
\cstar-algebras not elementarily equivalent to a nuclear \cstar-algebra was found 
in  \cite{boutonnet2015ii_1}. 
This  is a consequence of the existence of $2^{\aleph_0}$ distinct theories of McDuff factors. 
Notably, all known proofs that a \cstar-algebra is not elementarily equivalent to a nuclear \cstar-algebra use 
Connes' theorem that the weak closure of a nuclear \cstar-algebra in each of its representations is an injective factor. 

\begin{proposition} The 
 set of complete theories of \cstar-algebras with a nuclear model is Borel. 
\end{proposition} 

\begin{proof} 
Since being nuclear is  \udt, 
a theory $T$ has a nuclear model if and only if the Borel condition (2) 
from  \cite[Theorem~6.3]{FaMa:Omitting} holds. 
\end{proof}

\begin{problem}  
 Is there a natural characterization of theories of \cstar-algebras that have nuclear models? 
\end{problem}

\chapter{The Cuntz semigroup} 
\label{S.elfun}
The Cuntz semigroup was introduced by Cuntz in the 1970s and recently rose to prominence as an  invariant of  \cstar-algebras (see e.g., \cite{APT} or \cite{CoElIv}). Although this semigroup does not belong to $A^{\eq}$, we shall see that some of its relevant first-order properties  are reflected 
in the theory of~$A$.

 \section{Cuntz subequivalence} 
\label{S.Cuntz} 
For positive $a$ and $b$ in $A$ we write 
$a\precsim b$ and say that $b$ \emph{Cuntz-dominates}\index{Cuntz-dominates} $a$ when $\inf_x \|a-xbx^*\|=0$, where 
the infimum is taken over \emph{all} elements of the \cstar-algebra, and not just contractions. 
The equivalence classes obtained by symmetrizing $\precsim$ on $\bigcup_n M_n(A)$ gives rise to 
the \emph{Cuntz semigroup},\index{Cuntz semigroup}
denoted $W(A)$,\index{W@$W(A)$}
  of $A$.  In this semigroup, the sum of the equivalence classes of $a$ and $b$ is defined to be the equivalence class
  of $\diag(a,b)$ (see \S\ref{S.K0}).  
  An alternative definition of the Cuntz semigroup of~$A$, denoted $\Cu(A)$,\index{Cu@$\Cu(A)$}  considers equivalence classes of positive elements 
  in  $A\otimes \cK$ instead of $\bigcup_n M_n(A)$.  
   The two structures are slightly different but
the distinction will not play a role here. 
The relation $\precsim$ 
is not compatible with the logic of metric structures. 
Since for every $a\geq 0$ and every $n$, we have 
$a\precsim \frac 1n a$, the closure of the graph of $\precsim$ 
includes $A_+\times \{0\}$, so this graph is not a closed subset of 
$A^2$ whenever $A$ is a nontrivial \cstar-algebra. 

 The following straightforward fact is worth mentioning. 
 
  \begin{lemma} \label{L.Cuntz.3} 
  For all $n\geq 1$ and 
    $a$ and $b$ in $A$ we have $(a\precsim b)^A$ if and only if $(a\precsim b)^{M_n(A)}$. 
      \end{lemma} 

\begin{proof} Let $e$ be a minimal projection in $M_n(\bbC)$ 
and identify $A$ with $eM_n(A)e$. 
If $x\in M_n(A)$ is such that $\|a-xbx^*\|<\e$ then 
$x_1:=exe$ belongs to $A$ and satisfies $a-x_1bx_1^*=a-xbx^*$. 
\end{proof} 

The following is well-known.

\begin{lemma} \label{L.Cuntz} 
Assume $a$ and $b$ are positive contractions in $A$. 
If  there exists $s$ such that $a=s^* (b-\delta)_+ s$ then there
exists $x$ with  $\|x\|\leq \delta^{-1/2}$ such that $a=x^*bx$.  
In particular, if $a\precsim (b-\delta)_+$ then for every $\e>0$ there is  
$x\in A$ such that $\|x\|\leq \delta^{-1/2}$ and $\|x^*bx-a\|<\e$. 
\end{lemma} 

\begin{proof} 
Say $a=s^* (b-\delta)_+ s$.  Set $y := (b-\delta)_+ ^ {1/2} s$. 
Then  $a=y^*y$ and therefore 
 $\|y\|\leq 1$.  Now choose a nonnegative function $f \in C_0((0,1])$ such that
$\| f \| \leq \delta^{-1}$
and
$f(t) = t^{-1}$, for $t \in  [\delta,1]$.
It follows that
$t f(t) (t-\delta)_+ = (t-\delta)_+$.
Set
$x := f(b)^{1/2} y$,
so that $\|x\| \leq  \delta^{-1/2}$ and
$x^*bx = s^* f(b) b (b-\delta)_+ s = s^*(b-\delta)_+s = a$.

For the second part, note that if $a \precsim (b-\delta)_+$ then for $\e>0$ there exists $x_0 \in A$ such that $x_0^*(b-\delta)_+x_0$ is within $\e$ of  $a$.  By rescaling, we may assume $a':=x_0^*(b-\delta)_+x_0$ is additionally a contraction. Then by the first part, there exists $x \in A$ such that $\|x\|\leq \delta^{-1/2}$ and $x^*bx=a'$ is within $\e$ of $a$, as required.
  \end{proof}

  \begin{lemma} \label{L.Cuntz.2} Assume $A\prec B$ 
    and $a$ and $b$ are in $(A\otimes \cK)_+$. 
  Then $(a\precsim b)^A$ if and only if $(a\precsim b)^B$. 
  \end{lemma} 
  
  \begin{proof} 
Clearly $A\subseteq B$ implies that   if $(a \precsim b)^A$ then $(a \precsim b)^B$.
Conversely, if $(a \precsim b)^B$ then by \cite[Proposition 2.4]{Rordam:UHFII}, for every $\e > 0$, there exists $\delta>0$ such that $((a-\e)_+ \precsim (b-\delta)_+)^B$.
By Lemma~\ref{L.Cuntz}, this is equivalent to the existence of $x \in M_n(B)$ of norm at most $\delta^{-1/2}$ such that $\|(a-\e)_+ - x^*bx\|$ is arbitrarily small.
Since the embedding is elementary, this is in turn equivalent to the existence of $x \in M_n(A)$ of norm at most $\delta^{-1/2}$ such that $\|(a-\e)_+-x^*bx\|$ is arbitrarily small, which implies that $(a-\e)_+ \precsim b$.
Since~$\e>0$ is arbitrary, it follows that $a \precsim b$.
\end{proof}

For $x\in A$ and $k\geq 1$ we shall write $x^{\oplus k}$ for the diagonal 
matrix $\diag(x,x,\dots, x)$ in $M_k(A)$
where  $x$ is repeated 
$k$ times.  
Similarly,  $x^{\oplus k}\oplus y^{\oplus l}$ stands for $\diag(x,\dots, x, y,\dots, y)$ in $M_{k+l}(A)$, 
where $x$ is repeated $k$ times and $y$ is repeated $l$ times.

Recall that by Lemma~\ref{L.Cuntz.3} the relation of Cuntz-subequivalence between elements 
of $M_m(A)$ is unchanged when computed in $M_n(A)$ for $n\geq m$. 

For a unital \cstar-algebra  $A$
and nonzero natural numbers $m$ and $n$ consider the following statement
\begin{enumerate}
\item [$\RC^A_{m,n}$:]
For all positive $x,y$ in $A$ 
we have that 
\[
x^{\oplus (n+1)}\oplus 1^{\oplus m}
\precsim y^{\oplus n}
\] 
in $M_{m+n+1}(A)$ implies $x\precsim y$ in $A$. 
\end{enumerate}
This implication is a first-order property of the Cuntz semigroup. 
(Since the positive elements in $\bigcup_n M_n(A)$ form a  dense subset  of the positive 
elements of $A\otimes \cK$, its truth does not depend on whether we consider $W(A)$ 
or $\Cu(A)$.  This also applies to properties of the Cuntz semigroup considered 
later on in this subsection.) 
Since the latter (typically being an uncountable 
discrete ordered semigroup) does not appear to be in $A^{\eq}$, the following 
 proposition that will be used in \S\ref{S.rc}
is not immediate.

\begin{prop}\label{L.rc.mn} 
For all $m$ and $n$ the class
   of all unital $A$ such that   $\RC^A_{m,n}$ holds 
   is elementary.  
\end{prop} 

\begin{proof} Fix $m$ and $n$. 
By Theorem~\ref{T.Ax} we need to show that this class (we temporarily denote it $\cC(m,n)$)
is closed under isomorphisms, ultraproducts, and elementary submodels. 
Closure under isomorphisms is obvious. 

Fix   $A\prec B$. 
Lemma \ref{L.Mn} implies 
 $M_{m+n+1}(A)$ is in $A^{\eq}$ 
and therefore    $M_{m+n+1}(A)\prec M_{m+n+1}(B)$.  
Therefore  by Lemma \ref{L.Cuntz.2}  any counterexample to $\RC^A_{m,n}$  is also a counterexample 
to $\RC^B_{m,n}$  and $\RC^B_{m,n}$ implies $\RC^A_{m,n}$.  
Therefore $B\in \cC(m,n)$ implies $A\in \cC(m,n)$ and we have proved that $\cC(m,n)$ 
is closed under elementary submodels. 

It remains to prove that $\cC(m,n)$ is closed under ultraproducts. 
Fix an index set $I$, an ultrafilter $\cU$ on $I$ and $A_\xi\in \cC(m,n)$ for all $\xi\in I$.

Assume for the sake of  contradiction  that  $\RC^A_{m,n}$ fails with  $A:=\prod_{\cU}A_\xi$. 
Choose witnesses  $a$ and $b$  in the unit ball of $A$. 
We  have 
$a^{\oplus (n+1)}\oplus1^{\oplus m}\precsim b^{\oplus n}$
 but not $a\precsim b$. 
Thus for some $\e>0$ we have  $(a-\e)_+\not\precsim b$. We may choose $\e<1$. 
Since for all $x,y$ and $\epsilon$ we have
\[
((x^{\oplus k} \oplus y^{\oplus l})-\e)_+=
 (x^{\oplus k}-\e)_+ \oplus   (y^{\oplus l}-\e)_+,
\] 
by \cite[Proposition 2.4]{Rordam:UHFII}, 
 there exists $\delta>0$ such that 
\[
(a^{\oplus (n+1)}-\e)_+ \oplus  (1^{\oplus m}-\e)_+
\precsim (b^{\oplus n}-\delta)_+
\]
 in $ M_{m+n+1}(A)$. 
Let $a_\xi$ and  $b_\xi$,    for $\xi\in I$, be representing sequences of $a$ and $b$ respectively. Since $1$ is a compact element (which is equivalent to being Cuntz equivalent to some $(1-\epsilon)_+$, for some $\epsilon > 0$) and $\e<1$,  
we have 
that 
$ (a^{\oplus (n+1)}_\xi-\e)_+ \oplus  1^{\oplus m}$ is Cuntz-equivalent to \\
$ (a^{\oplus (n+1)}_\xi-\e)_+ \oplus   (1^{\oplus m}-\e)_+$. 
 We now have 
   \[
(a^{\oplus (n+1)}_\xi-\e)_+ \oplus   1^{\oplus m}
   \precsim  (b^{\oplus n}_\xi-\delta/2)_+ 
     \] 
for $\cU$-many $\xi\in J$.
  Using $\RC^{A_\xi}_{m,n}$  
 for every such $\xi$ we have $(a_\xi-\e)_+\precsim (b_\xi-\delta/2)_+$, and by  
   Lemma \ref{L.Cuntz} there exists $x_\xi\in A_\xi$ 
   such that  $\|x_\xi\|\leq(\delta/2)^{-1/2}$ and  
  $(a_\xi-\e)_+=x_\xi b_\xi x_\xi^*$. Therefore in $A$  we have $(a-\e)_+\precsim b$, 
  a contradiction. 
  \end{proof}

The Cuntz semigroup of a \cstar-algebra $A$ 
is  said to be \emph{unperforated}\index{Cuntz semigroup!unperforated}  if for all $n\geq 1$ 
\begin{enumerate}
\item [$\UNP^A_{n}$:]
For all  positive $x$ and $y$ in $A\otimes \cK$
we have that $x^{\oplus n} \precsim y^{\oplus n}$ 
implies~$x\precsim y$.
\end{enumerate}

\begin{prop}\label{L.unp.n} 
The class of \cstar-algebras with unperforated Cuntz semigroup is elementary. 
 For every $n\geq 1$ the class
   of all $A$ such that   $\UNP^A_n$ holds 
   is elementary.  
\end{prop} 

\begin{proof} Since it suffices to consider $x$ and $y$ in $M_m(A)$
for a large enough $m$,  Lemma \ref{L.Mn} and  Lemma \ref{L.Cuntz.2}  
together imply that if $A\prec B$ and $\UNP^B_n$ holds then $\UNP^A_n$ holds.  
 In \cite[Lemma~2.3]{robert2013nuclear} 
it was proved that the class of \cstar-algebras with unperforated Cuntz semigroup is 
closed under products and quotients, and therefore closed under ultraproducts. 
The proof shows that $\UNP^A_n$ is preserved by ultraproducts for all $n\geq 1$.   

By Theorem~\ref{T.Ax} a class of \cstar-algebras closed under isomorphisms, 
elementary submodels, and 
ultraproducts is elementary. 
\end{proof}

The Cuntz semigroup of a \cstar-algebra $A$ 
is  said to be \emph{almost unperforated}\index{Cuntz semigroup!almost 
unperforated}  if for all $n\geq 1$ 
\begin{enumerate}
\item [$\AUNP^A_{n}$:]
For all  positive $x$ and $y$ in $A\otimes \cK$
we have that $x^{\oplus n+1} \precsim y^{\oplus n}$ 
implies~$x\precsim y$.
\end{enumerate}
The proof of the following result used in \S\ref{S.elfunsc}
 is analogous to the proof of Proposition~\ref{L.unp.n}. 

\begin{prop}\label{L.aunp.n} 
The class of \cstar-algebras with almost unperforated Cuntz semigroup is elementary. 
 For $n$ the class
   of all $A$ such that   $\AUNP^A_n$ holds 
   is elementary.  \qed
\end{prop}

The proof of Lemma~\ref{L.rc.mn} appears to be flexible enough to 
 show that, although $W(A)$ does not belong to $A^{\eq}$, 
  numerous properties $P$ of $W(A)$  are `elementary' 
 in the sense that the class $\{A: W(A)$ satisfies~$P\}$ is elementary. 
 This applies for example to 
 having $n$-comparison, 
 \index{Cuntz semigroup!@n!$n$-comparison} 
being $n$-divisible,\index{Cuntz semigroup!@n!$n$-divisible}
 (\cite[\S 2.3]{RobertTikuisis}),  
or being $(m,n)$-pure (\cite[Definition~2.6]{winter2012nuclear}). We remind the reader of Question \ref{Q1} from the introduction.
\begin{question*}
Assume $A$ and $B$ are  two elementarily equivalent, simple, separable, nuclear, unital \cstar-algebras
with the same Elliott invariant. Are $A$ and $B$ 
necessarily isomorphic? 
\end{question*} 
While one might imagine an approach to this question using the Cuntz semigroup as the distinguishing invariant, this discussion suggests that providing a negative answer would  
require a novel construction of \cstar-algebras. 
It would be interesting to have a characterization of properties of $W(A)$ that are 
 `elementary' in the above sense.

\section{Strict comparison of positive elements}\label{S.elfunsc}
Let $A$ be a \cstar-algebra.  A \emph{2-quasitrace}\index{2-quasitrace} on $A$ is a map $\tau\colon (A\otimes \cK)_+\to [0,\infty]$  that is linear on pairs of commuting positive elements, maps 0 to 0, and satisfies that $\tau(x^*x)=\tau(xx^*)$ for all $x\in A\otimes \cK$ (\cite[Definition II.1.1]{BH82} and \cite[Section 4]{ERS}; the apparently 
different definitions are equivalent by a result of Blanchard and Kirchberg, 
see \cite[p. 984]{ERS}). For one such $\tau$  one defines the \emph{dimension function}\index{dimension function}
 $d_\tau\colon (A\otimes \cK)_+\to [0,\infty]$ by 
\[
d_\tau(a):=\lim_n \tau(a^{1/n}). 
\]
It is not  difficult to check that  $a\precsim b$ implies $d_\tau(a)\leq d_\tau(b)$. 
Moreover, by  \cite[Theorem~II.2.2]{BH82},
2-quasitraces on $A$ correspond to `lower semicontinuous rank functions,' 
i.e., functionals on $W(A)$ (in \cite{BH82} $K_0^*(A)$ denotes the Grothendieck group 
associated to $W(A)$). 

A \cstar-algebra is said to have the property of \emph{strict comparison
of positive elements by 2-quasitraces}\index{strict comparison} 
 if for all $a,b\in (A\otimes \cK)_+$,
if for some $\epsilon>0$, $d_\tau(a)\leq (1-\epsilon)d_\tau(b)$ for all lower semicontinuous, $[0,\infty]$-valued, 2-quasitraces $\tau$ then 
$a\precsim b$. 

\begin{theorem}\label{T.SCQT} 
The class of \cstar-algebras with strict comparison of 
positive elements by 2-quasitraces is an \aea\  class.
\end{theorem}

\begin{proof} 
By \cite[Proposition 6.2]{ERS}, strict comparison by 2-quasitraces is equivalent to the property of almost unperforation in the Cuntz semigroup of $A$.
By Proposition~\ref{L.aunp.n} 
the class of \cstar-algebras with almost unperforated Cuntz semigroups---equivalently, with strict comparison of positive elements by 2-quasitraces---is elementary. 

It is clear that the union of a chain of \cstar algebras with strict comparison by 2-quasitraces also has strict comparison. By Proposition~\ref{P.Ax}  we conclude  that this class is \aea.
\end{proof}

If in the above definition of strict comparison 
we let $\tau$ range through traces rather than 2-quasitraces  then we 
say that $A$ has \emph{strict comparison of positive elements by traces}.
That is, strict comparison of positive elements by traces holds if  we have that  for all $a,b\in (A\otimes \cK)_+$ 
if for some $\epsilon>0$, $d_\tau(a)\leq (1-\epsilon)d_\tau(b)$ for all lower semicontinuous, $[0,\infty]$-valued,  traces $\tau$ then 
$a\precsim b$. 

\begin{theorem}
The class of \cstar-algebras with strict comparison of positive elements by traces is an \aea\  class.
\end{theorem}

\begin{proof}  
We first show that the class of \cstar-algebras with the strict comparison of positive elements is elementary using Theorem~\ref{T.Ax}.
Suppose that $A\prec B$ and~$B$ has strict comparison of positive elements by traces.  If we have $a,b\in (A\otimes \cK)_+$ 
and $\epsilon>0$ such that $d_\tau(a)\leq (1-\epsilon)d_\tau(b)$ for all lower semicontinuous, $[0,\infty]$-valued, traces $\tau$ on $A$ then since the restriction of a trace on $B$ to $A$ is a trace on $A$, we see that we have the same on~$B$ and since $B$ has strict comparison, $a \precsim b$ in $B$.  But by Lemma~\ref{L.Cuntz.2}, we have $a \precsim b$ in $A$ and so $A$ has strict comparison by traces.
It is shown in the course of the proof of \cite[Theorem 4.1]{NgRobert} that an ultraproduct of \cstar-algebras with strict comparison of positive elements by traces
again has strict comparison of positive elements by traces. 
By Theorem~\ref{T.Ax}, this class is elementary.

As before, it is clear that this class is closed under  unions of chains.
By Proposition~\ref{P.Ax}   the \cstar-algebras with this property  also form an \aea\ class.
\end{proof}

For  exact (and in particular nuclear) \cstar-algebras strict comparison by 2-quasitraces is equivalent to strict comparison 
by traces because on exact \cstar-algebras all quasitraces are traces 
by \cite{haagerup2014quasitraces} (see also \S\ref{S.Exact}).  
It is not known whether this is true in general (see  \cite[II.6.8.16]{Black:Operator}).

\section{The Toms--Winter conjecture} 
\index{Toms--Winter conjecture} \label{S.TWC} 
Three prominent and diverse regularity properties 
are conjecturally equivalent in the case of
 \snus{} \cstar-algebras. 

\begin{conjecture}[Toms--Winter] \label{C.TW} 
Suppose $A$ is a \snus{} \cstar-algebra not isomorphic to a 
matrix algebra. Then the following are equivalent: 
\begin{enumerate}
\item $A$ has finite nuclear dimension (see \S\ref{S.dimnuc}). 
\item $A$ is $\cZ$-stable. 
\item $A$ has strict comparison (see  \S\ref{S.elfunsc}). 
\end{enumerate}
\end{conjecture} 

The implications  from (1) to (2) and from (2) to (3) have been established 
in \cite{Winter:drSH} and \cite{winter2012nuclear} and  the implications 
from (3) to (2) and from (2) to (1) have  been established in many cases (see 
 the introduction to~\cite{SaWhWi:Nuclear}, as well as~\cite{bosa2014covering}).

 Being $\cZ$-stable is \aea{}  (see \S\ref{Ex.ssa-stable}) and having strict comparison 
 is also \aea{} (Theorem~\ref{T.SCQT}). On the other hand, having finite nuclear dimension 
 is not even axiomatizable, although having nuclear dimension $\leq n$ for any fixed $n$ is 
 \udt{} (Theorem~\ref{T1}). Therefore every implication in     Conjecture~\ref{C.TW} can be restated 
in purely model-theoretic terms. 

Regularity properties of a \cstar-algebra $A$
affect the definability and structure of $A^{\eq}$. 
For example, by  \cite[Corollary 4.3]{NgRobert} (see Theorem~\ref{T.tau-definable.0}) 
  if $A$ has strict comparison of positive elements  by traces then $A_0$
is definable. 
By   \cite[Theorem 6]{Oza:Dixmier}, if $A$ is \emph{exact} and $\cZ$-stable 
then $A_0$
is definable. (`Exact' can be weakened to `elementarily equivalent to an exact \cstar-algebra'
by the argument given in the proof of  Proposition~\ref{P.Exact}.) 
Finally, if $A$ is \emph{nuclear} and has finite nuclear dimension then it is  
$\cZ$-stable by \cite{winter2012nuclear} and therefore $A_0$ is definable by the above.

The Toms--Winter 
conjecture in particular predicts that no simple nuclear \cstar-algebra of infinite nuclear dimension can be elementarily equivalent to a simple \cstar-algebra of finite nuclear dimension. 
Presently evidence suggests that the nuclear dimension of a $\cZ$-stable \cstar-algebra is always either 0,1, or $\infty$.  
One could also ask: for which $n$ is there a \cstar-algebra that has nuclear dimension (or decomposition rank) 
$n+1$ which is not elementarily equivalent to any \cstar-algebra of smaller nuclear dimension (decomposition rank, respectively)? 
Here are two partial answers to this question:
\begin{enumerate}[leftmargin=*,topsep=6pt]
\item For $n=0$ one easily has an affirmative answer to this question. All Kirchberg algebras, as well as  $\cZ$,   have 
nuclear dimension one (by \cite{bosa2014covering}) and their theories have no models with nuclear dimension zero. 
This follows from three facts. 

First, a \cstar-algebra has nuclear dimension 0 if and only if it is AF.
AF algebras have stable rank one and real rank zero and both of these are axiomatizable properties by \S\ref{S.sr}  and \S\ref{S.rr0.revisited}. Second, Kirchberg algebras, being infinite, do not have stable rank one, and finally $\mathcal Z$ does not have real rank zero since it does not have a nontrivial projection.


\item We have a full answer in the abelian case.
The class of unital abelian \cstar-algebras is axiomatizable, and for a unital separable abelian \cstar-algebra $\mathrm{C}(X)$ the nuclear dimension and decomposition rank both coincide with the dimension of $X$.  We know by Proposition \ref{P.RR}, if $\mathrm{C}(X)$ is elementarily equivalent to $\mathrm{C}(Y)$ then $\mbox{rr}(\mathrm{C}(X)) =\mbox{rr}(\mathrm{C}(Y))$ and by \cite{brown1991c}, $\mbox{rr}(\mathrm{C}(X)) =\dim(X)$.  This was also noticed in a different context by Bankston, \cite[Theorem 3.2.4]{Bank}. 
\end{enumerate}

\section{Radius of comparison} 
\label{S.rc} 
In \cite{To:Infinite} and \cite{To:Comparison}
an infinite and then uncountable family of nonisomorphic \snus{} \cstar-algebras with the
 same Elliott
invariant was constructed. The distinguishing invariant for these algebras was the radius of comparison defined below. We shall prove these algebras can be distinguished by their theories. 

The \emph{radius of comparison},\index{radius of comparison}
 $\rc(A)$, of a unital \cstar-algebra $A$ is the infimum of all 
 $r\geq 0$ 
such that for all positive $a$ and $b$ in $\bigcup_n M_n(A)$ 
if $d_\tau(a)+r<d_\tau(b)$ for all 2-quasitracial states $\tau$ 
 then $a\precsim b$. If there is no such $r$ then $\rc(A)=\infty$.

\begin{thm} \label{P.rc} 
For every $r\geq 0$ each of the following is an elementary class:
\begin{enumerate}
\item The class of \cstar-algebras with radius of comparison $\leq r$.
\item The class of \cstar-algebras with radius of comparison $\geq r$.
\item The class of \cstar-algebras with radius of comparison $r$.
\end{enumerate}
\end{thm}

\begin{proof} We shall explain why this follows from Proposition~\ref{L.rc.mn} and known results.  

In   \cite[\S 3.2]{BlRoTiToWi} the assertion $\rc(A)\leq r$
was denoted by {\bf R1}$(r)$ and the assertion 
``$\RC^{M_k(A)}_{m,n}$ holds whenever $r<m/n$, for all $k$'' was denoted by {\bf R2}$(r)$. 
It is easy to see that {\bf R1}$(r)$ implies {\bf R2}$(r)$ and 
in    \cite[Proposition 3.2.1]{BlRoTiToWi} it was proved that 
{\bf R2}$(r)$ implies {\bf R1}$(r+\e)$ for all $\e>0$. 

This result and Proposition~\ref{L.rc.mn} imply the desired conclusion. 
 \end{proof} 
 
 \begin{corollary}
 The examples in \cite{To:Infinite} and \cite{To:Comparison} have the same Elliott invariant but different theories.
 \end{corollary}
 
 \begin{proof} These examples were constructed to have the same Elliott invariant. They were distinguished by their radius of comparison, and by the previous theorem, this is part of their theory.
 \end{proof}

\appendix

\chapter{\cstar-algebras} 
A Banach algebra is a complex algebra endowed with a complete norm 
satisfying $\left\Vert xy\right\Vert \leq \left\Vert x\right\Vert \left\Vert
y\right\Vert $. A \emph{\cstar-algebra}\index{C@\cstar-algebra}
 is a Banach algebra endowed moreover with an
involution $x\mapsto x^{\ast }$ satisfying the \cstar-identity $\left\Vert
x^{\ast }x\right\Vert =\left\Vert x\right\Vert ^{2}$. There are many good
textbooks and monographs on the theory of \cstar-algebras such as \cite{Black:Operator,BrOz:cstar,Murphy,Dav:cstar,Pedersen,Take1,Take2,Take3}. For convenience we will only refer to \cite{Black:Operator} in what follows.

A natural notion of morphism for the category of \cstar-algebras is that of
a \emph{$^*$-homo\-morphism}\index{S@$^*$-homomorphism}. If $A,B$ are \cstar-algebras, then a linear map $\phi :A\rightarrow
B$ is a  $^*$-homomorphism
 if it is \emph{multiplicative}, i.e., $\phi \left(
ab\right) =\phi \left( a\right) \phi \left( b\right) $ for every $a,b\in A$ and is {\em self-adjoint}\index{linear map!self-adjont}, i.e., $\phi(a^*) = \phi(a)^*$ for all $a \in A$.
A $^*$-homomorphism $\phi $ is necessarily a contraction, i.e., it satisfies 
$\left\Vert \phi \left( a\right) \right\Vert \leq \left\Vert a\right\Vert $ 
\cite[II.1.6.6]{Black:Operator}. Moreover it is an isometry if and only if
it is injective. A bijective $^*$-homomorphism is called a $^*$-isomorphism.

Suppose that $H$ is a Hilbert space, and denote by $B( H) $\index{B@$B(H)$} the
space of bounded operators on $H$. If $x\in B\left( H\right) $, let 
$\left\Vert x\right\Vert $ be the operator norm of $x$, and $x^{\ast }$ be
the operator on $H$ implicitly defined by $\left\langle x^{\ast }\xi ,\eta
\right\rangle =\left\langle \xi ,x\eta \right\rangle $ for $\xi ,\eta \in H$. 
This defines on $B\left( H\right) $ a \cstar-algebra structure. In particular
if $H$ is a Hilbert space of finite dimension $n$, then $B\left( H\right) $
is a \cstar-algebra that can be identified with the algebra $M_{n}(\bbC)$\index{M@$M_n(\bbC)$}
 of $n\times
n $ complex matrices. Any finite-dimensional \cstar-algebra is $^*$-isomorphic to a
direct sum of algebras of this form \cite[II.8.3.2]{Black:Operator}.

 A fundamental result of Gelfand--Naimark and Segal
asserts that any \cstar-algebra is
$^*$-isomorphic to a norm-closed self-adjoint subalgebra of $B(H)$ for some Hilbert space $H$. 
A \emph{representation}\index{representation}
 of a \cstar-algebra $A$ on a Hilbert space $H$ is a
$^*$-homo\-morphism $\pi :A\rightarrow B\left( H\right) $. The dimension of $\pi $
is the dimension of the Hilbert space $H$. A representation on a $1$-dimensional Hilbert space is called a \emph{character}.\index{character} A representation is 
\emph{faithful}\index{representation!faithful} if it is injective or, equivalently, isometric. 
It is \emph{irreducible}\index{representation!irreducible} 
if there is no nontrivial subspace of $H$ that is $\pi(a) $-invariant for every $a\in A$. 
A GNS representation is irreducible if and only if the corresponding state is pure. 
A \cstar-algebra is \emph{simple}\index{C@\cstar-algebra!simple} if it has no 
 two-sided, self-adjoint, norm-closed ideals. Equivalently, $A$ is simple if all of its representations are faithful.

If $T$ is a locally compact Hausdorff space, then the space $C_{0}(
T) $ of complex-valued continuous functions on $T$ that vanish at
infinity is a \cstar-algebra with pointwise operations and the
uniform norm. \cstar-algebras of the form $C_{0}( T) $ are
precisely the commutative \cstar-algebras. If $T$ is compact, then $C_{0}(
T) $ is {\em unital}\index{C@\cstar-algebra!unital}  (has a multiplicative identity) and simply denoted by $C( T) $. More
generally if $B$ is a \cstar-algebra one can define the algebra $C_{0}(
T,B) $ of $B$-valued continuous functions on $T$ that vanish at
infinity.

The taxonomy of elements of \cstar-algebras is imported wholesale 
from operator theory. $x$ is \emph{normal}\index{normal} if it commutes with
its adjoint.
An element $x$ of a \cstar-algebra $A$ is \emph{self-adjoint}\index{self-adjoint} if it equals its adjoint.
It is \emph{positive}\index{positive} if it is of the form $b^{\ast }b$ for some 
$b\in A$.   A \emph{projection}\index{projection} $p$ of $A$ is a self-adjoint idempotent
element, i.e., it satisfies $p=p^{\ast }=p^{2}$. 
A \emph{unit }$1$ of a
\cstar-algebra is a multiplicative identity which is necessarily unique.
An element $u$ of a unital
\cstar-algebra is a \emph{unitary}\index{unitary}
 provided that $uu^{\ast }=u^{\ast }u=1$.
 An element $v$ of a \cstar-algebra is an \emph{isometry}\index{isometry}  if $v^*v=1$. 
 $v$ is a  \emph{partial isometry}\index{partial isometry}  if $v^*v$ is a projection. 
 If $v$ is a partial isometry then $vv^*$ is necessarily a projection. Two projections $p$ and $q$ in a \cstar-algebra $A$
 are \emph{Murray--von Neumann equivalent}\index{Murray--von Neumann equivalence} 
  if there exists a partial isometry $v \in A$ such that $v^*v=p$ and $vv^*=q$.

The spectral theorem for normal operators and continuous 
functional calculus (see \S\ref{S.cfc}) 
provide a powerful insight into the behaviour of normal elements of a \cstar-algebra.  
If $A$ is a \cstar-algebra and $x\in A$, then the \emph{spectrum}\index{spectrum} 
$\sigma \left(
x\right) $ of $x$ is the set of $\lambda \in \mathbb{C}$ such that $\lambda
-x$ is not invertible in $A$ (or in the unitization of $A$ in the absence of a unit). The spectrum is always a nonempty compact
subset of $\mathbb{C}$. The spectral theorem 
easily implies that a normal element $x$ of a \cstar-algebra $A$ is self-adjoint
(respectively positive, unitary, a projection) if and only if $\sigma \left(
x\right) $ is contained in $\mathbb{R}$ (respectively $\left[ 0,+\infty
\right) $, $\mathbb{T}$, $\left\{ 0,1\right\} $); see \cite[II.2.3.4 and
II.3.1.2]{Black:Operator}.

A linear map $\phi $ between \cstar-algebras is
\emph{positive}\index{positive!linear map} if it maps positive elements to positive elements. In
particular a linear functional $\phi $ on $A$ is positive if $\phi \left(
a\right) \geq 0$ for every positive element $a$ of $A$. Since every
self-adjoint element can be written as difference of two positive elements,
a positive map is automatically self-adjoint. A \emph{state}\index{state} on $A$ is a
positive linear functional of norm $1$. The set $S(A) $\index{S@$S(A)$} of
states of $A$ is a $G_{\delta }$ subset of the unit ball of the dual 
$A^{\ast }$ of $A$ endowed with the weak$^*$ topology. If $A$ is unital, then 
$\phi \in A^{\ast }$ is a state if and only if $1=\phi ( 1)
=\left\Vert \phi \right\Vert $ \cite[II.6.2.5]{Black:Operator}. Moreover, 
in the unital case, $S(A) $ is weak$^*$-compact. 
The extremal points of $S(A)$ are the \emph{pure states}\index{state!pure}
and the set of pure states is denoted by $P(A)$. 
A state is \emph{faithful} if $\phi(a^*a)=0$ implies $a=0$.\index{state!faithful} 

One can assign to any state $\phi $ of $A$ a representation $\pi _{\phi }$
of $A$ on a Hilbert space $H_{\phi }$ through the so called \emph{GNS
construction}\index{GNS construction}  \cite[II.6.4]{Black:Operator}. The Hilbert space $H_{\phi }$
is the  completion of $A$ with respect to the pre-inner product $\left\langle a,b\right\rangle =\phi \left( b^{\ast }a\right) $. The
representation $\pi _{\phi }$ maps $a$ to the operator
of left-multiplication by $a$ on $H_{\phi }$. This simple yet fundamental construction is the key
idea in the proof that every \cstar-algebra admits a faithful representation on a
Hilbert space. 

A \emph{trace}\index{trace} (or \emph{tracial state})\index{state!tracial}
of a unital \cstar-algebra is a state $\tau $ satisfying the \emph{trace
identity} $\tau(ab) =\tau(ba) $. A 
trace is \emph{faithful}\index{trace!faithful} if $\tau ( a^{\ast }a) =0$ implies $a=0$. The set $T(A) $\index{T@$T(A)$} of traces on $A$ forms a (possibly empty) \emph{Choquet
simplex}, which is metrizable if $A$ is separable \cite[II.4.4]{BH82}. A
\cstar-algebra is called \emph{tracial}\index{C@\cstar-algebra!tracial} 
if $T( A) $ is nonempty, and 
\emph{monotracial}\index{C@\cstar-algebra!monotracial} if $T(A) $ is a singleton.

We say a few words about the category of \cstar-algebras with $^*$-homo\-morphisms as morphisms.
If $\left( A_{i}\right) _{i\in I}$ is a collection of \cstar-algebras, then the
direct product $\prod_{i}A_{i}$ is the subset of the Cartesian product
consisting of indexed families $\left( a_{i}\right) $ such that $\sup_{i}\left\Vert
a_{i}\right\Vert <+\infty $ endowed with the norm $\left\Vert \left(
a_{i}\right) \right\Vert =\sup_{i}\left\Vert a_{i}\right\Vert $. The direct
sum $\bigoplus_{i}A_{i}$ is the closure inside $\prod_{i}A_{i}$ of the set
of finitely supported families. It is not difficult to verify that these
constructions are the product and the coproduct in the category of
\cstar-algebras.

The category of \cstar-algebras is also closed under {\em tensor products}.
There are in general many ways to define tensor products between two
\cstar-algebras $A$ and $B$. One canonical choice is to faithfully represent $A,B $ on a Hilbert space and then consider the norm-closure inside $B\left(
H\otimes H\right) $ of the algebraic tensor product $A\odot B$. (Here $H\otimes H$ denotes the unique tensor product of Hilbert spaces.) This construction
yields the \emph{minimal} or \emph{spatial} tensor product $A\otimes B$ (sometimes
also denoted $A\otimes_{\min} B$)\index{tensor product!minimal (spatial)}
  \cite[II.9.1.3]{Black:Operator}. There are in general other ways to equip the algebraic tensor product of $A$ and $B$ with a \cstar-norm (see \cite[\S 3]{BrOz:cstar}). 
  If $A$ is unital then $B$ can be identified with the subalgebra $\{1\}\otimes B$ of $A\otimes B$. 
  One can therefore define, by taking a direct limit, a tensor product of an infinite family of unital \cstar-algebras. 
  Infinite tensor products of matrix algebras $M_n(\bbC)$ are \emph{uniformly hyperfinite (UHF)} algebras.\index{C@\cstar-algebra!UHF}

Since $^*$-homomorphisms between \cstar-algebras are always contractive, 
 the category of \cstar-algebras admits
direct limits\index{direct limit} (or \emph{inductive limits})\index{inductive limit} 
  \cite[II.8.2.1]{Black:Operator}.
These are explicitly constructed as follows.
Suppose that $\left( A_{i}\right) _{i\in I}$ is an inductive system with
connecting maps $\phi _{i,j}:A_{i}\rightarrow A_{j}$. There is a canonical
\cstar-seminorm on the algebraic direct limit obtained by setting%
\begin{equation*}
\| a\| :=\lim_{j}\| \phi _{i,j}\left( a\right)\| 
\end{equation*}
for $a \in A_i$.
The corresponding  completion is a \cstar-algebra, which is the direct
limit of the inductive system $\left( A_{i}\right) _{i\in I}$ in the
category of \cstar-algebras \cite[II.8.2]{Black:Operator}. Many important
classes of \cstar-algebras are defined in terms of direct limits. Thus a
\cstar-algebra is \emph{approximately 
finite-dimensional} (AF)\index{C@\cstar-algebra!approximately finite (AF)}  
 if it is the
direct limit of finite-dimensional \cstar-algebras 
and  \emph{approximately homogeneous} 
(AH)\index{C@\cstar-algebra!approximately homogeneous (AH)}  \cite[II.8.2.2(iv)]{Black:Operator},
 if it is the direct
limit of homogeneous \cstar-algebras.

\backmatter
\bibliographystyle{plain}
\bibliography{mt-nucUpdated}

\printindex

\end{document}